\documentclass{ucthesisnew}

\usepackage{amsmath, amsfonts, amsthm, amssymb}
\usepackage[all]{xy}
\usepackage{graphicx}

\makeatletter
\newtheorem*{rep@theorem}{\rep@title}
\newcommand{\newreptheorem}[2]{%
\newenvironment{rep#1}[1]{%
 \def\rep@title{#2 \ref{##1}}%
 \begin{rep@theorem}}%
 {\end{rep@theorem}}}
\makeatother

\theoremstyle{plain}
\newtheorem{theorem}{Theorem}[section]
\newreptheorem{theorem}{Theorem}
\newtheorem{lemma}[theorem]{Lemma}
\newtheorem{proposition}[theorem]{Proposition}

\newtheorem{corollary}[theorem]{Corollary}

\newtheorem{question}[theorem]{Question}

\newtheorem{reduction}[theorem]{Reduction}

\makeatletter
\@addtoreset{theorem}{chapter}
\makeatother

\theoremstyle{definition}
\newtheorem{definition}[theorem]{Definition}

\theoremstyle{remark}
\newtheorem{remark}[theorem]{Remark}
\newtheorem{example}[theorem]{Example}

\newcommand{\SC}{\mbox{ ; }}

\newcommand{\for}{,\mbox{\hspace{4pt} for  }}

\newcommand{\isom}{\simeq}
\newcommand{\inj}{\hookrightarrow}

\newcommand{\dd}{{\rm d}}
\newcommand{\brac}{\{\cdot,\cdot\}}

\newcommand{\TT}{\mathcal{T}}
\newcommand{\JJ}{\mathcal{J}}
\newcommand{\II}{\mathcal{I}}
\newcommand{\OO}{\mathcal{O}}

\newcommand{\KK}{\mathbb{K}}
\newcommand{\NN}{\mathbb{N}}
\newcommand{\ZZ}{\mathbb{Z}}

\newcommand{\RR}{\mathbb{R}}

\newcommand{\CC}{\mathbb{C}}
\newcommand{\bbA}{\mathbb{A}}

\newcommand{\g}{\mathfrak{g}}
\newcommand{\gl}{\mathfrak{gl}_2}
\newcommand{\slt}{\mathfrak{sl}_2}
\newcommand{\m}{\mathfrak{m}}
\newcommand{\h}{\mathfrak{h}}

\newcommand{\zb}{\overline{z}}
\newcommand{\wb}{\overline{w}}

\newcommand{\pt}{\{0\}}
\newcommand{\half}{\frac{1}{2}}

\newcommand{\ad}{\rm{ad}}
\newcommand{\Span}{{\rm Span}}
\newcommand{\Sym}{{\rm Sym}}

\newcommand{\Der}{{\rm Der}}
\newcommand{\End}{{\rm End}}
\newcommand{\Spec}{{\rm Spec}}

\newcommand{\pderiv}[2]{\frac{\partial #1}{\partial #2}}
\newcommand{\vect}[1]{\frac{\partial}{\partial #1}}
\newcommand{\bivect}[2]{\frac{\partial}{\partial #1}\wedge\frac{\partial}{\partial #2}}

\newcommand{\nbd}[3]{\mathcal{N}^{#1}_{#2}(#3)}

\begin{document}


\title{On Embedding Singular Poisson Spaces}
\author{Aaron Fraenkel McMillan}
\degreesemester{Spring}
\degreeyear{2011}
\degree{Doctor of Philosophy}
\chair{Professor Alan Weinstein}
\othermembers{Professor Denis Auroux  \\
   Professor Venkat Anantharam}
\numberofmembers{3}
\prevdegrees{B.A. (University of California, Berkeley) 2004} 
\field{Mathematics}
\campus{Berkeley}

\maketitle


\begin{abstract}

This dissertation investigates the problem of locally embedding singular Poisson spaces.  Specifically, it seeks to understand when a singular symplectic quotient $V/G$ of a symplectic vector space $V$ by a group $G\subseteq{\rm Sp}_{2n}(\RR)$ is realizable as a Poisson subspace of some Poisson manifold $(\RR^n,\brac)$.

The local embedding problem is recast in the language of schemes and reinterpreted as a problem of extending the Poisson bracket to infinitesimal neighborhoods of an embedded singular space.  Such extensions of a Poisson bracket near a singular point $p$ of $V/G$ are then related to the cohomology and representation theory of the cotangent Lie algebra at $p$.  

Using this framework, it is shown that the real 4-dimensional quotient $V/\ZZ_n$ ($n$ odd) is not realizable as a Poisson subspace of any $(\RR^{2n+6},\brac)$, even though the underlying variety algebraically embeds into $\RR^{2n+6}$.  The proof of this nonembedding result hinges on a refinement of the Levi decomposition for Poisson manifolds to partially linearize any extension with respect to the Levi decomposition of the cotangent Lie algebra of $V/G$ at the origin.  Moreover, in the case $n=3$, this nonembedding result is complemented by a concrete realization of $V/\ZZ_3$ as a Poisson subspace of $\RR^{78}$.

\end{abstract}

\begin{frontmatter}



\tableofcontents




\end{frontmatter}

\pagestyle{headings}


\chapter{Introduction}

This dissertation studies the local Poisson geometry of singular symplectic quotients.  It is particularly concerned with the problem of realizing such a quotient as an embedded Poisson subspace of some Poisson structure on $\RR^d$.  Such an embedding then realizes the quotient as a union of singular leaves of the symplectic foliation of the Poisson structure on $\RR^d$.  

The main results of this dissertation concern the singular symplectic quotient $V/\ZZ_n$ of a $4$-dimensional symplectic real vector space $V$ by the linear action of a cyclic group $\ZZ_n\subseteq{\rm Sp}_{4}(\RR)$ of odd order.   First, the Poisson structures on these quotients are shown to exhibit an obstruction to extension:

\begin{theorem} 
\label{thm:NonEmb}
Let $n$ be an odd integer greater than 2.  The symplectic quotient $V/\ZZ_n$ embeds into $\RR^{2n+6}$, yet is {\it not} realizable as a Poisson subspace of $\RR^{2n+6}$ for {\it any} Poisson structure on $\RR^{2n+6}$.
\end{theorem}

\begin{remark} The requirement that $n$ be odd is likely unnecessary; the proof of the theorem requires this assumption at only one step.
\end{remark}

This result provides further evidence for such rigidity of Poisson structures previously established by Egilsson \cite{Egilsson95} and Davis \cite{Davis01} (see Example \ref{ex:ResSpaces}).  

Second, a procedure is given that constructs an explicit realization of $V/\ZZ_3$ as a Poisson subspace of some $\RR^{N}$.  More specifically,

\begin{theorem} 
\label{thm:PoissExt}
There is a Poisson structure on $\RR^{78}$ and an embedding $V/\ZZ_3\inj\RR^{78}$ that realizes $V/\ZZ_3$ as a Poisson subspace of $\RR^{78}$.
\end{theorem}

In conjunction, these two theorems paint a picture of a singular symplectic space that (1) smoothly embeds in $\RR^{12}$, (2) whose Poisson structure cannot be extended to a Poisson structure on $\RR^{12}$, yet (3) can be extended to a Poisson structure on $\RR^{78}$.  It is unclear why this dimensional obstruction disappears between 13 and 78.  However, it seems to be a local property: in all known examples, the obstruction to embedding such singular quotients appears when attempting to infinitesimally extend the Poisson structure at a zero-dimensional symplectic leaf (a singular point with respect to both the smooth and Poisson structures). \\

While the motivating examples and corresponding results of this dissertation lie in the category of smooth manifolds, the techniques developed to handle them are algebraic in nature.  Why would one answer questions about quotients of {\it smooth symplectic manifolds} by studying {\it algebraic/formal Poisson structures}?  The more general category of Poisson structures is better equipped for studying such quotients for a number of reasons:

\begin{enumerate}
\item The symplectic quotients considered have isolated singularities. While a symplectic structure is defined in terms of a smooth structure on the space, Poisson structures have no such requirement for smoothness. 
\item Every smooth Poisson structure has a (possibly singular) symplectic foliation.  Thus, one may attempt to ``desingularize'' a symplectic quotient by realizing it as the union of (a finite number of) symplectic leaves of a smooth Poisson manifold.  The ``desingularization'' is then a nearby, smooth symplectic leaf. 
\end{enumerate}

\begin{remark} There is another notion of the desingularization of singular symplectic quotients, dual to the one above, called the ``symplectic resolution of singularities.''  This approach to understanding symplectic singularities allows one to work exclusively with symplectic structures (though there is still value in considering their corresponding Poisson geometry --- see \cite{Kaledin06}).
\end{remark}

Our method for either finding, or proving the nonexistence of, a Poisson embedding is roughly as follows (see Chapter 4):  Given a singular symplectic quotient $X=V/G$ and an embedding $X \inj \KK^d$ for $\KK=\RR$ or $\CC$, we attempt to extend the Poisson structure on $X$ locally near the zero-dimensional symplectic leaf $\pt\subseteq \KK^d$.  Such an extension would be a Poisson bracket on the algebra of formal functions $\KK[[x_1,\ldots, x_d]]$ vanishing at $\pt$.  The necessary and sufficient conditions (Theorem \ref{prop:EmbeddingCriteria}) for a local extension at $\pt$ are obtained in terms of
\begin{enumerate}
\item a sequence of conditions given in terms of the decomposition of the bracket by degree in $\KK[[x_1,\ldots,x_d]]$, and 
\item a local invariant (the cotangent Lie algebra) of the Poisson structure on $X/G$ at the singular point. 
\end{enumerate}

While the characterization above describes a (local) embedding of algebraic varieties $V/G\inj\KK^d$, it is perhaps better to interpret it in the language of schemes.  From this perspective, solving the degree $k$ condition amounts to extending the Poisson bracket on $X$ to a Poisson bracket on the $k^{\rm th}$ infinitesimal neighborhood of $X$ (which is not a variety, but is a scheme).  The condition for each such extension is naturally expressed as the solution of a Maurer-Cartan type equation on the cotangent Lie algebra.

When working with examples (a large part of this dissertation) we will take advantage of the computational power of the first point of view, then reflect on the procedure and the result in the language of schemes.

\section{Outline}

\subsubsection{Chapter 2}

Chapter 2 contains an introduction to the basic structures of Poisson geometry from an algebraic point of view --- focusing on Poisson structures in the category of schemes.  Following \cite{Kaledin06}, we introduce singular symplectic structures and discuss the extent to which local properties of their Poisson structure carry over from the smooth case to the algebraic case.  Lastly, using these concepts, we examine Poisson brackets on infinitesimal neighborhoods of a Poisson subscheme that will be used later as the framework in which one ``infinitesimally extends'' the Poisson bracket on a singular space.

\subsubsection{Chapter 3}

Chapter 3 describes the connection between the usual differential geometric description of a Poisson structure and the scheme theoretic development in Chapter 2.  It begins by reviewing (following \cite{Dufour05}) the correspondence, on smooth spaces, between Poisson brackets and bivector fields.  Moreover, the definition of such bivector fields is extended to infinitesimal neighborhoods and related to the homogeneous decomposition of a bivector field.

\subsubsection{Chapter 4}

Chapter 4 serves an an introduction to the problem of embedding a singular symplectic space.  The first section explains the problem and the extent to which the smooth embedding problem reduces to a algebraic/formal Poisson embedding problem.  The second section consists of many examples --- both of concrete embeddings and nonembedding results of singular Poisson spaces.  The last section sketches a first approach to constructing an embedding: the Zariski cotangent space at the singular point has a natural Lie algebra structure, called the cotangent Lie algebra, that defines a Poisson-embedding {\it up to first order}.  Any Poisson structure extending the given one is a deformation of the dual of this Lie algebra.

\subsubsection{Chapter 5}

Chapter 5 summarizes (following \cite{Dufour05}) the linearization problem in Poisson geometry.  In particular, the Levi decomposition for Poisson manifolds determines to what extent a potential extension of the Poisson structure on the quotient will be isomorphic to the dual of its cotangent Lie algebra.  This semi-linearization is key to proving Theorem \ref{thm:NonEmb}.

\subsubsection{Chapter 6}

Chapter 6 contains the proof of Theorem \ref{thm:NonEmb}.  Using the Levi decomposition for Poisson structures, the proof simplifies the ``infinitesimal extension'' conditions formulated in Chapter 4 to reveal an obstruction to extension in the $(n-1)^{\rm th}$-infinitesimal neighborhood of the singular point.

\subsubsection{Chapter 7}

Chapter 7 reasons through a series of reductions in the ``infinitesimal extension'' conditions formulated in Chapter 4 to construct a concrete extension of the Poisson structure on $V/\ZZ_3$, for a 4-dimensional real vector space $V$ (i.e. Theorem \ref{thm:PoissExt}).

\subsubsection{Appendix}

The Appendix contains two parts: (1) a section reviewing the representation theory of $\slt(\CC)$ and (2) a section explaining the Macaulay2 code used to construct the extension given in Chapter 7.

\chapter{Poisson Geometry Basics}

This chapter contains preliminary materials on Poisson structures that form the framework for the remainder of the thesis.   As the central objects of study are Poisson structures on singular spaces, we will take an algebraic approach to the subject.  A Poisson structure on a space is an additional algebraic structure on the sheaf of functions of the space.  

Some of the following material may be developed over a ground field $\KK$ of positive characteristic (see \cite{Kaledin06}).  This is a difficult endeavor that will not be pursued here.  Throughout, $\KK$ will be a field of characteristic $0$.  Much of the later material will assume that $\KK=\RR$ or $\CC$.  All schemes will be Noetherian schemes of finite type over $\KK$.

The material in this section draws from the following work of Kaledin and Polishchuk (see \cite{Kaledin06, Polishchuk97}). 

\section{Poisson Algebras}
\label{sec:PoissonAlg}

\begin{definition}
\label{def:PoissonAlg}
A Poisson algebra is a pair $(A,\brac)$, where $A$ is a commutative, associative algebra over $\KK$ and $\brac$ is a Lie bracket on $A$, satisfying the Leibniz identity:
$$\{fg,h\}=f\{g,h\}+g\{f,h\}  \mbox{ for all $f,g,h\in A$}$$   
The Leibniz identity is a compatibility condition between the Lie algebra $(A,\brac)$ and associative algebra $(A,\cdot)$.
\end{definition}

\begin{remark}  As a Poisson algebra $A$ is both an associative algebra and a Lie algebra, an algebraic property/structure of $A$ will by default mean a property/structure of the associative algebra of $A$.  For example, the statement ``$A$ is commutative'' means that the associative algebra $A$ is a commutative algebra.  The property that $\{f,g\}=\{g,f\}$ would be expressed as ``$(A,\brac)$ is a commutative Lie algebra.''
\end{remark}

\begin{definition} A homomorphism of Poisson algebras $\psi:(A,\brac_A)\to(B,\brac_B)$ is a homomorphism of associative and Lie algebras.
\end{definition}

The Leibniz identity implies that a Poisson bracket is determined by where it sends the generators of $A$.

\begin{example}
\label{ex:SympVS}  
Let $A=\KK[x_1,y_1,\dots,x_n,y_n]$ be a polynomial algebra in $(2n)$-variables.  Define a Poisson bracket by $\{x_i,y_j\}=\delta_{ij}$. 
\end{example}

\begin{example}
\label{ex:LP}
Let $(\g,[\cdot,\cdot])$ be a Lie algebra over $\KK$.  Then the symmetric algebra $\Sym(\g)$ is a Poisson algebra, where $\brac$ is defined on generators $v_i,v_j\in\g$ by $\{v_i,v_j\}=[v_i,v_j]$ and extended to $\Sym(\g)$ via the Leibniz identity.  
\end{example}

\begin{definition}  An element $f$ of a Poisson algebra $(A,\brac)$ is called a {\it Casimir} element if $\{f,g\}=0$ for all $g\in A$.  The Casimir elements form a subalgebra of $A$.
\end{definition}

\begin{remark} The Casimir elements of the polynomial algebra $\KK[x_1,y_1,\dots,x_n,y_n]$ in Example \ref{ex:SympVS} are the constants, while the subalgebra formed by the Casimir elements of $\Sym(\g)$ in Example \ref{ex:LP} is isomorphic to the center of the universal enveloping algebra $Z(\mathcal{U(\g)})$.
\end{remark}

Poisson algebras behave well with respect to many common ring operations including localization, quotients, and tensor products:

\begin{proposition}
\label{prop:MultSub}
 Let $S\subset A$ be a multiplicative subset of a Poisson algebra $(A,\brac)$.  The localization $S^{-1}A$ is naturally a Poisson algebra with the following bracket:  let $s^{-1}f,t^{-1}g\in S^{-1}A$,
$$ \{s^{-1}f,t^{-1}g\}:=(s^2t^2)^{-1}\left[ \{s,t\}fg\right]-(t^{2}s)^{-1}\left[\{f,t\}g\right]-(s^{2}t)^{-1}\left[\{s,g\}f\right]+(st)^{-1}\left[\{f,g\}\right]$$
Moreover, the natural homomorphism $A\to S^{-1}A$ is a homomorphism of Poisson algebras.
\end{proposition}

Analogous to the case of Lie algebras, in order to take quotients in the category of Poisson algebras we need the notion of a Poisson ideal:

\begin{definition} An ideal $I$ of a Poisson algebra $A$ is called a {\it Poisson ideal} if it is also a Lie-ideal: $\{f,I\}\in I$ for all $f\in A$.
\end{definition}

\begin{proposition} Given a Poisson ideal $I$ of a Poisson algebra $(A,\brac)$, the quotient $A/I$ is naturally a Poisson algebra with Poisson bracket $\{f+I,g+I\}=\{f,g\}+I$.
\end{proposition}

An important special subclass of the Poisson ideals are the Poisson prime ideals:

\begin{definition}
\label{def:PoissonPrime}
A Poisson ideal $I$ is a called a {\it Poisson prime ideal} if it is both a Poisson ideal and a prime ideal.
\end{definition}

\begin{proposition}
\label{prop:PoissonTProd}
 Given two Poisson algebras $(A,\brac_A)$ and $(B,\brac_B)$, one can define a Poisson algebra on the tensor product $A\otimes B$ as follows:
$$\{f\otimes g,f'\otimes g'\}_{A\otimes B}:=ff'\otimes\{g,g'\}_B+\{f,f'\}_A\otimes gg'$$
\end{proposition}

A Poisson module over a Poisson algebra $A$ is both an $A$-module and a Lie algebra module, 

\begin{definition}
\label{def:PoissonModule}
A {\it Poisson Module} is an $A$-module $M$ over a Poisson algebra $A$, and a $\KK$-linear bracket $\brac_M:A\times M\to M$ that
\begin{enumerate}
\item is a derivation of the commutative product: $\{fg,m\}_M=f\{g,m\}_M+g\{f,m\}_M $
\item satisfies the Leibniz-type identity: $ \{f,gm\}_M= \{f,g\}m+g\{f,m\}_M$
\end{enumerate}
for all $m\in M$ and $f,g\in A$.
\end{definition}

\begin{example} It is straightforward to check that a Poisson ideal of a Poisson algebra is a Poisson module.  
\end{example}

The algebraic development above also carries over to the graded case:

\begin{definition}
\label{def:GrPoissonAlg} A graded associative, commutative $\KK$-algebra $A=\bigoplus A_n$  is a {\it graded Poisson algebra of degree $l$} if the Poisson bracket satisfies 
$$ \brac: A_i\otimes A_j\to A_{i+j+l}$$
\end{definition}

\begin{example} With the usual grading on the associative algebra $\KK[x_1,y_1,\dots,x_n,y_n]$ considered in Example \ref{ex:SympVS} the Poisson
bracket $\{x_i,y_j\}=\delta_{ij}$ has degree (-2).   On the other hand, with the usual grading on $\Sym(\g)=\bigoplus_{n}\Sym^n(\g)$ the Poisson bracket on $\Sym(\g)$ considered in Example \ref{ex:LP} has degree (-1).
\end{example}

\section{Poisson Structures on Schemes}  

The algebraic preliminaries above cover the basic requirements needed to develop the subject of Poisson algebraic geometry.  

\begin{definition}  A scheme $(X,\OO_X)$ is a {\it Poisson scheme} if its sheaf of functions $\OO_X$ is a sheaf of Poisson algebras.  That is, for each open set $U\subset X$, the algebra of functions $\OO_X(U)$ is a Poisson algebra.  Moreover, these brackets are required to be compatible with the restriction maps.  A morphism of Poisson schemes is a morphism of schemes with the additional requirement that the pullback is a homomorphism of Poisson algebras. 
\end{definition}

\begin{remark}
\label{rmk:Anchor} Note that any such Poisson structure is given by an $\OO_X$-linear morphism of sheaves, called the {\it anchor map} $\sharp:\Omega_X\to\Der(\OO_X,\OO_X) = \TT_X\subset\End(\OO_X)$ satisfying $\sharp(df)(g)=\{f,g\}$.
\end{remark}

\begin{remark} In fact, one may instead develop Poisson structures on other spaces in a similar manner.  The most common such Poisson structures are listed below:
\begin{itemize}
\item A smooth manifold $(M,C^{\infty}(M))$ is a {\it Poisson manifold} if  $C^{\infty}(M)$ is a Poisson algebra.
\item A complex manifold $M$ is a {\it complex-analytic Poisson manifold} if the sheaf of complex-analytic functions is a sheaf of Poisson algebras.
\item A formal scheme $(\hat{X},\hat{\OO}_X)$ is a {\it formal Poisson scheme} if $\hat{\OO}_X$ is a sheaf of Poisson algebras.
\item An algebraic variety $X$ is a {\it Poisson variety} if the sheaf of regular functions is a sheaf of Poisson algebras.
\end{itemize}
\end{remark}

The only Poisson schemes considered in this thesis are affine Poisson schemes.  Given a Poisson algebra $A$, Proposition \ref{prop:MultSub} implies the affine scheme $\Spec(A)$ is in fact a Poisson scheme.  

\begin{example} $2n$-dimensional affine space $\bbA^{2n}_k:=\Spec(\KK[x_1,y_1,\dots,x_n,y_n])$ is a Poisson scheme with the sheaf of Poisson algebras defined in Example \ref{ex:LP}.  This is {\it the standard symplectic} structure on $\bbA^{2n}_k$.  
\end{example}

\begin{example} Given a Lie algebra $\g$ over $\KK$, the Poisson bracket in Example \ref{ex:LP} makes the dual space $\g^*$ a Poisson scheme.
\end{example}

The remainder of the material from the previous algebraic section easily translates to their corresponding geometric properties:

\begin{definition}
Let $X$ be a Poisson scheme.  A Poisson scheme $Y$ is a {\it Poisson subscheme} of X if 
\begin{enumerate}
\item $Y$ is a subscheme of $X$, and
\item the inclusion map $Y\inj X$ is a Poisson map.
\end{enumerate}
\end{definition}

\begin{proposition} Any open subscheme $U\subset X$ of a Poisson scheme is a Poisson subscheme of $X$.
\end{proposition}

\begin{proposition} Suppose that $Y\subset X$ is a closed subscheme of a Poisson scheme defined locally by a sheaf of ideals $\II$.  Then $Y$ is a Poisson sub-scheme if and only if $\II$ is a sheaf of Poisson ideals.
\end{proposition}

Combining this proposition with the fact that irreducible subschemes correspond to quasi-coherent sheaves of prime ideals (see \cite{Hartshorne77}) gives:

\begin{proposition}
\label{prop:IrredPoisson}
Let $X$ be a Poisson scheme.  The irreducible Poisson subschemes of $X$ correspond precisely to $V(\II)$, where $\II\subseteq\OO_X$ is a quasi-coherent sheaf of Poisson prime ideals.
\end{proposition}

Lastly, in the language of affine schemes, Proposition \ref{prop:PoissonTProd} implies that the product of affine Poisson schemes is again an affine Poisson scheme.  This result globalizes to give the following proposition:

\begin{proposition} If $X,Y$ are Poisson schemes, then $X\times_\KK Y$ is a Poisson scheme.
\end{proposition}

\subsection{Symplectic structures}

Singular symplectic varieties will provide the primary examples in later chapters.  Unlike Poisson structures, which require no smoothness to define, symplectic structures are defined via differential 2-forms.  Thus the notion of a {\it singular} symplectic space requires a little care.

 Defining symplectic structures and their relationship to Poisson structures requires the notion of the rank of a Poisson bracket:

\begin{definition}
\label{def:Rank}
Let $(X,\brac)$ be a Poisson scheme.  The {\it rank} of $\brac$ at a point $x\in X$ is the rank of the anchor map $\sharp_{x}:\Omega_{X,x}\to\TT_{X,x}$.  If the rank of $\brac$ is maximal for all $x$, the Poisson bracket is called {\it non-degenerate}.
\end{definition}

\begin{remark} More concretely, the rank is computed as follows:  Pick an affine open set $x\in U$, and a local system of generators $f_1,\ldots,f_n\in\OO_X(U)$.  The rank of $\brac$ at $x$ is the rank of the matrix with entries $\{f_i,f_j\}$, for $1\leq i,j\leq n$.
\end{remark}

\begin{example} The Poisson bracket on $\KK[x_1,y_1,\ldots,x_n,y_n]$ given in Example \ref{ex:SympVS} is non-degenerate, as it has rank $n$ everywhere.
\end{example} 

For smooth spaces, the above definition provides the link between Poisson structures and the symplectic structures defined below:

\begin{definition}
\label{def:Symp}[Smooth symplectic structures]
On a smooth scheme $X$, a closed, non-degenerate 2-form $\omega\in\Omega^2_{X}$ is a {\it smooth symplectic structure} on $X$.
\end{definition}

\begin{remark} We append the adjective ``smooth'' to differentiate between the usual notion of a symplectic structure in manifold theory and the singular symplectic varieties considered in \cite{Kaledin06} and below.
\end{remark}

There is a one-to-one correspondence between non-degenerate smooth Poisson structures and smooth symplectic structures given by the bundle isomorphism $\sharp:\Omega_X\to\TT_X$.   Starting with a non-degenerate Poisson structure, define a 2-form $\omega_x(u,v)=<\sharp^{-1}(u),v>$ where $<,>:\Omega_X\times\TT_X\to k$ is the usual non-degenerate pairing.

\begin{example} The Poisson structure of Example \ref{ex:SympVS} on $\bbA^{2n}_k$ corresponds to the constant 2-form  $\omega=dx_1\wedge dy_1+\ldots+dx_n\wedge dy_n$
\end{example}

Following \cite{Kaledin06}, this correspondence between non-degenerate Poisson spaces and symplectic spaces may be used to extend the definition of a symplectic structure to a possibly singular algebraic variety.   Let $(X,\brac)$ be an integral, normal, Poisson scheme of finite type over $\KK$ (that is, $X$ is a normal algebraic variety).  Moreover, denote the maximal, open smooth subset of $X$ by $X_{sm}$.

\begin{definition} A normal Poisson variety $(X,\brac)$ is called {\it symplectic} if the induced Poisson bracket on $X_{sm}$ is non-degenerate.
\end{definition}

\begin{remark} The requirement that the variety be normal is motivated by the fact (see \cite{Hartshorne77}) that normal varieties are the most general spaces for which an algebraic version of Hartog's theorem holds.  This condition ensures that $\brac$ is determined by its restriction to the smooth locus $X_{sm}$ .
\end{remark}

By \cite{Kaledin06}, Definition 2.1, such a bracket is equivalent to a non-degenerate, closed 2-form on $X_{sm}$ that extends to  a possibly degenerate 2-form on some resolution of singularities of $X$.  If $X$ is smooth, this clearly coincides with the definition of a smooth symplectic space.

\begin{example} Let $X=\bbA^{2n}_k$ with the standard Poisson structure, and let $G\subseteq{\rm Sp}_{2n}(\KK)$ be a finite algebraic group acting linearly on $X$ by Poisson isomorphisms. In the Chapter \ref{sec:PoissEmbProb}, it will be shown that the quotient $X/G=\Spec(\OO_X(X)^G)$ is a singular symplectic variety.
\end{example}

\subsection{Decompositions of Poisson structures}

In the case of real manifolds, Poisson structures are built from non-degenerate, symplectic pieces.  Weinstein proved \cite{Weinstein83} that any Poisson manifold is foliated by symplectic leaves: 

\begin{theorem}[Weinstein's splitting theorem \cite{Weinstein83}]
\label{thm:WeinsteinDecomp} Let $(M,\brac)$ be a smooth Poisson manifold, and $p\in M$.  Then there is a system of coordinates $(x_1,y_1,\ldots,x_r,y_r,z_1,\ldots,z_d)$ centered at $p$ such that $\brac=\brac_S+\brac_N$, where
$$\{f,g\}_S=\sum_{i=0}^{r}(\pderiv{f}{x_i}\pderiv{g}{y_i}-\pderiv{g}{x_i}\pderiv{f}{x_j})$$
$$\{f,g\}_N=\sum_{0\leq i,j\leq d} \phi(x,y,z)\pderiv{f}{z_i}\pderiv{g}{z_j}$$
where $f,g,\phi\in C^{\infty}(M)$ and $\phi(p)=0$.  In other words, $M\isom S\times N$ is locally a product of a symplectic manifold $S$ and a transverse Poisson manifold $N$ with a bracket $\brac_N$ that vanishes at the origin.
\end{theorem}

If $M$ is a symplectic manifold to begin with, one recovers Darboux's theorem:

\begin{corollary}[Darboux's Theorem]
\label{thm:Darboux}
  Let $(M,\omega)$ be a symplectic manifold, and $p\in M$.  Then there exists a system of coordinates $(x_1,y_1,\ldots,x_r,y_r)$ centered at $p$ such that 
$$\omega=\sum_{i=0}^{r}dx_i\wedge dy_i$$
\end{corollary}

Thus, all symplectic manifolds are locally isomorphic to Example \ref{ex:SympVS}, and the study of local properties of Poisson manifolds reduces to the consideration of Poisson structures that vanish at the origin.
\\

The analogous properties for algebraic Poisson structures do {\it not} hold as (1) the change in coordinates is rarely algebraic and (2) the symplectic leaves are rarely algebraic.  However, in certain cases, Poisson schemes still have a stratification by symplectic leaves.  For example, a smooth symplectic scheme trivially decomposes into a product of a single symplectic leaf and a trivial Poisson scheme.  The next case --- that of a singular symplectic variety --- is the subject of Theorem 2.3 of \cite{Kaledin06}.

\begin{theorem}[Kaledin \cite{Kaledin06}]
\label{thm:Kaledin06}
Let $X$ be a singular symplectic variety.  There exists a finite stratification $X_i\subseteq X$ where
\begin{enumerate}
\item each stratum $X_i$ is an irreducible Poisson subscheme,
\item the open smooth loci $(X_i)_{sm}$ are symplectic,
\item for any closed point $x\in(X_i)_{sm}$, the formal completion of $X$ decomposes as a product
$$ \widehat{X}_x=\mathcal{Y}_x\times \widehat{(X_{i})_{sm}}_x $$
\end{enumerate}
where $\mathcal{Y}_x$ is a local formal Poisson scheme.
\end{theorem}

The completed strata $\widehat{(X_{i})_{sm}}_x$  are the symplectic leaves of $\widehat{X}$.
The use of formal schemes in the theorem is unavoidable, as the symplectic leaves leaves are seldom algebraic.  However, the theorem still gives a one-to-one correspondence between the algebraic strata $X_i$ and the symplectic leaves of $X$.  Therefore,

\begin{corollary}
There is a bijection between the Poisson prime ideals of a singular symplectic variety $X$ and the symplectic leaves of $X$.  Moreover, a singular symplectic variety has finitely many symplectic leaves.
\end{corollary}

\subsection{Poisson Structures on $k^{\rm th}$-Infinitesimal Neighborhoods}
\label{sec:InfNbd}

The following proposition states that the $k$-infinitesimal neighborhood of a closed Poisson subscheme is again a Poisson subscheme.

\begin{proposition} Suppose $Y=\Spec(A/I)$ is a closed Poisson subscheme of  an affine Poisson scheme $X=\Spec(A)$. The $k$-infinitesimal neighborhood of $Y$, defined locally as $\nbd{k}{Y}{X}:=\Spec(A/I^{k+1})$, is also a Poisson subscheme of $X$. 
\end{proposition}

\begin{proof} It needs to be shown that $I^k$ is a Poisson ideal.  In fact, given $f\in A$ and $g=g_1\ldots g_k\in I^k$, the Leibniz identity implies that
$$\{f,g\}_{X}=\{f,g_1\ldots g_k\}_{X}=\sum_{i=1}^k g_1\ldots\hat{g_i}\ldots g_k\{f,g_i\}_{X}$$
As each $g_i\in I$ and $I$ is a Poisson ideal, it follows that each $\{f,g_i\}_{X}\in I$. Hence $\{f,g\}_{X}\in I^k$.
\end{proof}

Thus, given a closed Poisson subscheme of an affine scheme, one automatically gets compatible Poisson structures on the $k$-infinitesimal neighborhoods of $Y$. That is, $\forall k>0$ the map $p_k\!:A\to A/I^{k+1}$ is a homomorphism of Poisson algebras.  Furthermore, the Poisson structures $(\nbd{k}{Y}{X},\brac_{k})$ are compatible with each other in the following sense:

\begin{proposition} The natural projection $\phi_k\!:A/I^{k+1}\to A/I^k$ is a homomorphism of Poisson algebras.
\end{proposition}

\begin{proof}  By the proposition above, the top arrows below are homomorphisms of Poisson algebras:

\[
\xymatrix{
& (A,\brac) \ar[ld]_{p_{k-1}} \ar[rd]^{p_k} &  \\
(A/I^k,\brac_{k-1}) & & (A/I^{k+1},\brac_k) \ar[ll]^{\phi_k}
}
\]

Given $f,g\in A/I^k$, let $\tilde{f},\tilde{g}$ be lifts of $f,g$ to $A$.  By the diagram above, 
$$\left\{\begin{array}{l} \phi_k(f)=\phi_k\circ p_k(\tilde{f})=p_{k-1}(\tilde{f}) \\ \phi_k(g)=\phi_k\circ p_k(\tilde{g})=p_{k-1}(\tilde{g})\end{array}\right.$$

Using these lifts to $A$, and the fact that the $p_i$ are Poisson algebra homomorphisms, 
$$\begin{array}{rclr}
\{\phi_k f,\phi_k g\}_{k-1} &=& \{\phi_k\circ p_k(\tilde{f}),\phi_k\circ p_k(\tilde{g})\}_{k-1} &\\
					& =& \{p_{k-1}(\tilde{f}),p_{k-1}(\tilde{g})\}_{k-1} &\\
					&=& p_{k-1}(\{\tilde{f},\tilde{g}\}) &  \mbox{     (as $p_{k-1}$ is a Poisson homomorphism) } \\
					&=& \phi_k\circ p_k(\{\tilde{f},\tilde{g}\}) &\\
					&=& \phi_k(\{f,g\}_k)  &\mbox{     (as $p_{k}$ is a Poisson homomorphism) }
\end{array}
$$
Thus, $\phi_k$ is a homomorphism of Poisson algebras.

\end{proof}

The propositions above are nicely summarized by the following commutative diagram of Poisson algebras:

\[
\xymatrix{
A \ar[d]_{\exists p_\infty} \ar[rrd]|-{p_3} \ar[rrrd]|-{p_2} \ar[rrrrd]|-{p_1}\\
\varprojlim A/I^k \ar[r] & \ldots  \ar[r]_{\phi_3} & A/I^3 \ar[r]_{\phi_2} & A/I^2 \ar[r]_{\phi_1} & A/I
}
\]

On the level of spaces, this diagram corresponds to a chain of Poisson embeddings:
$$(Y,\brac_Y)\hookrightarrow(\nbd{1}{Y}{X},\brac_{1})\hookrightarrow\ldots\hookrightarrow(\nbd{\infty}{Y}{X},\brac_{\infty})\hookrightarrow(X,\brac)$$

In subsequent chapters, the primary technique to (1) construct Poisson embeddings, or (2) show the impossibility of constructing a Poisson embedding will be as follows: 

Given a Poisson scheme $(Y,\brac_Y)$ and a closed embedding $\iota\!:Y\hookrightarrow\mathbb{A}^n_k$ into $n$-dimensional affine space, either (1) construct a Poisson structure on $X$ (making $\iota$ a Poisson map) by extending $\brac_Y$ to Poisson structures on each $\nbd{k}{Y}{\mathbb{A}^n_k}$, or (2) find an obstruction to such an extension, for some $\KK$.

\subsubsection{The cotangent Lie algebra}

Let $X$ be a smooth Poisson scheme.  Define a Lie bracket on the cotangent sheaf $\Omega_X$ by defining it locally on exact 1-forms as follows: $[df,dg]=d\{f,g\}$.  This Lie algebra structure turns the anchor map $\sharp:\Omega_X\to\TT_X$ into a Lie algebra homomorphism.  Moreover, this construction is functorial with respect to Poisson subschemes:  Given $Y\subseteq X$ with defining Poisson ideal $\II$, the diagram below is a commutative diagram of Lie algebras:

\[
\xymatrix{
\Omega_{X/Y} \ar[d]\ar[rd]^{(\sharp_X)_{|_Y}}   \\
\Omega_Y \ar[r]_{\sharp_Y} & \TT_Y
}
\]

The standard exact sequence
$$0\to\II/\II^2\to\Omega_{X/Y}\to\Omega_Y\to 0$$
becomes an exact sequence of Lie algebras.

\begin{definition}
\label{def:Linearization} The $\OO_Y$-linear lie bracket on $\II/\II^2$ obtained from the above exact sequence is the {\it linearization of the Poisson bracket at $Y$}
\end{definition}

\begin{example} Let $(X,\brac)$ be a Poisson scheme and $x\in X$ is a Poisson subscheme with corresponding maximal ideal $\m_x$ (so $\brac$ vanishes at $x$).  The linearization of $\brac$ at $x$ is a Lie bracket on Zariski cotangent space $\m_x/\m_x^2$, sometimes called the {\it isotropy cotangent space at $x$} (see \cite{Fernandes04}).
\end{example}

\begin{remark} The above example, in terms of the splitting theorem for real Poisson manifolds (Theorem \ref{thm:WeinsteinDecomp}), consists of only a transverse Poisson structure $\brac=\brac_N$.  The linearization of $\brac_N$ is given by what would be expected:  
$$[df,dg]=d\left(\sum_{0\leq i,j\leq d}\phi^1\pderiv{f}{z_i}\pderiv{g}{z_j}\right)$$
where $\phi^1$ is the linear term of $\phi$.
\end{remark}

\chapter{Poisson Bivectors and the Schouten Calculus}
\label{sec:Ch3}

This section summarizes basic properties of {\it smooth} Poisson structures using multivector fields.  As the majority of this material is applied to the case $X=\bbA^{n}_k$, for the purpose of explicit computation, we'll adopt the coordinate intensive treatment of Dufour/Zung in \cite{Dufour05}.  

\section{The Poisson bivector and the Schouten bracket}

Fix a Poisson scheme $(X,\brac)$.   The Poisson bracket $\brac$ is equivalent to a global section of the sheaf $\mathcal{H}{\it\! om}_{\OO_X}(\bigwedge^2\Omega_X,\OO_X)$.  That is, we may define $\pi\in\Gamma(X,\mathcal{H}{\it\! om}_{\OO_X}(\bigwedge^2\Omega_X,\OO_X))$ by 
$$\pi(df\wedge dg)=\{f,g\}$$
Furthermore, when $X$ is smooth, $\mathcal{H}{\it\! om}_{\OO_X}(\bigwedge^2\Omega_X,\OO_X)$ may be identified with the sheaf of bivector fields $\bigwedge^2\TT_X$, and the defining relationship between the bivector and bracket becomes:
$$ \iota_{df}\pi = \sharp(df) $$
where $\iota$ is the contraction operator, $\sharp$ is the anchor (see Remark \ref{rmk:Anchor}), and equality takes place in $\TT_X$.  It's clear that the properties of the bilinearity, skew-symmetry of $\brac$, as well as the satisfaction of the Leibniz identity, are equivalent to $\pi\in\Gamma(X,\mathcal{H}{\it\! om}_{\OO_X}(\bigwedge^2\Omega_X,\OO_X))$.  To see what property of $\pi$ corresponds to $\brac$ satisfying the Jacobi identity, wee need to introduce the Schouten bracket.

\subsection{Multivector fields and the Schouten bracket}

This section introduces the theory of multivector fields and translates properties of Poisson brackets into the language of bivectors (see \cite{Dufour05}).  Throughout, $(X,\brac)$ will be a {\it smooth} Poisson scheme.  \\

Denote the graded complex of multivector fields on $X$ by $\bigwedge^{*}\TT_X=\bigoplus\bigwedge^{k}\TT_X$.   The Lie bracket of vector fields on $X$ uniquely extends to a graded Lie bracket on $\bigwedge^{*}\TT_X$ satisfying the graded Leibniz rule.  That is:

\begin{definition}[Schouten-Nijenhuis] There is a unique bilinear bracket $[\cdot,\cdot]$ of degree-$(-1)$, called the {\it Schouten bracket} satisfying the following properties:  for multivector fields $\alpha$ and $\beta\in\bigwedge^{\cdot}\TT_X$, with respective degrees $a,b\in\NN$, $\gamma\in\bigwedge^{c}\TT_X$, and $X\in\TT_X$, $f\in\OO_X$
\\
$$
\begin{array}{lclr}
(1) &\qquad & [\alpha,\beta]+(-1)^{(a-1)(b-1)}[\beta,\alpha]=0 &  \quad\mbox{(skew-symmetry)} \\
\\
(2) &\qquad & [\alpha,\beta\wedge\gamma]=[\alpha,\beta]\wedge\gamma+(-1)^{(a-1)b}\beta\wedge[\alpha,\gamma] & \mbox{(Leibniz identity)} \\
\\
(3) & \qquad & (-1)^{(a-1)(c-1)}[\alpha,[\beta,\gamma]]+(-1)^{(b-1)(a-1)}[\beta,[\gamma,\alpha]]+ & \mbox{(Jacobi identity)} \\
 & \qquad & \qquad\qquad\qquad\qquad\quad  +(-1)^{(c-1)(b-1)}[\gamma,[\alpha,\beta]]=0 & \\
\\
(4) & \qquad & [X,Y]=\mathcal{L}_XY \mbox{ and } [X,f]=\langle df,X \rangle & \mbox{(extends the Lie bracket)} \\
\end{array}
$$
\end{definition}

The Schouten bracket succinctly characterizes when a bivector corresponds to a Poisson bracket \cite{Dufour05}:
$$ \pi \mbox{ corresponds to a Poisson bracket $\brac$ if and only if } [\pi,\pi]=0 $$

\begin{definition}
\label{def:Jacobiator}
Given a bivector field $\pi\in\bigwedge^2\TT_X$, the trivector field $[\pi,\pi]\in\bigwedge^3\TT_X$ is called the {\it Jacobiator} of $\pi$.
\end{definition}

The jacobiator of a bivector has the explicit characterization:

$$[\pi,\pi](df,dg,dh)=2\left(\pi(df,d\pi(dg,dh))+{\rm c.p}\right)$$

An explicit formula for the Schouten bracket in local coordinates will prove useful for explicit calculations.  We adopt the conventions in \cite{Dufour05}.  Given a local coordinate system $(x_1,\ldots,x_n)$ in X, the $\OO_X$-module $\bigwedge^*\TT_X$ is generated by elements of the form $\vect{x_{i_1}}\wedge\ldots\wedge\vect{x_{i_k}}$.   To ease notation, we replace $\{\vect{x_1},\ldots,\vect{x_n}\}$ with anti-commuting variables $\{\xi_1,\ldots,\xi_n\}$.   The Schouten bracket of $\alpha,\beta\in\bigwedge^{*}\TT_X$ is then locally written as a ``super Poisson bracket.''

$$[\alpha,\beta]=\sum_{i=1}^{n}\left(\pderiv{\alpha}{\xi_i}\pderiv{\beta}{x_i}-(-1)^{(a-1)(b-1)}\pderiv{\beta}{\xi_i}\pderiv{\alpha}{x_i}\right)$$

\begin{remark} Defining $\vect{\xi_i}$ for anti-commuting variables takes a little care.  We use the sign convention that
$$\pderiv{\xi_{i_1}\cdots\xi_{i_d}}{\xi_{i_k}}=(-1)^{d-k}\xi_{i_1}\cdots\widehat{\xi_{i_k}}\cdots\xi_{i_d}$$
where $\widehat{\mbox{  }}$ denotes a removed term.
\end{remark}

\section{Poisson Cohomology}

Using the structures defined in the previous section, Lichnerowicz \cite{Lichnerowicz77} defined a cohomology for Poisson manifolds.  This construction transfers mostly unchanged to the category of {\it smooth} Poisson schemes.

\begin{definition}
Let $X$ be a smooth Poisson scheme with Poisson bivector $\pi$.  We define the {\it Poisson cohomology} $H_\pi^*(X,\mathcal{F}) $ of $(X,\pi)$ with coefficients in a sheaf of Poisson-modules $\mathcal{F}$ to be the cohomology of the complex:
$$\ldots\to\wedge^3\TT_X\otimes_{\OO_X}\mathcal{F}^*\to\wedge ^2\TT_X\otimes_{\OO_X}\mathcal{F}^*\to\TT_X\otimes_{\OO_X}\mathcal{F}^*\to\mathcal{F}^*\to 0$$
where $\dd_\pi:\wedge^k\TT_X\to\wedge^{k-1}\TT_X$ is defined as $\dd_\pi\beta:=[\pi,\beta]$.  If $\mathcal{F}=\OO_X$, then we write $H^*_\pi(X)$
\end{definition}

\begin{remark} That the $d_\pi$ squares to zero is equivalent to the condition that $[\pi,\pi]=0$
\end{remark}

\begin{remark} As this construction is built from multivector fields, it unsurprisingly depends heavily on the smoothness of $X$. There is an alternative cohomology theory for non-smooth Poisson schemes, called the Harrison cohomology, built from the Hochschild complex of $\OO_X$.  See \cite{Loday98}.
\end{remark}

Unsurprisingly, the $2^{{\rm nd}}$ Poisson cohomology of a Poisson scheme $(X,\brac)$ classifies (formal) deformations of $\brac$ on $X$.  The low degree cohomology groups have the following interpretation:

\begin{enumerate}
\item the zeroth cohomology $H^0_\pi(X)$ is the algebra of Casimir functions of $\brac$.
\item the first cohomology $H^1_\pi(X)$ is the algebra of vector fields leaving $\brac$ invariant, modulo Hamiltonian vector fields.
\item the second cohomology $H^2_\pi(X)$ is the space of formal deformations modulo ``trivial deformations''.
\item the third cohomology  $H^3_\pi(X)$ consists of obstructions to deformations of $\pi$.
\end{enumerate}

These Poisson cohomology groups tend to be difficult to compute, infinite dimensional groups --- even for well-behaved spaces.  In the category of smooth manifolds, the following theorem illustrates the range of possibilities for $H_\pi^*(M)$.  

\begin{theorem}[Lichnerowicz] Given a Poisson manifold $(M,\pi)$, the pullback of the anchor map descends to a homomorphism
$$\sharp^*:H^*_{dR}(M)\to H^*_\pi(M)$$
from the de Rham cohomology of $M$ to the Poisson cohomology of $M$.  Moreover:
\begin{itemize}
\item $\sharp^*$ is, in general, neither injective nor surjective.
\item If $M$ is symplectic, $\sharp^*$ is an isomorphism.
\end{itemize}
\end{theorem}

This theorem reinforces the result of Darboux (Theorem \ref{thm:Darboux}): symplectic manifolds are locally rigid, with only their topology affecting their deformations.

Another example for which the Poisson cohomology groups are known is the case of the Lie-Poisson structure $(\g^*,\brac)$  for  a Lie algebra $\g$ (see Example \ref{ex:LP}).  In this case they are related to the familiar Chevalley-Eilenberg cohomology.

\begin{theorem}[\cite{Lichnerowicz77}]
\label{thm:CECoh}
 Let $(\g^*,\pi)$ be the Lie-Poisson structure of a Lie algebra $\g$.  Then there is an isomorphism of cohomology groups:
$$H^*_{\pi}(\g^*)\isom H_{CE}^*(\g,\Sym(\g))$$
where $\Sym(\g)$ is the (infinite dimensional) representation given by the adjoint action of $\brac$. 
\end{theorem}

\begin{remark} This theorem actually holds on the level of cochains.
\end{remark}

\subsection{Poisson cohomology and infinitesimal neighborhoods}

Let $Y=V(\II)$ be a smooth Poisson subscheme of a Poisson scheme $X$.  Recall from Section \ref{sec:InfNbd} that the ideal sheaves $\II^k\subseteq\OO_X$ are in fact Poisson ideal sheaves for all $k\in\NN$.  Thus, each $k$ gives rise to an exact sequence of sheaves of Poisson $\OO_X$-modules:
$$0\to\II^k\to\OO_X\to\OO_{Y^k}\to0$$
where $Y^k=\mathcal{N}_Y^k(X)$.  There is a one-to-one correspondence between Poisson structures $\brac_{(k)}$ on $Y^k$ as defined in Section $\ref{sec:InfNbd}$ and Poisson bivectors $\pi^k$ in $\TT_X\otimes\OO_{Y^k}$.  Thus, although $Y^k$ is {\it not} smooth, it makes sense to talk about bivectors on $Y^k$ and to identify $H^*_{\pi^{(k)}}(Y^k)$ with $H^*_{\pi}(X,\OO_{Y^k})$.

\section{Graded Poisson structures on $\bbA^n_k$}
\label{sec:GrPoiss}

Let $\pi$ be a Poisson bivector on $\bbA^n_k$ vanishing at $0$.  The usual grading on $A=\KK[x_1,\ldots,x_n]$ decomposes $\pi$ into homogeneous components:
$$\pi=\pi^1+\pi^2+\pi^3+\ldots$$
where the coefficients $\pi^d(dx_i,dx_j)\in k_d[x_1,\ldots,x_n]$ of the $d^{\rm th}$ summand are homogeneous polynomials of degree $d$.  Note that the constant component $\pi^0$ vanishes because $\brac$ vanishes at the origin.

\begin{remark}
This grading does {\it not} agree with notion of the degree of a bracket introduced in Definition \ref{def:GrPoissonAlg}.  With the bracket corresponding to $\pi^d$ denoted by $\brac_{\pi^d}$, the $d^{\rm th}$ bracket $\brac_{\pi^d}$ has degree-$(d-2)$.
\end{remark}

One can now express the condition $[\pi,\pi]=0$ for $\pi$ to be a Poisson bivector in terms of the graded components:

\begin{eqnarray}
[\pi,\pi] &=& [\sum_{i\in\NN}\pi^i,\sum_{j\in\NN}\pi^j] \\
             &=& \sum_{i+j=k}[\pi^i,\pi^j]
\end{eqnarray}

Listing these conditions by $k\in\NN$ gives:\\

$\begin{array}{ll}
k=1: & [\pi^1,\pi^1]=0  \\
k=2: & [\pi^1,\pi^2]=0 \\
k=3: & [\pi^1,\pi^3]+\frac{1}{2}[\pi^2,\pi^2]=0 \\
k=4: & [\pi^1,\pi^4]+[\pi^2,\pi^3]=0 \\
k=5: & [\pi^1,\pi^5]+[\pi^2,\pi^4]+\frac{1}{2}[\pi^3,\pi^3]=0 \\
\vdots & 
\end{array}$\\

Analyzing the condition in each degree gives:
\begin{itemize}
\item[(1)] The degree-$1$ condition, $[\pi^1,\pi^1]=0$, is equivalent to the linear component $\pi^1$ being a Poisson bivector.  This linear Poisson structure is in fact the Lie-Poisson algebra associated to the {\it cotangent lie algebra} of the Poisson subscheme $\{0\}$.
\item[(2)] In light of the degree-$1$ condition, $\dd_{\pi^{1}}$ is a differential, and the degree-$2$ condition states that $\dd_{\pi^1}\pi^2=0$.  In other words, $\pi^2$ is a $\pi^1$-cocycle.
\item[\vdots] 
\item[(k)] The $k^{\rm th}$-degree condition for $k>2$: 
$$\dd_{\pi^1}\pi^k =\sum_{\stackrel{i+j=k+1}{i,j>1}}[\pi^i,\pi^j] $$

\end{itemize}

Thus a Poisson structure $\pi$ as above may be viewed as an infinitesimal deformation of its linear part $\pi^1$.

\subsection{Extending Poisson structures to infinitesimal neighborhoods of 0}
\label{sec:ExtPoissonInfNbd}

The graded decomposition above may be reinterpreted in terms of infinitesimal neighborhoods of the Poisson subscheme $\pt\subseteq\bbA^k_k$.  Using the notation from Section \ref{sec:InfNbd}, denote the Poisson bivector on $\nbd{k}{\pt}{\mathbb{A}^n_k}$ corresponding to $\brac_k$ by $\pi^{(k)}$.  A natural choice for a bivector on $\mathbb{A}^n_k$ that is $p_k$-related to $\pi^{(k)}$ is given by $\pi^1+\pi^2+\ldots+\pi^k\in\bigwedge^2\TT_{\mathbb{A}^n_k}$.  

\begin{eqnarray}
[\pi^{(k)},\pi^{(k)}] &=& p_{k*}\left([\pi^1+\ldots+\pi^k, \pi^1+\ldots+\pi^k]\right) \\
			   &=& p_{k*}\left([\pi^1,\pi^1]+2[\pi^1,\pi^2]+\ldots+\sum_{l=1}^{k}[\pi^l,\pi^{k-l}]\right)+\\
			   &  & \hspace{2cm} +p_{k*}\left(\sum_{2\leq l\leq k}[\pi^l,\pi^k]+\sum_{3\leq l\leq k}[\pi^l,\pi^{k-1}]+\ldots\right)
\end{eqnarray}

The second term (3.5) on the right-hand side vanishes, as each summand has coefficients in $\m_0^{k+1}$.  The first term (3.4) on the right-hand side vanishes {\it exactly when} the first $k$ degree conditions for $[\pi,\pi]=0$ are satisfied.  

Moreover, using the fact that $\pi^1$ defines a Poisson structure, one can view each $\pi^{(k)}$ as a solution of an equation: $${\rm d}_{\pi^1}\pi^{(k)}=[\phi_{k*}\pi^{(k-1)},\phi_{k*}\pi^{(k-1)}]$$

\begin{remark} The discussion above easily generalizes to the case where $Y=V(x_{i_1},\ldots,x_{i_k})$ is any coordinate subspace --- in this case, the relevant grading on $A$ is the one where $\{x_{i_1},\ldots,x_{i_k}\}$ have degree-1, and the other generators have degree-0.
\end{remark}

\begin{example}
Consider a Lie Poisson algebra $(\g,\pi^1)$ and the closed Poisson subscheme $\{0\}$.  The fact that Theorem \ref{thm:CECoh} holds on the level of cochains implies that the Poisson cohomology groups of infinitesimal neighborhoods of $\pt$ are isomorphic to {\it finite dimensional} cohomology groups:
$$H^*_{\pi^1}(\nbd{k}{\pt}{\g^*})\isom H^*_{CE}(\g,\Sym^{\leq k}\g)$$
\end{example}

\begin{remark}  This isomorphism allows one to reduce the problem of solving the condition $[\pi,\pi]=0$ to computing finite dimensional Chevalley-Eilenberg cohomology groups.
\end{remark}

\chapter{The Poisson Embedding Problem}
\label{sec:PoissEmbProb}

This chapter formulates and provides an initial approach to answering the motivating problem for this dissertation:

\begin{question}
\label{Q:SmEmb}
 Given a singular quotient of a smooth Poisson manifold by the linear action of a compact group, when does it smoothly embed as a Poisson submanifold of $\RR^n$ for some $n>0$ and some Poisson structure on $\RR^n$? 
\end{question}

This problem on {\it smooth} Poisson manifolds is then (locally) reduced to the following {\it algebraic} problem:

\begin{question}
\label{Q:AlgEmb}
Given a Poisson scheme $(X,\brac_X)$ and an embedding (in the category of schemes) of $X\inj\bbA^n_\RR$, when does there exist a Poisson structure on $\bbA^n_\RR$ extending $\brac_X$?
\end{question}

As our primary examples arise as quotients by subgroups of $U(1)$, we will discuss a technique for taking advantage of the natural complex coordinate --- further reducing the problem to finding {\it complex} algebraic Poisson embeddings of the complexification $X_\CC$ into $\CC^n$.  After discussing these reductions, we illustrate the problem with many examples framing the problem in the context of the previous chapters.

\section{Embedding Poisson Structures: Preliminaries}
\label{sec:PoissEmb}
In this chapter, we fix the following notation:
\begin{enumerate}
\item Denote by $V$ the real symplectic space $\bbA^{2n}_\RR$ with the  standard non-degenerate Poisson bracket $\brac_V$.  Denote the Poisson algebra of global sections by $\OO_V(V)$, or $\Sym(V^*)$
\item Let $G\subseteq{\rm Sp}(V)$ be a compact group acting linearly on $V$ by symplectic isomorphisms, and assume that $0$ is an isolated fixed point. 
\end{enumerate}

\begin{remark} Because $\pt$ is an isolated fixed point, the ring of polynomial invariants $\Sym(V^*)^G$ has no linear terms.
\end{remark}

The algebra of $G$-invariant polynomials on $V$ is finitely generated, by the Hilbert basis theorem.  Fixing a set of generators $\sigma:=\{\sigma_1,\ldots,\sigma_d\}\subseteq\Sym(V^*)^G$, we define the {\it Hilbert map} to be the map induced by the ring map: $$(x_1,\ldots,x_d)\mapsto(\sigma_1,\dots,\sigma_d)$$ 

The corresponding map of algebraic varieties is given by:
$$\begin{array}{l}
\tilde{F}_\sigma\!:V\to\RR^d \vspace{.5pc}\\

\tilde{F}_\sigma(v)=(\sigma_1(v),\ldots,\sigma_d(v))  \end{array}$$

As this map separates the orbits of $G$, the Hilbert map induces an embedding $F_\sigma\!:V/G\hookrightarrow\RR^d$ of the quotient.  This construction justifies the following definition:

\begin{definition}  The (algebraic) embedding dimension of $V/G$ is defined as
$$d={\rm min}\{l\in\NN\SC \mbox{ some } \sigma_1,\ldots,\sigma_l \mbox{ generates } \Sym(V^*)^G\}$$
\end{definition}

\begin{remark}  As $\Sym(V^*)^G$ is a finitely generated algebra, the minimal algebraic embedding dimension is finite.
\end{remark}

A useful characterization of the minimal embedding is given by the Zariski tangent space at the fixed point (see \cite{Hartshorne77}):

\begin{proposition}  A minimal algebraic embedding of $V/G$ is given by a choice of basis of the Zariski cotangent space $\m_0/\m_0^2$.  Therefore:
\begin{enumerate}
\item The Zariski tangent space at zero $\TT_{V/G,0}$ is naturally identified with the codomain of the minimal embedding of $V/G$
\item The embedding dimension of $V/G$ is given by $\dim_{\RR}(\TT_{V/G,0})$.
\end{enumerate}
\end{proposition}

\subsubsection{Smooth structures}
\label{sec:SmStr}

We give a definition, suitable for our purposes, of a smooth embedding of a (possibly singular) quotient:

\begin{definition} Let $M$ be a smooth manifold.  An injective map of sets $\phi\!:V/G\hookrightarrow M$ is called a {\it smooth embedding} if the pullback $\phi^*:C^\infty(P)\twoheadrightarrow C^\infty(V)^G$ is surjective.
\end{definition}

In the present context, the study of smooth embeddings reduces to the algebraic case described above: A theorem of Schwarz \cite{Schwarz75} states that $C^\infty(V)^G$ is generated by a finite number of {\it polynomial} invariants. That is, there exist $G$-invariant polynomials $\{\sigma_1,\ldots,\sigma_d\}$ that generate smooth functions on $V/G$ in the following sense: for all $F\in C^\infty(V)^G$, there exists an $\tilde{F}\in C^\infty(\RR^d)$ such that $F=\tilde{F}(\sigma_1,\ldots,\sigma_d)$.  Thus, any smooth embedding factors through a minimal algebraic embedding.

In more general contexts, the model for the smooth functions on a singular space is more subtle and understood using stratifications (see \cite{Pflaum01}).

\subsubsection{Poisson embedding}

Poisson structures behave very well with respect to quotients by group actions.  As $G$ acts on $V$ by symplectomorphisms, the Poisson bracket on $V$ descends to a Poisson bracket on $V/G$:  If $\phi,\psi$ are $G$-invariant functions on $V$, then $\{\phi,\psi\}_V$ is also a $G$-invariant function on $V$, as 
$$g^*\{\phi,\psi\}_V=\{g^*\phi,g^*\psi\}_V=\{\phi,\psi\}_V$$
Thus, $\brac_V$ restricts to a bracket on the subalgebra of $G$-invariant functions, defining a Poisson structure on the quotient $(V/G,\brac_{V/G})$.  This result holds in every category considered in previous sections (schemes, varieties, smooth spaces, analytic spaces, and formal schemes).  Thus, we have a uniform definition of a Poisson embedding of a quotient space $V/G$:

\begin{definition} Let $V$ and $G$ be as above, and $(P,\brac_P)$ be a Poisson space in a category $\mathcal{C}$ (e.g.  schemes, varieties, smooth manifolds, analytic spaces, or formal schemes).  A map $\phi: V/G\hookrightarrow P$ is a {\em $\mathcal{C}$-Poisson embedding} if it is an embedding and a Poisson map in $\mathcal{C}$.
\end{definition}

\begin{definition}
Define the {\it $\mathcal{C}$-Poisson embedding dimension} of $V/G$ to be the minimal dimension of any $\mathcal{C}$-Poisson embedding --- that is $\min\{d \SC \exists \phi:V/G\hookrightarrow P \mbox{ and } \dim P=d\}$.  If no embedding exists, then we say the $\mathcal{C}$-Poisson embedding dimension is infinity.
\end{definition}

In \cite{Sjamaar91}, R. Cushman conjectured that given an embedding $\phi:V/G\hookrightarrow\RR^n$, there is a Poisson structure on $\RR^n$ making $\phi$ a {\em Poisson embedding}.  A. Egilsson provided a negative answer to this conjecture in \cite{Egilsson95}, using an example that will be described in the following Section \ref{sec:Complexification}.

The answer to the following related question is still not known:

\begin{question}  Given a Poisson quotient $V/G$ in the category of smooth spaces is the smooth Poisson embedding dimension finite?  That is, can one realize $V/G$ as a Poisson subspace of some smooth Poisson structure on $\RR^n$?
\end{question}

In fact, the answer is unknown in every category considered in this dissertation.

\subsection{Reduction to the Algebraic Case}
\label{sec:ReductionAlg}

This reduction follows the same technique developed in \cite{Davis01}.  In the context of this dissertation, the singular quotients $V/G$ behave well enough that their smooth functions $C^\infty(V/G)$ as defined in Section \ref{sec:SmStr} are locally well-approximated by power series.

\begin{theorem}  Let  $A=\OO_V(V)^G$ be the algebra of $G$-invariant global sections of $\OO_V$. Then the inclusion
$$ \hat{\iota}: \widehat{A}_{\iota^{-1}(\m_0)}\to \widehat{C^{\infty}(V/G)}_{\m_0} $$
of the completions of the local rings of functions at $0$ is an isomorphism. 
\end{theorem}

By Theorem 4.2.15 in \cite{Davis01}, such an isomorphism is obtained any time the algebra $A$ is 
\begin{enumerate}
\item a subalgebra of a polynomial algebra 
\item finitely generated as an algebra by homogeneous polynomials.  
\end{enumerate}
These conditions are guaranteed by Schwarz \cite{Schwarz75} as described above.  

Within the context of embeddings, this isomorphism implies that any embedding $\phi:V/G\inj\RR^n$, after applying the completion functor, gives rise to a local embedding:

$$\xymatrix{
\widehat{C^{\infty}(\RR^n)}_{(x_1,\ldots,x_n)} \ar@{->>}[d]^{\phi^*} \ar[r]^{\isom} & \RR[[x_1,\ldots,x_n]] \ar@{-->>}[d] \\
\widehat{C^{\infty}(V/G)}_{\m_0}  \ar[r]^{\isom} & \widehat{A}_{\iota^{-1}(\m_0)} \\
}$$
Moreover, as the horizontal isomorphisms respect the bracket structures (see Section \ref{sec:PoissonAlg}, or \cite{Davis01}), if $\phi$ is a Poisson embedding, the above diagram induces a local Poisson embedding.  Thus we obtain a relationship between smooth and formal embedding dimensions.

\begin{theorem}[Davis \cite{Davis01}] 
If $d$ is the smooth Poisson embedding dimension of $V/G$, then the formal Poisson embedding dimension is less than or equal to $d$ at every point $p\in V/G$.
\end{theorem}

This theorem allows us to restrict our attention to algebraic and formal Poisson structures, as either:

\begin{enumerate}
\item we find an algebraic or formal Poisson embedding of $V/G\inj\RR^n$, or
\item we find an obstruction to formally extending the Poisson structure on $V/G$, which by the above theorem is also an obstruction to smoothly extending the Poisson structure on $V/G$.
\end{enumerate}

\begin{remark} It is possible that such a procedure could result in a third possibility: a formal Poisson structure on a neighborhood $V/G$ that is not extendable to a Poisson structure on the whole space.   Such an example has yet to be found, but would be quite interesting.
\end{remark}

\subsection{The Complexification Trick}
\label{sec:Complexification}

Many of the motivating examples of Poisson quotients arise from actions of subgroups $G\subseteq U(n)$.  When working with such actions, the natural coordinates to consider are the complex coordinates.  However, making use of this extra structure requires a development of how a Poisson embedding behaves under complexification (see also \cite{Egilsson95}).

Recall from Proposition \ref{prop:PoissonTProd}, that the tensor product of Poisson algebras over a field $\KK$ is a Poisson algebra.  Applying this proposition to a Poisson algebra $(A,\brac)$ over $\RR$ and the trivial Poisson algebra $\CC$ over $\RR$, we obtain a Poisson algebra $A\otimes_\RR\CC$ over $\RR$.  One may also consider $A\otimes_\RR\CC$ as a Poisson algebra over $\CC$, and the resulting bracket $\brac_\CC$ is precisely the $\CC$-linear extension of $\brac$.  Thus, if $(X,\brac)$ is a real affine Poisson variety, its complexification $(X_\CC,\brac_\CC)$ is naturally a complex affine Poisson variety.  Moreover, this construction is functorial: given a Poisson map of real Poisson varieties $F:X\to Y$, the induced map between their complexifications $F_\CC:X_\CC\to Y_\CC$ is a Poisson map.

\begin{example}[K\"{a}hler coordinates] 
\label{ex:Complexification}
Let $V_1$ and $V_2$ be real $2n$-dimensional vector spaces (considered as affine varieties), with coordinate rings 
$$\OO_{V_1}(V_1)=\RR[x_1,\ldots,x_n,y_1,\ldots, y_n] \mbox{ and } \OO_{V_2}(V_2)=\RR[u_1,\ldots,u_n,v_1,\ldots, v_n]$$ 
We give $V_1$ the standard symplectic structure and define the map between the complexifications: $$\kappa^*:\CC[x_1,\ldots,x_n,y_1,\ldots, y_n] \to \CC[u_1,\ldots,u_n,v_1,\ldots, v_n]$$

$$\kappa^*(w)=\left\{\begin{array}{l} \frac{1}{2}(u_k+v_k) \mbox{,   if $w=x_k$ } \\ \frac{-i}{2}(u_k-v_k) \mbox{,   if $w=y_k$ } \end{array}\right.$$

The discussion above implies that $\brac$ determines a unique Poisson bracket $\brac_\CC$ on $\OO_{V_2}(V_2)\otimes \CC$.  Denoting $\kappa^*(x_i)=z_i$ and $\kappa^*(y_i)=\zb_i$, the resulting Poisson structure on $(V_2)_\CC$ is given by the standard K\"{a}hler bracket:
$$ \{z_j,\zb_k\}=-2i\delta_{jk} \mbox{ and \hspace{4pt} } \{z_j,z_k\}=\{\zb_j,\zb_k\}=0  \mbox{ for all $j,k$ }$$
\end{example}

\begin{example}  A Poisson embedding of real Poisson varieties $F:X\inj\RR^d$ induces a Poisson embedding of complex Poisson varieties:

$$\xymatrix{
X_\CC \ar@{^{(}->}[r]^{F_\CC} \ar[d] & \CC^d \ar[d] \\
X \ar@{^{(}->}[r]^{F} & \RR^d }
$$ 

As complex coordinates are often more natural, we will usually work with the induced embedding of complex varieties $(X_\CC,\brac_\CC)$ and use this one-to-one correspondence between Poisson embeddings of the real variety $X\inj\RR^d$ and certain Poisson embeddings of the complexification $X_\CC\inj\CC^d$.
\end{example}

\section{Examples of Poisson Embeddings}

The following example, considered in \cite{Sjamaar91}, describes a quotient that Poisson embeds into a linear Poisson structure:
\begin{example}[Quadratic invariants] 
\label{ex:Quadratic}
Assume that the $G$-invariant functions on $V$ are generated by {\em quadratic} polynomials $\sigma_1,\ldots,\sigma_n$.  As the Poisson bracket $\brac_V$ has degree-$(-2)$, the Hamiltonian vector fields $X_{\sigma_i}:=\{\cdot,\sigma_i\}$ are linear; they span a Lie subalgebra $\h\subseteq\mathfrak{sp}(V)$.  The moment map $\mu:\! V\to\h^*$ is precisely the Hilbert map $F_\sigma$.  Moreover, as moment maps are homomorphisms of Poisson algebras, the induced map on the quotient $\mu:\! V/G\hookrightarrow\h^*\isom\RR^n$ is a Poisson embedding, identifying $V/G$ as a nilpotent cone in $\h^*$.
\end{example}

\subsection{Embedding 2-Dimensional Symplectic Orbifolds}

The Hilbert map associated to a low dimensional dimensional Poisson quotient often defines a Poisson embedding.  The following treats 2-dimensional (real and complex) orbifolds.

\subsubsection{ Real Symplectic Orbifolds}

Every local model $\RR^2/\Gamma$ for a real 2-dimensional symplectic orbifold Poisson embeds into $\RR^3$.  The 2-dimensional local models for general orbifolds are quotients $\RR^2/\ZZ_n$ and $\RR^2/D_{2n}$ with the standard real representations. However, only the action of $\ZZ_n$ is compatible with the symplectic structure on $\RR^2$.

\begin{example}[The 2-dimensional real orbifold model $\RR^2/\ZZ_n$]  
\label{ex:Z_n}
Consider the linear action of $\ZZ_n$ on $\RR^2$ by rotation: the generator $\zeta\in\ZZ_n$ acts via the matrix
$$\Phi_\zeta=\left(\begin{array}{cc} \cos (2\pi/n) & \sin(2\pi/n) \\ -\sin (2\pi/n) & \cos (2\pi/n) \end{array}\right)$$
The coordinate transformation $\kappa:\RR^2\to\RR^2\otimes\CC$ in example 2.3 becomes the matrix
$\kappa~=~\frac{1}{2}\left(\begin{array}{cc} 1 & -i \\ 1 & i \end{array}\right)$.  The induced Poisson action on the complexification $\RR^2\otimes\CC$ becomes the standard action of $\ZZ_n\subseteq{\rm SU}(2)$:
$$\kappa\Phi_\zeta\kappa^{-1}=\left(\begin{array}{cc} \zeta & 0 \\ 0 & \zeta^{-1} \end{array}\right) \mbox{ where $\zeta=e^\frac{2\pi i}{n}$}$$
 A minimal generating set (i.e.\! a Hilbert basis) for the algebra of invariant functions consists of polynomials: 
$$\sigma_0=z\zb \mbox{,  \hspace{4pt}    } \sigma_1=z^n+\zb^n \mbox{,   \hspace{4pt}    }  \sigma_2=i(z^n-\zb^n)   $$
These define an embedding into $\RR^3$ via the Hilbert map:
$$F_\sigma\!:\RR^2\isom\CC\to\RR^3\subseteq\CC^3$$
$$F_\sigma(z,\zb)=(\sigma_0(z,\zb),\sigma_1(z,\zb),\sigma_2(z,\zb))$$
and the quotient is realized as the vanishing locus $V(y^2+z^2-4x^n)\subseteq\RR^3$.  The Poisson bracket on the algebra of invariants is given by:
$$\{\sigma_0,\sigma_1\}=-2n\sigma_2 \mbox{, \hspace{4pt}  } \{\sigma_0,\sigma_2\}=2n\sigma_1  \mbox{, \hspace{4pt}  } \{\sigma_1,\sigma_2\}=(2n)^2\sigma_0^{n-1}$$
which extends to the Poisson structure on $\RR^3$ given by:
$$\{x,y\}_{\RR^3}=-2nz \mbox{,\hspace{4pt} } \{x,z\}_{\RR^3}=2ny \mbox{, \hspace{4pt} } \{y,z\}_{\RR^3}=(2n)^2 x^{n-1} $$
That this bracket satisfies the Jacobi identity is easily verified on the generators $\{x,y,z\}$:
$$\begin{array}{rcccccc}
\{x,\{y,z\}_{\RR^3}\}_{\RR^3} + {\rm c.p. } &=& \{x,(2n)^2x^{n-1}\}_{\RR^3} &+& \{z,-2nz\}_{\RR^3} &+& \{y,-2ny\}_{\RR^3}\\ 
&=& 0 &+& 0 &+& 0  \\
\end{array}$$
\end{example}

\begin{remark} If $n=2$, then the resulting quotient $V/\ZZ_2$ falls within the assumptions of Example \ref{ex:Quadratic}.  The resulting Poisson structure on $\RR^3$ is isomorphic to the Lie-Poisson structure on $\slt(\RR)^*$ and the image of the embedding is the nilpotent cone.
\end{remark}

\subsubsection{ Holomorphic Symplectic Orbifolds }

Other low dimensional examples arise as holomorphic Poisson structures on 2-dimensional complex orbifolds.  Let $V$ be a 2 dimensional {\it complex} symplectic vector space and $\Gamma\subseteq{\rm Sp}(2,\CC)$ a finite group. The resulting quotient $V/\Gamma$ is a local model for a holomorphic symplectic orbifold.

\begin{proposition} Every such ``holomorphic symplectic orbifold'' of dimension 2 has a Poisson embedding into $\CC^3$.  
\end{proposition}

As $\dim V=2$, the groups ${\rm Sp}(V)$ and ${\rm SL}(V)$ are equal.  Quotients of finite subgroups $\Gamma\subseteq{\rm SL}(V)$ in dimension two are the ``Kleinian singularities'' and are classified by their singularity type (i.e.\! ADE-singularities).  For all such quotients, the algebraic embedding dimension is equal to the Poisson embedding dimension. The proof of this equality follows from two specific properties of these quotients: 
\begin{enumerate}
\item The vanishing ideals of the quotients, by Krull's Hauptidealsatz, are principal (i.e. generated by a single polynomial), forcing both the resulting singularity and the moduli space of extensions to be well-behaved.
\item The low dimension of the embedding space allows one to easily verify the Jacobi identity of any extended bracket.
\end{enumerate}

The computations of the invariants are classical (e.g. see Chapter 2 of \cite{Klein93}).

\begin{example}[$A_n$-singularities]
\label{ex:A_n}
  Let $\Gamma=\ZZ_n$ act on $V$ by multiplication: 
  $$\zeta\cdot(x,y)=(\zeta x,\zeta^{-1} y)  \mbox{\quad where  } \zeta\in\Gamma $$
      A Hilbert basis for the action consists of polynomials: 
$$\sigma_0=xy \mbox{,  \hspace{4pt}    } \sigma_1=x^n \mbox{,   \hspace{4pt}    }  \sigma_2=y^n   $$
The induced Hilbert map $F_\sigma$, whose dual is given by
$$F_\sigma^*=\left\{\begin{array}{l} u\mapsto\sigma_0 \\ v\mapsto\sigma_1 \\ w\mapsto\sigma_2 \end{array}\right.$$
 realizes $\CC^2/\ZZ_n$ as the vanishing locus $V(vw-u^{n})$. The obvious extension of the Poisson bracket $\brac_{V/\Gamma}$ to $\CC^{3}$ is given by:
$$\{u,v\}=-nv \mbox{, \hspace{4pt}  } \{u,w\}=nw  \mbox{, \hspace{4pt}  } \{v,w\}=n^2 u^{n-1}$$

Checking the Jacobi identity by hand verifies that this defines a Poisson bracket on $\CC^{3}$.   In Example \ref{ex:ClassifyA_n} {\it all Poisson structures} on $\CC^3$ extending the one above are classified.
\end{example}

\begin{example}[$D_{(n+2)}$-singularities] Let the binary dihedral group $BD_{4n}\subseteq{\rm Sp}(V)$ be the finite group generated by the transformations:
$$  R_n=\left(\begin{array}{cc} \zeta_n & 0\\ 0 & \zeta^{-1}_n\end{array}\right)  \mbox{and }  T=\left(\begin{array}{cc} 0 & i\\ i & 0\end{array} \right),\mbox{ where } \zeta_n=e^\frac{2\pi i}{n}$$
The $BD_{4n}$-invariant polynomials must be invariant under both generators $R_n$ and $T$.  A minimal generating set for the polynomials invariant under the action of $R_n$ is given by $\{x^n,xy,y^n\}$, while the analogous set for $T$ is given by $\{x^2-y^2,xy(x^2+y^2),(xy)^2\}$.  Thus a Hilbert basis of the $BD_{4n}$-invariant polynomials is given by:
$$\sigma_0=(xy)^2 \mbox{,  \hspace{4pt}    } \sigma_1=x^{2n}+(-1)^ny^{2n} \mbox{,   \hspace{4pt}    }  \sigma_2=xy\left(x^{2n}-(-1)^ny^{2n}\right)  $$
The Hilbert map $F_\sigma$ realizes $V/BD_{4n}$ as the vanishing locus $V(4u^{n+1}+w^2-uv^2)\subseteq\CC^3$.  The Poisson bracket on $V/BD_{4n}$ extends to the Poisson bracket on $\CC^3$:
$$\{u,v\}=-4nw \mbox{, \hspace{4pt} } \{u,w\}=-4nuv \mbox{, \hspace{4pt} } \{v,w\}=-4n\left(2(n-1)(-1)^nu^n-\frac{1}{2}v^2\right) $$
One can easily check directly that this extended bracket satisfies the Jacobi identity.
\end{example}

\begin{example}[$E_6$-singularities] The 2-dimensional $E_6$-singularities are realized by the action of the Binary Tetrahedral group $BT_{24}$, generated by the following matrices:
$$S=\left(\begin{array}{cc} i & 0 \\ 0 & -i \end{array}\right) \mbox{, \hspace{1pt}} T=\left(\begin{array}{cc} 0 & -1 \\ 1 & 0 \end{array}\right) \mbox{, and \hspace{1pt} } U=\frac{1}{2}\left(\begin{array}{cc} 1+i & -1+i \\ 1+i & 1-i \end{array}\right)$$
A Hilbert basis of the $BT_{24}$-invariant polynomials is given by:
$$\sigma_0=xy(x^4-y^4) \mbox{,  \hspace{4pt}    } \sigma_1=x^{8}+14x^4y^4+y^8 \mbox{,   \hspace{4pt}    }  \sigma_2=x^{12}-33x^8y^4-33x^4y^8+y^{12}$$
The Hilbert map $F_\sigma$ realizes $V/BT_{24}$ as the vanishing locus $V(108u^4-v^3+w^2)\subseteq\CC^3$.  The Poisson bracket on $V/BT_{24}$ extends to the Poisson bracket on $\CC^3$:
$$\{u,v\}=-8w \mbox{, \hspace{4pt} } \{u,w\}=-12v^2 \mbox{, \hspace{4pt} } \{v,w\}=-1728u^3 $$
\end{example}

\begin{example}[$E_7$-singularities] The 2-dimensional $E_7$-singularities are realized by the action of the Binary Octahedral group $BO_{48}$, generated by $BT_{24}$ (i.e. the matrices $S,T,$ and $U$) and $R_8=\left(\begin{array}{cc} \zeta_8 & 0 \\ 0 & \zeta^{-1}_8\end{array}\right)$.  Noting that $R_8^2=S$, we find that the Hilbert basis is given by a subset of the $BT_{24}$-invariant polynomials:
$$\sigma_0=\left(xy(x^4-y^4)\right)^2 \mbox{,  \hspace{2pt}    } \sigma_1=x^{8}+14x^4y^4+y^8 \mbox{,   \hspace{2pt}    } $$ $$ \sigma_2=\left(x^{12}-33x^8y^4-33x^4y^8+y^{12}\right)\left(xy(x^4-y^4)\right)$$
The Hilbert map $F_\sigma$ realizes $V/BO_{48}$ as the vanishing locus $V(uv^3-108u^3-w^2)\subseteq\CC^3$.  The Poisson bracket on $V/BO_{48}$ extends to the Poisson bracket on $\CC^3$:
$$\{u,v\}=-16w \mbox{, \hspace{4pt} } \{u,w\}=-24uv^2 \mbox{, \hspace{4pt} } \{v,w\}=-2592u^2+8v^3 $$
\end{example}

\begin{example}[$E_8$-singularities] The 2-dimensional $E_8$-singularities are realized by the action of the Binary Icosahedral group $BI_{120}$, generated by:
$$V=\left(\begin{array}{cc} \zeta^3_5 & 0 \\ 0 & \zeta_5^2 \end{array}\right) \mbox{, and \hspace{1pt} } W=\frac{1}{\sqrt{5}}\left(\begin{array}{cc} -\zeta_5+\zeta_5^4 & \zeta_5^2-\zeta_5^3 \\ \zeta_5^2-\zeta_5^3 & \zeta_5-\zeta_5^4 \end{array}\right) \mbox{, where } \zeta_5=e^\frac{2\pi i}{5} $$
A Hilbert basis of the $BI_{120}$-invariant polynomials is given by:
$$\left\{\begin{array}{l}
\sigma_0=x^{30}+522x^{25}y^5-10005x^{20}y^{10}-10005x^{10}y^{20}-522x^5y^{25}+y^{30} \\
\sigma_1=x^{20}-228x^{15}y^5+494x^{10}y^{10}+228x^5y^{15}+y^{20} \\
\sigma_2=xy\left(x^{10}+11x^5y^5-y^{10}\right) 
\end{array}\right.$$
The Hilbert map $F_\sigma$ realizes $V/BI_{120}$ as the vanishing locus $V(u^2+\sigma^3_1-1728w^5)\subseteq\CC^3$.  The Poisson bracket on $V/BI_{120}$ extends to the Poisson bracket on $\CC^3$:
$$\{u,v\}=-86400w^4 \mbox{, \hspace{4pt} } \{u,w\}=30v^2 \mbox{, \hspace{4pt} } \{v,w\}=20u $$

\end{example}

\begin{remark} Higher dimensional orbifolds do not behave so nicely.  For example, in dimension 4, Theorem \ref{thm:NonEmb} proves that any minimal embedding $\RR^4/\ZZ_n\hookrightarrow\RR^{2n+6}$ cannot be a Poisson embedding.
\end{remark}

\subsubsection{$S^1$-resonance Spaces}
\label{ex:ResSpaces}

Denote the coordinate functions on $\RR^{2n}$ by $\{x_1,y_1,\ldots,x_n,y_n\}$ and fix $(d_1,\ldots,d_n)\in\ZZ^n$.  Define the $(d_1,\ldots,d_n)$-resonance action on $\RR^{2n}$ to be the linear, symplectic $S^1$-action given by a block matrix built from 2-dimensional rotation matrices $\theta\mapsto A_{d,\theta}=(A_{d_1,\theta},\ldots,A_{d_n,\theta})$, where $$A_{d_k,\theta}=\left(\begin{array}{cc} \cos (2\pi d_k \theta) & \sin(2\pi d_k \theta) \\ -\sin (2\pi d_k \theta) & \cos (2\pi d_k \theta) \end{array}\right)$$
Using the change of coordinates in example 2.3 and passing to complex coefficients rewrites the action as
$$\theta\cdot(z_1,\ldots,z_n,\zb_1,\ldots,\zb_n)=(e^{2\pi i\theta d_1} z_1\ldots,e^{2\pi i\theta d_n}z_n,e^{-2\pi i\theta d_1}\zb_1,\ldots,e^{-2\pi i\theta d_1}\zb_n)$$  The  quotient from a general $S^1$-resonance action is a cone over weighted projective $n$-space $\mathbb{CP}(d_1,\dots,d_n)$; Hilbert bases are determined in \cite{Egilsson00}.

\begin{example}[$(a,b)$-resonance space]  
Let $n=1$ and $(d_1,d_2)=(a,b)\in\NN$.  A Hilbert basis of invariants is
$$\sigma_1=z\zb \mbox{, \hspace{4pt}} \sigma_2=w\wb \mbox{, \hspace{4pt}} \sigma_3=(z^b\wb^a+\zb^b w^a) \mbox{, \hspace{4pt}} \sigma_4=i(z^b\wb^a-\zb^b w^a) $$  This basis embeds the quotient $\RR^4/S^1\isom\CC^2/S^1$ as $V(x_3^2+x_4^2-4x_1^bx_2^a)\subseteq\RR^4\subseteq\CC^4$.  Additionally, there is a Poisson structure on $\RR^4$ making this a Poisson embedding:
$$\{x_1,x_2\}=0    \mbox{, \hspace{4pt}}  \{x_1,x_3\}= 2bx_4    \mbox{, \hspace{4pt}}  \{x_1,x_4\}= -2bx_3   $$
$$\{x_2,x_3\}= -2ax_3   \mbox{, \hspace{4pt}}  \{x_2,x_4\}=2ax_4    \mbox{, \hspace{4pt}}  \{x_3,x_4\}= 2x_1^{b-1}x_2^{a-1}(b^2x_2-a^2x_1)  $$
The Jacobi identity can be checked directly.
\end{example}

\begin{remark}  If $(a,b)=(1,1)$, then the invariants are quadratic.  Applying the result of Example \ref{ex:Quadratic}, the quotient $\CC^2/S^1$ embeds as the nilpotent cone in $\mathfrak{u}(2)^*$.  The Hilbert map associated to the invariants $\{z\zb,w\wb,z\wb,\zb w\}$ is precisely the moment map 
$$\mu:\!\CC^2\to\mathfrak{u}(2)^*\isom\mathfrak{u}(2)$$
$$(z,w)\mapsto i\left(\begin{array}{cc} z\zb & \zb w \\ z\wb & w\wb \end{array}\right)$$
Thus, the image of the quotient $\mathcal{N}:=\CC^2/S^1$ is given by the semi-algebraic set 
$$\mathcal{N}=\{A\in\mathfrak{u}(2) \SC \det(A)=0 \mbox{ and } i{\rm tr}(A)\geq 0\}$$
Furthermore, the moment map for the $S^1$ action $\mu_{S^1}\!=-i({\rm tr}\circ\mu):\!\CC^2\to\RR$ describes the symplectic leaves of this nilpotent cone.  The symplectic leaves $S_\lambda=i{\rm tr}^{-1}(\lambda)\cap\mathcal{N}\subseteq\mathfrak{u}(2)$ are isomorphic to $\mu_{S^1}^{-1}(\lambda)/S^1\isom S^2$ for $\lambda\neq 0$ and $\pt$ for $\lambda=0$.
\end{remark}

\begin{example}[(1,1,2)-resonance space]   Unlike in the examples above, the Poisson embedding dimension of the $(1,1,2)$-resonance space {\it does not} equal the smooth embedding dimension.   Any Hilbert basis for the $(1,1,2)$-action has 11 elements, giving rise to an embedding $F:\!\CC^3/S^1\hookrightarrow\RR^{11}$ (the minimal smooth embedding dimension).   In \cite{Egilsson95}, A. Egilsson proves that there cannot exist a Poisson structure on $\RR^{11}$ extending the Poisson structure on the embedded variety.  Later, in \cite{Davis02}, B. Davis strengthens this result: Given an embedding of the $(1,1,2)$-resonance space into $\RR^{12}$, there does not exist a Poisson structure on $\RR^{12}$ extending the Poisson structure on the quotient.
\end{example}

\section{The Poisson Embedding Problem: First Approach}

\subsection{Infinitesimal Extensions of Embedded Poisson Varieties}

As discussed in the previous section, it is generally quite difficult to determine when an embedding can be turned into a Poisson embedding. The remainder of this thesis is concerned with this question:  

\begin{question} Given an embedding of a singular symplectic quotient $F\!:\!V/G\hookrightarrow\RR^n$, when is there a Poisson structure on $\RR^n$ such that $\phi$ becomes a Poisson embedding?
\end{question}

As a first approach to finding such a Poisson structure $(\RR^n,\brac)$, one can attempt to extend the Poisson structure on $V/G\isom{\rm Im}(F)\subseteq\RR^n$ to the infinitesimal neighborhoods $\mathcal{N}_{V/G}^k(\RR^n)$ of $V/G$ in $\RR^n$.  As mentioned in Chapter 3, a Poisson embedding $F\!:V/G\hookrightarrow\RR^n$ gives rise to a chain of Poisson maps:
$$V/G\hookrightarrow\mathcal{N}_{V/G}^1(\RR^n)\hookrightarrow\mathcal{N}_{V/G}^2(\RR^n)\hookrightarrow\ldots\hookrightarrow\mathcal{N}_{V/G}^\infty(\RR^n)\hookrightarrow\RR^n$$
on each infinitesimal neighborhood $\mathcal{N}_{V/G}^k(\RR^n)\!:=\Spec(A/I^k)$, where $I:=\ker(F^*)$.  A procedure for building successive extensions goes as follows:

We begin by extending the Poisson bracket $\brac_{V/G}$ on $V/G$ arbitrarily to a bilinear bracket $\brac$ on $\RR^n$ (that almost surely will {\it not} satisfy the Jacobi identity) and denote the bivector corresponding to $\brac$ by $\pi$.  All extensions of the Poisson structure on $V/G$ are of the form $\pi+\pi_I$, where $\pi_I$ has coefficients in $I$.  Therefore we may try to solve the equation
$$[\pi+\pi_I,\pi+\pi_I]=0$$ 
for $\pi_I$.   Moreover, we may also attempt to solve for $\pi_I$ step-by-step:

\begin{enumerate}
\item[0.] The bivector $\pi$ may be considered the zeroth extension $\pi^{(0)}\!:=\pi$, as the Schouten bracket vanishes $[\pi,\pi]=0 \mod I$.  Thus, $\pi^{(0)}$ defines a Poisson bracket on $\mathcal{N}_{V/G}^0(\RR^n)\isom V/G$.
\item[1.] The first extension would be a bivector $\pi^{(1)}:=\pi^{(0)}+\pi_I^1$ such that 
\begin{enumerate}
\item $\pi^{(1)}$ defines a Poisson structure on $\mathcal{N}_{V/G}^1(\RR^n)$.  In other words, $[\pi^{(1)},\pi^{(1)}]=0\mod I^2$).
\item $\pi^1_I$ has coefficients in $I$, so that $\pi^{(1)}$ is an extension of $\pi^{(0)}$.
\end{enumerate}
\item[k.]  Similarly, the $k^{\rm th}$ extension is a bivector $\pi^{(k)}:=\pi^{(k)}+\pi_I^k$ such that $\pi^{(k)}$ defines a Poisson structure on $\mathcal{N}_{V/G}^k(\RR^n)$ --- that is, $[\pi^{(k)},\pi^{(k)}]=0\mod I^k$) --- and $\pi_I^k$ has coefficients in $I$.
\end{enumerate}

The limit $\pi^{(\infty)}:=\varprojlim\pi^{(k)}$ then defines a Poisson structure on a formal neighborhood of $V/G\subseteq\RR^n$.

\begin{remark}  This procedure is in general very difficult, as the ideals $I^k$ are usually not easy to describe.  However, various refinements of this procedure are used later.
\end{remark}

To illustrate the work involved in solving $[\pi+\pi_I,\pi+\pi_I]=0$, consider the 2-dimensional holomorphic symplectic orbifold $\CC^2/\ZZ_n\subseteq\CC^3$ of Example \ref{ex:A_n}.

\subsection{Classification of Minimal Poisson Embeddings of $\CC^2/\ZZ_n$}

\begin{example}[The moduli of Poisson structures extending $\CC^2/\ZZ_n$]  
\label{ex:ClassifyA_n}
In Example \ref{ex:A_n}, a Poisson structure on $\CC^3$ is defined so that the embedding $F_\sigma\!:\!\CC^2/\ZZ_n\hookrightarrow\CC^3$ is a Poisson embedding. One can further ask if it is possible to determine all Poisson structures extending the quotient Poisson structure on $\CC^2/\ZZ_n$.  

\subsubsection{Multivector fields on $\CC^3$}
To investigate this problem, we take advantage of a few peculiarities of the deRham sequence of 3-dimensional affine space (see \cite{Pichereau06} for details):  Let $A=\CC[u,v,w]$ be the algebra of polynomials on $\CC^3$ and denote the $k$-multivectors on $\CC^3$ by $\mathfrak{X}^k(\CC^3):=\bigwedge^k\TT_{\CC^3}$.   Recall that $\mathfrak{X}^k(\CC^3)\isom A$ for $k=0,3$ and $\mathfrak{X}^j(\CC^3)\isom A^3$ for $j=1,2$.  Using these identifications, define the following operations (using the analogous, familiar, vector operations on $\CC^3$):  
\begin{enumerate}
\item Given a vector field $X=f\frac{\partial}{\partial u}+g\frac{\partial}{\partial v}+h\frac{\partial}{\partial w}\in\mathfrak{X}^1(\CC^3)$, define the function
$${\rm div}(X)=\left(\frac{\partial f}{\partial u}+\frac{\partial g}{\partial v }+\frac{\partial h}{\partial w}\right)\in\mathfrak{X}^0(\CC^3)\isom A$$
\item Given a bivector field $\beta=f\frac{\partial}{\partial  u}\wedge\frac{\partial}{\partial v }+g\frac{\partial}{\partial u }\wedge\frac{\partial}{\partial w }+h\frac{\partial}{\partial v }\wedge\frac{\partial}{\partial w }\in\mathfrak{X}^2(\CC^3)$, define the vector field
$${\rm Curl}(\beta)= \left(\frac{\partial f}{\partial v }-\frac{\partial g}{\partial w }\right)\frac{\partial}{\partial u }+\left(\frac{\partial h}{\partial w }-\frac{\partial f}{\partial u }\right)\frac{\partial}{\partial v }+\left(\frac{\partial g}{\partial u }-\frac{\partial h}{\partial v }\right)\frac{\partial}{\partial w }\in\mathfrak{X}^1(\CC^3)$$
\item Given a trivector field $\tau=f\frac{\partial}{\partial u }\wedge\frac{\partial}{\partial v }\wedge\frac{\partial}{\partial w }\in\mathfrak{X}^3(\CC^3)$, define the bivector field
$$\nabla(\tau)=\frac{\partial f}{\partial w }\frac{\partial}{\partial u }\wedge\frac{\partial}{\partial v }+\frac{\partial f}{\partial v }\frac{\partial}{\partial w }\wedge\frac{\partial}{\partial u }+\frac{\partial f}{\partial u }\frac{\partial}{\partial v }\wedge\frac{\partial}{\partial w }$$
\end{enumerate}

Using the identification of multi-derivations and differential forms $\mathfrak{X}^{3-k}(\CC^3)\isom\Omega^k(\CC^3)$ given by the star operator $*:\mathfrak{X}^{3-k}(\CC^3)\to\Omega^k(\CC^3)$, there is an isomorphism of exact sequences:

\[
\xymatrix{
0 & \Omega^3(\CC^3) \ar[l] & \Omega^2(\CC^3) \ar[l]_d & \Omega^1(\CC^3) \ar[l]_d & \Omega^0(\CC^3) \ar[l]_d & 0 \ar[l] \\
0 & \mathfrak{X}^0(\CC^3) \ar[l] \ar@{<->}[u]_{*} &\mathfrak{X}^1(\CC^3) \ar[l]^{{\rm div}} \ar@{<->}[u]_{*} & \mathfrak{X}^2(\CC^3) \ar[l]^{{\rm Curl}} \ar@{<->}[u]_{*} & \mathfrak{X}^3(\CC^3) \ar[l]^{\nabla} \ar@{<->}[u]_{*} & 0 \ar[l]
}
\]

\subsubsection{Characterization of extensions}
We recall the set up for the Poisson embedding $F_\sigma: \CC^2/\ZZ_n\hookrightarrow\CC^3$.  The Hilbert basis is given by $\{\sigma_0=xy,\sigma_1=x^n,\sigma_2=y^n\}$, and the map $F_\sigma$ is defined by 
$$\begin{array}{rcccc}
		      &&  u&v&w\\
F_\sigma(x,y) &=& (\sigma_0(x,y),&\sigma_1(x,y),&\sigma_2(x,y))\\
			\end{array}$$

This embedding realizes the quotient as the vanishing locus $V(\phi)\subseteq\CC^3$, where $\phi=u^n-vw$.  Recall that the extension of $\brac_{\CC^2/\ZZ_n}$ given by the bivector below defines a Poisson structure:

$$\pi=-2v \frac{\partial}{\partial u }\wedge\frac{\partial}{\partial v }-2w \frac{\partial}{\partial w }\wedge\frac{\partial}{\partial u } + u^{n-1}\frac{\partial}{\partial v }\wedge\frac{\partial}{\partial w }$$

This is a very special type of Poisson bivector: 
$$\pi=\nabla(*\phi)  \mbox{  , where } \phi=(vw-u^n) \mbox{ is a weighted homogeneous polynomial.}$$
The Poisson cohomology and deformation theory of such Poisson structures are computed in \cite{Pichereau06}.  Now, consider another bivector $\pi'$ corresponding to an extension of $\brac_{\CC^2/\ZZ_n}$.  As $\pi$ and $\pi'$ agree on $V/G$, it follows that $\pi'=\pi+\pi_\phi$, where $\pi_\phi$ vanishes on $V/G$.  Such a bivector $\pi_\phi$ has the form $\phi A$ for an arbitrary bivector $A$. That is
$$\pi_\phi=\phi A=\phi\left(A_3\frac{\partial}{\partial u }\wedge\frac{\partial}{\partial v }+A_2 \frac{\partial}{\partial w }\wedge\frac{\partial}{\partial u } + A_1\frac{\partial}{\partial v }\wedge\frac{\partial}{\partial w }\right)$$
The following proposition gives criteria for when $\pi'$ is also a Poisson bivector:

\begin{proposition}  The set of Poisson structures on $\CC^3$ extending the Poisson structure $(\CC^2/\ZZ_n,\brac_{\CC^2/\ZZ_n})$ is given by 
$$\left\{ \pi+\phi(*\alpha)\SC \alpha\in\Omega^1(\CC^3) \mbox{ {\rm and} } (d\phi-\phi\alpha)\wedge d\alpha=0\right\}$$
\end{proposition}

\begin{proof}  The proof relies heavily on the following formula for the Schouten bracket of bivectors:  Given bivectors $A,B\in\mathfrak{X}^2(\CC^3)$, their Schouten bracket is given by: 
\begin{eqnarray}  
[A,B] &=& A\cdot{\rm Curl}(B)+{\rm Curl}(A)\cdot B  \\
         &=& \alpha\wedge d\beta + d\alpha\wedge\beta   \mbox{,  \hspace{10pt} where $\alpha=*A$ and $\beta=*B$}
      \end{eqnarray}

The bivector $\pi+\pi_\phi$, by construction, extends the bracket given by Poisson structure on the quotient.  It defines a Poisson structure if and only if
\begin{eqnarray}
0 &=& \frac{1}{2}[\pi+\pi_\phi,\pi+\pi_\phi]  \\
   &=& \frac{1}{2}[\pi,\pi]+[\pi,\pi_\phi]+\frac{1}{2}[\pi_\phi,\pi_\phi] \\
   &=& [\pi,\pi_\phi]+\frac{1}{2}[\pi_\phi,\pi_\phi],
   \end{eqnarray}
where $[\pi,\pi]=0$ because $\pi$ itself is a Poisson structure.  Recall, that 
$$\pi_\phi=\phi A=\phi\left(A_3\frac{\partial}{\partial u }\wedge\frac{\partial}{\partial v }+A_2 \frac{\partial}{\partial w }\wedge\frac{\partial}{\partial u } + A_1\frac{\partial}{\partial v }\wedge\frac{\partial}{\partial w }\right) \for A\in\mathfrak{X}^2(\CC^3)$$
Straightforward computations give
\begin{enumerate}
\item $[\pi_\phi,\pi_\phi]=\phi^2[A,A]$, 
\item $[\pi,\pi_\phi]=\phi[\pi,A]$ holds because $\phi$ is Casimir for $\pi$ 
\end{enumerate}
Thus, formula $(2)$ implies
\begin{eqnarray}
0 &=& [\pi,\pi_\phi]+\frac{1}{2}[\pi_\phi,\pi_\phi] \\
   &=& \phi\left([\pi,A]+\phi[A,A]\right)\\
   &=& \phi\left([\nabla{(*\phi)},A]+\phi[A,A]\right)\\
   &=& *\phi\left(d\phi\wedge d\alpha+\phi\alpha\wedge d\alpha\right)
\end{eqnarray}

As $*$ is an isomorphism, this expression is zero if and only if $d\phi\wedge d\alpha+\phi\alpha\wedge d\alpha=0$
\end{proof}

Note that if $d\alpha=0$, then $\pi+\phi(*\alpha)$ is a Poisson structure.  As the De Rham complex for $\CC^3$ is exact, any closed form $\alpha=df$ is exact. Thus, functions give a wealth of Poisson structures lifting $\brac_{V/G}$:

$$\pi+\phi(*df)= \left(\frac{\partial\phi}{\partial w}+\phi\frac{\partial f}{\partial w}\right)\frac{\partial}{\partial u }\wedge\frac{\partial}{\partial v }+\left(\frac{\partial\phi}{\partial v}+\phi\frac{\partial f}{\partial v} \right)\frac{\partial}{\partial w }\wedge\frac{\partial}{\partial u } + \left(\frac{\partial\phi}{\partial u}+\phi\frac{\partial f}{\partial u}\right)\frac{\partial}{\partial v }\wedge\frac{\partial}{\partial w }$$

In this case, as the proof demonstrates, $d_{\pi}\pi_\phi=-[\pi,\pi_\phi]=0$, and so $\pi+\pi_\phi$ is a deformation of $\pi$. Using the cohomological computations in chapter 3 of \cite{Pichereau06}, conditions for when such deformations $\pi+\phi(*df)$ are trivial can be explicitly determined.  Using the notation from above, denote
$$\pi_\phi=\phi A=\phi\left(A_3\frac{\partial}{\partial u }\wedge\frac{\partial}{\partial v }+A_2 \frac{\partial}{\partial w }\wedge\frac{\partial}{\partial u } + A_1\frac{\partial}{\partial v }\wedge\frac{\partial}{\partial w }\right)$$
The deformation $\pi+\pi_\phi$ is a trivial deformation of $\pi$ if and only if $A_1=u^{n-2}\tilde{A}_1$ for some $\tilde{A}_1\in\CC[u,v,w]$.  The vector field which generates the deformation is then given by:
$$\left(\frac{\phi}{n(n-1)}\tilde{A}_1+\frac{n}{n-1}{\rm div}(\pi_\phi)\right)\frac{\partial}{\partial u}+\left(\phi A_2+w{\rm div}(\pi_\phi)\right)\frac{\partial}{\partial v}+\left(\phi A_3+v{\rm div}(\pi_\phi)\right)\frac{\partial}{\partial w}$$

The nontrivial deformations of $\pi$ are generated by: 
$$\pi+\pi_\phi=\pi+\phi^m u^k\frac{\partial}{\partial v}\wedge\frac{\partial}{\partial w}  \mbox{ \hspace{4pt} for $m>0$ and  } 0<k<(n-2)$$
\end{example}

\begin{remark} This sort of analysis also applies to the other 2-dimensional orbifolds as well.
\end{remark}

\section{Infinitesimal Extensions around $0$}

This section applies the analysis of infinitesimal extensions of Poisson structures in Section \ref{sec:InfNbd} to the Poisson embedding problem (Section \ref{sec:PoissEmb}).  Guided by the fact that a set of invariant polynomials $\{\sigma_1,\ldots,\sigma_n\}$ defines an embedding $F_\sigma\!:\! V/G\to\RR^n$ if and only if their images span $\m_0/\m_0^2$, it is natural to consider local problem of extending a Poisson bracket to infinitesimal neighborhoods of $0$.

Any embedding $F\!:\!V/G\hookrightarrow\RR^n$ such that $\phi(0)=0$ gives rise to embeddings into the infinitesimal $k$-neighborhoods of the origin $\mathcal{N}_{\{0\}}^k(\RR^n)$:
\[
\xymatrix{
V/G \ar@{^{(}->}[r]^{F} & \RR^n  \\
\mathcal{N}_{\{0\}}^\infty(V/G) \ar@{^{(}->}[r] \ar@{^{(}->}[u] & \mathcal{N}_{\{0\}}^\infty(\RR^n) \ar@{^{(}->}[u] \\
\vdots \ar@{^{(}->}[u] & \vdots \ar@{^{(}->}[u] \\
\mathcal{N}_{\{0\}}^1(V/G) \ar@{^{(}->}[r] \ar@{^{(}->}[u] & \mathcal{N}_{\{0\}}^1(\RR^n) \ar@{^{(}->}[u] \\
\{0\} \ar@{^{(}->}[r] \ar@{^{(}->}[u] & \{0\} \ar@{^{(}->}[u]
}
\]
Moreover, by Section \ref{sec:ExtPoissonInfNbd}, any Poisson structure on $\RR^n$ making $F$ a Poisson embedding in turn makes the diagram above a commutative diagram of Poisson maps.  Thus, one may attempt to construct a Poisson structure on $\RR^n$ by iteratively extending Poisson structures on each $\mathcal{N}^k_{\{0\}}(\RR^n)$.  This procedure, of course, begins with the bottom square:
\[
\xymatrix{
\mathcal{N}_{\{0\}}^1(V/G) \ar@{^{(}->}[r]  & \mathcal{N}_{\{0\}}^1(\RR^n) \\
\{0\} \ar@{^{(}->}[r] \ar@{^{(}->}[u] & \{0\} \ar@{^{(}->}[u]
}
\]

In the following proposition (1) describes the Poisson structures on $\nbd{1}{\pt}{\RR^n}$ that extend the given one on $\nbd{1}{\pt}{V/G}$, and (2) describes the conditions for each subsequent extension:

\begin{theorem}
\label{prop:EmbeddingCriteria}
Let $F\!:(Q,\brac_{Q})\hookrightarrow(\RR^n,\brac_\pi)$ be a Poisson embedding.  Denote by $\pi^{(k)}$ the Poisson bivector on $\nbd{k}{\pt}{\RR^n}$ that is $p_k$-related to $\sum_{l=1}^{k}\pi^l$, where $\pi^l$ has homogeneous coefficients of degree $l$, and let $\g\isom\RR^n$ be the Lie algebra corresponding to the Lie-Poisson structure $(\RR^n,\pi^1)$.  Then
\begin{enumerate}
\item The Lie algebra $\g$ is a Lie algebra extension of $\m_0/\m_0^2$
\item Each extension $\pi^{(k)}$ satisfies the following $\g$-cohomological equations:
$${\rm d}_\g\pi^{k}=[i_{k*}\pi^{(k-1)},i_{k*}\pi^{(k-1)}]   \mbox{\quad where  }  i_k:\nbd{k-1}{\pt}{\RR^n}\inj\nbd{k}{\pt}{\RR^n}$$
\end{enumerate}
\end{theorem}

\begin{remark} The proposition above may also be thought of as interpreting the Lie algebra extension $0\!\to\!K\!\to\!\g\!\to\!\m_0/\m_0^2\!\to\!0$ as a first approximation to the Poisson embedding $F$.  That is, $F$ is only a Poisson map {\it up to first order}.  The subsequent extensions are then nonlinear deformations of the Lie-Poisson structure $\pi^1$ to a Poisson structure $\pi$ that makes $F$ a Poisson embedding.
\end{remark}

\begin{proof}  The pullback map $F^*:\RR[x_1,\ldots,x_n]\twoheadrightarrow\Sym(V^*)$ is a surjective homomorphism of Poisson algebras.  Moreover, the pullback restricts to a surjective map $F^*(x_1,\ldots,x_n)\twoheadrightarrow\m_0$, and thus descends to a surjective linear map of vector spaces $\overline{F^*}:\Span\{x_1,\ldots,x_n\}\twoheadrightarrow\m_0/\m^2_0$.  As $F^*$ is a homomorphism of Poisson algebras, the induced map $\overline{F^*}$ is a Lie-algebra homomorphism.  Denoting the kernel of this linear map $K:=\ker(\overline{F^*})$, the Lie bracket on the Zariski cotangent space of $\RR^n$ is the Lie algebra extension of $\m_0/\m_0^2$ given by the exact sequence of Lie algebras:
$$0\to K\to\g\to\m_0/\m_0^2\to0$$
The proof of $(2)$ then follows from Section \ref{sec:GrPoiss} and the fact that the linear Poisson structure $(\RR^n,\brac_1)$ is isomorphic to the Lie-Poisson structure on $\g$.
\end{proof}

\begin{remark} This procedure describes, in different terminology, the procedure used by Egilsson \cite{Egilsson95} and Davis \cite{Davis02} to prove nonembedding results for the $(1,1,2)$-resonance space $\CC^3/S^1$.
\begin{itemize}
\item[i)] As $\dim(\m_0/\m_0^2)=11$, the first order approximation of any embedding $F: \CC^3/S^1\hookrightarrow\RR^{11}$ is completely determined (i.e. $\g\isom\m_0/\m_0^2$).  Then, using explicit computation, Egilsson proved the impossibility of solving the condition $\dd_\g\pi^{(2)}=[i_{1*}\pi^{(1)},i_{1*}\pi^{(1)}] =0$.
\item[ii)] Similarly, an embedding $F:\CC^3/S^1\hookrightarrow\RR^{12}$ is completely determined by the data attached to a 1-dimensional central extension$\tilde{\g}$ of $\g\isom\m_0/\m_0^2$.  For such $\tilde{\g}$, Davis showed the impossibility of solving $\dd_{\tilde{\g}}\pi^{(2)}=[i_{1*}\pi^{(1)},i_{1*}\pi^{(1)}] =0$
\item[iii)] This procedure quickly becomes intractable as the dimension of the target space increases.  To use this technique to prove that no Poisson embedding $F:\CC^3/S^1\hookrightarrow\RR^{11+n}$ exists, one first has to classify $n$-dimensional Lie algebra extensions of $\m_0/\m_0^2$ (which is no easy task), and then prove the insolvability of the higher degree equations {\it for each extension} of $\m_0/\m_0^2$.  
\end{itemize}
\end{remark}

The two approaches to infinitesimally extending a Poisson structure on an embedded variety $V(I)\subseteq\RR^n$ considered in this chapter have opposing limitations.
\begin{itemize} 
\item Extending a Poisson structure to formal neighborhoods of the variety itself $\mathcal{N}^k_{V(I)}(\RR^n)$ is a very tractable problem, in the sense that there are not many possible extensions at each step.  On the other hand, finding explicit descriptions of the extensions involves determining membership in the ideal $I$, which is often a difficult problem to solve.
\item Extending the germ of a Poisson structure to formal neighborhoods of the origin $\mathcal{N}^k_{\{0\}}(\RR^n)$ involves performing calculations degree-by-degree (that is, working modulo $\m_0^k$), which is computationally easy.  However, there are many extensions at each step, and the task of determining which of these arise from the Poisson structure results in having to perform a very large number of (easier) computations.
\end{itemize}

In the subsequent chapters, we will take the latter view, using Lie theoretic techniques to make the large number of computations manageable.

\chapter{Linearization of Poisson Structures}

One of the difficulties of studying extensions of Poisson algebras is that one must extend both the Poisson algebra as a Lie algebra and as an associative algebra, and satisfy compatibility between the two structures.  In the previous chapter, we laid out a procedure that (1) first extends the associative algebra, then (2) attempts to extend the Poisson bracket.  In other words, one finds a (perhaps infinitesimal) embedding of the Poisson space then tries to extend the Poisson structure. 

The final criteria for a Poisson embedding given in Proposition \ref{prop:EmbeddingCriteria} first establishes an embedding of a Poisson quotient space $Q$ {\it up to first order} into $(\g^*,\pi^1)$, for an extension $\g$ of its cotangent Lie algebra $\m_0/\m_0^2$.  Then, it states criteria for deforming the Lie-Poisson structure $\pi^1$ into a Poisson structure that extends the full Poisson structure on $Q$.  The nonlinear deformations of $(\g^*,\pi^1)$ are greatly influenced by the Lie algebra cohomology of $\g$.  One extreme is the result of Weinstein \cite{Weinstein83}:

\begin{theorem}[Weinstein] 
Let $\pi=\pi^1+\pi^2+\ldots$ be formal Poisson structure that vanishes at the origin with corresponding linear Poisson structure $(\g^*,\pi)$.  If $\g$ is semi-simple, then the Poisson structures $(\g^*,\pi)$ and $(\g^*,\pi^1)$ are formally isomorphic. 
\end{theorem}

This theorem may be applied to the embedding problem.  Suppose that the cotangent Lie algebra $\g:=\m_0/\m_0^2$ of $(Q,\brac_Q)$ is semi-simple.  Then there is a Poisson embedding into the minimal embedding space $\g^*$ if and only if there is a Poisson embedding into $(\g^*,\pi^1)$, where the linearization $\pi^1$ of $\pi$ is the linear Poisson structure on $\g^*$.  In fact, it seems likely (see Section \ref{sec:DeterminingPhi}) that, if $\g$ is semi-simple, then the natural map $Q\inj\nbd{1}{\pt}{Q}\isom\g^*$ is a Poisson embedding.

In general, when $\g$ is not semi-simple, the Poisson structure $(\g^*,\pi^1)$ tends to have nonlinear deformations.  However, one still has some control of the form of such deformations via the Levi decomposition for Poisson structures (see \cite{Dufour05}, \cite{Wade97}).    

\begin{definition}[Levi decomposition] Let $\g$ be a finite dimensional Lie algebra, $\mathfrak{r}$ be its radical and $\mathfrak{s}:=\g/\mathfrak{r}$ its semi-simple quotient.  By the Levi-Malcev Theorem, the exact sequence 
$$0\to\mathfrak{r}\to\g\to\mathfrak{s}\to0$$
splits.  The decomposition $\g\isom\mathfrak{s}\ltimes\mathfrak{r}$ given by this splitting is called the Levi decomposition of $\g$.
\end{definition}

One may partially linearize a Poisson bracket with respect to a Levi decomposition:

\begin{theorem}[Wade] 
\label{thm:LeviDecomp}
Let $\pi=\pi^1+\pi^2+\ldots$ be a formal Poisson structure that vanishes at the origin with corresponding linear Poisson structure $(\g^*,\pi)$, and let $\g=\mathfrak{s}\ltimes\mathfrak{r}$ be the Levi decomposition of $\g$, with coordinates $(\tilde{s}_1,\ldots,\tilde{s}_n,\tilde{r}_1,\ldots,\tilde{r}_m)$ and Lie brackets 
$$ [\tilde{s}_i,\tilde{s}_j]=\sum_{k}a_{ij}^k \tilde{s}_k  \mbox{\quad and \quad} [\tilde{s}_i,\tilde{r}_j]=\sum_{k}b_{ij}^k \tilde{r}_k$$
Then there exists a formal change of coordinates $(s_1,\ldots,s_n,r_1,\ldots,r_m)$ such that:
\begin{enumerate}
\item The change of coordinates is the identity up to first order.
\item The resulting Poisson structure has the form:
$$\pi=\sum_{i<j} a_{ij}^k s_k\bivect{s_i}{s_j}+\sum_{i<j}b_{ij}^k r_k\bivect{s_i}{r_j}+\sum_{i<j}P_{ij}\bivect{r_i}{r_j}$$
for formal functions $P_{ij}$.
\end{enumerate}
\end{theorem}

This theorem's proof relies on either Hermann's linearization of vector fields \cite{Hermann68} or on the Hochschild-Serre spectral sequence \cite{Hochschild54}, depending on the point of view one wants to take (See \cite{Dufour05} and the references therein).  Additionally, this semi-linearization improves the embedding criteria given by Proposition \ref{prop:EmbeddingCriteria} as follows:

\begin{proposition} Let $F\!:(Q,\brac_{Q})\hookrightarrow(\RR^n,\brac)$ be a Poisson embedding.  Denote by the Poisson bivector corresponding to $\brac$ by $\pi$, and the induced Poisson bivector on $\nbd{k}{\pt}{\RR^n}$ by $\pi^{(k)}$.   Let $\g\isom\RR^n$ be the Lie algebra corresponding to the Lie-Poisson structure $(\RR^n,\pi^{(1)})$.  Then
\begin{enumerate}
\item The Lie algebra $\g$ is a Lie algebra extension of $\m_0/\m_0^2$.
\item If the Levi decomposition of $\g\isom\mathfrak{s}\ltimes\mathfrak{r}$, then the extended Poisson structure $\pi^{(k)}$ on $\nbd{k}{\pt}{\RR^n}$ is a solution to the following $\mathfrak{r}$-cohomological equations:
$${\rm d}_\mathfrak{r}\pi^{(k)}=[i_{k*}\pi^{(k-1)},i_{k*}\pi^{(k-1)}]$$
where we may consider $\pi^{(i)}\in{\rm C^2}(\mathfrak{r},\Sym^{\leq i}(\g))^{\mathfrak{s}}$ via the identification $\bigwedge^2\TT_{\nbd{k}{\pt}{\RR^n}}\isom C^2(\g,\Sym^{i}\g)$ given in Theorem \ref{thm:CECoh}.
\end{enumerate}
\end{proposition}

To prove the nonembedding theorem in Chapter 6, we need to make a slight improvement to the semi-linearization given in the Levi decomposition.  For a relevant class of Lie algebras, this improvement will give a normal form for some of the nonlinear $P_{ij}$ in Theorem  \ref{thm:LeviDecomp}.   This improvement is based on a technique developed by J.P. Dufour and N.T. Zung in \cite{Dufour02}, and employs the Poincar\'{e}-Dulac normal form of a vector field.\\

\subsubsection{Poincar\'{e} -- Dulac normal form}

A thorough exposition of this subject matter may be found in Appendix A of \cite{Dufour05}.  Roughly stated, the Poincar\'{e}-Dulac normal form is a decomposition of a vector field into a linear part and a resonant part.   This section will make this statement precise.

Let $\KK^n$ have coordinates $\{v_1,\ldots,v_n\}$ and let $X$ be formal a vector field on $\KK^n$ that vanishes at $0$.  In a neighborhood of $0$, we may decompose $X$ into a linear and nonlinear part $X=X^1+X^{nl}$.  Furthermore, the linear vector field decomposes into semi-simple and nilpotent parts $X^1=X^1_{ss}+X^1_n$.  Then $X$ is in {\it Poincar\'{e}-Dulac form} if $[X,X^1_{ss}]=0$.  

In particular, when $\KK=\CC$, the semi-simple part diagonalizes as: 
$$X^1_{ss}=\sum_{k=1}^{n}\lambda_k x_k\vect{x_k}   \mbox{\quad where   } \lambda_k\in\CC$$

We write $X$ in coordinates as $X=\sum_{K} (x_1^{k_1}x_2^{k_2}\ldots x_n^{k_n})\vect{x_k}$, where $K=(k_1,\ldots,k_n,k)$ is multi-index notation.  The condition that $[X,X^1_{ss}]=0$ then becomes

\begin{eqnarray}
\nonumber 0    &=&  [X,X^1_{ss}] \\
 \nonumber	&=& [\sum_{K} (x_1^{k_1}x_2^{k_2}\ldots x_n^{k_n})\vect{x_k} ,\sum_{k=1}^{n}\lambda_k x_k\vect{x_k}] \\
	&=& \sum_{K}\left( \sum_{i=1}^{n} k_i \lambda_i-\lambda_k\right)(x_1^{k_1}x_2^{k_2}\ldots x_n^{k_n})\vect{x_k} 
\end{eqnarray} 

Therefore, a vector field is in Poincar\'{e}-Dulac normal form if the nonlinear coefficients $(x_1^{k_1}x_2^{k_2}\ldots x_n^{k_n})$ are resonant with respect to the eigenvalues $(\lambda_1,\ldots,\lambda_n)$ of $X^1$:
$$\left( \sum_{i=1}^{n} k_i \lambda_i-\lambda_k\right)=0$$ 

\begin{theorem}[Poincar\'{e} -- Dulac]  Let $X$ be a formal vector field vanishing at the origin.  Then there is a formal change of coordinates that puts $X$ into Poincar\'{e}-Dulac normal form.
\end{theorem}

\begin{remark} Under an additional diophantine condition, a theorem of Bruno actually gives an analytic version of the  Poincar\'{e}-Dulac theorem (see \cite{Dufour05}).
\end{remark}

\begin{example}If the nonlinear part of $X$ is nonresonant, then the Poincar\'{e}-Dulac theorem states that $X$ is linearizable.
\end{example}

\subsubsection{The example of $\g=(\KK\oplus\mathfrak{s})\ltimes \mathfrak{r}$ }

We now prove, for a certain Lie algebra $\g$, the existence of a slight improvement of the semi-linearization given by Theorem \ref{thm:LeviDecomp}.   Let $\g=(\KK\oplus\mathfrak{s})\ltimes \mathfrak{r}$ be the semi-direct product of (1) a 1-dimensional central extension of a semi-simple Lie algebra $\mathfrak{s}$ and (2) a nilpotent Lie algebra $\mathfrak{r}$.  Denote the basis of $\g=(\KK\oplus\mathfrak{s})\ltimes \mathfrak{r}$ by $\{\tilde{t},\tilde{s}_1,\ldots,\tilde{s}_k,\tilde{r}_1,\ldots,\tilde{r}_m\}$.

\begin{example} Let the semi-simple part be $\mathfrak{s}=\mathfrak{sl}_n(\CC)$ and the Lie ideal be $\mathfrak{r}=\CC^n$.  If the action of $\mathfrak{s}=\mathfrak{sl}_n(\CC)$ on $\CC^n$ is the usual one and $\CC$ acts via scalar multiplication, then $\g$ is the affine Lie-algebra $\mathfrak{aff}(\CC^n)$.  This example, considered in \cite{Dufour02}, originally necessitated the use of the Poincar\'{e}-Dulac normal form to improve the Levi decomposition.
\end{example}

\begin{proposition} 
\label{prop:LeviImprove}
Let $\pi=\sum\pi^k$ be a formal Poisson structure on $\KK^n$ that vanishes at the origin, with corresponding linear Poisson structure $(\g^*,\pi^1)$, where $\g=(\KK\oplus\mathfrak{s})\ltimes \mathfrak{r}$. There exists a formal change of coordinates $(t,s_1,\ldots,s_k,r_1,\ldots,r_m)$ such that 
\begin{enumerate}
\item The change of coordinates is the identity up to first order.
\item The resulting Poisson structure has the form:
$$\pi=\sum_{\stackrel{i<j}{k}} a_{ij}^k s_k\bivect{s_i}{s_j}+\sum_{\stackrel{i<j}{k}}b_{ij}^k r_k\bivect{s_i}{r_j}+\sum_{k} (\lambda_k r_k+R_{k})\bivect{t}{r_k}+\sum_{i<j}P_{ij}\bivect{r_i}{r_j}$$
where $a_{ij}^k,b_{ij}^k,\lambda_k\in\KK$, the $P_{ij}$ are formal nonlinear functions and the $R_{k}$ are nonlinear functions that are {\it resonant} with respect to the action of $\KK$ on $\mathfrak{r}$.
\end{enumerate}
\end{proposition}

\begin{proof}
By the Levi decomposition, we may assume that $\pi$ is already in the form of Theorem  \ref{thm:LeviDecomp} -- that is, $\pi$ has the form in the proposition, where the $R_k$ are not necessarily resonant.  The Hamiltonian vector fields $\{X_{s_1},\ldots,X_{s_k}\}$ form a Lie subalgebra of ${\rm Ham}(\g)$ isomorphic to $\mathfrak{s}$.  Moreover, the Hamiltonian vector field
$$X_{t}=\sum_{k} \lambda_k r_k\vect{r_k}+R_{k}\vect{r_k}$$ 
is invariant under the action of $\mathfrak{s}$ (as $\{t,s_i\}_1=0$), and so we may $\mathfrak{s}$-equivariantly apply the Poincar\'{e}-Dulac theorem to obtain coordinates such that
$$X_{t}=\sum_{k} \lambda_k r_k\vect{r_k}+R_{k}\vect{r_k}  \mbox{\quad  where $R_k$ are resonant }$$ 
Furthermore, this coordinate change does not affect the Levi decomposition -- that is,
$$\pi=\sum_{\stackrel{i<j}{k}} a_{ij}^k s_k\bivect{s_i}{s_j}+\sum_{\stackrel{i<j}{k}}b_{ij}^k r_k\bivect{s_i}{r_j}+\sum_{k} (\lambda_kr_k+R_{k})\bivect{t}{r_k}+\sum_{i<j}P_{ij}\bivect{r_i}{r_j}$$
\end{proof}

\begin{example}[ The cotangent Lie algebra $\g=\gl(\CC)\ltimes(V_n\oplus V_n)$ ]
\label{ex:ZnCotangent}
This example will apply Proposition \ref{prop:LeviImprove} to the cotangent Lie algebra of the $4$-dimensional symplectic quotient $V/\ZZ_n$ considered in Chapter 6.  Let $\g=(\CC\oplus\slt(\CC))\ltimes(V_n\oplus V_n)$ be the Lie algebra with basis $\{t,H,E,F,x_0,\ldots,x_n,y_0,\ldots,y_n\}$ and Lie brackets given by: 
\begin{enumerate}
\item The standard $\gl(\CC)=(\CC\oplus\slt(\CC))$ structure on the first factor.
\item The action of the semi-simple factor $\slt(\CC)$ on $V_n\oplus V_n$ given by $\rho_n\oplus-\rho_n$, where $\rho_n:\slt(\CC)\to V_n$ is the standard $(n+1)$-dimensional irreducible representation of $\slt(\CC)$. 
\item The action of $\CC$ on $V_n\oplus V_n$ is given by ${\rm Id}\oplus-{\rm Id}$.
\item The trivial Lie bracket on $V_n\oplus V_n$.
\end{enumerate}

The eigenvalues of the $\CC$-action are 
$$(\lambda_0,\ldots,\lambda_n,\lambda_{n+1},\ldots,\lambda_{2n+1})=(1,\ldots,1,-1,\ldots,-1)$$
and the $\lambda$-resonant polynomials in $\Sym(\g)$ are generated by
$$ \{ (x_{i}y_{j}) \mbox{ ;  } 0\leq i,j \leq n \}  $$

Therefore, if $\pi=\sum\pi^k$ is a Poisson structure on $\CC^{2(n+1)+4}$ with linear Poisson structure $((\gl(\CC)\ltimes(V_n\oplus V_n)^*,\pi^1)$, then Proposition \ref{prop:LeviImprove} provides coordinates in which $\pi$ has the form:

\begin{eqnarray}
\label{eqn:LeviDecomp}
\pi &=& \pi^1+\sum_{k}P^t_{k}\bivect{t}{w_k}+\sum_{i<j}P_{ij}\bivect{w_i}{w_j}   \mbox{\quad  where  } w_i\in\{ x_i, y_i\} \\
\nonumber   &   & \hspace{6cm} \mbox{ and } P^t_{k}\in\CC[x_iy_j;\, 0\leq i,j\leq n]
\end{eqnarray}

\end{example}

\chapter{A Nonembedding Result for $V/\ZZ_{\MakeLowercase{n}}$}

This chapter consists of the proof of Theorem \ref{thm:NonEmb}, which asserts that the Poisson embedding dimension of the singular symplectic quotient $V/\ZZ_n$ of 4-dimensional real affine space $V$ by the cyclic group $\ZZ_n$ of odd order is strictly greater than its smooth embedding dimension.  To prove such a result, we need a description of the quotient and its Poisson structure.

\section{Basic Properties of $V/\ZZ_n$}

Throughout this chapter, let $G=\ZZ_n$ be the cyclic group of order $n$.  Let $V=\RR^4$ be equipped with the standard symplectic structure (see Example \ref{ex:SympVS}), and let $\ZZ_n\subseteq S^1\subseteq{\rm Sp}(V)$ act diagonally on $V=\RR^2\times\RR^2$, with the action of $\ZZ_n$ on $\RR^2$ given in Example \ref{ex:Z_n}.  As in Example \ref{ex:Complexification}, denote the complexified coordinates of $V_\CC$ by $\{z,\zb,w,\wb\}$.  The resulting linear action of $\ZZ_n$ on the 4-dimensional complex vector space  $V_\CC$ is given by:
\begin{eqnarray}
\label{eqn:CxAction}
\zeta\cdot(z,\zb,w,\wb)=(\zeta z,\zeta^{-1}\zb,\zeta w,\zeta\wb^{-1}) \mbox{,   for  }  \zeta\in\ZZ_n\subseteq\CC.
\end{eqnarray}

By Section \ref{sec:Complexification}, if there is no Poisson structure on $\CC^{N}$ extending the Poisson structure on $V_\CC/\ZZ_n$, then there cannot be a Poisson embedding $V/\ZZ_n\inj\RR^N$.   We will take this approach and therefore only be concerned with $V_\CC$. Thus, to reduce notational clutter, we will drop the $\CC$ from $V_\CC$. 

$V$ is now a complex vector space equipped with the action in Equation \ref{eqn:CxAction}.  However, to prevent a proliferation of $i=\sqrt{-1}$ throughout the chapter, we will scale this Poisson bracket by $\frac{-1}{2i}$ so that $$\{z,\zb\}=\{w,\wb\}=1.$$
As this action is a Poisson action, the quotient $V/\ZZ_n$ inherits a natural Poisson structure.  Denote this Poisson variety $(V/\ZZ_n,\brac_{V/\ZZ_n})$ by $(Q,\brac_Q)$.  

The following lemma will be used repeatedly to prove properties of the quotient $Q$.  It arises from the fact that the $\ZZ_n$-action is a restriction of the standard action of $U(2)$ on $V$ and therefore defines a homomorphism of (complexified) Lie algebras:

$$\xymatrix{  \mathfrak{gl}_2(\CC) \ar[r]^{\mu^*\otimes\CC}  \ar@{-->}[dr]  & \Sym(V^*) \ar[r]^-{\rm Ham}  & {\rm Ham}(V)\subseteq \TT_V  \\
							& \Sym(V^*)^G  \ar@{^{(}->}[u] \ar[r]^-{\rm Ham} & {\rm Ham}(Q) \subseteq \TT_Q \ar@{^{(}->}[u]\\
							}
$$
where $\mu$ is the moment map for the $U(2)$-action, and ${\rm Ham}(f)=\{\cdot,f\}$ is the map that sends a function to its Hamiltonian vector field.  Explicitly, the homomorphism is given by:

\begin{lemma}
Let $\{t,H,E,F\}$ be a basis of $\gl(\CC)$ such that 
\begin{enumerate}
\item $t$ is in the center of $\gl(\CC)$
\item $\{H,E,F\}$ is the standard basis of $\slt(\CC)$ (see Appendix \ref{sec:AppendixSL2}).
\end{enumerate}
We may identify $\gl(\CC)$ with a subalgebra of $\TT_V$ via the map 
$\left\{\begin{array}{rcl}
t &\mapsto& \{\cdot,z\zb+w\wb\} \\
H &\mapsto& \{\cdot,z\zb-w\wb\} \\
E &\mapsto& \{\cdot,\zb w\} \\
F &\mapsto& \{\cdot,z\wb\} 
\end{array}\right.$.
\end{lemma}

Such an identification of $\gl(\CC)\subseteq\TT_V$  induces an action of $\gl(\CC)$ on $\Sym(V^*)$ and $\Sym(V^*)^{\ZZ_n}$.  Therefore

\begin{lemma} The action above makes $\Sym(V^*)$ and $\Sym(V^*)^{\ZZ_n}$ representations of the Lie algebra $\gl(\CC)$.   Moreover,
\begin{enumerate}
\item The action decomposes $\Sym(V^*)$ and $\Sym(V^*)^{\ZZ_n}$ into $t$-eigenspaces and $H$-eigenspaces.
\item As the Poisson bracket on $V$ has degree-$(-2)$ and $\gl(\CC)$ is generated by (the Hamiltonian vector fields of) quadratic functions, this action preserves the natural grading on $\Sym(V^*)=\bigoplus_{k}\Sym^k(V^*)$.
\end{enumerate}
\end{lemma}
 
\subsection{A Hilbert Basis for $V/\ZZ_n$}
\label{sec:HilbertBasis}

In this section, we will determine a Hilbert basis for the quotient $Q=V/\ZZ_n$ --- that is, a minimal generating set for the algebra of invariant polynomials $\Sym(V^*)^{\ZZ_n}$.  The previous section provides three different gradings on the algebras $\Sym(V^*)$ and $\Sym(V^*)^{\ZZ_n}$.  The polynomial algebra decomposes as 
$$\Sym(V^*)=\bigoplus_{i,j,k}\Sym^k(V^*)_{(i,j)}$$
where 
$$p\in\Sym^k(V^*)_{(i,j)}, \mbox{\quad if and only if }
\left\{\begin{array}{l} \mbox{ $p$ is a polynomial of degree $k$}, \\
t\cdot p:=\{p,z\zb+w\wb\}=ip, \\
H\cdot p:=\{p,z\zb-w\wb\}=jp.\\
\end{array}\right.  $$
This decomposition gives a useful characterization of the invariants $\Sym(V^*)^{\ZZ_n}\subseteq\Sym(V^*)$:

\begin{proposition} A homogeneous polynomial $p\in\Sym^k(V^*)_{(i,j)}$ is ${\ZZ_n}$-invariant if and only if the $t$-eigenvalue  $i \equiv0\bmod n$.  Thus:
$$\Sym(V^*)^{\ZZ_n}=\bigoplus_{\stackrel{j\in\ZZ,k\in\NN}{i\equiv0 \bmod n}}\Sym^k(V^*)_{(i,j)}.$$
\end{proposition} 

\begin{proof} A general element of $\Sym(V^*)$ may be written as a linear combination of monomials $p=z^a\zb^b w^c\wb^d$, where $a,b,c,d\in\NN$.  The action of $\zeta$ on $p$ gives $\zeta\cdot p = \zeta^{(a-b+c-d)} p$ and so $p$ is invariant if and only if $(a-b+c-d)\equiv 0 \bmod n$.  Whereas the action of $t$ on $p$ is given by:
\begin{eqnarray} 
\nonumber t\cdot p &=& \{p,z\zb\}+\{p,w\wb\}  \\
\nonumber	&=& (a-b)p+(c-d)p \\
\nonumber	&=&	\left[(a-b+c-d)\right]p. 
\end{eqnarray}

Thus $p$ is ${\ZZ_n}$-invariant if and only if the $t$-eigenvalue of $p$ is divisible by $n$.
\end{proof}

Now we present a Hilbert basis for $Q$:

\begin{proposition} A Hilbert basis for the action of $\ZZ_n$ on $V$ given above is given by:
$$\{z\zb,w\wb,\zb w,z\wb,w^n,zw^{n-1},\ldots z^n,\wb^n,\zb^{}\wb^{n-1},\ldots,\zb^n\}$$
\end{proposition}

\begin{proof} One can easily verify (using the proposition above) that each of these are ${\ZZ_n}$-invariant and that no member of this list can be written in terms of the other.  Thus, it only remains to show that any ${\ZZ_n}$-invariant monomial can be written in terms of those on the list.  Let $p$ be a ${\ZZ_n}$-invariant monomial, so that we may write $p=z^a\zb^b w^c\wb^d$, where $(a-b+c-d)\equiv 0\bmod n$.  There are four cases to consider:
\vspace{.5cm}
\\
(1) Suppose $a>b$ and $c<d$.  We factor $p=(z\zb)^b(w\wb)^c z^{a'}\wb^{d'}$.  As $t\cdot p\equiv 0\bmod n$, and $t\cdot(z\zb)=t\cdot(w\wb)=0$, we have that $t\cdot(z^{a'}\wb^{d'})=0$.  Thus, $a'-d'\equiv 0\bmod n$ and there exists a $k\in\ZZ$ such that $(a'-d')=kn$.  If $k\geq0$, then $a'=kn+d'\geq0$ and so $z^{a'}\wb^{d'}=z^{kn+d'}\wb^{d'}$.  Thus $p$ may be written as the product of elements in the proposed Hilbert basis:
$$ p=(z\zb)^b(w\wb)^c(z\wb)^{d'}(z^n)^{k}. $$
If $k<0$, then $d'=(-k)n+a'>0$ and so $z^{a'}\wb^{d'}=z^{n|k|+a'}\wb^{d'}$.  Thus, $p$ may be written as
$$p=(z\zb)^a(w\wb)^c(z\wb)^{a'}(\wb^n)^{|k|}.$$

\noindent (2) Suppose $a>b$ and $c>d$.  We can factor $p$ in a similar manner to the previous case: $p = (z\zb)^b(w\wb)^d(z^{a'})(w^{b'})$.  Then, writing $a'=nk_a+r_a$ and $b'=nk_b+r_a$ with $0\leq r_a,r_b<n$, one can further factor $p$:
$$p = (z\zb)^b(w\wb)^d(z^n)^{k_a}(w^n)^{k_b}(z^{r_a})(w^{r_b}).$$
Moreover, as $a'+b'\equiv0\bmod n$, it follows that $r_a+r_b\equiv 0\bmod n$.  As $0\leq r_a,r_b<n$, it follows that either 
$$\begin{array}{lr}
 p = (z\zb)^b(w\wb)^d(z^n)^{k_a}(w^n)^{k_b},   & \mbox{ if  } r_a=r_b=0 \\
p = (z\zb)^b(w\wb)^d(z^n)^{k_a}(w^n)^{k_b}(z^{r_a}w^{n-r_a}), & \mbox{ if } r_a+r_b=n 
\end{array}$$
Cases (3) and (4) are proved analogously.  Thus, an arbitrary $p\in\Sym(V^*)^{\ZZ_n}$ may be written as a product of elements in the list and therefore the list forms a Hilbert basis for the $\ZZ_n$-action on $V$.
\end{proof}

Thus the quotient $Q$ embeds into $\CC^{4+2(n+1)}$ via the Hilbert map $F:V/\ZZ_n\hookrightarrow\CC^{4+2(n+1)}$ given by
$$
\begin{array}{lcl}
		      &  & \hspace{2pt} (a_0,a_1,a_2,a_3) \qquad  (x_0,x_1,\ldots,x_n)  \quad(y_0,y_1,\ldots,y_n) \\
F(z,\zb,w,\wb)&=& (\overbrace{z\zb,w\wb,\zb w,z\wb},\overbrace{w^n,zw^{n-1},\ldots, z^n},\overbrace{\wb^n,\zb^{}\wb^{n-1},\ldots,\zb^n}),
\end{array}
$$
where we write $\CC^{4+2(n+1)}=\gl(\CC)^*\times W^*$ and $W:=V\times\overline{V}$, with coordinates:
$$\begin{array}{lcl}
\gl(\CC) &=& \Span\{a_0,\ldots,a_3\} \\
V &=& \Span\{x_0,\ldots\ldots,x_n\} \\
\overline{V} &=&   \Span\{y_0,\ldots\ldots,y_n\}.
\end{array}$$

\subsubsection{The defining ideal}
\label{sec:DefIdeal}
It will prove useful to have explicit description of the defining ideal of the image of $Q\subseteq\CC^{4+2(n+1)}$.  Table \ref{table:GenKer} lists the generators of the ideal $\II:=\ker(F^*)\subseteq\Sym(\gl(\CC)\times W)$.

\begin{table}[h!]
\caption{The generators of the defining ideal $\II$ of $V/\ZZ_n$ in $\CC^{4+2(n+1)}$}
$$
\begin{array}{|ccl|c|c|c|} \hline
	& & \mbox{Name}        & \mbox{Index Range}  & H-\deg  \\  \hline
C\!	&=& a_0a_1-a_2a_3  &  $n/a$					    &  0             \\ \hline
A_k\!	&=& (n-k)(a_0y_k-a_3y_{k+1})+k(a_1y_k-a_2y_{k-1})  & k=0\ldots n   & (2k-n) \\ 
B_k\!  &=& (a_0y_k-a_1y_k)+(a_2y_{k-1}-a_3y_{k+1})	& k=1\ldots (n-1)  & (2k-n)  \\ \hline \noalign{\smallskip}
\overline{A}_k\! &=& (n-k)(a_0x_k-a_2x_{k+1})+k(a_1x_k-a_3x_{k-1})   & k=0\ldots n  & (n-2k) \\
\overline{B}_k\! &=& (a_0x_k-a_1x_k)+(a_3x_{k-1}-a_2x_{k+1})  &  k=1\ldots (n-1)  & (n-2k)  \\ \hline
D_{ijk}\! &=&  y_iy_j-y_ky_{i+j-k}   & 0\leq i,j,k \leq n  & (i+j)-2n  \\
\overline{D}_{ijk}\! &=&  x_ix_j-x_kx_{i+j-k}  & 0\leq i,j,k \leq n  & 2n-(i+j)  \\ \hline
M_{ij}\! &=& x_iy_j-a_0^ja_1^{n-i}a_3^{i-j} & 0\leq j\leq i\leq n  & 2(j-i) \\
M_{ij}\! &=& x_iy_j-a_0^ia_1^{n-j}a_2^{j-i} & 0 \leq i\leq j \leq n  & 2(j-i) \\ \hline
\end{array}
$$
\label{table:GenKer}
\end{table}

\begin{remark} These generators are given in this form because of their representation theoretic properties.  It is tedious, but straightforward to check that these generate $\ker(F^*)$.
\end{remark}

\subsection{The Lie-Poisson Structure of the Minimal Embedding}
\label{sec:LPg}

Recall that the Hilbert basis in Section \ref{sec:HilbertBasis} has a Lie algebra structure (the cotangent Lie algebra) inherited from the Poisson bracket on $\Sym(V^*)^{\ZZ_n}$.  Moreover, as was proved in Chapter 3, this structure determines the linear part of any extension of the Poisson structure on $Q$ to $\CC^{4+2(n+1)}$ (one may also see this directly; there are no linear terms in $\II$).  Direct computation gives that the cotangent Lie algebra is isomorphic to
$$\g:=\gl(\CC)\ltimes(V_n\oplus V_n) $$
where $V_n$ is the $(n+1)$-dimensional irreducible representation of $\slt(\CC)$.  The action is given by $\rho_n\oplus n\cdot(Id)$ on the first $V_n$-factor and $-\left[\rho_n\oplus n\cdot(Id)\right]$ on the second $V_n$-factor (see Appendix \ref{sec:AppendixSL2} for an explanation of the notation).  The corresponding Lie-Poisson structure determined by the cotangent Lie algebra is given by the follow skew-symmetric coefficient matrix:

\begin{table}[h!]
\caption{The coefficient matrix for the Lie-Poisson algebra $(\gl(\CC)\times(V_n\oplus V_n))^*$}
$$
\begin{array}{c|cccc|ccc|ccc}
  &a_0&a_1&a_2&a_3&\ldots&x_k&\ldots&\ldots&y_k&\ldots\\ \hline
 a_0&0&0&-a_2&a_3& &-kx_k& & &ky_k& \\
 a_1&0&0&a_2&-a_3& &-(n-k)x_k& & &(n-k)y_k& \\
 a_2&a_2&-a_2&0&a_0-a_1& &-kx_{k-1}& & &(n-k)y_{k+1}& \\
 a_3&-a_3&a_3&a_1-a_0&0& &-(n-k)x_{k+1}& & &ky_{k-1}& \\ \hline
 \vdots &&&&&&&&&&\\
 x_k& & & & & & 0 &  & &  0 &\\
 \vdots &&&&&&&&&&\\ \hline
 \vdots &&&&&&&&&&\\
 y_k& & & & & &  &  & &  0 &\\
 \vdots &&&&&&&&&&\\ 
 \end{array}
$$
\end{table}
\begin{remark} The $\gl(\CC)$ subalgebra of the Lie-Poisson structure above is related to the usual basis by:
$$ t\mapsto (a_0+a_1), \qquad H\mapsto (a_0-a_1), \qquad E\mapsto a_2, \qquad F\mapsto a_3. $$
In the following, we will identify the elements of $\gl(\CC)$ with this subalgebra.
\end{remark}

\subsubsection{ The $\slt(\CC)$-action and $\gl(\CC)$-action}

The Lie algebra $\slt(\CC)\isom\Span\{a_0-a_1,a_2,a_3\}\subseteq\Sym(\gl(\CC)\ltimes W)$ acts on $\Sym(\gl(\CC)\ltimes W)$ via their Hamiltonian actions with respect to the Lie-Poisson structure: 
$$s\cdot f:=\{f,s\}_1 \mbox{\quad for } s\in\slt(\CC), \, f\in \Sym(\gl(\CC)\ltimes W)$$   
As the Lie-Poisson bracket has degree-$(-1)$, this action preserves polynomial degree and thus restricts to actions on each $\Sym^k(\gl(\CC)\ltimes W)$.  In particular, it acts on the generators of $\II$, decomposing them into direct sums of irreducible representations:

\begin{proposition} 
\label{prop:RepIdeal}
The generators given in Table \ref{table:GenKer} span the following representations of $\slt(\CC)$:
\begin{enumerate}
\item $\Span\{C\}\isom V_0$
\item $\Span\{\overline{A_k}\}\isom\Span\{A_k\}\isom V_n$
\item $\Span\{\overline{B_k}\}\isom\Span\{B_k\}\isom V_{n-2}$
\item $\Span\{\overline{D_{ijk}}\}\isom\Span\{D_{ijk}\} \isom\bigoplus_{k=0}^{\lfloor\frac{n-1}{2}\rfloor}V_{2(n-(2k+1))}$
\item $\Span(M_{ij})\isom\bigoplus_{k=0}^n V_{2(n-k)}$
\end{enumerate}
\end{proposition}

Similarly, the central element $t=(a_0+a_1)\in\gl(\CC)$ acts on $\Sym(\gl(\CC)\ltimes W)$ via 
$$t\cdot f:=\{f,t\}_1  \mbox{\quad for }  f\in \Sym(\gl(\CC)\ltimes W)$$  
which, as before, gives a notion of $t$-degree on $ \Sym(\gl(\CC)\ltimes W)$.

\subsection{ A non-Poisson extension of $\brac_Q$ }
\label{sec:BetaExt}

We will now fix an extension of $\brac_Q$ to a bilinear bracket $\brac_\beta$ with corresponding bivector $\beta=\beta^1+\beta^{n-1}$.  The linear part $\beta^1$ (with corresponding bracket $\brac_1$) must be the Lie-Poisson bivector on $\g^*$ in the previous section.  We define the nonlinear summand explicitly by the formula 
 $$\beta^{n-1}(dx_i,dy_j)= \left\{ \begin{array}{ll} 
 ((n-i)(n-j)a_0^ja_1^{n-i-1}+(ij)a_0^{j-1}a_1^{n-i})a_3^{i-j} & \mbox{if $i \geq j$};\\ 
 ((n-i)(n-j)a_0^ia_1^{n-j-1}+(ij)a_0^{i-1}a_1^{n-j})a_2^{j-i}& \mbox{if $i \leq j$}.\end{array} \right.$$
 where $\beta^{n-1}$ is defined to be zero on all other basis elements of $\g^*\wedge\g^*$.  This extension of $\brac_Q$ does {\it not} define a Poisson structure on $\g^*\isom\CC^{4+2(n+1)}$, as the vanishing condition on the Schouten bracket is {\it not} satisfied:
$$[\beta,\beta]=[\beta^1,\beta^1]+[\beta^1,\beta^{n-1}]+[\beta^{n-1},\beta^{n-1}]\neq 0.$$
The discussion in Section \ref{sec:LPg} implies that $[\beta^1,\beta^1]=0$.  The last term $[\beta^{n-1},\beta^{n-1}]=0$ also vanishes, which follows from the fact that $\beta^{n-1}$ is nonzero only on elements of the form $dx_i\wedge dy_j\in\g^*\wedge\g^*$, while $\beta^{n-1}$ has coefficients in $\Sym(\gl(\CC))$ --- i.e.  the image of the bracket corresponding to $\beta^{n-1}$ is a subset of its kernel.  On the other hand, one can directly note that $[\beta^1,\beta^{n-1}](da_2,dx_1,dy_3)\neq0$.  

\begin{proposition} A partial description of the jacobiator $[\beta,\beta]$ is given by:

$$[\beta,\beta](du,dv,dw)=\left\{\begin{array}{ll}  
r(a)C, \mbox{ with } r\in\Sym^{n-1}(\gl(\CC)), &  (da_i,dx_j,dy_k), \, i=2,3 \\
p(a)A_l+q(a)B_l \mbox{ with } p,q\in\Sym^{n-1}(\gl(\CC)), &  (dx_i,dy_j,dy_k) \\
\overline{p}(a)\overline{A}_l+\overline{q}(a)\overline{B}_l \mbox{ with } \overline{p},\overline{q}\in\Sym^{n-1}(\gl(\CC)), & (dy_i,dx_j,dx_k)
\end{array}\right.$$
For those $(du,dv,dw)$ not included above, or not achievable from the above using permutations, the jacobiator is zero. In particular, we have the contractions with $a_0,a_1$ are zero:
$$\iota_{da_0}[\beta,\beta]=0  \mbox{\quad and } \iota_{da_1}[\beta,\beta]=0 $$
\end{proposition}

\begin{proof} A direct computation using the definition of $\beta$ gives that the jacobiator is zero on those triples not on the above list.  That the jacobiator takes the form above follows from two observations.  For the case $(da_i,dx_j,dy_k)$,  direct computation using the definition of $\beta$ shows $[\beta,\beta](da_i,dx_j,dy_k)\in\Sym(\gl(\CC))$:
\begin{eqnarray}
\nonumber  [\beta,\beta](da_i,dx_j,dy_k) &=& [\beta^1,\beta^{n-1}](da_i,dx_j,dy_k) \\
\nonumber		&=& \dd_{\g}\beta^{n-1}(da_i,dx_j,dy_k) \\
\nonumber		&=&\! \{a_i,\beta^{n-1}(dx_j,dy_k)\}_1 +\!\{x_i,\overbrace{\beta^{n-1}(dy_k,da_i)}^0\}_1+\{y_k,\overbrace{\beta^{n-1}(da_i,dx_j)}^0\}_1\!+\\
\nonumber		&  &\!	+ \beta^{n-1}(da_i,\underbrace{d\{x_j,y_k\}_1}_0)+\beta^{n-1}(dx_j,d\{y_k,a_i\}_1)+\beta^{n-1}(dy_k,d\{a_i,x_j\}_1) \\
\nonumber		&=& \{a_i,\beta^{n-1}(dx_j,dy_k)\}_1+\beta^{n-1}(dx_j,d\{y_k,a_i\}_1)+\beta^{n-1}(dy_k,d\{a_i,x_j\}_1). 
\end{eqnarray}

The remaining nonzero terms are all in $\Sym(\gl(\CC))$ by definition of $\beta$. However, $[\beta,\beta]$ has coefficients in $\II$, as $\beta$ corresponds to the extension of a Poisson bracket on $V(\II)\isom Q$.  The only way $[\beta,\beta](da_i,dx_j,dy_k)\in\Sym(\gl(\CC))\cap\II$ is for $[\beta,\beta](da_i,dx_j,dy_k)=r(a)C$ as given in the proposition.  The same reasoning applies to the other triples.  

To prove the last statement, it is enough to show the contractions are zero for $t=a_0+a_1$ and $H=a_0-a_1$, as contraction is linear.  However, from the equality above, $[\beta,\beta](dt,dx_j,dy_k)=0$ if and only if 
$$\{t,\beta^{n-1}(dx_j,dy_k)\}_1=\beta^{n-1}(dx_j,d\{t,y_k\}_1)+\beta^{n-1}(d\{t,x_j\}_1,dy_k)$$
That is, it holds if and only if $\beta^{n-1}$ preserves the $t-\deg$.  Similarly, $[\beta,\beta](dH,dx_j,dy_k)=0$ if and only if $\beta^{n-1}$ preserves the $H-\deg$.  Both of these are easily verifiable from the definition of $\beta^{n-1}$.
\end{proof}

\begin{lemma} The restriction of the skew-symmetric, bilinear bracket $\brac_\beta$ to the image of the embedding $Q$ agrees with the Poisson bracket $\brac_Q$.
\end{lemma}

\section{Proof of the Nonembedding Result}

In this section, we begin the proof of Theorem \ref{thm:NonEmb}:
\begin{reptheorem}{thm:NonEmb} Given $n>2$ with $n$ odd, there does not exist a Poisson bracket on $\RR^{4+2(n+1)}$ extending the Poisson structure on $Q=V/\ZZ_n$.
\end{reptheorem}

\begin{remark} The proof given only requires $n$ to be odd at one step, which is indicated in bold.  It is likely that the result also holds for even $n$.
\end{remark}

\begin{remark}[Outline of proof]
The proof of Theorem \ref{thm:NonEmb} roughly follows the steps below:
\begin{enumerate}
\item For the sake of contradiction, we assume the existence of a Poisson structure $\pi$ extending the Poisson structure on $Q$.
\item Then, we apply the semi-linearization techniques of Chapter 5 to obtain a change of coordinates that transforms $\pi$ into an easy-to-work-with form.
\item We obtain an explicit description of the resulting embedding.  In particular, we show it coincides with the Hilbert map $F$ in the section above on the $\gl(\CC)$-coordinates.  Such a description in turn yields a  tangible description of $\pi$ in terms of the ``non-Poisson extension'' $\beta$ and a bivector $\alpha$ that vanishes on $Q$.
\item We analyze the jacobiator $[\pi,\pi]$ in terms of the description obtained in the step above.  This analysis results in $[\pi,\pi]\neq0$, contradicting the assumption that $\pi$ is a Poisson structure.  This final step consists of three parts:
\begin{enumerate}
\item We define a grading on $\Sym(\g)$ that corresponds to infinitesimal neighborhoods $\nbd{k}{\gl(\CC)}{\CC^{4+2(n+1)}}$ and examine the restriction of $[\pi,\pi]$ to 
$$\nbd{1}{\gl(\CC)}{\CC^{4+2(n+1)}}\cap\nbd{n-1}{\pt}{\CC^{4+2(n+1)}}.$$
\item We obtain a refinement of our description of $\pi$ on the neighborhood in the previous step.  By examining the restriction of $[\pi,\pi]$ applied to $(dt,dw_i,dw_j)$, we conclude the term $[\pi,\pi](dt,dw_i,dw_j)$ vanishes on this neighborhood if and only if the restriction of $\alpha$ vanishes on $V(C)=V(a_0a_1-a_2a_3)$.
\item Using this refinement of $\alpha$, we determine that the restriction of $[\pi,\pi](dw_i,dw_j,dw_k)$ on $\nbd{1}{\gl(\CC)}{\CC^{4+2(n+1)}}\cap\nbd{n-1}{\pt}{\CC^{4+2(n+1)}}$ cannot vanish, contrary to our first assumption.  Essentially, this contradiction follows from the fact that an $\alpha$ vanishing on $V(C)$ cannot cancel the nonzero terms of the jacobiator $[\beta,\beta]$ computed in the section above.
\end{enumerate}
\end{enumerate}
\end{remark}

\noindent {\it Proof:} As noted at the start of the chapter, it is enough to show that the Poisson bracket on $V_\CC/{\ZZ_n}$ does not extend to a Poisson bracket on $\CC^{4+2(n+1)}$ in a formal neighborhood of $\pt\subseteq\CC^{4+2(n+1)}$ (see Sections \ref{sec:Complexification} and \ref{sec:ReductionAlg}).  

Assume, for the sake of contradiction, there is such a Poisson bracket $\brac=\brac_\pi$ (with corresponding bivector $\pi$) extending the Poisson bracket $\brac_Q$.  The linear Poisson structure $\pi^1$ coincides with $\beta^1$ of Section \ref{sec:LPg}.   Applying the semi-linearization in Proposition \ref{prop:LeviImprove} (see Example \ref{ex:ZnCotangent}), we obtain a formal system of coordinates

$$\phi: \CC[[a_0,\ldots a_3,x_0,\ldots,x_n,y_0,\ldots,y_n]] \twoheadrightarrow\CC[[z\zb,\ldots,z\wb,\ldots z^iw^{n-i}\ldots \zb^i\wb^{n-i}\ldots]] $$
in which the Poisson bivector takes the form $\pi=\beta^1+\pi^{nl}$, where  $\pi^{nl}$ is the nonlinear part: 
$$\pi^{nl}=P_{ij}\frac{\partial}{\partial w_i}\wedge\frac{\partial}{\partial w_j} + \sum_{l,i,j}\lambda_l(x_iy_j)^l\bivect{t}{w_k}  \mbox{\qquad   where the $P_{ij}$ are nonlinear, and } w_i\in\{x_i,y_i\}$$

Denote by $F_\phi$ the embedding with pullback $\phi$.  By the semi-linearization procedure in Proposition  \ref{prop:LeviImprove}, $F$ coincides with $F_\phi$ up to first order.  In the next section, we will determine $F_\phi$ in higher orders.

\subsection{Determining the Embedding $F_\phi$}
\label{sec:DeterminingPhi}

Applying the Levi decomposition has the effect of putting the extension $\pi$ into a specific, easy to work with form, at the expense of losing a description of the embedding $F_\phi$.  The construction of the change of coordinates for the Levi decomposition proceeds in two steps: (see Theorem \ref{thm:LeviDecomp} and Theorem 3.2.3 in \cite{Dufour05})
\begin{enumerate}
\item First, a change of coordinates is applied to the semi-simple subalgebra $\g$, linearizing bracket $\brac$ on this subalgebra.
\item Then, the coordinates of the radical are deformed in a way that linearizes the cross terms of $\brac$.
\end{enumerate}

However, in the case at hand, the first change of coordinates is unnecessary:

\begin{lemma} The map $F_\phi$ coincides with the Hilbert map $F$ on the semi-simple coordinates.  That is, $\phi$ takes the subalgebra $\Sym(\slt(\CC))\subseteq\Sym(\g)$ to the subalgebra $\Sym(\slt(\CC))\subseteq\Sym(V^*)^{\ZZ_n}$:
$$\phi(a_0-a_1)=z\zb-w\wb, \quad \phi(a_2)=\zb w, \quad \phi(a_3)=z\wb. $$
\end{lemma}

\begin{proof}  The Levi factor $\slt(\CC)$ is unique up to conjugation by an element of $\Sym(\g)$.  Thus the change of coordinates that linearizes $\brac$ on $\slt(\CC)$ is given by 
$$u\mapsto u+\{p_{ss},u\}_1 \mbox{\qquad where } p_{ss}\in\Sym(\g) \mbox{ and } u\in\Sym(\g).$$
For this change of coordinates to define a Poisson bracket extending $\brac_Q$, the function $\{p_{ss},u\}\in\II$ must vanish on $V/\ZZ_n$ for all $u\in\Sym(\g)$.  We claim that the $p_{ss}$ has no linear terms.  If this were the case, then $\{p_{ss},u\}_1$ would have linear terms (by considering a linear $u$).  As $\{p_{ss},u\}_1\in\II$ and $\II$ contains no linear terms, this is impossible.  Thus, $p_{ss}$ has no linear terms.

We now use the fact that the two Levi factors are conjugate to obtain conditions on $p_{ss}$.  Let $s_i\in\slt(\CC)$, and write $\tilde{s}_i=s_i+\{p_{ss},s_i\}_1$.  As the span of $\tilde{s}_i$ is isomorphic to $\slt(\CC)$, we arrive at the following equalities:

\begin{eqnarray}
\{\tilde{s}_i,\tilde{s}_i\}_1 &=& \{s_i+\{p_{ss},s_i\}_1,s_j+\{p_{ss},s_j\}_1\}_1 \nonumber\\
s_k+\{p_{ss},\{s_i,s_j\}_1\}_1 &=& s_k + \{s_i,\{p_{ss},s_j\}_1\}_1+ \{\{p_{ss},s_i\}_1,s_j\}_1\}_1+ \nonumber\\
						&& \hspace{4cm} + \{\{p_{ss},s_i\}_1,\{p_{ss},s_j\}_1\}_1 \nonumber\\
2\{p_{ss},\{s_i,s_j\}_1\}_1 &=&  \{\{p_{ss},s_i\}_1,\{p_{ss},s_j\}_1\}_1,    \nonumber
\end{eqnarray}
where on the left hand side we used that
$$\{\tilde{s}_i,\tilde{s}_i\}_1=\tilde{s}_k=s_k+\{p_{ss},s_k\}_1=s_k+\{p_{ss},\{s_i,s_j\}_1\}_1.$$
Now suppose that the first nonzero term in $p_{ss}$ has degree $d_{ss}$.  Then the minimal degree of the left-hand term $2\{p_{ss},\{s_i,s_j\}_1\}_1$ is also $d_{ss}$, whereas the minimal degree of the right-hand term $\{\{p_{ss},s_i\}_1,\{p_{ss},s_j\}_1\}_1$ is $2d_{ss}-1$.  This only occurs if $d_{ss}=1$ or both sides are zero.  As, $p_{ss}\neq0$, it follows that 
$$2\{p_{ss},\{s_i,s_j\}_1\}_1=0.$$
As $\slt(\CC)$ is semi-simple, it follows that $p_{ss}$ commutes with every element of $\slt(\CC)$.  Thus the coordinate transformation
$$u\mapsto u+\{p_{ss},u\}_1=u+0=u$$
is the identity transformation.
\end{proof}

We can in fact improve the above lemma:

\begin{lemma} The surjection $\phi$ sends the subalgebra $\rm{Sym}(\gl(\CC))$ to the subalgebra $$\rm{Sym}(\gl(\CC))\subseteq\Sym(V^*)^{\ZZ_n}$$ generated by quadratic terms. That is:
$$\phi(t)=z\zb+w\wb.$$
\end{lemma}

\begin{proof} As $\slt(\CC)$ acts on both $\Sym(\g)$ and $\Sym(V^*)^{\ZZ_n}$, the lemma above implies that $\phi$ is a homomorphism of $\slt(\CC)$-representations.  This imposes conditions on the image of the coordinates $\{t,x_i,y_j\}$ under $\phi$:

\begin{enumerate}
\item $\Span\{\phi(t)\}\subseteq\Sym(V^*)^{\ZZ_n}$ is isomorphic to $V_0$, where $t=a_0+a_1$,
\item $\Span\{\phi(x_i),\, i=0\ldots n\} \subseteq\Sym(V^*)^{\ZZ_n}$ is isomorphic to $V_n$,
\item $\Span\{\phi(y_i),\, i=0\ldots n\}\subseteq\Sym(V^*)^{\ZZ_n}$ is isomorphic to $V_n$,
\end{enumerate}
where $V_k$ is the $(k+1)$-dimensional irreducible representation of $\slt(\CC)$.

As the 1-dimensional irreducible representations of $\slt(\CC)$ in $\Sym(V^*)^{\ZZ_n}$ are spanned by powers of $(z\zb+w\wb)$, the map $\phi(t)=p_t(z\zb+w\wb)$ is a polynomial in $(z\zb+w\wb)$ with zero constant term and linear term equal to $(z\zb+w\wb)$.

We need to show that the polynomial $p_t$ is linear.  As $\phi$ is a Poisson map, we have:
\begin{eqnarray}
\phi(\{t,x_i\}) &=& \{\phi(t),\phi(x_i)\}_Q \nonumber \\
-n\phi(x_i)+\phi(\sum_{K}\lambda_k(x_{k_1}y_{k_2})^k) &=& \{p_t(z\zb+w\wb),\phi(x_i)\}_Q, 
\end{eqnarray}
where $K=(k,k_1,k_2)$ is a multi-index with $k>0$.  Writing $\phi=F^*+\psi$ as the sum of the first order terms and higher order terms, we examine the term $\phi(\sum_{K}\lambda_k(x_{k_1}y_{k_2})^k)$:

\begin{eqnarray}
\phi(\sum_{K}\lambda_k(x_{k_1}y_{k_2})^k) &=& \sum_{K}\lambda_k(\phi(x_{k_1})\phi(y_{k_2}))^k \nonumber \\
    &=& \sum_K \lambda_k\left((z^{k_1}w^{n-k_1}+\psi(x_{k_1}))(\zb^{k_2}\wb^{n-k_2}+\psi(y_{k_2}))\right)^k. \nonumber
    \end{eqnarray}
    
Expanding the terms inside each summand, the term $(z^{k_1}w^{n-k_1})(\zb^{k_2}\wb^{n-k_2})$ has degree strictly less than the terms involving $\psi$.  Furthermore, $(z^{k_1}w^{n-k_1})(\zb^{k_2}\wb^{n-k_2})$ has $t$-degree 0 (i.e.\! an equal number of $\{z,w\}$-terms and $\{\zb,\wb\}$-terms).  The right hand side of Equation 6.2 has no $t$-degree 0 component as $\{t,\cdot\}_Q$ vanishes on $t$-degree zero functions.  It follows that each $\lambda_k(z^{k_1}w^{n-k_1})(\zb^{k_2}\wb^{n-k_2})$ term must be cancelled by the $\phi(x_i)$ term.  However, any nontrivial cancellation is impossible as: 
\begin{enumerate}
\item $\Span\{\phi(x_l);\, l=0\ldots n\}\isom V_n$ and {\bf $n$ is odd}, 
\item while each term
$$\lambda_k\left((z^{k_1}w^{n-k_1})(\zb^{k_2}\wb^{n-k_2})\right)^k\in\CC[z\zb,w\wb,\zb w,z\wb]\isom \bigoplus_{k\in\NN}V_{2k}$$
is contained in a direct sum of even-dimensional irreducible representations.
\end{enumerate}
Thus $\phi(\sum_{K}\lambda_k(x_{k_1}y_{k_2})^k)=0$ and Equation 6.2 becomes 
$$-n\phi(x_i)= \{p_t(z\zb+w\wb),\phi(x_i)\}_Q.$$
The only way for this equality to hold is if $p_t(z\zb+w\wb)=z\zb+w\wb$.  
\end{proof}

Finally, we give a more explicit description of $\phi(x_i)$ and $\phi(y_i)$. Using the lemma above, we see that
 $$-n\phi(x_i)=\{z\zb+w\wb,\phi(x_i)\}_Q.$$
This implies:
\begin{enumerate}
\item $\phi(x_i)=p_i\cdot( z^iw^{n-i})$ where $p_i\in\CC[z\zb,w\wb,\zb w,z\wb]$ has $t$-degree 0.  
\item As $\phi$ is a homomorphism of $\slt(\CC)$-representations, the image $\{\phi(x_i);\, i=0\ldots n\}$
 spans a subspace isomorphic to the irreducible representation $V_n$.  
 \item $\Span\{z^iw^{n-i};\, i=0\ldots n\}\isom V_n$.
 \end{enumerate}
It follows that 
$$\Span\{\phi(x_i);\, i=0,\ldots,n\}=\Span\{p_i \cdot(z^iw^{n-i});\, i=0,\ldots,n\}\isom V_0\otimes V_n.$$ 
Thus $p_i=p(z\zb+w\wb)$ is a polynomial in $z\zb+w\wb$ with constant term 1.  Making the same argument with the $y_i$ implies:
$$\phi(y_i) = q(z\zb+w\wb) \zb^i\wb^{n-i}, $$
where $q$ is a polynomial in $z\zb+w\wb$ with constant term 1.  The following proposition summarizes the above:
 
\begin{proposition} The map $\phi$ defining the embedding $F_\phi$ has the form:\vspace{.5cm}\\ 
$\left\{\begin{array}{lcl}
\phi(a_0) &=& z\zb \\
\phi(a_1) &=& w\wb  \\
\phi(a_2) &=& \zb w  \\
\phi(a_3) &=& z\wb  \\
\end{array}\right.,$ \hspace{2cm}
$\left\{\begin{array}{lcl}
\phi(x_i) &=& p(z\zb+w\wb) z^iw^{n-i}\\
\phi(y_i) &=& q(z\zb+w\wb) \zb^i\wb^{n-i} \\
\end{array}\right.,$
\vspace{.5cm}

where $p,q$ are polynomials of $(z\zb+w\wb)$ with constant term $1$.
\end{proposition}

From the above description of the embedding $F_\Phi$, we easily see:

\begin{corollary} The ideal $\ker(\phi)$ agrees with the na\"{i}ve ideal of definition $\II=\ker(F^*)$ given in Table \ref{table:GenKer}, with the exception of the $M_{ij}$ which are deformed to
$$\left\{\begin{array}{lcll}
M_{ij} &=& x_iy_j-p(t)q(t)a_0^ja_1^{n-i}a_3^{i-j} & 0\leq j\leq i\leq n \\
M_{ij} &=& x_iy_j-p(t)q(t)a_0^ia_1^{n-j}a_2^{j-i} & 0 \leq i\leq j \leq n \\
\end{array}\right.$$
\end{corollary}

\begin{corollary} The Poisson bivector $\pi$ has the decomposition $\pi=(\beta+\alpha)+\tilde{\pi}$ where
\begin{enumerate}
\item $\beta=\beta^1+\beta^{n-1}$
\item $\alpha$ vanishes on $\ker(\phi)$
\item $\tilde{\pi}=\sum_{k\geq n}\pi^k$ is the sum of homogeneous components of $\pi$ greater than or equal to $n$.
\end{enumerate}
\end{corollary}

\begin{proof}  We will show
$$\phi\left(\{w_i,w_j\}-\{w_i,w_j\}_{\beta}\right)\in\bigoplus_{k\geq n}\Sym^{k}(V^*)^{\ZZ_n} \mbox{ for $w_i\in\{x_i,y_i,t\}$ }$$

To prove this, simply calculate the term:
$$\begin{array}{rcl}
\phi(\{x_i,y_j\}-\{x_i,y_j\}_{\beta}) &=& \{\phi(x_i),\phi(y_j)\}_Q-  \{z^iw^{n-i},\zb^j\wb^{n-j}\}_Q  \vspace{2pt}\\
&=& \{z^iw^{n-i},\psi(y_j)\}_Q+ \{\psi(x_i),\zb^j\wb^{n-j}\}_Q+\{\psi(x_i),\psi(y_j)\}_Q  \vspace{2pt}\\
\end{array}$$

However, $\psi(x_i),\psi(y_j)\in\Sym(V^*)^{\ZZ_n}$ both have degree greater than or equal $n+2$ as
\begin{enumerate}
\item both have degree greater than $n$, and
\item the algebra of invariants $\Sym(V^*)^{\ZZ_n}\subseteq\Sym(V^*)$ is generated by elements of degrees $2$ and $n$.  
\end{enumerate}
Therefore, as the Poisson bracket $\brac_Q$ has degree $-2$, the right-hand side has degree greater than or equal to $n$.  The other cases are analogous.

\end{proof}

\subsection{The $W$-grading}

While our approach often uses the usual grading on the polynomial algebra $\Sym(\g)$, we will now introduce a grading on the coordinate algebra of $(\gl(\CC)\ltimes W)^*\isom\CC^{4+2(n+1)}$ that takes advantage of the semi-linearization of $\pi$.  

\begin{definition}  Define the $W$-grading on ${\rm Sym}(\g)$ by giving the generators the following degrees: $\deg_W(a_i)=0$ and $\deg_W(w_i)=1$ for all $a_i\in\gl(\CC)$ and $w_i\in W$.  This grading results in the decomposition 
$${\rm Sym}(\g)=\bigoplus_{k\in\NN}\Sym^k(W)$$
where each $\Sym^k(W)$ is considered as a $\Sym(\gl(\CC))$-algebra via the associative multiplication.  To distinguish between the $W$-grading and the usual polynomial grading, we will denote the $k^{\rm th}$ $W$-graded homogeneous component of $\Sym(\g)$ by $\Sym(\g)_{k}$.   Thus $\Sym_\CC(\g)_k=\Sym^k_{\gl(\CC)}(W)$.
\end{definition}

Recall from Definition \ref{def:GrPoissonAlg}, that a grading on an algebra of functions induces a grading on the space of brackets and the space of multivectors.  We will use subscripts to denote the decomposition of a multivector $X$ with respect to the $W$-grading: 
$$X=\sum_{d\in\ZZ} X_d  \mbox{,     where $X_d$ has $W$-degree $d$ }.$$
Moreover, with respect to a grading on the space of multivectors, the Schouten bracket $[\cdot,\cdot]$ preserves the grading. 

While the usual graded decomposition of $\pi$ corresponds to extensions of Poisson structures on infinitesimal neighborhoods of the origin $\nbd{k}{\pt}{\g^*}$, the $W$-grading corresponds to extending Poisson structures on infinitesimal neighborhoods of $\gl(\CC)^*$ --- i.e. on $\nbd{k}{\gl(\CC)^*}{\g^*}$.

We will need two properties with respect to this grading:

\begin{proposition}
The linear Poisson bivector $\beta^1$ has $W$-deg 0,  and therefore the $\beta^1$-differential ${\rm d}_{\beta^1}$ preserves the $W$-grading.
\end{proposition}
\begin{proof}
This follows from the fact that $\gl(\CC)\ltimes W$ is a semi-direct product and $\{W,W\}_1$ is identically zero.
\end{proof}

Moreover, for any $k$, the decomposition of the space of $k$-vectors with respect to $W$-degree is bounded below: 

\begin{proposition} The minimal $W$-degree of a $W$-homogeneous $p$-vector is $-p$.
\end{proposition}

\begin{proof} This proposition is a consequence of the Leibniz identity. 
\end{proof}

\begin{example} $\beta$ decomposes as $\beta=\beta_{-2}+\beta_{0}$, where $\beta_{-2}=\beta^{n-1}$ and $\beta_{0}=\beta^1$.  Thus, we have the decomposition $\pi=\sum_{k\geq-2} \pi_k$ where $\pi_k=0$ for $k<-2$.  In particular,	 
$$\begin{array}{lcl}
\pi_{-2} &=&\beta^{n-1}+(\alpha_{-2}+\tilde{\pi}_{-2})\\

\pi_{-1}&=&(\alpha_{-1}+\tilde{\pi}_{-1})\\

\pi_{0}&=&\beta^1+(\alpha_{0}+\tilde{\pi}_{0})\\
\end{array}
$$
In the next section, we will examine the following $W$-graded components of the jacobiator $[\pi,\pi]$:
$$  [\pi,\pi]_{-2} = 2[\pi_{-2},\pi_0]+[\pi_{-1},\pi_{-1}] $$
$$ [\pi,\pi]_{-1} = 2[\pi_{-2},\pi_1]+[\pi_{-1},\pi_0]  $$

\end{example}

\subsection{Degree Conditions}

When expressing $\pi$ as a sum of its homogeneous components (with respect to the {\it usual} polynomial degree), the vanishing of the Schouten bracket becomes

$$0=[\pi,\pi]=[\sum_{k\in\NN}\pi^k,\sum_{k\in\NN}\pi^k]=\sum_{k\in\NN}\left(\sum_{i=1}^{k}[\pi^i,\pi^{k-i+1}]\right).$$
Note that for each $k$, the sum inside the parentheses on the right hand side must be zero.  In this section, we will examine the $W$-degree $(-2)$ and $(-1)$ components of $[\pi,\pi]$ for each $k<n$.

\subsubsection{The component $\pi^k_{-2}$ preserves the $t$-degree for $k<n$ }

\begin{proposition}[Facts about the bivector $\pi$]
\label{prop:Facts}
Let $1<k<n$.  The following easily verifiable facts of the bivector $\pi^k=\beta^k+\alpha^k$ will be used repeatedly:
\begin{enumerate}
\item[](Fact 1)  For all $a_i,a_j\in\gl(\CC)$, the bivector $\alpha^k(da_i,da_j)=0$.
\item[](Fact 2)  For $d=-2,-1,0$, the bivector $\alpha^k_{d}(dt,dw_i)=0$ for all $w_i\in\{x_i,y_i\}$.
\item[](Fact 3)  The $(-2)$-component vanishes on $V(C)$, where $C=a_0a_1-a_2a_3$.  Specifically,
 $$\alpha_{-2}^k(dw_i,dw_j)=p^k_{ij}C \mbox{\qquad  for  } p_{ij}^k\in\Sym(\gl(\CC))$$
\end{enumerate}
\end{proposition}

\begin{proof}  Fact 1 follows from the normal form given by the Levi decomposition: the subalgebra $\gl(\CC)$ is a 1-dimensional central extension of the Levi factor $\slt(\CC)$ (on which $\pi$ is linear).   Fact 2 follows directly from the definition of the $W$-grading: as the $W$-degree of $t$ is $0$, the bivector $\alpha^k_{d}(dt,dw_i)\in\Sym^k(\g)_{\leq 1}$.  However, the Levi-decomposition implies that $\alpha^k(dt,dw_i)=\sum(x_iy_j)^l\in\Sym(\g)_{\geq 2}$.   Finally, Fact 3 follows from the description of $\ker(\phi)$ given in Section \ref{sec:DefIdeal}:  $\ker(\phi)\cap\Sym(\gl(\CC))=\langle C\rangle$.
\end{proof}

We will prove that $\pi_{-2}^k$ preserves the $t$-degree, which will determine the form of $\pi^k_{-2}$ for $k<n$.

\begin{proposition}
\label{prop:tdeg}
The $k^{\rm th}$ component $\pi_{-2}^k$ preserves the $t$-degree for all $k<n$.  That is:
$$\{t,\pi_{-2}^k(dw_i,dw_j)\}_1=\pi_{-2}^k(d\{t,dw_i\}_1,dw_j)+\pi_{-2}^k(dw_i,d\{t,dw_j\}_1)$$
\end{proposition}
 
 \begin{proof} For all $k<(n-1)$, we have: 
\begin{equation}
\label{eqn:degk1}
0=[\beta^1,\alpha^k]_{-2}+\sum_{l=2}^{k-2}[\alpha^l,\alpha^{k-l}]_{-2}=\overbrace{[\beta^1,\alpha^k_{-2}]}^{\mbox{term I}}+\overbrace{\sum_{l=2}^{k-2}[\alpha_{-2}^l,\alpha_0^{k-l}]}^{\mbox{term II}}+\overbrace{\sum_{l=2}^{k-2}[\alpha_{-1}^l,\alpha_{-1}^{k-l}]}^{\mbox{term III}}
\end{equation}

Applying these three terms to an element of the form $(dt,dw_i,dw_j)$ yields the following lemmas:

\begin{lemma}[Term I]
\label{lemma:TermI}
The first term in Equation \ref{eqn:degk1} measures the extent to which $\alpha^k_{-2}$ preserves the $t$-degree:
$$ [\beta^1,\alpha_{-2}^k](dt,dw_i,dw_j) =\{t,\alpha_{-2}^k(dw_i,dw_j)\}_1-\alpha_{-2}^k(d\{t,w_i\}_1,dw_j)-\alpha_{-2}^k(dw_i,d\{t,w_j\}_1) $$
\end{lemma}

\begin{proof}  Calculating the Schouten bracket gives
$$\begin{array}{rcl}
  [\beta^1,\alpha_{-2}^k](dt,dw_i,dw_j)  &=& \beta^1(dt,d\alpha_{-2}^k(dw_i,dw_j))-\alpha_{-2}^k(d\beta^1(dt,dw_i),dw_j)+\\
    & &\qquad -\alpha_{-2}^k(dw_i,d\beta^1(dt,dw_j))+ \beta^1(dw_i,d\alpha^k_{-2}(dw_j,dt))+\\
    & &\qquad +\beta^1(dw_j,d\alpha^k_{-2}(dt,dw_i)) \\
    &=& \{t,\alpha_{-2}^k(dw_i,dw_j)\}_1-\alpha_{-2}^k(d\{t,w_i\}_1,dw_j)-\alpha_{-2}^k(dw_i,d\{t,w_j\}_1).
\end{array}$$
The terms in the middle line vanish Fact 2 of Proposition \ref{prop:Facts}.  
\end{proof}

\begin{lemma}[Terms II and III]  The second and third terms in Equation \ref{eqn:degk1} are identically zero.
\end{lemma}
\begin{proof} We write out the second term:
$$\begin{array}{rcl}
[\alpha_{-2}^l,\alpha_0^{k-l}](dt,dw_i,dw_j) &=& \alpha_{-2}^l(dt,d\alpha_0^{k-l}(dw_i,dw_j))+\alpha_{-2}^l(dw_i,d\alpha_0^{k-l}(dw_j,dt))+ \\
   & & +\alpha_{-2}^l(dw_j,d\alpha_0^{k-l}(dt,dw_i))+ \alpha_{0}^{k-l}(dt,d\alpha_{-2}^{l}(dw_i,dw_j))+ \\
   & & +\alpha_{0}^{k-l}(dw_i,d\alpha_{-2}^{l}(dw_j,dt))+\alpha_{0}^{k-l}(dw_j,d\alpha_{-2}^{l}(dt,dw_i)) \\
   &=& 0
\end{array}$$
The fourth term is zero by Fact 1 and the other terms vanish by Fact 2 of Proposition \ref{prop:Facts}.  Similar reasoning implies that term III (i.e. $[\alpha_{-1}^l,\alpha_{-1}^{k-l}](dt,dw_i,dw_j)$) is zero.
\end{proof}

In other words,$[\pi,\pi]^k_{-2}=0$ if and only if $\alpha_{-2}^k:W\wedge W\to\rm{Sym}^k(\gl(\CC))\cap\II$ preserves the $t$-degree for $k<(n-1)$.  The $k=(n-1)$ case is similar, as there is only one new term in the vanishing condition:

$$0=[\pi,\pi]^{n-1}_{-2}=[\beta^1,\alpha^k_{-2}]+[\beta^{n-1},\beta^1]+\sum_{l=2}^{k-2}[\alpha_{-2}^l,\alpha_0^{k-l}]+\sum_{l=2}^{k-2}[\alpha_{-1}^l,\alpha_{-1}^{k-l}]$$
However, as was shown in Section \ref{sec:BetaExt}, $\beta^{n-1}$ preserves the $t$-degree and 
$$[\beta^1,\beta^{n-1}](dt,dw_i,dw_j)=0$$ proving Proposition \ref{prop:tdeg}.
\end{proof}
 
That $\alpha^k_{-2}$ preserves the $t$-degree forces $\alpha_{-2}^k(dw_i,dw_j)$ to have the following form:

\begin{proposition}
\label{prop:Deg2Form}
  For all $k<n$, the $W$-degree $(-2)$ summand takes the values
$$ \left\{\begin{array}{l}
 \alpha_{-2}^k(dx_i,dx_j)=0\\
 \alpha_{-2}^k(dy_i,dy_j)=0\\
 \alpha_{-2}^k(dx_i,dy_j)=p^k_{ij}C  \mbox{\qquad where  } p^k_{ij}\in\Sym(\gl(\CC)) \end{array}\right.$$
\end{proposition}

\begin{proof} To prove the first equality, we use that the $t$-degree of $x_i$ is $(-n)$ and $\alpha^k_{-2}(dw_i,dw_j)\in\Sym(\gl(\CC))$:
\begin{eqnarray}
\{t,\alpha^k_{-2}(dx_i,dx_j)\}_1 &=& \alpha^k_{-2}(d\{t,x_i\}_1,dx_j)+\alpha^k_{-2}(dx_i,d\{t,x_j\}_1) \nonumber\\
0						&=&  \alpha^k_{-2}(-ndx_i,dx_j)+\alpha^k_{-2}(dx_i,-ndx_j) \nonumber\\
0						&=& -2n\alpha^k_{-2}(dx_i,dx_j) \nonumber
\end{eqnarray}
The second equality is proved similarly, while the third equality follows from Fact 3 of Proposition \ref{prop:Facts}.

\end{proof}

 \subsubsection{The component $\pi^k_{-1}$ preserves the $t$-degree $\bmod\, C$ for $k<n$ }
 
As in the $W$-degree $(-2)$ case, the following (slightly weaker) condition on the $t$-degree will determine the form of $\pi^k_{-1}$ for $k<n$.

\begin{proposition}
\label{prop:degt2}
  In the case of $W$-degree $(-1)$, the component $\pi^k_{-1}$ only preserves the $t$-degree up to a factor in the ideal $\langle C\rangle$:

$$\{t,\pi^k_{-1}(dw_i,dw_j)\}_1=\pi^k_{-1}(d\{t,w_i\}_1,dw_j)+\pi^k_{-1}(dw_i,d\{t,w_j\}_1)+q_{ij}$$
where $q_{ij}\in\langle C\rangle\subseteq\Sym(\g)$

\end{proposition}

\begin{proof} One proves this proposition in the same way as Proposition \ref{prop:tdeg}.  Consider the $W$-degree $(-1)$ component of the Schouten bracket $[\pi,\pi]^k$:

\begin{equation}
\label{eqn:degk2}
0=[\beta^1,\alpha^k]_{-1}+\sum_{l=2}^{k-2}[\alpha^l,\alpha^{k-l}]_{-1}=\overbrace{[\beta^1,\alpha^k_{-1}]}^{\mbox{term I}}+\overbrace{\sum_{l=2}^{k-2}[\alpha_{-1}^l,\alpha_0^{k-l}]}^{\mbox{term II}}+\overbrace{\sum_{l=2}^{k-2}[\alpha_{-2}^l,\alpha_{1}^{k-l}]}^{\mbox{term III}}
\end{equation}
 
Applying these three terms to an element of the form $(dt,dw_i,dw_j)$ yields three more lemmas:

\begin{lemma}[Term I]  The first term in Equation \ref{eqn:degk2} measures the extent to which $\alpha^k_{-1}$ preserves the $t$-degree:
$$ [\beta^1,\alpha_{-1}^k](dt,dw_i,dw_j) =\{t,\alpha_{-1}^k(dw_i,dw_j)\}_1-\alpha_{-1}^k(d\{t,w_i\}_1,dw_j)-\alpha_{-1}^k(dw_i,d\{t,w_j\}_1) $$
\end{lemma}
 
\begin{proof} This is identical to the proof of Lemma \ref{lemma:TermI} above.
\end{proof}
 
 \begin{lemma}[Term II]  The second term vanishes identically: $$\sum[\alpha_{-1}^l,\alpha_0^{k-l}](dt,dw_i,dw_j)=0$$
 \end{lemma}
 
 \begin{proof}  This follows from Fact 1 of Proposition \ref{prop:Facts}.
 \end{proof}
 
 \begin{lemma}[Term III]  The third term $\sum_{l=2}^{k-2}[\alpha_{-2}^l,\alpha_{1}^{k-l}]$ belongs to the ideal generated by $\langle C\rangle\subseteq\Sym(W)_1$.
 \end{lemma}
 
 \begin{proof}  We examine each summand $[\alpha_{-2}^l,\alpha_{1}^{k-l}]$ on $(dt,dw_i,dw_j)$.  Using Fact 1 from Proposition \ref{prop:Facts}, the nonzero terms are:
$$ [\alpha_{-2}^l,\alpha_{1}^{k-l}](dt,dw_i,dw_j) = \alpha^l_{-2}(dw_i,d\alpha_1^{k-l}(dw_j,dt))+ \alpha^l_{-2}(dw_j,d\alpha_1^{k-l}(dt,dw_i))$$
The result then follows from applying Fact 3 of Proposition \ref{prop:Facts}.
\end{proof}
As in the $W$-degree $-2$ case, the above lemmas easily imply the conclusion of Proposition \ref{prop:degt2}.  
\end{proof}

That $\pi^k_{-1}$ {\it almost} preserves the $t$-degree still greatly determines the form of $\pi^k_{-1}$:

\begin{proposition} 
\label{prop:Deg1Form}
For all $k<n$, the $W$-degree $(-1)$ summand is 
$$\alpha_{-1}^k(dw_i,dw_j)=r^k_{ij}C  \mbox{\qquad where  } r^k_{ij}\in\Sym(\g)_{1}$$
In particular, $\alpha_{-1}^k(dw_i,dw_j)\notin\langle A_l,B_l,\overline{A}_l,\overline{B}_l\rangle\subseteq\Sym(\g)_{1}$
\end{proposition}

\begin{proof}
 
Proposition \ref{prop:degt2} implies 
$$[\pi,\pi]^k_{-1}(dt,dw_i,dw_j)=0 \mbox{ only if } \alpha^k_{-1}(dw_i,dw_j) \bmod\langle C\rangle \mbox{ preserves the $t$-degree}.$$  
Therefore, any component 
$$\widetilde{\alpha^k_{-1}}(dw_i,dw_j)\in\Span\{A_l,B_l,\overline{A}_l,\overline{B}_l\}\otimes\Sym(\gl(\CC))$$
of $\alpha^k_{-1}$ that {\it does not} vanish on $V(C)$ satisfies
$$
\{t,\widetilde{\alpha^k_{-1}}(dw_i,dw_j)\}_1 =\left\{\begin{array}{l} \pm 2n\widetilde{\alpha^k_{-1}}(dw_i,dw_j) \\
0\\
\end{array}\right.$$
As any nonzero element of $\Span\{A_l,B_l,\overline{A}_l,\overline{B}_l\}\otimes\Sym(\gl(\CC))$ has $t$-degree $\pm n$, the component $\widetilde{\alpha^k_{-1}}$ must be zero --- proving the proposition. 
\end{proof}
 
 \subsubsection{The degree $(n-1)$ condition}
 
Propositions \ref{prop:Deg1Form} and \ref{prop:Deg2Form} imply that $\pi$ has the following form:

\begin{proposition} 
\label{prop:Alphak}
Recall that $\pi=\beta+\alpha+\widetilde{\pi}$, where $\alpha$ vanishes on the quotient $Q$ and $\widetilde{\pi}$ vanishes up to degree $n$.  The previous sections imply that $\alpha$ has the following form:
$$\alpha^k(dw_i,dw_j)=p_{ij}^k\cdot C+\sum_{l\geq 0}\alpha^k_l(dw_i,dw_j),$$
where $p_{ij}^k(g)\in{\rm Sym}^{k-2}(\g)$.
\end{proposition}

We now examine the $W$-degree $(-2)$ component of the ``degree $(n-1)$'' condition applied to an element of $(dw_i,dw_j,dw_k)$:
\begin{equation}\label{eqn:LastCondition}
[\pi,\pi]^{n-1}_{-2}=\overbrace{[\beta^1,\beta^{n-1}]}^{\mbox{term 1}}+\overbrace{[\beta^1,\alpha^{n-1}_{-2}]}^{\mbox{term 2}}+\overbrace{\sum_{k=0}^{n-2}[\alpha^k_{-1},\alpha^{n-k}_{-1}]}^{\mbox{term 3}}+\overbrace{\sum_{k=0}^{n-2}[\alpha^k_{-2},\alpha^{n-k}_{0}]}^{\mbox{term 4}}=0
\end{equation}
For concreteness, choose $(dx_i,dx_j,dy_0)\in \bigwedge^3 W^*$, with $i>j$.  We will prove that 
$$[\pi,\pi]_{-2}^{n-1}(dx_i,dx_j,dy_0)\neq 0$$
contrary to the assumption that $\pi$ is a Poisson structure (and therefore $[\pi,\pi]=0$).

\begin{lemma}[1st Term]  The first term in Equation \ref{eqn:LastCondition} is given by the following: 

$$[\beta^1,\beta^{n-1}](dx_i,dx_j,dy_0)= \hspace{10cm}$$
$$=(n-i)(n-j)\left[ ja_1^{n-j-2}a_3^{j-1}(\overline{A}_{i+1}-(i+1)\overline{B}_{i+1}) +ia_1^{n-i-2}a_3^{i-1}(\overline{A}_{j+1}-(j+1)\overline{B}_{j+1})\right] +$$
$$+n(n-1)(n-i)(n-j)\sum_{l=0}^{i-j-2}a_1^{(n-3)-(j+l)}a_3^{j+l}(\overline{A}_{i-l}-(i-l)\overline{B}_{i-l}),$$
where the above term belongs to 
$$[\beta^1,\beta^{n-1}](dx_i,dx_j,dy_0)\in\Span\{\overline{A}_{j+1}-(j+1)\overline{B}_{j+1},\ldots,\overline{A}_{i+1}-(i+1)\overline{B}_{i+1}\}\subseteq\Sym(\g)_1$$

\end{lemma}

\begin{proof}  Using the explicit formula for $\beta^{n-1}$ given earlier:
 $$\begin{array}{rcc}
 [\beta^1,\beta^{n-1}](dx_i,dx_j,dy_0) &=& \beta^1(dx_i,\beta^{n-1}(dx_j,dy_0))+\beta^1(\beta^{n-1}(dx_i,dy_0),dx_j), \\
 							  & &   (I)  \hspace{35pt} + \hspace{25pt} (II)   \\
\end{array}$$
where
$$\begin{array}{ccr}
(I) &=& (n-j)(n-i)[n(n-1)a_1^{n-j-2}a_3^jx_i+ja_1^{n-j-2}a_3^{j-1}(\overline{A}_{i+1}-(i+1)\overline{B}_{i+1})] \\
(II) &=& -(n-j)(n-i)[n(n-1)a_1^{n-i-2}a_3^ix_j+ia_1^{n-i-2}a_3^{i-1}(\overline{A}_{j+1}-(j+1)\overline{B}_{j+1})]
\end{array}$$ 

Assembling the above terms gives
$$[\beta^1,\beta^{n-1}](dx_i,dx_j,dy_0)= n(n-i)(n-j)(n-1)\left(a_1^{n-j-2}a_3^jx_i-v^{n-i-2}a_3^ix_j\right)+$$
$$\hspace{1cm}+(n-i)(n-j)\left(ja_1^{n-j-2}a_3^{j-1}(\overline{A}_{i+1}-(i+1)\overline{B}_{i+1})+ia_1^{n-i-2}a_3^{i-1}(\overline{A}_{j+1}-(j+1)\overline{B}_{j+1})\right)$$

Then, the result easily follows from the observation:

$$\left(a_1^{n-j-2}a_3^jx_i-a_1^{n-i-2}a_3^ix_j\right)=\sum_{l=0}^{i-j-2}a_1^{(n-3)-(j+l)}a_3^{j+l}(\overline{A}_{i-l}-(i-l)\overline{B}_{i-l}) $$
\end{proof}

\begin{lemma}[2nd Term] The second term of the condition is given by
$$[\beta^1,\alpha_{-2}^{n-1}](dx_i,dx_j,dy_0)=  p_{j0} \overline{A}_i-p_{i0}\overline{A}_j+ q_{ij}$$,
for polynomials $q_{ij}$ vanishing on $V(C)=V(a_0a_1-a_2a_3)$.
\end{lemma}

\begin{proof} We use the semi-explicit description of $\alpha^{n-1}_{-2}$ obtained earlier:

$$\begin{array}{rcl}
[\beta^1,\alpha_{-2}^{n-1}](dx_i,dx_j,dy_0) &=& \{x_i,p_{j0}C\}_1+\{p_{i0}C,x_j\}_1+\{y_0,\alpha^{n-1}_{-2}(dx_i,dx_j)\}_1\\
 &=& (\{x_i,C\}_1p_{j0}+p_{i0}\{C,x_j\}_1)+(\{x_i,p_{j0}\}_1+\{p_{i0},x_j\}_1)C+\\
 &  & \hspace{7cm}+ \{y_0,\alpha^{n-1}_{-2}(dx_i,dx_j)\}_1 \\
\end{array}$$
The first two terms in the above sum are:
$$(\{x_i,C\}_1p_{j0}+p_{i0}\{C,x_j\}_1)=\overline{A}_i p_{j0} -\overline{A}_j p_{i0}. $$
The last term in the sum $\{y_0,\alpha^{n-1}_{-2}(dx_i,dx_j)\}_1$ vanishes as $\alpha^{n-1}_{-2}(dx_i,dx_j)=0$.
\end{proof}

\begin{lemma}[3rd and 4th terms]  The third and fourth terms of the ``degree (n-1) condition'' send $(dx_i,dx_j,dy_0)$ into the ideal generated by $C$.  That is

$$\sum_{k=0}^{n-2}[\alpha^k_{-1},\alpha^{n-k}_{-1}](dx_i,dx_j,dy_0)+\sum_{k=0}^{n-2}[\alpha^k_{-2},\alpha^{n-k}_{0}](dx_i,dx_j,dy_0)\in \langle C\rangle\subseteq\Sym(\g)$$
\end{lemma}

\begin{proof}
The third term is identically zero.  For all $k<(n-1)$,
$$\alpha^k_{-1}(dx_i,d\alpha^{n-k}_{-1}(dx_j,y_0))=\alpha^k_{-1}(dx_i,d\left(p_{j0}^{n-k}C\right)). $$
Since $\alpha^k_{-1}(dx_i,a_l)=0$ for all $a_l\in\gl(\CC)$ (Proposition \ref{prop:Alphak}), a simple application of the Leibniz rule implies $\alpha^k_{-1}(dx_i,d\alpha^{n-k}_{-1}(dx_j,dy_0))=0$.  Moreover, this holds for all the cyclic permutations in the third term.
  Since $\alpha^l_{-2}(dw_i,dw_j)\in Sym^{l}(\gl(\CC))$ and $\alpha^l_0(da,dw)=0$ for all $l<(n-1)$, $a\in\gl(\CC)$, $w\in W$, the composition $\alpha^k_0(dw_i,d\alpha^{n-k}_{-2}(dw_j,dw_k))=0$.  Thus
$$[\alpha^k_{-2},\alpha^{n-k}_{0}](dx_i,dx_j,dy_0)= \alpha^k_{-2}(dx_i,\alpha_0^{n-k}(dx_j,dy_0)) + \mbox{c.p.} $$

As $\alpha^{k}_{-2}(dw_i,dw_j)\in\Span\{C\}\otimes\Sym(\g) $, the result follows.
\end{proof}

One can now verify that the ($W$-degree $-2$ part of the) ``degree $(n-1)$ condition'' 

$$[\pi,\pi]^{n-1}_{-2}= [\beta^1,\beta^{n-1}]+[\beta^1,\alpha^{n-1}_{-2}]+\sum_{k=0}^{n-2}[\alpha^k_{-1},\alpha^{n-k}_{-1}]+\sum_{k=0}^{n-2}[\alpha^k_{-2},\alpha^{n-k}_{0}]=0$$
cannot possibly be satisfied, as 
$$[\pi,\pi]^{n-1}_{-2}(dx_i,dx_j,dy_0) \bmod \langle C\rangle \neq 0$$

\begin{remark}
This last conclusion is most easily seen by noticing the following:  define the ideal $\JJ:=\langle C,\overline{A}_j,\ldots\overline{A}_i,\overline{B}_j,\ldots,\overline{B}_i\rangle$.  Then
$$[\pi,\pi]^{n-1}_{-2}(dx_i,dx_j,dy_0) \bmod\JJ =(n-i)(n-j)ja_1^{n-j-2}a_3^{j-1}(\overline{A}_{i+1}-(i+1)\overline{B}_{i+1})\neq0$$
\end{remark}

Thus, we arrive at a contradiction to the assumption that $\pi$ is a Poisson structure extending the Poisson structure on $Q$, as the jacabiator $[\pi,\pi]$ of any such Poisson structure must be identically zero.

\chapter{Constructing Explicit Poisson Embeddings}

This section will explore the problem of constructing an explicit Poisson embedding of one of the singular symplectic quotients considered in the previous chapter.  This discussion will result in an explicit Poisson embedding of $V/\ZZ_3$ into $\RR^{78}$.  Existence of such an embedding, along with the non-embedding theorem of the previous chapter, implies the following theorem:

\begin{theorem} The four dimensional symplectic orbifold $V/\ZZ_3$ has the following embedding properties:
 \begin{itemize}
 \item The minimal smooth embedding dimension of $V/\ZZ_3$ is $12$.
 \item The minimal Poisson embedding dimension of $V/\ZZ_3$ is $12<d\leq 78$.
 \end{itemize}
 \end{theorem}
 
Thus, the Poisson structure (and {\it not} the geometric structure) on the singular quotient $\CC^3/\ZZ_3$ has an obstruction to extension in $\RR^{12}$ that disappears in $\RR^{78}$.  It's not known whether there exists a Poisson embedding of $V/\ZZ_3$ into a Euclidean space of smaller dimension.
 
 \section{Generalities on Constructing Poisson Embeddings}
 
This section will continue to use the notation of Chapter 5 and later specialize to the setup in Chapter 6.  Fix a symplectic linear action of a compact group on the standard $2n$-dimensional symplectic vector space $V$ (denoting the resulting quotient $(Q,\brac_Q)$), and a minimal embedding $Q\hookrightarrow \g^*$ that is a Poisson embedding up to first order.   Recall the following theorem proved in Chapter 4:

\begin{reptheorem}{prop:EmbeddingCriteria} 
Let $\phi\!:(Q,\brac_{Q})\hookrightarrow(\RR^n,\brac)$ be a Poisson embedding.  Denote by $\pi=\sum\pi^k$ the decomposition of the Poisson bivector corresponding to $\brac$ into its homogeneous components, and let $\h\isom\RR^n$ be the Lie algebra corresponding to the Lie-Poisson structure $(\RR^n,\pi^1)$.  Then
\begin{enumerate}
\item The Lie algebra $\h$ is a Lie algebra extension of $\g:=\m_0/\m_0^2$
\item The extended Poisson structure $\pi$ is constructed from the following $\g$-cohomological equations:
$${\rm d}_\g\pi^k=\sum_{i<k}[\pi^i,\pi^{k-i}] \mbox{,  where } \pi^i\in{\rm C^2}(\g,\Sym^i(\g))$$
\end{enumerate}
\end{reptheorem}

In the following section we will consider the obvious ``extension'' of the fixed minimal embedding:
 $$Q\hookrightarrow\g^*\hookrightarrow\g^*\times\{0\}\hookrightarrow(\g\times Z)^*$$
for some finite dimensional vector space $Z$, with coordinates $\{z_i\}$, and fix the following extension of $\brac_Q$:
\begin{itemize}
\item Let $\brac$ be an extension of $\brac_Q$ to $(\g\times Z)^*$, and denote by $\pi$ the corresponding Poisson bivector.  This Poisson bracket is the extension that we are attempting to determine.
\item Let $\brac_\beta$ be an extension of $\brac_Q$ to the minimal embedding space $\g^*$ that {\it does not} necessarily satisfy the Jacobi identity.  Denote the corresponding bivector by $\beta=\beta^1+\beta^2+\beta^3+\ldots$.  This bracket will be the concrete extension.
\item Let $\alpha:=(\pi-\beta)$ be the difference of the above bivectors.  We note that $\alpha$ has coefficients in the defining ideal $\II$ of $Q\isom V(\II)\subseteq(\g\times Z)^*$.  From this we may deduce that $\alpha^1$ has coefficients in $Z$.
\end{itemize}

\subsection{The General Degree 1 Condition}

We will attempt to satisfy condition (1) in the proposition above by finding a (possibly nonabelian) Lie algebra extension:
$$0\to Z\to \g\times Z\to \g\to 0$$ 

The degree 1 vanishing condition has the following characterization:

\begin{proposition}  
\label{prop:LieAlgExt}
The Schouten bracket $[\pi^1,\pi^1]=0$ vanishes if and only if the following four conditions hold:
\begin{enumerate}
\item The bivector $\alpha^1$ defines a Lie bracket $[\cdot,\cdot]_Z: Z\wedge Z\to Z$.
\item The bivector $\alpha^1$ defines a linear map $\rho_Z:\g\to {\rm End}(Z)$ whose image lies within the derivations on $Z$, denoted by ${\rm Der}(Z)\subseteq{\rm End}(Z)$.
\item Furthermore, $\rho_Z:\g\to {\rm Der}(Z)\to{\rm Out}(Z)$ defines a homomorphism of Lie algebras (and is therefore ``a $\g$-representation modulo inner derivations'')
\item $\alpha^1$ defines a cocycle $c:\g\wedge\g\to Z$ with respect to the above data. Moreover, this cocycle measures the failure of $\rho_Z$ to be a true $\g$-representation:
$$\rho_Z([u,v])-[\rho_Z(u),\rho_Z(v)] = \ad_{c(u,v)} $$
\end{enumerate}
\end{proposition}

\begin{proof} As the vanishing ideal of $Q$ in $\g^*$ contains no linear terms, the linear components have the form:
$$\beta^1=\sum_{i,j}b_{ij}^k v_k \bivect{v_i}{v_j} \mbox{ and } \alpha^1=\sum_{i,j}a_{ij}^k z_k \bivect{z_i}{z_j}+\sum_{i,j}a_{ij}^{'k} z_k \bivect{v_i}{z_j}+\sum_{i,j}a_{ij}^{''k} z_k \bivect{v_i}{v_j}$$

\begin{enumerate} 
\item Define the Lie bracket on $Z$ by $[z_i,z_j]_Z:=\sum_{k}a^{k}_{ij}z_k$.  This Lie bracket satisfies the Jacobi identity precisely when $[\pi^1,\pi^1](dz_i,dz_j,dz_k)=0$, as $\beta^1$ vanishes on $Z$.
\item Let $\rho_Z(v)(z):=\alpha^1(dv,dz)=\sum_k a^{'k}_{ij}z_k\in Z$.  The Schouten bracket is
$$[\pi^1,\pi^1](dv,dz_i,dz_j) = [\beta^1,\beta^1](dv,dz_i,dz_j) + 2[\beta^1,\alpha^1](dv,dz_i,dz_j) + [\alpha^1,\alpha^1](dv,dz_i,dz_j)$$
The first and second terms vanish, as $\beta^1$ vanishes on $Z$, leaving $[\pi^1,\pi^1]=0$ precisely when
\begin{eqnarray}
\nonumber 0 &=& \frac{1}{2}[\alpha^1,\alpha^1](dv,dz_i,dz_j) \\
\nonumber    &=& \alpha^1(dv,d\alpha^1(dz_i,dz_j))+\alpha^1(dz_i,d\alpha^1(dz_j,dv)])+\alpha^1(dz_j,d\alpha^1(dv,dz_i))
\end{eqnarray}

Rearranging these terms, we obtain the following precisely when $[\pi^1,\pi^1]=0$:
\begin{eqnarray}
 \rho_Z(v)([z_i,z_j]_Z) &=& \alpha^1(dv,\alpha^1(dz_i,dz_j)) \nonumber\\
					&=& -\alpha^1(dz_i,d\alpha^1(dz_j,dv))-\alpha^1(dz_j,d\alpha^1(dv,dz_i)) \nonumber\\
					&=& \alpha^1(dz_i,d\alpha^1(dv,dz_j))+\alpha^1(d\alpha^1(dv,dz_j),dz_i)\nonumber\\
					&=& \alpha^1(dz_i,\rho_Z(v)(z_j))+\alpha^1(\rho_Z(v)(z_i),z_j)\nonumber\\
					&=&[z_i,\rho_Z(v)(z_j)]_Z+[\rho_Z(v)(z_i),z_j]_Z \nonumber
\end{eqnarray}

\item To show that the map $\rho_Z$ is {\it almost} a $\g$-representation, note that
$$[\pi^1,\pi^1](dz,dv_i,dv_j) = [\beta^1,\beta^1](dz,dv_i,dv_j) + 2[\beta^1,\alpha^1](dz,dv_i,dv_j) + [\alpha^1,\alpha^1](dz,dv_i,dv_j)$$

As $\beta^1$ vanishes on $Z$, we have $[\beta^1,\beta^1](dz,dv_i,dv_j)=0$, as well as
$$[\beta^1,\alpha^1](dz,dv_i,dv_j) = \alpha^1(dz,d\beta^1(dv_i,dv_j))=-\rho_Z([v_i,v_j]_\g)(z)$$
Moreover, 
\begin{eqnarray}
\nonumber \half[\alpha^1,\alpha^1](dz,dv_i,dv_j) &=& \alpha^1(dv_i,d\alpha^1(dv_j,dz))+\alpha^1(dv_j,d\alpha^1(dz,dv_i))+\alpha^1(dz,d\alpha^1(dv_i,dv_j))\\
\nonumber						&=& \rho_Z(v_i)\circ\rho_Z(v_j)(z)-\rho_Z(v_j)\circ\rho_Z(v_i)(z)-[\alpha^1(dv_i,dv_j),z]_Z
\end{eqnarray}
Thus, 
$$\half[\pi^1,\pi^1](dz,dv_i,dv_j)=-\rho_Z([v_i,v_j]_\g)(z)+[\rho_Z(v_i),\rho_Z(v_j)](z)-{\rm ad}_{\alpha^1(dv_i,dv_j)}z$$
And $[\pi^1,\pi^1]=0$ if and only if
$$\rho_Z([v_i,v_j]_\g)(z)=[\rho_Z(v_i),\rho_Z(v_j)](z)-{\rm ad}_{\alpha^1(dv_i,dv_j)}z$$

\item Defining $c(v_i,v_j)=\alpha^1(dv_i,dv_j)$, the analogous Schouten bracket calculations are:
\begin{itemize}
\item $\half[\beta^1,\beta^1](dv_i,dv_j,dv_k)=0$, as $\beta^1$ is a Lie-Poisson bracket.
\item $[\alpha^1,\beta^1](dv_i,dv_j,dv_k)= c(v_i,[v_j,v_k]_\g)+ c(v_k,[v_i,v_j]_\g)+ c(v_j,[v_k,v_i]_\g)$
\item $\half[\alpha^1,\alpha^1](dv_i,dv_j,dv_k)=-\rho_Z(v_i)(c(v_j,v_k))-\rho_Z(v_k)(c(v_i,v_j))-\rho_Z(v_j)(c(v_k,v_i))$
\end{itemize}
Thus,
\begin{eqnarray} dc(v_i,v_j,v_k) &=& c(v_i,[v_j,v_k]_\g)-c(v_j,[v_i,v_k]_\g) c(v_k,[v_i,v_j]_\g) + \nonumber\\
\nonumber						&&\quad -\rho_Z(v_i)(c(v_j,v_k))+\rho_Z(v_j)(c(v_i,v_k)) -\rho_Z(v_k)(c(v_i,v_j))\\
						&=& [\alpha^1,\beta^1](dv_i,dv_j,dv_k)+\half[\alpha^1,\alpha^1](dv_i,dv_j,dv_k) \nonumber\\
						&=& [\pi^1,\pi^1](dv_i,dv_j,dv_k)\nonumber
						\end{eqnarray}
						
and so $[\pi^1,\pi^1]=0$ exactly when $dc=0$.
\end{enumerate}
\end{proof}

\begin{remark}  These data appear in different guises in the literature.  The first appearance of this cohomological perspective   is due to Hochschild \cite{Hochschild54}.   One may also package these data into a morphism of Lie 2-algebras $\g\to{\rm DER}(Z)$ (see HDA6 \cite{Baez04}).
\end{remark}

\begin{example}  Suppose that $Z$ is 1-dimensional.  Then $\pi^1$ on $(\g\times Z)^*$ is just a one dimensional extension of  $\beta^1$ on $\g^*$.  The conditions in the proposition above become:
\begin{enumerate}
\item $(Z,[\cdot,\cdot]_Z)$ is necessarily abelian. 
\item Any $\rho_Z: \g\to {\rm End}(Z)$ is automatically a derivation.
\item $\rho_Z$ is the trivial representation.
\item $c$ is a genuine cocycle, with respect to the trivial representation.
\end{enumerate}
\end{example}

\begin{example} Consider the example of the $(1,1,2)$-resonance space $\CC^3/S^1$.  In \cite{Egilsson95}, Egilsson proved that there is no Poisson structure on $\RR^{11}$ extending the Poisson structure on $\CC^3/S^1$.  In this case $\RR^{11}$ is a minimal embedding and $\RR^{11}\isom\g^*=(\RR\oplus\gl(\RR))^*\ltimes(\RR^3\oplus\RR^3)^*$).  This result was later strengthened by Davis \cite{Davis01}, who proved there does not exist a Poisson embedding of the quotient into $\RR^{12}$.  Davis implicitly considered all possible cocycles (as in the above example), in order to prove the degree 2 condition could not be satisfied. 
\end{example}

When $Z$ is of dimension greater than one, the problem of determining when the degree 1 condition is satisfied quickly becomes intractable.  After all, the first criteria alone demands that one keeps track of all, possibly nonabelian, extensions of $\g$.  Thus, in an effort to construct explicit embeddings of $(Q,\brac_Q)$, we will assume that $(Z,[\cdot,\cdot]_Z)$ is an abelian Lie algebra.  Such an assumption imposes quite a bit of structure on any resulting extension and so we will expect to have the dimension of $Z$ get quite large before finding any embeddings. 

\subsubsection{The case of an abelian $Z$}

When we impose the condition that $Z$ be an abelian Lie algebra, the criteria in Proposition \ref{prop:LieAlgExt} become:
\begin{enumerate}
\item $(Z,[\cdot,\cdot]_Z)$ is an abelian Lie algebra.
\item $\rho_Z:\g\to{\rm End}(Z)$ is a linear map.
\item $\rho_Z$ is a homomorphism of Lie algebras, and therefore a $\g$-representation.
\item $c:\g\wedge\g\to Z$ is a cocycle with respect to $\rho_Z$.
\end{enumerate}
Therefore the embedding space $(\g\times Z)^*$ becomes a representation of $\g$; this condition imposes  $\g$-structure on any such choice of an embedding.

\section{A Poisson Embedding of $V/\ZZ_3$}

\subsubsection{Introduction}

The contruction of the Poisson embedding of $Q=V/\ZZ_3$ follows a simple idea.  As above, fix an algebraic embedding $Q\subseteq \g\times Z$ into an $n$-dimensional affine space and an extension $\brac_\beta$ of the Poisson structure $(Q,\brac_Q)$ on the quotient to a skew-symmetric bilinear bracket on $\Sym(\g\times Z)$.  Finding a Poisson structure that extends the one on $Q$ amounts to solving
$$[\beta+\alpha,\beta+\alpha]=0$$
for a bivector $\alpha$ that vanishes on $Q$.  This requires tackling three issues:
\begin{enumerate}
\item choosing the dimension of the embedding space $\g\times Z$,
\item fixing the extension of $\brac_Q$ up to first order, and
\item determining the higher-order bivector $(\alpha^2+\alpha^3+\alpha^4+\ldots)$
\end{enumerate}

The first two issues are closely related and are largely determined by the simplifying assumption that the Lie algebra $\g\times Z$ defined by the choice of $\pi^1=(\beta^1+\alpha^1)$ is an {\em abelian} extension of the Lie algebra $\g$.  The approach to the third issue is very crude:\! we write $\alpha^2$ as a bivector in general form and use a computer to calculate $[\beta+\alpha,\beta+\alpha]$.  We then set the resulting expression equal to 0 and attempt to find constraints on the coefficients of $\alpha^2$.  At the start, this method is extraordinarily unfeasible: the potential embedding space is $78$-dimensional, making the vector space of quadratic bivectors $(18,270,252)$-dimensional.  However, choosing such a large candidate for the embedding space has two advantages:  
\begin{enumerate}
\item The aim is to merely find {\em one} Poisson structure.  As $\g\times Z$ is so large, one may hope there is probably more Poisson structure extending $\brac_Q$.  This allows us to merely try to cancel the non-zero terms of $[\beta,\beta]$ when choosing $\alpha$.
\item The large size of $\g\times Z$ seems to allow the Poisson structure on $Q$ to impose a clear structure on $\alpha$ without being overly restrictive.  By choosing $\g\times Z$ to be an abelian extension of $\g$, all of the relevant structures $(\g\times Z, \alpha, \II)$ inherit a grading from $\g$, and all such structures must be $\g$-compatible as described below.

\end{enumerate}

Thus, by choosing a large embedding dimension, we allow ourselves the freedom to both apply reductions and take advantage of structure imposed by $\g$, in order to restrict the form of $\alpha$ to something computationally manageable.  Using the techniques mentioned above, the $(18,270,252)$-dimensional vector space of quadratic bivectors becomes $199$-dimensional --- small enough for a computer to handle.\\

\begin{remark}  The computations in this chapter are done with Macaulay2 \cite{M2}, Python Scripts, and  Maple${}^\mathrm{TM}$ \cite{Maple10}.  The majority of the work --- including all the invariant theoretic calculations and the Schouten bracket computations -- are done using Macaulay2.  The Python scripts are used both to construct $\alpha^2$ according to a set of rules (defined by a grading) and to scan very long expressions for certain coefficients (to aid in eliminating extraneous coefficients).  Lastly,  Maple${}^\mathrm{TM}$ is used only a handful of times, to symbolically solve systems of linear equations.  A detailed explanation of the process and code used is included in Appendix \ref{sec:AppendixM2}.
\end{remark}

\subsection{The setup and minimal embedding}

As in Chapter 6, for the case of $Q=V/\ZZ_3$ we will work with the complexification of the vector spaces involved -- that is, while the minimal embedding space of $Q$ is $\RR^{12}\isom\g^*$, we will consider the complexification $\g_\CC\isom\gl(\CC)\ltimes(\CC^4\oplus\CC^4)$ when doing calculations -- which we will simply denote by $\g$ to ease notation.  As before, denote the basis of $\g$ by 
$$\g=\gl(\CC)\ltimes(\CC^4\oplus\CC^4)\isom\Span\{a_0,a_1,a_2,a_3,x_0,x_1,x_2,x_3,y_0,y_1,y_2,y_3\}$$
and the embedding $\phi^*:V/\ZZ_3\hookrightarrow \g^*$ is the one induced by the algebra map $\phi:\Sym(\g^*)\to\Sym(V^*)^{\ZZ_3}$ defined on generators by:
$$
\phi:\left\{\begin{array}{lll}
a_0 \mapsto z\zb, &   x_0 \mapsto w^3, & y_0 \mapsto \wb^3 \\
a_1 \mapsto w\wb, & x_1 \mapsto zw^2, & y_1 \mapsto \zb\wb^2 \\
a_2 \mapsto \zb w, & x_2 \mapsto z^2w, & y_2 \mapsto \zb^2\wb  \\
a_3 \mapsto z\wb, &  x_3 \mapsto z^3,  & y_3 \mapsto \zb^3 \\
\end{array}\right.
$$

The previously considered extension $\beta=\beta^1+\beta^2$ is given explicitly by \\

\begin{itemize}
\item[$\beta^1=$]
$(a_2\bivect{a_0}{a_2}-a_2\bivect{a_1}{a_2} -a_3\bivect{a_0}{a_3} +a_3\bivect{a_1}{a_3} + (a_0-a_1)\bivect{a_2}{a_3})+ \\-3x_0\bivect{a_1}{x_0} -3x_1\bivect{a_3}{x_0}-x_1\bivect{a_0}{x_1}-2x_1\bivect{a_1}{x_1}-x_0\bivect{a_2}{x_1}+
\\-2x_2\bivect{a_3}{x_1}-2x_2\bivect{a_0}{x_2}-x_2\bivect{a_1}{x_2}-2x_1\bivect{a_2}{x_2}-x_3\bivect{a_3}{x_2} +
\\-3x_3\bivect{a_0}{x_3}-3x_2\bivect{a_2}{x_3}+3y_0\bivect{a_1}{y_0}+3y_1\bivect{a_2}{y_0}+y_1\bivect{a_0}{y_1}+2y_1\bivect{a_1}{y_1}+2y_2\bivect{a_2}{y_1}+y_0\bivect{a_3}{y_1}+2y_2\bivect{a_0}{y_2}+y_2\bivect{a_1}{y_2}+y_3\bivect{a_2}{y_2}+2y_1\bivect{a_3}{y_2}+ 3y_3\bivect{a_0}{y_3}$
\\
\item[$\beta^2=$] $3a_2^2\bivect{x_0}{y_2}+6a_1a_3\bivect{x_1}{y_0}+3a_3^2\bivect{x_2}{y_0}+6a_1a_2\bivect{x_0}{y_1}+(a_1^2+4a_2a_3)\bivect{x_1}{y_1}+(2a_0a_3+2a_1a_3)\bivect{x_2}{y_1}+ 3a_3^2\bivect{x_3}{y_1}+(2a_0a_2+2a_1a_2)\bivect{x_1}{y_2}+(a_0^2+4a_2a_3)\bivect{x_2}{y_2}+6a_0a_3\bivect{x_3}{y_2}+3y_2\bivect{a_3}{y_3}+3a_2^2\bivect{x_1}{y_3}+6a_0a_2\bivect{x_2}{y_3}+9a_0^2\bivect{x_3}{y_3}+9a_1^2\bivect{x_0}{y_0}$
\end{itemize}
Recalling the isomorphism of Lie algebras $\gl(\CC)\to\Span\{a_0,a_1,a_2,a_3\}$ given by
$$\{ H\mapsto(a_0-a_1), E\mapsto a_2, F\mapsto a_3, t\mapsto(a_0+a_1)\} $$
one notes that $\beta^1$ has the structure:
$$ \beta^1=\sum_{0\leq i,j\leq3}[a_i,a_j]_{\gl}\bivect{a_i}{a_j}+\sum_{0\leq i,j\leq3}\rho(a_i)(x_j)\bivect{a_i}{x_j}+\sum_{0\leq i,j\leq3}\rho(a_i)(y_j)\bivect{a_i}{y_j}$$
where $\rho:\gl(\CC)\to{\rm End}(V_4\oplus V_4)$ is given by $\rho=(\rho_3\oplus{\rm Id})\oplus(-\rho_3\oplus-{\rm Id})$.

The quadratic term $\beta^2$ has the structure:
$$\beta^2=\sum_{0\leq i,j\leq3}p(x_i,y_j)\bivect{x_i}{y_j} \mbox{ where the $p(x_i,y_j)$ are quadratic in the $a_i$.}$$

\subsection{ The Degree 1 Condition}

An explicit embedding of $Q=\CC^3/\ZZ_3$ into $\g^*\isom\CC^{12}$ was fixed in the above section.  However, this embedding is not a Poisson embedding --- in fact, by Theorem 1.1, there cannot exist a Poisson embedding into $\g^*$.  

\begin{reduction}Following the above section, consider an embedding $Q\hookrightarrow(\g\times Z)^*$ and an abelian Lie algebra extension of $\g$.  That is,
\begin{enumerate}
\item a $\g$-representation $\rho_Z:\g\to{\rm End}(Z)$, and
\item a cocycle $c:\g\wedge\g\to Z$.  
\end{enumerate}
\end{reduction}
These conditions become quite strong when one uses the Levi decomposition $\g\isom\slt(\CC)\ltimes(\CC\oplus\CC^4\oplus\CC^4)$.  We can thus restrict $\rho_Z$ to the Levi factor $\slt(\CC)$, and use the natural action of $\slt(\CC)$ on $\g$ to deduce:
\begin{enumerate}
\item the representation $(Z,\rho_Z |_{\slt})$ is an $\slt(\CC)$-representation,
\item $c:\g\wedge\g\to Z$ is $\slt(\CC)$-equivariant, (i.e. a homomorphism of $\slt(\CC)$-representations).
\end{enumerate}
Therefore, a natural starting candidate for an embedding space would be to take 
$$\g^*\times Z^*\isom\g^*\times(\g\wedge\g)^*\isom\CC^{78}$$ 
and to assume that $c:\g\wedge\g\to\g\wedge\g$ restricts to scalar multiples of the identity on the irreducible subrepresentations of $\g\wedge\g$ (see Proposition \ref{cor:SchurEnd}).  So as not to confuse the vector spaces $Z$ and $\g\wedge\g$, fix the following basis for $Z$:
$$Z=\Span\{\ldots (aa)_{ij} \dots, (ax)_{ij} \dots,(ay)_{ij} \dots,(xx)_{ij} \dots,(xy)_{ij} \dots,(yy)_{ij} \ldots\}$$
The choice of notation, while perhaps unfortunate for exposition, was chosen to ease computations in Macaulay2.  Such a first order extension yields a bivectors $\alpha^1=\alpha^1_\rho+\alpha^1_c$.  While the explicit formulae for this bivector is included in Appendix \ref{sec:AppendixM2}, the following computations are intended to explain the structure of the formulae and how they are computed:\\
\begin{enumerate}
\item The first term $\alpha^1_\rho$ is defined via the natural representation of $\g$ on $\g\wedge\g$.  
For example $\alpha^1_\rho(dxy_{02},da_3)$ is defined as follows:
\begin{eqnarray}
\nonumber \alpha^1_\rho(dxy_{02},da_3) &=& \beta^1(dx_0\wedge dy_2,da_3)  \\
\nonumber					&=& \beta^1(dx_0,da_3)\wedge y_2 + x_0\wedge\beta^1(dy_2,da_3)  \\
\nonumber					&=& 3x_1\wedge y_2+x_0\wedge(-2y_1) \\
\nonumber					&=& 3xy_{12}-2xy_{01}
\end{eqnarray}

\item The second term $\alpha^1_c:\g\wedge\g\to\g\wedge\g$ is a multiple of the identity on the irreducible $\slt(\CC)$-subrepresentations of $\g\wedge\g$. For example:
$$ \alpha^1_c(dx_i,dx_j)=L_{ij}(xx)_{ij} \mbox{ for scalars $L_{ij}$ }$$
\end{enumerate}

\subsection{The Degree 2 Condition}

The condition for $[\pi,\pi]$ to vanish in degree 2 is:
\begin{eqnarray}
\nonumber 0 &=& [\pi^1,\pi^2] \\
 \nonumber  &=& [\beta^1+\alpha^1,\beta^2+\alpha^2]\\
   &=& [\beta^1,\beta^2]+[\beta^1,\alpha^2]+[\alpha^1,\beta^2]+[\alpha^1,\alpha^2]
   \end{eqnarray}
   
As noted in Chapter 6, the jacobiator $[\beta,\beta]=[\beta^1,\beta^2]\neq 0$ is nonvanishing and therefore is the term that we are trying to cancel when choosing $\alpha^2$.  Explicitly, for the $\ZZ_3$-action the jacobiator is generically nonzero on a $28$-dimensional subspace of $\g\wedge\g\wedge\g$:

\begin{flushleft}
\noindent $[\beta^1,\beta^2]= (6\text{a}_{0}\text{a}_{1}-6\text{a}_{2}\text{a}_{3})\partial\text{a}_{2}\partial\text{x}_{1}\partial\text{y}_{0}+(18\text{a}_{3}\text{x}_{0}-18\text{a}_{1}\text{x}_{1})\partial\text{x}_{0}\partial\text{x}_{1}\partial\text{y}_{0}+(18\text{a}_{3}\text{x}_{1}-18\text{a}_{1}\text{x}_{2})\partial\text{x}_{0}\partial\text{x}_{2}\partial\text{y}_{0}+(6\text{a}_{3}\text{x}_{2}-6\text{a}_{1}\text{x}_{3})\partial\text{x}_{1}\partial\text{x}_{2}\partial\text{y}_{0}+(-6\text{a}_{0}\text{a}_{1}+6\text{a}_{2}\text{a}_{3})\partial\text{a}_{3}\partial\text{x}_{0}\partial\text{y}_{1}+(6\text{a}_{3}\text{x}_{0}+6\text{a}_{0}\text{x}_{1}-6\text{a}_{1}\text{x}_{1}-6\text{a}_{2}\text{x}_{2})\partial\text{x}_{0}\partial\text{x}_{2}\partial\text{y}_{1}+(-2\text{a}_{3}\text{x}_{1}+4\text{a}_{0}\text{x}_{2}+2\text{a}_{1}\text{x}_{2}-4\text{a}_{2}\text{x}_{3})\partial\text{x}_{1}\partial\text{x}_{2}\partial\text{y}_{1}+(18\text{a}_{3}\text{x}_{1}-18\text{a}_{1}\text{x}_{2})\partial\text{x}_{0}\partial\text{x}_{3}\partial\text{y}_{1}+(18\text{a}_{2}\text{y}_{0}-18\text{a}_{1}\text{y}_{1})\partial\text{x}_{0}\partial\text{y}_{0}\partial\text{y}_{1}+(4\text{a}_{3}\text{x}_{0}-2\text{a}_{0}\text{x}_{1}-4\text{a}_{1}\text{x}_{1}+2\text{a}_{2}\text{x}_{2})\partial\text{x}_{1}\partial\text{x}_{2}\partial\text{y}_{2}+(-6\text{a}_{0}\text{a}_{1}+6\text{a}_{2}\text{a}_{3})\partial\text{a}_{2}\partial\text{x}_{3}\partial\text{y}_{2}+(18\text{a}_{0}\text{x}_{1}-18\text{a}_{2}\text{x}_{2})\partial\text{x}_{0}\partial\text{x}_{3}\partial\text{y}_{2}+(6\text{a}_{3}\text{x}_{1}+6\text{a}_{0}\text{x}_{2}-6\text{a}_{1}\text{x}_{2}-6\text{a}_{2}\text{x}_{3})\partial\text{x}_{1}\partial\text{x}_{3}\partial\text{y}_{2}+(18\text{a}_{2}\text{y}_{1}-18\text{a}_{1}\text{y}_{2})\partial\text{x}_{0}\partial\text{y}_{0}\partial\text{y}_{2}+(6\text{a}_{2}\text{y}_{0}+6\text{a}_{0}\text{y}_{1}-6\text{a}_{1}\text{y}_{1}-6\text{a}_{3}\text{y}_{2})\partial\text{x}_{1}\partial\text{y}_{0}\partial\text{y}_{2}+(6\text{a}_{2}\text{y}_{2}-6\text{a}_{1}\text{y}_{3})\partial\text{x}_{0}\partial\text{y}_{1}\partial\text{y}_{2}+(-2\text{a}_{2}\text{y}_{1}+4\text{a}_{0}\text{y}_{2}+2\text{a}_{1}\text{y}_{2}-4\text{a}_{3}\text{y}_{3})\partial\text{x}_{1}\partial\text{y}_{1}\partial\text{y}_{2}+(4\text{a}_{2}\text{y}_{0}-2\text{a}_{0}\text{y}_{1}-4\text{a}_{1}\text{y}_{1}+2\text{a}_{3}\text{y}_{2})\partial\text{x}_{2}\partial\text{y}_{1}\partial\text{y}_{2}+(6\text{a}_{0}\text{y}_{0}-6\text{a}_{3}\text{y}_{1})\partial\text{x}_{3}\partial\text{y}_{1}\partial\text{y}_{2}+(6\text{a}_{0}\text{a}_{1}-6\text{a}_{2}\text{a}_{3})\partial\text{a}_{3}\partial\text{x}_{2}\partial\text{y}_{3}+(6\text{a}_{0}\text{x}_{0}-6\text{a}_{2}\text{x}_{1})\partial\text{x}_{1}\partial\text{x}_{2}\partial\text{y}_{3}+(18\text{a}_{0}\text{x}_{1}-18\text{a}_{2}\text{x}_{2})\partial\text{x}_{1}\partial\text{x}_{3}\partial\text{y}_{3}+(18\text{a}_{0}\text{x}_{2}-18\text{a}_{2}\text{x}_{3})\partial\text{x}_{2}\partial\text{x}_{3}\partial\text{y}_{3}+(18\text{a}_{2}\text{y}_{1}-18\text{a}_{1}\text{y}_{2})\partial\text{x}_{1}\partial\text{y}_{0}\partial\text{y}_{3}+(18\text{a}_{0}\text{y}_{1}-18\text{a}_{3}\text{y}_{2})\partial\text{x}_{2}\partial\text{y}_{0}\partial\text{y}_{3}+(6\text{a}_{2}\text{y}_{1}+6\text{a}_{0}\text{y}_{2}-6\text{a}_{1}\text{y}_{2}-6\text{a}_{3}\text{y}_{3})\partial\text{x}_{2}\partial\text{y}_{1}\partial\text{y}_{3}+(18\text{a}_{0}\text{y}_{1}-18\text{a}_{3}\text{y}_{2})\partial\text{x}_{3}\partial\text{y}_{1}\partial\text{y}_{3}+(18\text{a}_{0}\text{y}_{2}-18\text{a}_{3}\text{y}_{3})\partial\text{x}_{3}\partial\text{y}_{2}\partial\text{y}_{3}$
\end{flushleft}

\begin{remark} An element ``$\partial\text{v}$'' is shorthand for the anti-commuting variable ``$\vect{v}$.''
\end{remark}

\noindent Therefore, in choosing $\alpha^2$, we will first focus on the $28$-dimensional subspace above.

\begin{reduction} If there exists a Poisson structure $\pi=\beta+\alpha$ on $(\g\times Z)^*$ with linear Poisson structure $\pi^1=\beta^1+\alpha^1$, then there exists a semi-linearized Poisson structure with respect to the Levi decomposition $(\slt(\CC)\ltimes(\CC\oplus W\oplus Z))^*$.  Therefore, we need only consider $\pi^2=\beta^2+\alpha^2$ that are supported on the radical of $\g\times Z$.  That is,
$$\alpha^2(ds,\cdot)=0 \mbox{ for } s\in\slt(\CC)$$
\end{reduction}

The statement in this reduction only holds for our embedding up to order $(n-1)=2$, by the proof in Lemma \ref{sec:DeterminingPhi}.  However, our $\alpha$ will conveniently be a second order extension $\alpha=\alpha^1+\alpha^2$.

\subsubsection{Reductions of $\alpha^2$ via $H$-grading}

The most effective reduction in reducing the possible choices of $\alpha^2$ comes from the $H$-grading defined by the natural action of $\slt(\CC)$. The following properties were either proved in Chapter 6 (or may be verified by inspection or Macaulay2):
\begin{itemize}
\item[] $\pi^1=\beta^1+\alpha^1$ preserves the $H$-grading.
\item[] $\beta^2$ preserves the $H$-grading.
\end{itemize}

\begin{reduction}As we are attempting to solve equation (7.1) above, only the summands in $\alpha^2$ that preserve the $H$-grading need to be considered.  
\end{reduction}
This observation, along with an analysis of the $H$-grading on the vanishing ideal $\II$ of $Q\subseteq\g^*$, will provide the first attempt at determining $\alpha^2$.  Recall that in degree 2 part of the vanishing ideal $\II$ is spanned by:

\begin{table}[h!]
\caption{The $t,H$-degrees of relevant generators of the vanishing ideal $\II$}
$$\begin{array}{|l|r|c|}\hline
\mbox{Element} 						&  H-\mbox{deg} & t-\mbox{deg} \\ \hline 
C=\text{a}_0 \text{a}_1-\text{a}_2 \text{a}_3 & 0 & 0 \\ \hline \noalign{\smallskip}
\overline{A}_0=3 \text{a}_0 \text{x}_0-3 \text{a}_2 \text{x}_1& 3 & -3 \\
\overline{A}_1=-\text{a}_3 \text{x}_0+2 \text{a}_0 \text{x}_1+\text{a}_1 \text{x}_1-2 \text{a}_2 \text{x}_2& 1 &-3 \\
\overline{A}_2=-2 \text{a}_3 \text{x}_1+\text{a}_0 \text{x}_2+2 \text{a}_1 \text{x}_2-\text{a}_2 \text{x}_3& -1 &-3 \\
\overline{A}_3=-3 \text{a}_3 \text{x}_2+3 \text{a}_1 \text{x}_3& -3 &-3 \\ \hline \noalign{\smallskip}
\overline{B}_0=\text{a}_3 \text{x}_0+\text{a}_0 \text{x}_1-\text{a}_1 \text{x}_1-\text{a}_2 \text{x}_2& 1 &-3 \\
\overline{B}_1=\text{a}_3 \text{x}_1+\text{a}_0 \text{x}_2-\text{a}_1 \text{x}_2-\text{a}_2 \text{x}_3& -1 &-3 \\ \hline
A_0=3 \text{a}_0 \text{y}_0-3 \text{a}_3 \text{y}_1& -3 &3 \\
A_1=-\text{a}_2 \text{y}_0+2 \text{a}_0 \text{y}_1+\text{a}_1 \text{y}_1-2 \text{a}_3 \text{y}_2& -1 &3 \\
A_2=-2 \text{a}_2 \text{y}_1+\text{a}_0 \text{y}_2+2 \text{a}_1 \text{y}_2-\text{a}_3 \text{y}_3& 1 &3 \\
A_3=-3 \text{a}_2 \text{y}_2+3 \text{a}_1 \text{y}_3& 3 & 3 \\ \hline
B_0=\text{a}_2 \text{y}_0+\text{a}_0 \text{y}_1-\text{a}_1 \text{y}_1-\text{a}_3 \text{y}_2& -1 &3 \\
B_1=\text{a}_2 \text{y}_1+\text{a}_0 \text{y}_2-\text{a}_1 \text{y}_2-\text{a}_3 \text{y}_3& 1 &3 \\
\hline
\end{array}$$
\label{table:M2generator2}
\end{table}

\begin{remark}  The sub-sections of the above table actually form irreducible representations of $\slt(\CC)$.
\end{remark}

\begin{reduction} Additionally there are degree 2 elements in $\II$ of the form:
$$\{\text{y}_2^{2}-\text{y}_1 \text{y}_3,\text{y}_1 \text{y}_2-\text{y}_0
       \text{y}_3,\text{y}_1^{2}-\text{y}_0 \text{y}_2,\text{x}_2^{2}-\text{x}_1
       \text{x}_3,\text{x}_1 \text{x}_2-\text{x}_0 \text{x}_3,\text{x}_1^{2}-\text{x}_0
       \text{x}_2\}$$

However, these elements do not appear in the coefficients of the jacobiator $[\beta^1,\beta^2]$, and therefore will not appear in the coefficients of $\alpha^2$.
\end{reduction}

With the above descriptions of the $H$-gradings of $\g\times Z$ and $\II$, we are almost ready to make a first attempt at defining $\alpha^2$.  However, we will first make a few more preliminary reductions guided by the idea that we are merely trying to cancel the nonzero part of $[\beta^1,\beta^{n-1}]$:

\begin{reduction} The coefficients of $\alpha^2$ are in the vanishing ideal $\II$ (and not in $\Sym(Z)$) and $\alpha^2$ preserves the $H$-grading.
\begin{itemize}
\item  The following terms are of the form: $\alpha^2(dx_i,dy_j)=K_{ij}C_0$, $\alpha^2(dx_i,dx_j)=K'_{ij}C_0$, $\alpha^2(dy_i,dy_j)=K''_{ij}C_0$ for scalars $K_{ij},K'_{ij},K''_{ij}$, as the $H$-degree of the terms $dx_i\wedge dy_j,dx_i\wedge dx_j,dy_i\wedge dy_j$  are even and $C_0$ is the only even quadratic element of $\II$ in the table above.
\item Analogous reasoning implies that the following have odd $H$-degree:
$$\{\alpha^2(dxx_{ij},dy_{k}),\alpha^2(dxx_{ij},dx_k),\alpha^2(dyy_{ij},dx_k)\},$$
$$\{\alpha^2(dyy_{ij},dy_k),\alpha^2(dxy_{ij},dx_k),\alpha^2(dxy_{ij},dy_k)\}$$ and therefore they lie in $Span\{\overline{A}_i,\overline{B}_i,A_i,B_i\}$.
\end{itemize}
\end{reduction}

\begin{reduction} Through an analysis of the coefficients of the jacobiator $[\beta^1,\beta^2]$, we can improve upon the reduction above:
\begin{itemize}
\item Elements of the form $\alpha^2(dxx_{ij},dy_k)$ and $\alpha^2(dxy_{ij},dx_k)$ lie in $\Span\{\overline{A}_l,\overline{B}_l\}$.
\item Elements of the form $\alpha^2(dyy_{ij},dx_k)$ and $\alpha^2(dxy_{ij},dy_k)$ lie in $\Span\{A_l,B_l\}$.
\item The elements from above $\alpha^2(dx_i,dx_j)=\alpha^2(dy_i,dy_j)=0$.
\end{itemize}
\end{reduction}

The justification for the reduction that $\alpha^2(dxx_{ij},dy_k)\in\Span\{\overline{A}_l,\overline{B}_l\}$ is as follows:  we only want to consider $\alpha^2(dxx_{ij},dy_k)$ that will cancel a nontrivial term of the jacobiator $[\beta,\beta]$, which is supported on $\g\wedge\g\wedge\g$.  Recall that $\alpha^1(dx_i,dx_j)=xx_{ij}$, so $\alpha^2(d\alpha^1(x_i,x_j)],dy_k)$.  This suggests we examine:
$$ [\pi^1,\pi^2](dx_i,dx_j,dy_k) = \left([\beta^1,\beta^2]+[\beta^1,\alpha^2]+[\alpha^1,\beta^2]+[\alpha^1,\alpha^2]\right)(dx_i,dx_j,dy_k)$$
First, by inspection, we see that $[\beta^1,\beta^2](dx_i,dx_j,dy_k)\in\Span\{\overline{A}_l,\overline{B}_l\}$.  The other term that's already fixed $[\alpha^1,\beta^2](dx_i,dx_j,dy_k)\in\Sym(Z)$ won't be relevant to canceling the terms of $[\beta^1,\beta^2]$, so we will ignore it.  The two terms involving $\alpha^2$ are:
\begin{eqnarray}
\nonumber [\beta^1,\alpha^2](dx_i,dx_j,dy_k) &=& \beta^1(dx_i, d\alpha^2(dx_j,dy_k))+\beta^1(dx_j, d\alpha^2(dy_k,dx_i))+ \\
\nonumber							& & +\beta^1(dy_k, d\alpha^2(dx_i,dx_j)) \\
\nonumber 							&=& K_{jk}\beta^1(dx_i,dC_0) - K_{ik}\beta^1(dx_j,dC_0)+K'_{ij}\beta^1(dy_k,dC_0) \\
									&=& K_{jk}\overline{A}_i-K_{ik}\overline{A}_j+K'_{ij}A_k
\end{eqnarray}
\begin{eqnarray}
\nonumber [\alpha^1,\alpha^2](dx_i,dx_j,dy_k) &=&  \alpha^2(dx_i, d\alpha^1(dx_j,dy_k))+\alpha^2(dx_j, d\alpha^1(dy_k,dx_i))+ \\
\nonumber							 &  & \alpha^2(dy_k, d\alpha^1(dx_i,dx_j))+\mbox{terms in $\Sym(Z)$} \\
									 &=& L_{33}\alpha^2(dx_i,dxy_{jk})-L_{33}\alpha^2(dx_j,dxy_{ik})+\alpha^2(dy_k,dxx_{ij}) 
\end{eqnarray}

Thus, we see from Equation $(7.2)$ that it's reasonable to make $\alpha^2(dx_i,dx_j)=K'_{ij}C_0=0$, and from Equation $(7.3)$ that it's reasonable to make $\alpha^2(dxx_{ij},dy_k)$, $\alpha^2(dxy_{jk},dx_i)$, and $\alpha^2(dxy_{ik},dx_j)$ lie in $\Span\{\overline{A}_l,\overline{B}_l\}$.\\

Using the above reasoning to propose a first candidate for $\alpha^2$ produces a bivector which is included in its entirety in Appendix \ref{sec:AppendixM2} (the scalar coefficients $K_i$ range from $i=0...195$, and $l_i$ for $i=0...3$):

\begin{flushleft}
\begin{itemize}
\item[$\alpha^2 =$] $l_0C_0\partial\text{x}_{0}\partial\text{y}_{0}+l_1C_0\partial\text{x}_{1}\partial\text{y}_{1}+l_2C_0\partial\text{x}_{2}\partial\text{y}_{2}+l_3C_0\partial\text{x}_{3}\partial\text{y}_{3}+(K_{1}A_1+K_{2}B_0)\partial\text{x}_{0}\partial\text{yy}_{01}+(K_{3}A_2+K_{4}B_1)\partial\text{x}_{0}\partial\text{yy}_{02}+K_{5}A_3\partial\text{x}_{0}\partial\text{yy}_{03}+K_{6}A_3\partial\text{x}_{0}\partial\text{yy}_{12}+K_{7}\overline{A}_0\partial\text{x}_{0}\partial\text{xy}_{00}+(K_{8}\overline{A}_1+K_{9}\overline{B}_0)\partial\text{x}_{0}\partial\text{xy}_{10}+K_{10}\overline{A}_0\partial\text{x}_{0}\partial\text{xy}_{11}+(K_{11}\overline{A}_2+K_{12}\overline{B}_1)\partial\text{x}_{0}\partial\text{xy}_{20}+(K_{13}\overline{A}_1+K_{14}\overline{B}_0)\partial\text{x}_{0}\partial\text{xy}_{21}+K_{15}\overline{A}_0\partial\text{x}_{0}\partial\text{xy}_{22}+K_{16}\overline{A}_3\partial\text{x}_{0}\partial\text{xy}_{30}+(K_{17}\overline{A}_2+K_{18}\overline{B}_1)\partial\text{x}_{0}\partial\text{xy}_{31}+(K_{19}\overline{A}_1+K_{20}\overline{B}_0)\partial\text{x}_{0}\partial\text{xy}_{32}+K_{21}\overline{A}_0\partial\text{x}_{0}\partial\text{xy}_{33}+K_{22}A_0\partial\text{x}_{1}\partial\text{yy}_{01}+(K_{23}A_1+K_{24}B_0)\partial\text{x}_{1}\partial\text{yy}_{02}+(K_{25}A_2+K_{26}B_1)\partial\text{x}_{1}\partial\text{yy}_{03} \ldots$\\
\end{itemize}
\end{flushleft}

While this candidate for $\alpha^2$ is too general to compute with --- the resulting jacobiator $[\pi,\pi]$ would be around 100k lines long --- we can evaluate it on the space of $\g^*\wedge\g^*\wedge\g^*$ spanned by $[\beta,\beta]$ to determine some of the $K_i$.  The jacobiator of the candidate $\pi=(\beta^1+\alpha^1)+(\beta^2+\alpha^2)$ given above is applied to the 28 vectors below:
\\

\noindent $\{(\text{da}_2,\text{dx}_1,\text{dy}_0),( \text{dx}_0,\text{dx}_1,\text{dy}_0),( \text{dx}_0,\text{dx}_2,\text{dy}_0),$
$( \text{dx}_1,\text{dx}_2,\text{dy}_0),( \text{da}_3,\text{dx}_0,\text{dy}_1),( \text{dx}_0,\text{dx}_2,\text{dy}_1),$\\
$( \text{dx}_1,\text{dx}_2,\text{dy}_1),( \text{dx}_0,\text{dx}_3,\text{dy}_1),( \text{dx}_0,\text{dy}_0,\text{dy}_1),$
$(\text{dx}_1,\text{dx}_2,\text{dy}_2), ( \text{da}_2,\text{dx}_3,\text{dy}_2),( \text{dx}_0,\text{dx}_3,\text{dy}_2),$\\
$( \text{dx}_1,\text{dx}_3,\text{dy}_2),( \text{dx}_0,\text{dy}_0,\text{dy}_2),( \text{dx}_1,\text{dy}_0,\text{dy}_2),
( \text{dx}_0,\text{dy}_1,\text{dy}_2),( \text{dx}_1,\text{dy}_1,\text{dy}_2),( \text{dx}_2,\text{dy}_1,\text{dy}_2),$\\
$(\text{dx}_3,\text{dy}_1,\text{dy}_2),( \text{da}_3,\text{dx}_2,\text{dy}_3),( \text{dx}_1,\text{dx}_2,\text{dy}_3),
(\text{dx}_1,\text{dx}_3,\text{dy}_3),( \text{dx}_2,\text{dx}_3,\text{dy}_3),( \text{dx}_1,\text{dy}_0,\text{dy}_3),$\\
$( \text{dx}_2,\text{dy}_0,\text{dy}_3),( \text{dx}_2,\text{dy}_1,\text{dy}_3),( \text{dx}_3,\text{dy}_1,\text{dy}_3),(
      \text{dx}_3,\text{dy}_2,\text{dy}_3)\}$
\\

The condition that $[\pi,\pi]$ vanishes on each of these vectors imposes linear constraints on the coefficients $K_i$ of $\alpha^2$:    
$$
\left\{\begin{array}{l}
2L_{33}K_{8}+L_{33}K_{9}-2L_{33}K_{30}-L_{33}K_{31}+2L_{10}K_{97}+L_{10}K_{98}-2l_0=0,\\
L_{33}K_{8}-L_{33}K_{9}-L_{33}K_{30}+L_{33}K_{31}+L_{10}K_{97}-L_{10}K_{98}-l_0-18=0,\\
 L_{33}K_{11}+L_{33}K_{12}-L_{33}K_{57}-L_{33}K_{58}+L_{10}K_{99}+L_{10}K_{100}-l_0=0,\\
 2L_{33}K_{11}-L_{33}K_{12}-2L_{33}K_{57}+L_{33}K_{58}+2L_{10}K_{99}-L_{10}K_{100}-2l_0-18=0,\\
\vdots \\
\vdots  \quad \mbox{(A total of 44 equations linear in $K_i$)} \\
\vdots \\

\end{array}\right.$$

There are a total of 132 coefficients $K_i$ in the above linear system of equations (in the $K_i$).  Using Maple, we eliminate 44 of the $K$-variables and two of the $l$-variables to simplify $\alpha^2$ to a bivector with 156 degrees of freedom.  Moreover, the resulting jacobiator $[\beta+\alpha,\beta+\alpha]$ now vanishes on the 28-dimensional subspace of $\g\wedge\g\wedge\g$ given above (that is, we have now cancelled $[\beta,\beta]$. 

The next step is to determine restrictions on the $K_i$ to make $[\pi,\pi]=0$ on the entire $\g^*\wedge\g^*\wedge\g^*$ and then on $\bigwedge^3(\g\times Z)^*$.  This is done by iterating the process above --- each time solving a new set of linear equations.  The end result is a Poisson bivector $\pi=(\beta^1+\alpha^1)+(\beta^2+\alpha^2)$ that extends the Poisson structure on $V/\ZZ_3$.

\bibliography{references}
\appendix
\chapter{Representations of $\mathfrak{\MakeLowercase{sl}}_2(\CC)$}
\label{sec:AppendixSL2}

The purpose of this section is to fix notation concerning the representation theory of the complex semi-simple Lie algebra $\slt(\CC)$.  We follow the treatment in \cite{Fulton91}. Denote the standard basis of $\slt(\CC)$ by $\{H,E,F\}$ where

$$[H,E]=2E, \mbox{\qquad   } [H,F]=-2F, \mbox{\qquad   } [E,F]=H $$

The standard representation of $\slt(\CC)$ is given by $\rho_2:\slt(\CC)\to{\rm End}(\CC^2)$, where

$$ \rho(H)=\left(\begin{array}{cc} 1 & 0 \\ 0 & -1 \end{array}\right), \mbox{\qquad   }
      \rho(E)=\left(\begin{array}{cc} 0 & 1 \\ 0 & 0 \end{array}\right), \mbox{\qquad   }
      \rho(F)=\left(\begin{array}{cc} 0 & 0 \\ 1 & 0 \end{array}\right), \mbox{\qquad   } $$
      
 In each dimension, there is exactly one irreducible representation.  We denote the $(n+1)$-dimensional representation by $V_n$ with corresponding basis $\{v_0,\ldots,v_n\}$.  The representation $\rho_n:\slt(\CC)\to{\rm End}(V_n)$ is given as follows:
 $$ \rho_n(H)(v_k)=(n-2k)v_k, \mbox{ \qquad  } \rho_n(E)(v_k)=(n-k)v_{k+1} \mbox{  \qquad } \rho_n(F)(v_k)=kv_{k-1} $$
 
 The Cartan element $H$ acts diagonally, decomposing $V_n$ into eigenspaces $V_n=\bigoplus_{i=0}^n E_i$, and the action of $\slt(\CC)$ acts on each one dimensional eigenspace:
 $$
 \xymatrix{
0  & E_0 \ar@/^/[l]^F \ar@(ru,ul)[]|{H} \ar@/^/[r]^{E} & E_1 \ar@(ru,ul)[]|{H} \ar@/^/[r]^E \ar@/^/[l]^F & E_2 \ar@(ru,ul)[]|{H} \ar@/^/[r]^E \ar@/^/[l]^F & \ldots\ldots  \ar@/^/[l]^F  \ar@/^/[r]^E & E_{n-1} \ar@(ru,ul)[]|{H} \ar@/^/[r]^E \ar@/^/[l]^F & E_n \ar@(ru,ul)[]|{H} \ar@/^/[l]^F \ar@/^/[r]^{E}& 0 
 }$$
 
 As $\slt(\CC)$ is semi-simple, a finite dimensional representation $V$ is the direct sum of the irreducible representations described above: $V=\bigoplus V_{d_i}$, for $d_i\in\NN$.  We will often have $\slt(\CC)$-equivariant linear maps of two representations and use Schur's Lemma to deconstruct the map:
 
 \begin{lemma}[Schur] Suppose that $V$ and $W$ are irreducible representations of a complex Lie algebra $\g$ and $\phi: V\to W$ is a $\g$-module homomorphism. Then
\begin{enumerate}
\item Either $\phi=0$ or $\phi$ is an isomorphism.
\item Moreover, if $V=W$,  then $\phi=\lambda {\rm Id}$ for $\lambda\in\CC$.
 \end{enumerate}
 \end{lemma}
 This implies the following corollary, which we use liberally throughout chapters 6 and 7:

\begin{corollary} 
\label{cor:SchurEnd}
Let $V$ and $V'$ be finite dimensional $\slt(\CC)$ representations, and $\phi:V\to V'$ be a $\g$-module homomorphism.  Denote the decomposition of $V=\bigoplus_{i}V_{d_i}$ and $V'=\bigoplus_{j}V_{c_j}$ into their irreducible representations.  Then $\phi$ decomposes as
$$\phi=\sum_{i}\phi|_{V_{d_i}}, \mbox{ \qquad where   }\phi|_{V_{d_i}}= \sum_{d_i=c_j}\lambda_{(d_i,c_j)}{\rm Id}   \mbox{\quad and  } \lambda_{(d_i,c_j)}\in\CC$$
\end{corollary}
 
This corollary is often applied to the tensor product, symmetric product, and wedge product of representations:

\begin{proposition} Let $V_n$ be an irreducible representation of $\slt(\CC)$. The decompositions of the following representations into irreducible sub-representations are:

$$V_n\otimes V_n = \bigoplus_{k=0}^n V_{2(n-k)} \mbox{ \qquad and \qquad }  V_n\wedge V_n  = \bigoplus_{k=0}^{\lfloor \frac{n-1}{2} \rfloor } V_{2(n-(2k+1))} $$

\end{proposition}

\begin{example} A simple, yet typical application of this is the case $n=1$, then: $V_1\otimes V_1=V_2\oplus V_0$ and $V_1\wedge V_1=V_0$, and so any $\slt(\CC)$-equivariant map $\phi:V_1\otimes V_1\to V_1\wedge V_1$ sends 
$$(v_0\otimes v_1+v_1\otimes v_0) \mapsto \lambda v_0\wedge v_1$$
and sends the complementary subspace to zero.  That is, 
$$\phi: \left\{\begin{array}{l} v_0\otimes v_0 \mapsto 0 \\
					v_0\otimes v_1 \mapsto \frac{\lambda}{2} v_0\wedge v_1 \\
					v_1\otimes v_0 \mapsto \frac{\lambda}{2} v_0\wedge v_1 \\
					v_1\otimes v_1 \mapsto 0 
\end{array}\right.$$
\end{example}

\begin{remark}  Using Corollary \ref{cor:SchurEnd} as in the example above, one can actually explicitly determine the extensions up to order $(n-1)$ in Chapters 6 and 7.

\end{remark}

\chapter{Macaulay2 Computations}
\label{sec:AppendixM2}

This appendix contains an explicit description, using Macaulay2 \cite{M2}, of the extension of the Poisson structure on $V/\ZZ_3$ given in Chapter 7 to $\RR^{78}$.  A sketch of the procedure used to construct this Poisson structure is also given (using the syntax/notation of Macaulay2), as it is likely that such a process is applicable to other examples (e.g. $V/\ZZ_n$ for other $n$, or for the $(2,1,1)$-resonance action considered in \cite{Egilsson95,Davis02}).  Throughout, we use $\mathbb{Q}$ for the coefficient ring to avoid decimals and issues associated to floating point.  \\

First define the coordinate ring of the embedding space:
\begin{itemize}
\item[$\text{i1}:$] $\text{T=QQ[a$_{0}$,a$_{1}$,a$_{2}$,a$_{3}$,x$_{0}$,x$_{1}$,x$_{2}$,x$_{3}$,y$_{0}$,y$_{1}$,y$_{2}$,y$_{3}$,aa$_{01}$,$\ldots$,aa$_{23}$,ax$_{00}$,$\ldots$,ax$_{33}$,ay$_{00}$,$\ldots$,ay$_{33}$,}$\\ $\text{xx$_{01}$,$\ldots$,xx$_{23}$,yy$_{01}$,$\ldots$,yy$_{23}$,xy$_{00}$,$\ldots$,xy$_{33}$];}$
\end{itemize}
To realize the quotient as an embedded subvariety of $\RR^{78}$, we will examine the graph of the embedding in $V\times\RR^{78}$.  Define the coordinate ring of the product $V\times\RR^{78}$ and a minimal generating set of $\Sym(V)^{\ZZ_n}$:
\begin{itemize}
\item[$\text{i2}:$] $\text{R}=\text{T}[z,\zb,w,\wb ];$
\end{itemize}
\begin{itemize}
\item[$\text{i3}:$]  $\text{HilbertBasis}=\text{ideal}(\{z \zb,w \wb,\zb w,z \wb,w^{3},z w^{2},z^{2}w,z^{3},\wb^{3},\zb \wb^{2},\zb^{2}\wb,\zb^{3}\});$
\end{itemize}
The image of the embedding is then given by $V(I)\isom Q$, where $Q=V/\ZZ_3$:
\begin{itemize}
\item[$\text{i4}:$]  $\text{Inv}=\text{matrix}\{\{z \zb,w \wb,\zb w,z \wb,w^{3},z w^{2},z^{2}w,z^{3},\wb^{3},\zb \wb^{2},\zb^{2}\wb,\zb^{3},0,0,0,0,0,\ldots\}\};$ \\
$I=\text{ideal}(\text{vars}(\text{T})-\text{Inv});$ \\
$Q=\text{R}/I ;$ \\
$\text{use R};$ 
\end{itemize}
We then define the Poisson bracket $\text{Brac}$ on $\Sym(V)$
\begin{itemize}
\item[$\text{i5}:$] $\text{Brac} = $\\
$(f,g) \to \text{diff}(z,f)\text{diff}(\zb,g)-\text{diff}(\zb,f)\text{diff}(z,g)+\text{diff}(w,f)\text{diff}(\wb,g)-\text{diff}(\wb,f)\text{diff}(w,g);$
\end{itemize}
and the skew-commutative ring of polynomial multivector fields on $\RR^{78}$
\begin{itemize}
\item[$\text{i6}:$] $ \text{XT}=\text{T}[\text{$\partial$a$_0$,$\partial$a$_1$,$\partial$a$_2$,$\partial$a$_3$,$\partial$x$_0$,$\partial$x$_1$,\ldots,$\partial$xy$_{32}$,$\partial$xy$_{33}$, SkewCommutative} \Rightarrow \text{true}];$
\end{itemize}
where $\partial\text{u}$ is the skew-commuting variable $\vect{\text{u}}$ for $\text{u}\in\text{T}$.  We may then extend the Poisson bracket on $Q$ to a bivector on $\RR^{78}$:
\begin{itemize}
\item[$\text{i7}:$]  $ \text{phi}= \text{map}(\text{XT},\text{T},\{\text{$\partial$a$_0$,$\partial$a$_1$,$\partial$a$_2$,$\partial$a$_3$,$\partial$x$_0$,$\ldots,\partial$xy$_{33}$}\}); $\\
 $ \text{psi} = f \to \text{lift}(\text{promote}(f,Q),\text{T}); $ \\
 $ \text{Br} = (f,g) \to \text{lift}(\text{promote}(\text{Brac}(f,g),Q),\text{T})\text{phi}(\text{psi}(f))\text{phi}(\text{psi}(g)); $
\end{itemize}
Br extends the bracket $\{f,g\}_Q$ for any $f,g\in\Sym(V)^{\ZZ_n}$, which is then used to construct the bivector Pi.
\begin{itemize}
\item[$\text{i8}:$] $\text{H} = \text{flatten entries Inv};  $
\item[$\text{i9}:$]  $ \beta= (1/2)(\text{fold (flatten table(H,H,Br)},(i,j)\to i+j)) $
\item[$\text{o9}:$] $\text{a$_2$}\text{$\partial$a$_0$}\text{$\partial$a$_2$}-\text{a$_2$}\text{$\partial$a$_1$}\text{$\partial$a$_2$}-\text{a$_3$}\text{$\partial$a$_0$}\text{$\partial$a$_3$}+\text{a$_3$}\text{$\partial$a$_1$}\text{$\partial$a$_3$}+(\text{a$_0$}-\text{a$_1$})\text{$\partial$a$_2$}\text{$\partial$a$_3$}-3\text{x$_0$}\text{$\partial$a$_1$}\text{$\partial$x$_0$}-3\text{x$_1$}\text{$\partial$a$_3$}\text{$\partial$x$_0$}-\text{x$_1$}\text{$\partial$a$_0$}\text{$\partial$x$_1$}-2\text{x$_1$}\text{$\partial$a$_1$}\text{$\partial$x$_1$}-\text{x$_0$}\text{$\partial$a$_2$}\text{$\partial$x$_1$}-2\text{x$_2$}\text{$\partial$a$_3$}\text{$\partial$x$_1$}-2\text{x$_2$}\text{$\partial$a$_0$}\text{$\partial$x$_2$}-\text{x$_2$}\text{$\partial$a$_1$}\text{$\partial$x$_2$}-2\text{x$_1$}\text{$\partial$a$_2$}\text{$\partial$x$_2$}-\text{x$_3$}\text{$\partial$a$_3$}\text{$\partial$x$_2$}-3\text{x$_3$}\text{$\partial$a$_0$}\text{$\partial$x$_3$}-3\text{x$_2$}\text{$\partial$a$_2$}\text{$\partial$x$_3$}+3\text{y$_0$}\text{$\partial$a$_1$}\text{$\partial$y$_0$}+3\text{y$_1$}\text{$\partial$a$_2$}\text{$\partial$y$_0$}+9\text{a$_1$}^{2}\text{$\partial$x$_0$}\text{$\partial$y$_0$}+6\text{a$_1$}\text{a$_3$}\text{$\partial$x$_1$}\text{$\partial$y$_0$}+3\text{a$_3$}^{2}\text{$\partial$x$_2$}\text{$\partial$y$_0$}+\text{y$_1$}\text{$\partial$a$_0$}\text{$\partial$y$_1$}+2\text{y$_1$}\text{$\partial$a$_1$}\text{$\partial$y$_1$}+2\text{y$_2$}\text{$\partial$a$_2$}\text{$\partial$y$_1$}+\text{y$_0$}\text{$\partial$a$_3$}\text{$\partial$y$_1$}+6\text{a$_1$}\text{a$_2$}\text{$\partial$x$_0$}\text{$\partial$y$_1$}+(\text{a$_1$}^{2}+4\text{a$_2$}\text{a$_3$})\text{$\partial$x$_1$}\text{$\partial$y$_1$}+(2\text{a$_0$}\text{a$_3$}+2\text{a$_1$}\text{a$_3$})\text{$\partial$x$_2$}\text{$\partial$y$_1$}+3\text{a$_3$}^{2}\text{$\partial$x$_3$}\text{$\partial$y$_1$}+2\text{y$_2$}\text{$\partial$a$_0$}\text{$\partial$y$_2$}+\text{y$_2$}\text{$\partial$a$_1$}\text{$\partial$y$_2$}+\text{y$_3$}\text{$\partial$a$_2$}\text{$\partial$y$_2$}+2\text{y$_1$}\text{$\partial$a$_3$}\text{$\partial$y$_2$}+3\text{a$_2$}^{2}\text{$\partial$x$_0$}\text{$\partial$y$_2$}+(2\text{a$_0$}\text{a$_2$}+2\text{a$_1$}\text{a$_2$})\text{$\partial$x$_1$}\text{$\partial$y$_2$}+(\text{a$_0$}^{2}+4\text{a$_2$}\text{a$_3$})\text{$\partial$x$_2$}\text{$\partial$y$_2$}+6\text{a$_0$}\text{a$_3$}\text{$\partial$x$_3$}\text{$\partial$y$_2$}+3\text{y$_3$}\text{$\partial$a$_0$}\text{$\partial$y$_3$}+3\text{y$_2$}\text{$\partial$a$_3$}\text{$\partial$y$_3$}+3\text{a$_2$}^{2}\text{$\partial$x$_1$}\text{$\partial$y$_3$}+6\text{a$_0$}\text{a$_2$}\text{$\partial$x$_2$}\text{$\partial$y$_3$}+9\text{a$_0$}^{2}\text{$\partial$x$_3$}\text{$\partial$y$_3$}$
\end{itemize}
The bivector $\beta$ decomposes into homogeneous components $\beta=\beta^1+\beta^2$:
\begin{itemize}
\item[$\text{i10}:$] $\text{tt}=\text{matrix}\{\{\text{a}_0,\text{a}_1,\text{a}_2,\text{a}_3,\text{x}_0,\text{x}_1,\text{x}_2,\text{x}_3,\text{y}_0,\text{y}_1,\text{y}_2,\text{y}_3\}\};$ \\
M2=ideal(transpose(tt)*tt);\\
XM2=XT/promote(M2,XT);\\
use XT;
\item[$\text{i11}:$] $\beta^1=\text{lift(promote($\beta$,XM2),XT)}$
\item[$\text{o11}:$] $\text{a$_2$} \text{$\partial$a$_0$} \text{$\partial$a$_2$}-\text{a$_2$} \text{$\partial$a$_1$} \text{$\partial$a$_2$}-\text{a$_3$}      \text{$\partial$a$_0$} \text{$\partial$a$_3$}+\text{a$_3$} \text{$\partial$a$_1$} \text{$\partial$a$_3$}+(\text{a$_0$}-\text{a$_1$})      \text{$\partial$a$_2$} \text{$\partial$a$_3$}-3 \text{x$_0$} \text{$\partial$a$_1$} \text{$\partial$x$_0$}-3 \text{x$_1$} \text{$\partial$a$_3$}      \text{$\partial$x$_0$}-\text{x$_1$} \text{$\partial$a$_0$} \text{$\partial$x$_1$}-2 \text{x$_1$} \text{$\partial$a$_1$}      \text{$\partial$x$_1$}-\text{x$_0$} \text{$\partial$a$_2$} \text{$\partial$x$_1$}-2 \text{x$_2$} \text{$\partial$a$_3$} \text{$\partial$x$_1$}-2      \text{x$_2$} \text{$\partial$a$_0$} \text{$\partial$x$_2$}-\text{x$_2$} \text{$\partial$a$_1$} \text{$\partial$x$_2$}-2 \text{x$_1$}      \text{$\partial$a$_2$} \text{$\partial$x$_2$}-\text{x$_3$} \text{$\partial$a$_3$} \text{$\partial$x$_2$}-3 \text{x$_3$} \text{$\partial$a$_0$}      \text{$\partial$x$_3$}-3 \text{x$_2$} \text{$\partial$a$_2$} \text{$\partial$x$_3$}+3 \text{y$_0$} \text{$\partial$a$_1$} \text{$\partial$y$_0$}+3      \text{y$_1$} \text{$\partial$a$_2$} \text{$\partial$y$_0$}+\text{y$_1$} \text{$\partial$a$_0$} \text{$\partial$y$_1$}+2 \text{y$_1$}      \text{$\partial$a$_1$} \text{$\partial$y$_1$}+2 \text{y$_2$} \text{$\partial$a$_2$} \text{$\partial$y$_1$}+\text{y$_0$} \text{$\partial$a$_3$}      \text{$\partial$y$_1$}+2 \text{y$_2$} \text{$\partial$a$_0$} \text{$\partial$y$_2$}+\text{y$_2$} \text{$\partial$a$_1$}      \text{$\partial$y$_2$}+\text{y$_3$} \text{$\partial$a$_2$} \text{$\partial$y$_2$}+2 \text{y$_1$} \text{$\partial$a$_3$} \text{$\partial$y$_2$}+3      \text{y$_3$} \text{$\partial$a$_0$} \text{$\partial$y$_3$}+3 \text{y$_2$} \text{$\partial$a$_3$} \text{$\partial$y$_3$}$
\item[$\text{i12}:$] $\beta^2=\beta-\beta^1$
\item[$\text{o12}:$] $9 \text{a$_1$}^{2} \text{$\partial$x$_0$} \text{$\partial$y$_0$}+6 \text{a$_1$} \text{a$_3$} \text{$\partial$x$_1$}
      \text{$\partial$y$_0$}+3 \text{a$_3$}^{2} \text{$\partial$x$_2$} \text{$\partial$y$_0$}+6 \text{a$_1$} \text{a$_2$}
      \text{$\partial$x$_0$} \text{$\partial$y$_1$}+(\text{a$_1$}^{2}+4 \text{a$_2$} \text{a$_3$}) \text{$\partial$x$_1$}
      \text{$\partial$y$_1$}+(2 \text{a$_0$} \text{a$_3$}+2 \text{a$_1$} \text{a$_3$}) \text{$\partial$x$_2$} \text{$\partial$y$_1$}+3
      \text{a$_3$}^{2} \text{$\partial$x$_3$} \text{$\partial$y$_1$}+3 \text{a$_2$}^{2} \text{$\partial$x$_0$} \text{$\partial$y$_2$}+(2
      \text{a$_0$} \text{a$_2$}+2 \text{a$_1$} \text{a$_2$}) \text{$\partial$x$_1$}
      \text{$\partial$y$_2$}+(\text{a$_0$}^{2}+4 \text{a$_2$} \text{a$_3$}) \text{$\partial$x$_2$} \text{$\partial$y$_2$}+6
      \text{a$_0$} \text{a$_3$} \text{$\partial$x$_3$} \text{$\partial$y$_2$}+3 \text{a$_2$}^{2} \text{$\partial$x$_1$}
      \text{$\partial$y$_3$}+6 \text{a$_0$} \text{a$_2$} \text{$\partial$x$_2$} \text{$\partial$y$_3$}+9 \text{a$_0$}^{2}
      \text{$\partial$x$_3$} \text{$\partial$y$_3$}$
\end{itemize}
We define the (even degree) Schouten bracket (on 2-vector fields), denoted by SBr, on $\RR^{78}$:
\begin{itemize}
\item[$\text{i13}:$]  $\text{Sbr}= (A,B) \to (\text{diff($\partial$a$_0$,$A$)diff(a$_0$,$B$)}+\text{diff($\partial$a$_0$,$B$)diff(a$_0$,$A$)})+\ldots$\\$\quad\qquad\qquad\ldots+(\text{diff($\partial$ay$_{33}$,$A$)diff(ay$_{33}$,$B$)}+\text{diff($\partial$ay$_{33}$,$B$)diff(ay$_{33}$,$A$)});$
\end{itemize}
We are now able to verify several facts from Chapter 6.
\begin{itemize}
\item[$\text{i14}:$] $\text{Sbr($\beta^1$,$\beta^1$)}==0 \text{ and Sbr($\beta^2$,$\beta^2$)}==0 $
\item[$\text{o14}:$] true
\item[$\text{i15}:$] jacobiator=Sbr($\beta$,$\beta$)
\item[$\text{o15}:$] \begin{flushleft} $(3 \text{a$_0$} \text{a$_1$}-3 \text{a$_2$} \text{a$_3$}) \text{$\partial$a$_2$} \text{$\partial$x$_1$}
      \text{$\partial$y$_0$}+(9 \text{a$_3$} \text{x$_0$}-9 \text{a$_1$} \text{x$_1$}) \text{$\partial$x$_0$} \text{$\partial$x$_1$}
      \text{$\partial$y$_0$}+(9 \text{a$_3$} \text{x$_1$}-9 \text{a$_1$} \text{x$_2$}) \text{$\partial$x$_0$} \text{$\partial$x$_2$}
      \text{$\partial$y$_0$}+(3 \text{a$_3$} \text{x$_2$}-3 \text{a$_1$} \text{x$_3$}) \text{$\partial$x$_1$} \text{$\partial$x$_2$}
      \text{$\partial$y$_0$}+(-3 \text{a$_0$} \text{a$_1$}+3 \text{a$_2$} \text{a$_3$}) \text{$\partial$a$_3$} \text{$\partial$x$_0$}
      \text{$\partial$y$_1$}+(3 \text{a$_3$} \text{x$_0$}+3 \text{a$_0$} \text{x$_1$}-3 \text{a$_1$} \text{x$_1$}-3
      \text{a$_2$} \text{x$_2$}) \text{$\partial$x$_0$} \text{$\partial$x$_2$} \text{$\partial$y$_1$}+(-\text{a$_3$} \text{x$_1$}+2
      \text{a$_0$} \text{x$_2$}+\text{a$_1$} \text{x$_2$}-2 \text{a$_2$} \text{x$_3$}) \text{$\partial$x$_1$}
      \text{$\partial$x$_2$} \text{$\partial$y$_1$}+(9 \text{a$_3$} \text{x$_1$}-9 \text{a$_1$} \text{x$_2$}) \text{$\partial$x$_0$}
      \text{$\partial$x$_3$} \text{$\partial$y$_1$}+(9 \text{a$_2$} \text{y$_0$}-9 \text{a$_1$} \text{y$_1$}) \text{$\partial$x$_0$}
      \text{$\partial$y$_0$} \text{$\partial$y$_1$}+(2 \text{a$_3$} \text{x$_0$}-\text{a$_0$} \text{x$_1$}-2 \text{a$_1$}
      \text{x$_1$}+\text{a$_2$} \text{x$_2$}) \text{$\partial$x$_1$} \text{$\partial$x$_2$} \text{$\partial$y$_2$}+(-3 \text{a$_0$}
      \text{a$_1$}+3 \text{a$_2$} \text{a$_3$}) \text{$\partial$a$_2$} \text{$\partial$x$_3$} \text{$\partial$y$_2$}+(9 \text{a$_0$}
      \text{x$_1$}-9 \text{a$_2$} \text{x$_2$}) \text{$\partial$x$_0$} \text{$\partial$x$_3$} \text{$\partial$y$_2$}+(3 \text{a$_3$}
      \text{x$_1$}+3 \text{a$_0$} \text{x$_2$}-3 \text{a$_1$} \text{x$_2$}-3 \text{a$_2$} \text{x$_3$})
      \text{$\partial$x$_1$} \text{$\partial$x$_3$} \text{$\partial$y$_2$}+(9 \text{a$_2$} \text{y$_1$}-9 \text{a$_1$} \text{y$_2$})
      \text{$\partial$x$_0$} \text{$\partial$y$_0$} \text{$\partial$y$_2$}+(3 \text{a$_2$} \text{y$_0$}+3 \text{a$_0$} \text{y$_1$}-3
      \text{a$_1$} \text{y$_1$}-3 \text{a$_3$} \text{y$_2$}) \text{$\partial$x$_1$} \text{$\partial$y$_0$} \text{$\partial$y$_2$}+(3
      \text{a$_2$} \text{y$_2$}-3 \text{a$_1$} \text{y$_3$}) \text{$\partial$x$_0$} \text{$\partial$y$_1$}
      \text{$\partial$y$_2$}+(-\text{a$_2$} \text{y$_1$}+2 \text{a$_0$} \text{y$_2$}+\text{a$_1$} \text{y$_2$}-2
      \text{a$_3$} \text{y$_3$}) \text{$\partial$x$_1$} \text{$\partial$y$_1$} \text{$\partial$y$_2$}+(2 \text{a$_2$}
      \text{y$_0$}-\text{a$_0$} \text{y$_1$}-2 \text{a$_1$} \text{y$_1$}+\text{a$_3$} \text{y$_2$})
      \text{$\partial$x$_2$} \text{$\partial$y$_1$} \text{$\partial$y$_2$}+(3 \text{a$_0$} \text{y$_0$}-3 \text{a$_3$} \text{y$_1$})
      \text{$\partial$x$_3$} \text{$\partial$y$_1$} \text{$\partial$y$_2$}+(3 \text{a$_0$} \text{a$_1$}-3 \text{a$_2$} \text{a$_3$})
      \text{$\partial$a$_3$} \text{$\partial$x$_2$} \text{$\partial$y$_3$}+(3 \text{a$_0$} \text{x$_0$}-3 \text{a$_2$} \text{x$_1$})
      \text{$\partial$x$_1$} \text{$\partial$x$_2$} \text{$\partial$y$_3$}+(9 \text{a$_0$} \text{x$_1$}-9 \text{a$_2$} \text{x$_2$})
      \text{$\partial$x$_1$} \text{$\partial$x$_3$} \text{$\partial$y$_3$}+(9 \text{a$_0$} \text{x$_2$}-9 \text{a$_2$} \text{x$_3$})
      \text{$\partial$x$_2$} \text{$\partial$x$_3$} \text{$\partial$y$_3$}+(9 \text{a$_2$} \text{y$_1$}-9 \text{a$_1$} \text{y$_2$})
      \text{$\partial$x$_1$} \text{$\partial$y$_0$} \text{$\partial$y$_3$}+(9 \text{a$_0$} \text{y$_1$}-9 \text{a$_3$} \text{y$_2$})
      \text{$\partial$x$_2$} \text{$\partial$y$_0$} \text{$\partial$y$_3$}+(3 \text{a$_2$} \text{y$_1$}+3 \text{a$_0$} \text{y$_2$}-3
      \text{a$_1$} \text{y$_2$}-3 \text{a$_3$} \text{y$_3$}) \text{$\partial$x$_2$} \text{$\partial$y$_1$} \text{$\partial$y$_3$}+(9
      \text{a$_0$} \text{y$_1$}-9 \text{a$_3$} \text{y$_2$}) \text{$\partial$x$_3$} \text{$\partial$y$_1$} \text{$\partial$y$_3$}+(9
      \text{a$_0$} \text{y$_2$}-9 \text{a$_3$} \text{y$_3$}) \text{$\partial$x$_3$} \text{$\partial$y$_2$} \text{$\partial$y$_3$}$
      \end{flushleft}
\end{itemize}
Following the procedure laid out in Chapter 7, we construct the linear bivector $\alpha^1=\alpha^1_\rho+\alpha^1_c$, where $\alpha^1_\rho$ is constructed from the $\g$-representation $Z\isom\g\wedge\g$, and $\alpha^1_c$ is constructed from the ``identity cocycle'' $\g\wedge\g\to Z\isom\g\wedge\g$.  We construct $\alpha^1_\rho$ using an identical process to the process used to extend $\brac_Q$ to $\beta$, obtaining 
\begin{itemize}
\item[$\text{i16}:$]\begin{flushleft} $\alpha_\rho^1=(\text{aa}_{02}+\text{aa}_{12})\partial\text{a}_{2}\partial\text{aa}_{01}+(-\text{aa}_{03}-\text{aa}_{13})\partial\text{a}_{3}\partial\text{aa}_{01}+3\text{ax}_{00}\partial\text{x}_{0}\partial\text{aa}_{01}+(2\text{ax}_{01}-\text{ax}_{11})\partial\text{x}_{1}\partial\text{aa}_{01}+(\text{ax}_{02}-2\text{ax}_{12})\partial\text{x}_{2}\partial\text{aa}_{01}-3\text{ax}_{13}\partial\text{x}_{3}\partial\text{aa}_{01}-3\text{ay}_{00}\partial\text{y}_{0}\partial\text{aa}_{01}+(-2\text{ay}_{01}+\text{ay}_{11})\partial\text{y}_{1}\partial\text{aa}_{01}+(-\text{ay}_{02}+2\text{ay}_{12})\partial\text{y}_{2}\partial\text{aa}_{01}+3\text{ay}_{13}\partial\text{y}_{3}\partial\text{aa}_{01}+\text{aa}_{02}\partial\text{a}_{0}\partial\text{aa}_{02}-\text{aa}_{02}\partial\text{a}_{1}\partial\text{aa}_{02}+(\text{aa}_{01}-\text{aa}_{23})\partial\text{a}_{3}\partial\text{aa}_{02}+(\text{ax}_{00}-\text{ax}_{21})\partial\text{x}_{1}\partial\text{aa}_{02}+(2\text{ax}_{01}-2\text{ax}_{22})\partial\text{x}_{2}\partial\text{aa}_{02}+(3\text{ax}_{02}-3\text{ax}_{23})\partial\text{x}_{3}\partial\text{aa}_{02}-3\text{ay}_{01}\partial\text{y}_{0}\partial\text{aa}_{02}+(-2\text{ay}_{02}+\text{ay}_{21})\partial\text{y}_{1}\partial\text{aa}_{02}+(-\text{ay}_{03}+2\text{ay}_{22})\partial\text{y}_{2}\partial\text{aa}_{02}+3\text{ay}_{23}\partial\text{y}_{3}\partial\text{aa}_{02}-\text{aa}_{03}\partial\text{a}_{0}\partial\text{aa}_{03}+\text{aa}_{03}\partial\text{a}_{1}\partial\text{aa}_{03}+(-\text{aa}_{01}-\text{aa}_{23})\partial\text{a}_{2}\partial\text{aa}_{03}+3\text{ax}_{01}\partial\text{x}_{0}\partial\text{aa}_{03}+(2\text{ax}_{02}-\text{ax}_{31})\partial\text{x}_{1}\partial\text{aa}_{03}+(\text{ax}_{03}-2\text{ax}_{32})\partial\text{x}_{2}\partial\text{aa}_{03}-3\text{ax}_{33}\partial\text{x}_{3}\partial\text{aa}_{03}+(-\text{ay}_{00}+\text{ay}_{31})\partial\text{y}_{1}\partial\text{aa}_{03}+(-2\text{ay}_{01}+2\text{ay}_{32})\partial\text{y}_{2}\partial\text{aa}_{03}+(-3\text{ay}_{02}+3\text{ay}_{33})\partial\text{y}_{3}\partial\text{aa}_{03}+\text{aa}_{12}\partial\text{a}_{0}\partial\text{aa}_{12}-\text{aa}_{12}\partial\text{a}_{1}\partial\text{aa}_{12}+(\text{aa}_{01}+\text{aa}_{23})\partial\text{a}_{3}\partial\text{aa}_{12}-3\text{ax}_{20}\partial\text{x}_{0}\partial\text{aa}_{12}+(\text{ax}_{10}-2\text{ax}_{21})\partial\text{x}_{1}\partial\text{aa}_{12}+(2\text{ax}_{11}-\text{ax}_{22})\partial\text{x}_{2}\partial\text{aa}_{12}+3\text{ax}_{12}\partial\text{x}_{3}\partial\text{aa}_{12}+(-3\text{ay}_{11}+3\text{ay}_{20})\partial\text{y}_{0}\partial\text{aa}_{12}+(-2\text{ay}_{12}+2\text{ay}_{21})\partial\text{y}_{1}\partial\text{aa}_{12}+(-\text{ay}_{13}+\text{ay}_{22})\partial\text{y}_{2}\partial\text{aa}_{12}-\text{aa}_{13}\partial\text{a}_{0}\partial\text{aa}_{13}+\text{aa}_{13}\partial\text{a}_{1}\partial\text{aa}_{13}+(-\text{aa}_{01}+\text{aa}_{23})\partial\text{a}_{2}\partial\text{aa}_{13}+(3\text{ax}_{11}-3\text{ax}_{30})\partial\text{x}_{0}\partial\text{aa}_{13}+(2\text{ax}_{12}-2\text{ax}_{31})\partial\text{x}_{1}\partial\text{aa}_{13}+(\text{ax}_{13}-\text{ax}_{32})\partial\text{x}_{2}\partial\text{aa}_{13}+3\text{ay}_{30}\partial\text{y}_{0}\partial\text{aa}_{13}+(-\text{ay}_{10}+2\text{ay}_{31})\partial\text{y}_{1}\partial\text{aa}_{13}+(-2\text{ay}_{11}+\text{ay}_{32})\partial\text{y}_{2}\partial\text{aa}_{13}-3\text{ay}_{12}\partial\text{y}_{3}\partial\text{aa}_{13}+(-\text{aa}_{02}+\text{aa}_{12})\partial\text{a}_{2}\partial\text{aa}_{23}+(-\text{aa}_{03}+\text{aa}_{13})\partial\text{a}_{3}\partial\text{aa}_{23}+3\text{ax}_{21}\partial\text{x}_{0}\partial\text{aa}_{23}+(2\text{ax}_{22}-\text{ax}_{30})\partial\text{x}_{1}\partial\text{aa}_{23}+(\text{ax}_{23}-2\text{ax}_{31})\partial\text{x}_{2}\partial\text{aa}_{23}-3\text{ax}_{32}\partial\text{x}_{3}\partial\text{aa}_{23}+3\text{ay}_{31}\partial\text{y}_{0}\partial\text{aa}_{23}+(-\text{ay}_{20}+2\text{ay}_{32})\partial\text{y}_{1}\partial\text{aa}_{23}+(-2\text{ay}_{21}+\text{ay}_{33})\partial\text{y}_{2}\partial\text{aa}_{23}-3\text{ay}_{22}\partial\text{y}_{3}\partial\text{aa}_{23}-3\text{ax}_{00}\partial\text{a}_{1}\partial\text{ax}_{00}-\text{ax}_{20}\partial\text{a}_{2}\partial\text{ax}_{00}+(-3\text{ax}_{01}+\text{ax}_{30})\partial\text{a}_{3}\partial\text{ax}_{00}-\text{xx}_{01}\partial\text{x}_{1}\partial\text{ax}_{00}-2\text{xx}_{02}\partial\text{x}_{2}\partial\text{ax}_{00}-3\text{xx}_{03}\partial\text{x}_{3}\partial\text{ax}_{00}+\text{xy}_{01}\partial\text{y}_{1}\partial\text{ax}_{00}+2\text{xy}_{02}\partial\text{y}_{2}\partial\text{ax}_{00}+3\text{xy}_{03}\partial\text{y}_{3}\partial\text{ax}_{00}-\text{ax}_{01}\partial\text{a}_{0}\partial\text{ax}_{01}-2\text{ax}_{01}\partial\text{a}_{1}\partial\text{ax}_{01}+(-\text{ax}_{00}-\text{ax}_{21})\partial\text{a}_{2}\partial\text{ax}_{01}+(-2\text{ax}_{02}+\text{ax}_{31})\partial\text{a}_{3}\partial\text{ax}_{01}-2\text{xx}_{12}\partial\text{x}_{2}\partial\text{ax}_{01}-3\text{xx}_{13}\partial\text{x}_{3}\partial\text{ax}_{01}+\text{xy}_{11}\partial\text{y}_{1}\partial\text{ax}_{01}+2\text{xy}_{12}\partial\text{y}_{2}\partial\text{ax}_{01}+3\text{xy}_{13}\partial\text{y}_{3}\partial\text{ax}_{01}-2\text{ax}_{02}\partial\text{a}_{0}\partial\text{ax}_{02}-\text{ax}_{02}\partial\text{a}_{1}\partial\text{ax}_{02}+(-2\text{ax}_{01}-\text{ax}_{22})\partial\text{a}_{2}\partial\text{ax}_{02}+(-\text{ax}_{03}+\text{ax}_{32})\partial\text{a}_{3}\partial\text{ax}_{02}+\text{xx}_{12}\partial\text{x}_{1}\partial\text{ax}_{02}-3\text{xx}_{23}\partial\text{x}_{3}\partial\text{ax}_{02}+\text{xy}_{21}\partial\text{y}_{1}\partial\text{ax}_{02}+2\text{xy}_{22}\partial\text{y}_{2}\partial\text{ax}_{02}+3\text{xy}_{23}\partial\text{y}_{3}\partial\text{ax}_{02}-3\text{ax}_{03}\partial\text{a}_{0}\partial\text{ax}_{03}+(-3\text{ax}_{02}-\text{ax}_{23})\partial\text{a}_{2}\partial\text{ax}_{03}+\text{ax}_{33}\partial\text{a}_{3}\partial\text{ax}_{03}+\text{xx}_{13}\partial\text{x}_{1}\partial\text{ax}_{03}+2\text{xx}_{23}\partial\text{x}_{2}\partial\text{ax}_{03}+\text{xy}_{31}\partial\text{y}_{1}\partial\text{ax}_{03}+2\text{xy}_{32}\partial\text{y}_{2}\partial\text{ax}_{03}+3\text{xy}_{33}\partial\text{y}_{3}\partial\text{ax}_{03}-3\text{ax}_{10}\partial\text{a}_{1}\partial\text{ax}_{10}+\text{ax}_{20}\partial\text{a}_{2}\partial\text{ax}_{10}+(-3\text{ax}_{11}-\text{ax}_{30})\partial\text{a}_{3}\partial\text{ax}_{10}-2\text{xx}_{01}\partial\text{x}_{1}\partial\text{ax}_{10}-\text{xx}_{02}\partial\text{x}_{2}\partial\text{ax}_{10}+3\text{xy}_{00}\partial\text{y}_{0}\partial\text{ax}_{10}+2\text{xy}_{01}\partial\text{y}_{1}\partial\text{ax}_{10}+\text{xy}_{02}\partial\text{y}_{2}\partial\text{ax}_{10}-\text{ax}_{11}\partial\text{a}_{0}\partial\text{ax}_{11}-2\text{ax}_{11}\partial\text{a}_{1}\partial\text{ax}_{11}+(-\text{ax}_{10}+\text{ax}_{21})\partial\text{a}_{2}\partial\text{ax}_{11}+(-2\text{ax}_{12}-\text{ax}_{31})\partial\text{a}_{3}\partial\text{ax}_{11}+3\text{xx}_{01}\partial\text{x}_{0}\partial\text{ax}_{11}-\text{xx}_{12}\partial\text{x}_{2}\partial\text{ax}_{11}+3\text{xy}_{10}\partial\text{y}_{0}\partial\text{ax}_{11}+2\text{xy}_{11}\partial\text{y}_{1}\partial\text{ax}_{11}+\text{xy}_{12}\partial\text{y}_{2}\partial\text{ax}_{11}-2\text{ax}_{12}\partial\text{a}_{0}\partial\text{ax}_{12}-\text{ax}_{12}\partial\text{a}_{1}\partial\text{ax}_{12}+(-2\text{ax}_{11}+\text{ax}_{22})\partial\text{a}_{2}\partial\text{ax}_{12}+(-\text{ax}_{13}-\text{ax}_{32})\partial\text{a}_{3}\partial\text{ax}_{12}+3\text{xx}_{02}\partial\text{x}_{0}\partial\text{ax}_{12}+2\text{xx}_{12}\partial\text{x}_{1}\partial\text{ax}_{12}+3\text{xy}_{20}\partial\text{y}_{0}\partial\text{ax}_{12}+2\text{xy}_{21}\partial\text{y}_{1}\partial\text{ax}_{12}+\text{xy}_{22}\partial\text{y}_{2}\partial\text{ax}_{12}-3\text{ax}_{13}\partial\text{a}_{0}\partial\text{ax}_{13}+(-3\text{ax}_{12}+\text{ax}_{23})\partial\text{a}_{2}\partial\text{ax}_{13}-\text{ax}_{33}\partial\text{a}_{3}\partial\text{ax}_{13}+3\text{xx}_{03}\partial\text{x}_{0}\partial\text{ax}_{13}+2\text{xx}_{13}\partial\text{x}_{1}\partial\text{ax}_{13}+\text{xx}_{23}\partial\text{x}_{2}\partial\text{ax}_{13}+3\text{xy}_{30}\partial\text{y}_{0}\partial\text{ax}_{13}+2\text{xy}_{31}\partial\text{y}_{1}\partial\text{ax}_{13}+\text{xy}_{32}\partial\text{y}_{2}\partial\text{ax}_{13}+\text{ax}_{20}\partial\text{a}_{0}\partial\text{ax}_{20}-4\text{ax}_{20}\partial\text{a}_{1}\partial\text{ax}_{20}+(-\text{ax}_{00}+\text{ax}_{10}-3\text{ax}_{21})\partial\text{a}_{3}\partial\text{ax}_{20}-2\text{xx}_{01}\partial\text{x}_{2}\partial\text{ax}_{20}-3\text{xx}_{02}\partial\text{x}_{3}\partial\text{ax}_{20}+3\text{xy}_{01}\partial\text{y}_{0}\partial\text{ax}_{20}+2\text{xy}_{02}\partial\text{y}_{1}\partial\text{ax}_{20}+\text{xy}_{03}\partial\text{y}_{2}\partial\text{ax}_{20}-3\text{ax}_{21}\partial\text{a}_{1}\partial\text{ax}_{21}-\text{ax}_{20}\partial\text{a}_{2}\partial\text{ax}_{21}+(-\text{ax}_{01}+\text{ax}_{11}-2\text{ax}_{22})\partial\text{a}_{3}\partial\text{ax}_{21}+\text{xx}_{01}\partial\text{x}_{1}\partial\text{ax}_{21}-3\text{xx}_{12}\partial\text{x}_{3}\partial\text{ax}_{21}+3\text{xy}_{11}\partial\text{y}_{0}\partial\text{ax}_{21}+2\text{xy}_{12}\partial\text{y}_{1}\partial\text{ax}_{21}+\text{xy}_{13}\partial\text{y}_{2}\partial\text{ax}_{21}-\text{ax}_{22}\partial\text{a}_{0}\partial\text{ax}_{22}-2\text{ax}_{22}\partial\text{a}_{1}\partial\text{ax}_{22}-2\text{ax}_{21}\partial\text{a}_{2}\partial\text{ax}_{22}+(-\text{ax}_{02}+\text{ax}_{12}-\text{ax}_{23})\partial\text{a}_{3}\partial\text{ax}_{22}+\text{xx}_{02}\partial\text{x}_{1}\partial\text{ax}_{22}+2\text{xx}_{12}\partial\text{x}_{2}\partial\text{ax}_{22}+3\text{xy}_{21}\partial\text{y}_{0}\partial\text{ax}_{22}+2\text{xy}_{22}\partial\text{y}_{1}\partial\text{ax}_{22}+\text{xy}_{23}\partial\text{y}_{2}\partial\text{ax}_{22}-2\text{ax}_{23}\partial\text{a}_{0}\partial\text{ax}_{23}-\text{ax}_{23}\partial\text{a}_{1}\partial\text{ax}_{23}-3\text{ax}_{22}\partial\text{a}_{2}\partial\text{ax}_{23}+(-\text{ax}_{03}+\text{ax}_{13})\partial\text{a}_{3}\partial\text{ax}_{23}+\text{xx}_{03}\partial\text{x}_{1}\partial\text{ax}_{23}+2\text{xx}_{13}\partial\text{x}_{2}\partial\text{ax}_{23}+3\text{xx}_{23}\partial\text{x}_{3}\partial\text{ax}_{23}+3\text{xy}_{31}\partial\text{y}_{0}\partial\text{ax}_{23}+2\text{xy}_{32}\partial\text{y}_{1}\partial\text{ax}_{23}+\text{xy}_{33}\partial\text{y}_{2}\partial\text{ax}_{23}-\text{ax}_{30}\partial\text{a}_{0}\partial\text{ax}_{30}-2\text{ax}_{30}\partial\text{a}_{1}\partial\text{ax}_{30}+(\text{ax}_{00}-\text{ax}_{10})\partial\text{a}_{2}\partial\text{ax}_{30}-3\text{ax}_{31}\partial\text{a}_{3}\partial\text{ax}_{30}-3\text{xx}_{01}\partial\text{x}_{0}\partial\text{ax}_{30}-2\text{xx}_{02}\partial\text{x}_{1}\partial\text{ax}_{30}-\text{xx}_{03}\partial\text{x}_{2}\partial\text{ax}_{30}+\text{xy}_{00}\partial\text{y}_{1}\partial\text{ax}_{30}+2\text{xy}_{01}\partial\text{y}_{2}\partial\text{ax}_{30}+3\text{xy}_{02}\partial\text{y}_{3}\partial\text{ax}_{30}-2\text{ax}_{31}\partial\text{a}_{0}\partial\text{ax}_{31}-\text{ax}_{31}\partial\text{a}_{1}\partial\text{ax}_{31}+(\text{ax}_{01}-\text{ax}_{11}-\text{ax}_{30})\partial\text{a}_{2}\partial\text{ax}_{31}-2\text{ax}_{32}\partial\text{a}_{3}\partial\text{ax}_{31}-2\text{xx}_{12}\partial\text{x}_{1}\partial\text{ax}_{31}-\text{xx}_{13}\partial\text{x}_{2}\partial\text{ax}_{31}+\text{xy}_{10}\partial\text{y}_{1}\partial\text{ax}_{31}+2\text{xy}_{11}\partial\text{y}_{2}\partial\text{ax}_{31}+3\text{xy}_{12}\partial\text{y}_{3}\partial\text{ax}_{31}-3\text{ax}_{32}\partial\text{a}_{0}\partial\text{ax}_{32}+(\text{ax}_{02}-\text{ax}_{12}-2\text{ax}_{31})\partial\text{a}_{2}\partial\text{ax}_{32}-\text{ax}_{33}\partial\text{a}_{3}\partial\text{ax}_{32}+3\text{xx}_{12}\partial\text{x}_{0}\partial\text{ax}_{32}-\text{xx}_{23}\partial\text{x}_{2}\partial\text{ax}_{32}+\text{xy}_{20}\partial\text{y}_{1}\partial\text{ax}_{32}+2\text{xy}_{21}\partial\text{y}_{2}\partial\text{ax}_{32}+3\text{xy}_{22}\partial\text{y}_{3}\partial\text{ax}_{32}-4\text{ax}_{33}\partial\text{a}_{0}\partial\text{ax}_{33}+\text{ax}_{33}\partial\text{a}_{1}\partial\text{ax}_{33}+(\text{ax}_{03}-\text{ax}_{13}-3\text{ax}_{32})\partial\text{a}_{2}\partial\text{ax}_{33}+3\text{xx}_{13}\partial\text{x}_{0}\partial\text{ax}_{33}+2\text{xx}_{23}\partial\text{x}_{1}\partial\text{ax}_{33}+\text{xy}_{30}\partial\text{y}_{1}\partial\text{ax}_{33}+2\text{xy}_{31}\partial\text{y}_{2}\partial\text{ax}_{33}+3\text{xy}_{32}\partial\text{y}_{3}\partial\text{ax}_{33}+3\text{ay}_{00}\partial\text{a}_{1}\partial\text{ay}_{00}+(3\text{ay}_{01}-\text{ay}_{20})\partial\text{a}_{2}\partial\text{ay}_{00}+\text{ay}_{30}\partial\text{a}_{3}\partial\text{ay}_{00}+\text{xy}_{10}\partial\text{x}_{1}\partial\text{ay}_{00}+2\text{xy}_{20}\partial\text{x}_{2}\partial\text{ay}_{00}+3\text{xy}_{30}\partial\text{x}_{3}\partial\text{ay}_{00}+\text{yy}_{01}\partial\text{y}_{1}\partial\text{ay}_{00}+2\text{yy}_{02}\partial\text{y}_{2}\partial\text{ay}_{00}+3\text{yy}_{03}\partial\text{y}_{3}\partial\text{ay}_{00}+\text{ay}_{01}\partial\text{a}_{0}\partial\text{ay}_{01}+2\text{ay}_{01}\partial\text{a}_{1}\partial\text{ay}_{01}+(2\text{ay}_{02}-\text{ay}_{21})\partial\text{a}_{2}\partial\text{ay}_{01}+(\text{ay}_{00}+\text{ay}_{31})\partial\text{a}_{3}\partial\text{ay}_{01}+\text{xy}_{11}\partial\text{x}_{1}\partial\text{ay}_{01}+2\text{xy}_{21}\partial\text{x}_{2}\partial\text{ay}_{01}+3\text{xy}_{31}\partial\text{x}_{3}\partial\text{ay}_{01}+2\text{yy}_{12}\partial\text{y}_{2}\partial\text{ay}_{01}+3\text{yy}_{13}\partial\text{y}_{3}\partial\text{ay}_{01}+2\text{ay}_{02}\partial\text{a}_{0}\partial\text{ay}_{02}+\text{ay}_{02}\partial\text{a}_{1}\partial\text{ay}_{02}+(\text{ay}_{03}-\text{ay}_{22})\partial\text{a}_{2}\partial\text{ay}_{02}+(2\text{ay}_{01}+\text{ay}_{32})\partial\text{a}_{3}\partial\text{ay}_{02}+\text{xy}_{12}\partial\text{x}_{1}\partial\text{ay}_{02}+2\text{xy}_{22}\partial\text{x}_{2}\partial\text{ay}_{02}+3\text{xy}_{32}\partial\text{x}_{3}\partial\text{ay}_{02}-\text{yy}_{12}\partial\text{y}_{1}\partial\text{ay}_{02}+3\text{yy}_{23}\partial\text{y}_{3}\partial\text{ay}_{02}+3\text{ay}_{03}\partial\text{a}_{0}\partial\text{ay}_{03}-\text{ay}_{23}\partial\text{a}_{2}\partial\text{ay}_{03}+(3\text{ay}_{02}+\text{ay}_{33})\partial\text{a}_{3}\partial\text{ay}_{03}+\text{xy}_{13}\partial\text{x}_{1}\partial\text{ay}_{03}+2\text{xy}_{23}\partial\text{x}_{2}\partial\text{ay}_{03}+3\text{xy}_{33}\partial\text{x}_{3}\partial\text{ay}_{03}-\text{yy}_{13}\partial\text{y}_{1}\partial\text{ay}_{03}-2\text{yy}_{23}\partial\text{y}_{2}\partial\text{ay}_{03}+3\text{ay}_{10}\partial\text{a}_{1}\partial\text{ay}_{10}+(3\text{ay}_{11}+\text{ay}_{20})\partial\text{a}_{2}\partial\text{ay}_{10}-\text{ay}_{30}\partial\text{a}_{3}\partial\text{ay}_{10}+3\text{xy}_{00}\partial\text{x}_{0}\partial\text{ay}_{10}+2\text{xy}_{10}\partial\text{x}_{1}\partial\text{ay}_{10}+\text{xy}_{20}\partial\text{x}_{2}\partial\text{ay}_{10}+2\text{yy}_{01}\partial\text{y}_{1}\partial\text{ay}_{10}+\text{yy}_{02}\partial\text{y}_{2}\partial\text{ay}_{10}+\text{ay}_{11}\partial\text{a}_{0}\partial\text{ay}_{11}+2\text{ay}_{11}\partial\text{a}_{1}\partial\text{ay}_{11}+(2\text{ay}_{12}+\text{ay}_{21})\partial\text{a}_{2}\partial\text{ay}_{11}+(\text{ay}_{10}-\text{ay}_{31})\partial\text{a}_{3}\partial\text{ay}_{11}+3\text{xy}_{01}\partial\text{x}_{0}\partial\text{ay}_{11}+2\text{xy}_{11}\partial\text{x}_{1}\partial\text{ay}_{11}+\text{xy}_{21}\partial\text{x}_{2}\partial\text{ay}_{11}-3\text{yy}_{01}\partial\text{y}_{0}\partial\text{ay}_{11}+\text{yy}_{12}\partial\text{y}_{2}\partial\text{ay}_{11}+2\text{ay}_{12}\partial\text{a}_{0}\partial\text{ay}_{12}+\text{ay}_{12}\partial\text{a}_{1}\partial\text{ay}_{12}+(\text{ay}_{13}+\text{ay}_{22})\partial\text{a}_{2}\partial\text{ay}_{12}+(2\text{ay}_{11}-\text{ay}_{32})\partial\text{a}_{3}\partial\text{ay}_{12}+3\text{xy}_{02}\partial\text{x}_{0}\partial\text{ay}_{12}+2\text{xy}_{12}\partial\text{x}_{1}\partial\text{ay}_{12}+\text{xy}_{22}\partial\text{x}_{2}\partial\text{ay}_{12}-3\text{yy}_{02}\partial\text{y}_{0}\partial\text{ay}_{12}-2\text{yy}_{12}\partial\text{y}_{1}\partial\text{ay}_{12}+3\text{ay}_{13}\partial\text{a}_{0}\partial\text{ay}_{13}+\text{ay}_{23}\partial\text{a}_{2}\partial\text{ay}_{13}+(3\text{ay}_{12}-\text{ay}_{33})\partial\text{a}_{3}\partial\text{ay}_{13}+3\text{xy}_{03}\partial\text{x}_{0}\partial\text{ay}_{13}+2\text{xy}_{13}\partial\text{x}_{1}\partial\text{ay}_{13}+\text{xy}_{23}\partial\text{x}_{2}\partial\text{ay}_{13}-3\text{yy}_{03}\partial\text{y}_{0}\partial\text{ay}_{13}-2\text{yy}_{13}\partial\text{y}_{1}\partial\text{ay}_{13}-\text{yy}_{23}\partial\text{y}_{2}\partial\text{ay}_{13}+\text{ay}_{20}\partial\text{a}_{0}\partial\text{ay}_{20}+2\text{ay}_{20}\partial\text{a}_{1}\partial\text{ay}_{20}+3\text{ay}_{21}\partial\text{a}_{2}\partial\text{ay}_{20}+(-\text{ay}_{00}+\text{ay}_{10})\partial\text{a}_{3}\partial\text{ay}_{20}+\text{xy}_{00}\partial\text{x}_{1}\partial\text{ay}_{20}+2\text{xy}_{10}\partial\text{x}_{2}\partial\text{ay}_{20}+3\text{xy}_{20}\partial\text{x}_{3}\partial\text{ay}_{20}+3\text{yy}_{01}\partial\text{y}_{0}\partial\text{ay}_{20}+2\text{yy}_{02}\partial\text{y}_{1}\partial\text{ay}_{20}+\text{yy}_{03}\partial\text{y}_{2}\partial\text{ay}_{20}+2\text{ay}_{21}\partial\text{a}_{0}\partial\text{ay}_{21}+\text{ay}_{21}\partial\text{a}_{1}\partial\text{ay}_{21}+2\text{ay}_{22}\partial\text{a}_{2}\partial\text{ay}_{21}+(-\text{ay}_{01}+\text{ay}_{11}+\text{ay}_{20})\partial\text{a}_{3}\partial\text{ay}_{21}+\text{xy}_{01}\partial\text{x}_{1}\partial\text{ay}_{21}+2\text{xy}_{11}\partial\text{x}_{2}\partial\text{ay}_{21}+3\text{xy}_{21}\partial\text{x}_{3}\partial\text{ay}_{21}+2\text{yy}_{12}\partial\text{y}_{1}\partial\text{ay}_{21}+\text{yy}_{13}\partial\text{y}_{2}\partial\text{ay}_{21}+3\text{ay}_{22}\partial\text{a}_{0}\partial\text{ay}_{22}+\text{ay}_{23}\partial\text{a}_{2}\partial\text{ay}_{22}+(-\text{ay}_{02}+\text{ay}_{12}+2\text{ay}_{21})\partial\text{a}_{3}\partial\text{ay}_{22}+\text{xy}_{02}\partial\text{x}_{1}\partial\text{ay}_{22}+2\text{xy}_{12}\partial\text{x}_{2}\partial\text{ay}_{22}+3\text{xy}_{22}\partial\text{x}_{3}\partial\text{ay}_{22}-3\text{yy}_{12}\partial\text{y}_{0}\partial\text{ay}_{22}+\text{yy}_{23}\partial\text{y}_{2}\partial\text{ay}_{22}+4\text{ay}_{23}\partial\text{a}_{0}\partial\text{ay}_{23}-\text{ay}_{23}\partial\text{a}_{1}\partial\text{ay}_{23}+(-\text{ay}_{03}+\text{ay}_{13}+3\text{ay}_{22})\partial\text{a}_{3}\partial\text{ay}_{23}+\text{xy}_{03}\partial\text{x}_{1}\partial\text{ay}_{23}+2\text{xy}_{13}\partial\text{x}_{2}\partial\text{ay}_{23}+3\text{xy}_{23}\partial\text{x}_{3}\partial\text{ay}_{23}-3\text{yy}_{13}\partial\text{y}_{0}\partial\text{ay}_{23}-2\text{yy}_{23}\partial\text{y}_{1}\partial\text{ay}_{23}-\text{ay}_{30}\partial\text{a}_{0}\partial\text{ay}_{30}+4\text{ay}_{30}\partial\text{a}_{1}\partial\text{ay}_{30}+(\text{ay}_{00}-\text{ay}_{10}+3\text{ay}_{31})\partial\text{a}_{2}\partial\text{ay}_{30}+3\text{xy}_{10}\partial\text{x}_{0}\partial\text{ay}_{30}+2\text{xy}_{20}\partial\text{x}_{1}\partial\text{ay}_{30}+\text{xy}_{30}\partial\text{x}_{2}\partial\text{ay}_{30}+2\text{yy}_{01}\partial\text{y}_{2}\partial\text{ay}_{30}+3\text{yy}_{02}\partial\text{y}_{3}\partial\text{ay}_{30}+3\text{ay}_{31}\partial\text{a}_{1}\partial\text{ay}_{31}+(\text{ay}_{01}-\text{ay}_{11}+2\text{ay}_{32})\partial\text{a}_{2}\partial\text{ay}_{31}+\text{ay}_{30}\partial\text{a}_{3}\partial\text{ay}_{31}+3\text{xy}_{11}\partial\text{x}_{0}\partial\text{ay}_{31}+2\text{xy}_{21}\partial\text{x}_{1}\partial\text{ay}_{31}+\text{xy}_{31}\partial\text{x}_{2}\partial\text{ay}_{31}-\text{yy}_{01}\partial\text{y}_{1}\partial\text{ay}_{31}+3\text{yy}_{12}\partial\text{y}_{3}\partial\text{ay}_{31}+\text{ay}_{32}\partial\text{a}_{0}\partial\text{ay}_{32}+2\text{ay}_{32}\partial\text{a}_{1}\partial\text{ay}_{32}+(\text{ay}_{02}-\text{ay}_{12}+\text{ay}_{33})\partial\text{a}_{2}\partial\text{ay}_{32}+2\text{ay}_{31}\partial\text{a}_{3}\partial\text{ay}_{32}+3\text{xy}_{12}\partial\text{x}_{0}\partial\text{ay}_{32}+2\text{xy}_{22}\partial\text{x}_{1}\partial\text{ay}_{32}+\text{xy}_{32}\partial\text{x}_{2}\partial\text{ay}_{32}-\text{yy}_{02}\partial\text{y}_{1}\partial\text{ay}_{32}-2\text{yy}_{12}\partial\text{y}_{2}\partial\text{ay}_{32}+2\text{ay}_{33}\partial\text{a}_{0}\partial\text{ay}_{33}+\text{ay}_{33}\partial\text{a}_{1}\partial\text{ay}_{33}+(\text{ay}_{03}-\text{ay}_{13})\partial\text{a}_{2}\partial\text{ay}_{33}+3\text{ay}_{32}\partial\text{a}_{3}\partial\text{ay}_{33}+3\text{xy}_{13}\partial\text{x}_{0}\partial\text{ay}_{33}+2\text{xy}_{23}\partial\text{x}_{1}\partial\text{ay}_{33}+\text{xy}_{33}\partial\text{x}_{2}\partial\text{ay}_{33}-\text{yy}_{03}\partial\text{y}_{1}\partial\text{ay}_{33}-2\text{yy}_{13}\partial\text{y}_{2}\partial\text{ay}_{33}-3\text{yy}_{23}\partial\text{y}_{3}\partial\text{ay}_{33}-\text{xx}_{01}\partial\text{a}_{0}\partial\text{xx}_{01}-5\text{xx}_{01}\partial\text{a}_{1}\partial\text{xx}_{01}-2\text{xx}_{02}\partial\text{a}_{3}\partial\text{xx}_{01}-2\text{xx}_{02}\partial\text{a}_{0}\partial\text{xx}_{02}-4\text{xx}_{02}\partial\text{a}_{1}\partial\text{xx}_{02}-2\text{xx}_{01}\partial\text{a}_{2}\partial\text{xx}_{02}+(-\text{xx}_{03}-3\text{xx}_{12})\partial\text{a}_{3}\partial\text{xx}_{02}-3\text{xx}_{03}\partial\text{a}_{0}\partial\text{xx}_{03}-3\text{xx}_{03}\partial\text{a}_{1}\partial\text{xx}_{03}-3\text{xx}_{02}\partial\text{a}_{2}\partial\text{xx}_{03}-3\text{xx}_{13}\partial\text{a}_{3}\partial\text{xx}_{03}-3\text{xx}_{12}\partial\text{a}_{0}\partial\text{xx}_{12}-3\text{xx}_{12}\partial\text{a}_{1}\partial\text{xx}_{12}-\text{xx}_{02}\partial\text{a}_{2}\partial\text{xx}_{12}-\text{xx}_{13}\partial\text{a}_{3}\partial\text{xx}_{12}-4\text{xx}_{13}\partial\text{a}_{0}\partial\text{xx}_{13}-2\text{xx}_{13}\partial\text{a}_{1}\partial\text{xx}_{13}+(-\text{xx}_{03}-3\text{xx}_{12})\partial\text{a}_{2}\partial\text{xx}_{13}-2\text{xx}_{23}\partial\text{a}_{3}\partial\text{xx}_{13}-5\text{xx}_{23}\partial\text{a}_{0}\partial\text{xx}_{23}-\text{xx}_{23}\partial\text{a}_{1}\partial\text{xx}_{23}-2\text{xx}_{13}\partial\text{a}_{2}\partial\text{xx}_{23}+\text{yy}_{01}\partial\text{a}_{0}\partial\text{yy}_{01}+5\text{yy}_{01}\partial\text{a}_{1}\partial\text{yy}_{01}+2\text{yy}_{02}\partial\text{a}_{2}\partial\text{yy}_{01}+2\text{yy}_{02}\partial\text{a}_{0}\partial\text{yy}_{02}+4\text{yy}_{02}\partial\text{a}_{1}\partial\text{yy}_{02}+(\text{yy}_{03}+3\text{yy}_{12})\partial\text{a}_{2}\partial\text{yy}_{02}+2\text{yy}_{01}\partial\text{a}_{3}\partial\text{yy}_{02}+3\text{yy}_{03}\partial\text{a}_{0}\partial\text{yy}_{03}+3\text{yy}_{03}\partial\text{a}_{1}\partial\text{yy}_{03}+3\text{yy}_{13}\partial\text{a}_{2}\partial\text{yy}_{03}+3\text{yy}_{02}\partial\text{a}_{3}\partial\text{yy}_{03}+3\text{yy}_{12}\partial\text{a}_{0}\partial\text{yy}_{12}+3\text{yy}_{12}\partial\text{a}_{1}\partial\text{yy}_{12}+\text{yy}_{13}\partial\text{a}_{2}\partial\text{yy}_{12}+\text{yy}_{02}\partial\text{a}_{3}\partial\text{yy}_{12}+4\text{yy}_{13}\partial\text{a}_{0}\partial\text{yy}_{13}+2\text{yy}_{13}\partial\text{a}_{1}\partial\text{yy}_{13}+2\text{yy}_{23}\partial\text{a}_{2}\partial\text{yy}_{13}+(\text{yy}_{03}+3\text{yy}_{12})\partial\text{a}_{3}\partial\text{yy}_{13}+5\text{yy}_{23}\partial\text{a}_{0}\partial\text{yy}_{23}+\text{yy}_{23}\partial\text{a}_{1}\partial\text{yy}_{23}+2\text{yy}_{13}\partial\text{a}_{3}\partial\text{yy}_{23}+3\text{xy}_{01}\partial\text{a}_{2}\partial\text{xy}_{00}-3\text{xy}_{10}\partial\text{a}_{3}\partial\text{xy}_{00}+\text{xy}_{01}\partial\text{a}_{0}\partial\text{xy}_{01}-\text{xy}_{01}\partial\text{a}_{1}\partial\text{xy}_{01}+2\text{xy}_{02}\partial\text{a}_{2}\partial\text{xy}_{01}+(\text{xy}_{00}-3\text{xy}_{11})\partial\text{a}_{3}\partial\text{xy}_{01}+2\text{xy}_{02}\partial\text{a}_{0}\partial\text{xy}_{02}-2\text{xy}_{02}\partial\text{a}_{1}\partial\text{xy}_{02}+\text{xy}_{03}\partial\text{a}_{2}\partial\text{xy}_{02}+(2\text{xy}_{01}-3\text{xy}_{12})\partial\text{a}_{3}\partial\text{xy}_{02}+3\text{xy}_{03}\partial\text{a}_{0}\partial\text{xy}_{03}-3\text{xy}_{03}\partial\text{a}_{1}\partial\text{xy}_{03}+(3\text{xy}_{02}-3\text{xy}_{13})\partial\text{a}_{3}\partial\text{xy}_{03}-\text{xy}_{10}\partial\text{a}_{0}\partial\text{xy}_{10}+\text{xy}_{10}\partial\text{a}_{1}\partial\text{xy}_{10}+(-\text{xy}_{00}+3\text{xy}_{11})\partial\text{a}_{2}\partial\text{xy}_{10}-2\text{xy}_{20}\partial\text{a}_{3}\partial\text{xy}_{10}+(-\text{xy}_{01}+2\text{xy}_{12})\partial\text{a}_{2}\partial\text{xy}_{11}+(\text{xy}_{10}-2\text{xy}_{21})\partial\text{a}_{3}\partial\text{xy}_{11}+\text{xy}_{12}\partial\text{a}_{0}\partial\text{xy}_{12}-\text{xy}_{12}\partial\text{a}_{1}\partial\text{xy}_{12}+(-\text{xy}_{02}+\text{xy}_{13})\partial\text{a}_{2}\partial\text{xy}_{12}+(2\text{xy}_{11}-2\text{xy}_{22})\partial\text{a}_{3}\partial\text{xy}_{12}+2\text{xy}_{13}\partial\text{a}_{0}\partial\text{xy}_{13}-2\text{xy}_{13}\partial\text{a}_{1}\partial\text{xy}_{13}-\text{xy}_{03}\partial\text{a}_{2}\partial\text{xy}_{13}+(3\text{xy}_{12}-2\text{xy}_{23})\partial\text{a}_{3}\partial\text{xy}_{13}-2\text{xy}_{20}\partial\text{a}_{0}\partial\text{xy}_{20}+2\text{xy}_{20}\partial\text{a}_{1}\partial\text{xy}_{20}+(-2\text{xy}_{10}+3\text{xy}_{21})\partial\text{a}_{2}\partial\text{xy}_{20}-\text{xy}_{30}\partial\text{a}_{3}\partial\text{xy}_{20}-\text{xy}_{21}\partial\text{a}_{0}\partial\text{xy}_{21}+\text{xy}_{21}\partial\text{a}_{1}\partial\text{xy}_{21}+(-2\text{xy}_{11}+2\text{xy}_{22})\partial\text{a}_{2}\partial\text{xy}_{21}+(\text{xy}_{20}-\text{xy}_{31})\partial\text{a}_{3}\partial\text{xy}_{21}+(-2\text{xy}_{12}+\text{xy}_{23})\partial\text{a}_{2}\partial\text{xy}_{22}+(2\text{xy}_{21}-\text{xy}_{32})\partial\text{a}_{3}\partial\text{xy}_{22}+\text{xy}_{23}\partial\text{a}_{0}\partial\text{xy}_{23}-\text{xy}_{23}\partial\text{a}_{1}\partial\text{xy}_{23}-2\text{xy}_{13}\partial\text{a}_{2}\partial\text{xy}_{23}+(3\text{xy}_{22}-\text{xy}_{33})\partial\text{a}_{3}\partial\text{xy}_{23}-3\text{xy}_{30}\partial\text{a}_{0}\partial\text{xy}_{30}+3\text{xy}_{30}\partial\text{a}_{1}\partial\text{xy}_{30}+(-3\text{xy}_{20}+3\text{xy}_{31})\partial\text{a}_{2}\partial\text{xy}_{30}-2\text{xy}_{31}\partial\text{a}_{0}\partial\text{xy}_{31}+2\text{xy}_{31}\partial\text{a}_{1}\partial\text{xy}_{31}+(-3\text{xy}_{21}+2\text{xy}_{32})\partial\text{a}_{2}\partial\text{xy}_{31}+\text{xy}_{30}\partial\text{a}_{3}\partial\text{xy}_{31}-\text{xy}_{32}\partial\text{a}_{0}\partial\text{xy}_{32}+\text{xy}_{32}\partial\text{a}_{1}\partial\text{xy}_{32}+(-3\text{xy}_{22}+\text{xy}_{33})\partial\text{a}_{2}\partial\text{xy}_{32}+2\text{xy}_{31}\partial\text{a}_{3}\partial\text{xy}_{32}-3\text{xy}_{23}\partial\text{a}_{2}\partial\text{xy}_{33}+3\text{xy}_{32}\partial\text{a}_{3}\partial\text{xy}_{33}$ 
\end{flushleft}
\end{itemize}
Initially, $\alpha^1_c$ is written with general coefficients (i.e. a different $L_{ij}$ for each element of $\g\wedge\g$).  Displaying $[\beta^1+\alpha^1,\beta^1+\alpha^1]$ and setting the expression equal to zero determines the coefficients $L_{ij}$:
\begin{itemize}
\item[$\text{i17}:$]\begin{flushleft}$\alpha_c^1=L_{10}\text{xx}_{01} \partial\text{x}_{0} \partial\text{x}_{1}+L_{10}\text{xx}_{02} \partial\text{x}_{0} \partial\text{x}_{2}+L_{10}\text{xx}_{12}
      \partial\text{x}_{1} \partial\text{x}_{2}+L_{10}\text{xx}_{03} \partial\text{x}_{0} \partial\text{x}_{3}+L_{10}\text{xx}_{13} \partial\text{x}_{1}
      \partial\text{x}_{3}+L_{10}\text{xx}_{23} \partial\text{x}_{2} \partial\text{x}_{3}+{L}_{33} \text{xy}_{00} \partial\text{x}_{0}
      \partial\text{y}_{0}+{L}_{33} \text{xy}_{10} \partial\text{x}_{1} \partial\text{y}_{0}+{L}_{33} \text{xy}_{20} \partial\text{x}_{2}
      \partial\text{y}_{0}+{L}_{33} \text{xy}_{30} \partial\text{x}_{3} \partial\text{y}_{0}+{L}_{33} \text{xy}_{01} \partial\text{x}_{0}
      \partial\text{y}_{1}+{L}_{33} \text{xy}_{11} \partial\text{x}_{1} \partial\text{y}_{1}+{L}_{33} \text{xy}_{21} \partial\text{x}_{2}
      \partial\text{y}_{1}+{L}_{33} \text{xy}_{31} \partial\text{x}_{3} \partial\text{y}_{1}+\text{yy}_{01} \partial\text{y}_{0}
      \partial\text{y}_{1}+{L}_{33} \text{xy}_{02} \partial\text{x}_{0} \partial\text{y}_{2}+{L}_{33} \text{xy}_{12} \partial\text{x}_{1}
      \partial\text{y}_{2}+{L}_{33} \text{xy}_{22} \partial\text{x}_{2} \partial\text{y}_{2}+{L}_{33} \text{xy}_{32} \partial\text{x}_{3}
      \partial\text{y}_{2}+L_{14}\text{yy}_{02} \partial\text{y}_{0} \partial\text{y}_{2}+L_{14}\text{yy}_{12} \partial\text{y}_{1} \partial\text{y}_{2}+{L}_{33}
      \text{xy}_{03} \partial\text{x}_{0} \partial\text{y}_{3}+{L}_{33} \text{xy}_{13} \partial\text{x}_{1} \partial\text{y}_{3}+{L}_{33}
      \text{xy}_{23} \partial\text{x}_{2} \partial\text{y}_{3}+{L}_{33} \text{xy}_{33} \partial\text{x}_{3} \partial\text{y}_{3}+L_{14}\text{yy}_{03}
      \partial\text{y}_{0} \partial\text{y}_{3}+L_{14}\text{yy}_{13} \partial\text{y}_{1} \partial\text{y}_{3}+L_{14}\text{yy}_{23} \partial\text{y}_{2} \partial\text{y}_{3}$
      \end{flushleft}
\end{itemize}
One may verify that $\pi^1=\beta^1+\alpha^1$ is a Poisson bivector:
\begin{itemize}
\item[$\text{i18}:$] $\pi^1=\beta^1+\alpha^1$; \\
				$\text{Sbr}(\pi^1,\pi^1)==0$
\item[$\text{o18}:$]  true
\end{itemize}
Using the reductions in Chapter 7, we construct a general $\alpha^2$ coefficients $K_{ij},l_{k}$:
\begin{flushleft}
\begin{itemize}
\item[$\text{i19}:$] $\alpha^2 = l_0C\partial\text{x}_{0}\partial\text{y}_{0}+l_1C\partial\text{x}_{1}\partial\text{y}_{1}+l_2C\partial\text{x}_{2}\partial\text{y}_{2}+l_3C\partial\text{x}_{3}\partial\text{y}_{3}+(K_{1}A_1+K_{2}B_0)\partial\text{x}_{0}\partial\text{yy}_{01}+(K_{3}A_2+K_{4}B_1)\partial\text{x}_{0}\partial\text{yy}_{02}+K_{5}A_3\partial\text{x}_{0}\partial\text{yy}_{03}+K_{6}A_3\partial\text{x}_{0}\partial\text{yy}_{12}+K_{7}\overline{A}_0\partial\text{x}_{0}\partial\text{xy}_{00}+(K_{8}\overline{A}_1+K_{9}\overline{B}_0)\partial\text{x}_{0}\partial\text{xy}_{10}+K_{10}\overline{A}_0\partial\text{x}_{0}\partial\text{xy}_{11}+(K_{11}\overline{A}_2+K_{12}\overline{B}_1)\partial\text{x}_{0}\partial\text{xy}_{20}+(K_{13}\overline{A}_1+K_{14}\overline{B}_0)\partial\text{x}_{0}\partial\text{xy}_{21}+K_{15}\overline{A}_0\partial\text{x}_{0}\partial\text{xy}_{22}+K_{16}\overline{A}_3\partial\text{x}_{0}\partial\text{xy}_{30}+(K_{17}\overline{A}_2+K_{18}\overline{B}_1)\partial\text{x}_{0}\partial\text{xy}_{31}+(K_{19}\overline{A}_1+K_{20}\overline{B}_0)\partial\text{x}_{0}\partial\text{xy}_{32}+K_{21}\overline{A}_0\partial\text{x}_{0}\partial\text{xy}_{33}+K_{22}A_0\partial\text{x}_{1}\partial\text{yy}_{01}+(K_{23}A_1+K_{24}B_0)\partial\text{x}_{1}\partial\text{yy}_{02}+(K_{25}A_2+K_{26}B_1)\partial\text{x}_{1}\partial\text{yy}_{03}+(K_{27}A_2+K_{28}B_1)\partial\text{x}_{1}\partial\text{yy}_{12}+K_{29}A_3\partial\text{x}_{1}\partial\text{yy}_{13}+(K_{30}\overline{A}_1+K_{31}\overline{B}_0)\partial\text{x}_{1}\partial\text{xy}_{00}+K_{32}\overline{A}_0\partial\text{x}_{1}\partial\text{xy}_{01}+(K_{33}\overline{A}_2+K_{34}\overline{B}_1)\partial\text{x}_{1}\partial\text{xy}_{10}+(K_{35}\overline{A}_1+K_{36}\overline{B}_0)\partial\text{x}_{1}\partial\text{xy}_{11}+K_{37}\overline{A}_0\partial\text{x}_{1}\partial\text{xy}_{12}+K_{38}\overline{A}_3\partial\text{x}_{1}\partial\text{xy}_{20}+(K_{39}\overline{A}_2+K_{40}\overline{B}_1)\partial\text{x}_{1}\partial\text{xy}_{21}+(K_{41}\overline{A}_1+K_{42}\overline{B}_0)\partial\text{x}_{1}\partial\text{xy}_{22}+K_{43}\overline{A}_0\partial\text{x}_{1}\partial\text{xy}_{23}+K_{44}\overline{A}_3\partial\text{x}_{1}\partial\text{xy}_{31}+(K_{45}\overline{A}_2+K_{46}\overline{B}_1)\partial\text{x}_{1}\partial\text{xy}_{32}+(K_{47}\overline{A}_1+K_{48}\overline{B}_0)\partial\text{x}_{1}\partial\text{xy}_{33}+K_{49}A_0\partial\text{x}_{2}\partial\text{yy}_{02}+(K_{50}A_1+K_{51}B_0)\partial\text{x}_{2}\partial\text{yy}_{03}+(K_{52}A_1+K_{53}B_0)\partial\text{x}_{2}\partial\text{yy}_{12}+(K_{54}A_2+K_{55}B_1)\partial\text{x}_{2}\partial\text{yy}_{13}+K_{56}A_3\partial\text{x}_{2}\partial\text{yy}_{23}+(K_{57}\overline{A}_2+K_{58}\overline{B}_1)\partial\text{x}_{2}\partial\text{xy}_{00}+(K_{59}\overline{A}_1+K_{60}\overline{B}_0)\partial\text{x}_{2}\partial\text{xy}_{01}+K_{61}\overline{A}_0\partial\text{x}_{2}\partial\text{xy}_{02}+K_{62}\overline{A}_3\partial\text{x}_{2}\partial\text{xy}_{10}+(K_{63}\overline{A}_2+K_{64}\overline{B}_1)\partial\text{x}_{2}\partial\text{xy}_{11}+(K_{65}\overline{A}_1+K_{66}\overline{B}_0)\partial\text{x}_{2}\partial\text{xy}_{12}+K_{67}\overline{A}_0\partial\text{x}_{2}\partial\text{xy}_{13}+K_{68}\overline{A}_3\partial\text{x}_{2}\partial\text{xy}_{21}+(K_{69}\overline{A}_2+K_{70}\overline{B}_1)\partial\text{x}_{2}\partial\text{xy}_{22}+(K_{71}\overline{A}_1+K_{72}\overline{B}_0)\partial\text{x}_{2}\partial\text{xy}_{23}+K_{73}\overline{A}_3\partial\text{x}_{2}\partial\text{xy}_{32}+(K_{74}\overline{A}_2+K_{75}\overline{B}_1)\partial\text{x}_{2}\partial\text{xy}_{33}+K_{76}A_0\partial\text{x}_{3}\partial\text{yy}_{03}+K_{77}A_0\partial\text{x}_{3}\partial\text{yy}_{12}+(K_{78}A_1+K_{79}B_0)\partial\text{x}_{3}\partial\text{yy}_{13}+(K_{80}A_2+K_{81}B_1)\partial\text{x}_{3}\partial\text{yy}_{23}+K_{82}\overline{A}_3\partial\text{x}_{3}\partial\text{xy}_{00}+(K_{83}\overline{A}_2+K_{84}\overline{B}_1)\partial\text{x}_{3}\partial\text{xy}_{01}+(K_{85}\overline{A}_1+K_{86}\overline{B}_0)\partial\text{x}_{3}\partial\text{xy}_{02}+K_{87}\overline{A}_0\partial\text{x}_{3}\partial\text{xy}_{03}+K_{88}\overline{A}_3\partial\text{x}_{3}\partial\text{xy}_{11}+(K_{89}\overline{A}_2+K_{90}\overline{B}_1)\partial\text{x}_{3}\partial\text{xy}_{12}+(K_{91}\overline{A}_1+K_{92}\overline{B}_0)\partial\text{x}_{3}\partial\text{xy}_{13}+K_{93}\overline{A}_3\partial\text{x}_{3}\partial\text{xy}_{22}+(K_{94}\overline{A}_2+K_{95}\overline{B}_1)\partial\text{x}_{3}\partial\text{xy}_{23}+K_{96}\overline{A}_3\partial\text{x}_{3}\partial\text{xy}_{33}+(K_{97}\overline{A}_1+K_{98}\overline{B}_0)\partial\text{y}_{0}\partial\text{xx}_{01}+(K_{99}\overline{A}_2+K_{100}\overline{B}_1)\partial\text{y}_{0}\partial\text{xx}_{02}+K_{101}\overline{A}_3\partial\text{y}_{0}\partial\text{xx}_{03}+K_{102}\overline{A}_3\partial\text{y}_{0}\partial\text{xx}_{12}+K_{103}A_0\partial\text{y}_{0}\partial\text{xy}_{00}+(K_{104}A_1+K_{105}B_0)\partial\text{y}_{0}\partial\text{xy}_{01}+(K_{106}A_2+K_{107}B_1)\partial\text{y}_{0}\partial\text{xy}_{02}+K_{108}A_3\partial\text{y}_{0}\partial\text{xy}_{03}+K_{109}A_0\partial\text{y}_{0}\partial\text{xy}_{11}+(K_{110}A_1+K_{111}B_0)\partial\text{y}_{0}\partial\text{xy}_{12}+(K_{112}A_2+K_{113}B_1)\partial\text{y}_{0}\partial\text{xy}_{13}+K_{114}A_0\partial\text{y}_{0}\partial\text{xy}_{22}+(K_{115}A_1+K_{116}B_0)\partial\text{y}_{0}\partial\text{xy}_{23}+K_{117}A_0\partial\text{y}_{0}\partial\text{xy}_{33}+K_{118}\overline{A}_0\partial\text{y}_{1}\partial\text{xx}_{01}+(K_{119}\overline{A}_1+K_{120}\overline{B}_0)\partial\text{y}_{1}\partial\text{xx}_{02}+(K_{121}\overline{A}_2+K_{122}\overline{B}_1)\partial\text{y}_{1}\partial\text{xx}_{03}+(K_{123}\overline{A}_2+K_{124}\overline{B}_1)\partial\text{y}_{1}\partial\text{xx}_{12}+K_{125}\overline{A}_3\partial\text{y}_{1}\partial\text{xx}_{13}+(K_{126}A_1+K_{127}B_0)\partial\text{y}_{1}\partial\text{xy}_{00}+(K_{128}A_2+K_{129}B_1)\partial\text{y}_{1}\partial\text{xy}_{01}+K_{130}A_3\partial\text{y}_{1}\partial\text{xy}_{02}+K_{131}A_0\partial\text{y}_{1}\partial\text{xy}_{10}+(K_{132}A_1+K_{133}B_0)\partial\text{y}_{1}\partial\text{xy}_{11}+(K_{134}A_2+K_{135}B_1)\partial\text{y}_{1}\partial\text{xy}_{12}+K_{136}A_3\partial\text{y}_{1}\partial\text{xy}_{13}+K_{137}A_0\partial\text{y}_{1}\partial\text{xy}_{21}+(K_{138}A_1+K_{139}B_0)\partial\text{y}_{1}\partial\text{xy}_{22}+(K_{140}A_2+K_{141}B_1)\partial\text{y}_{1}\partial\text{xy}_{23}+K_{142}A_0\partial\text{y}_{1}\partial\text{xy}_{32}+(K_{143}A_1+K_{144}B_0)\partial\text{y}_{1}\partial\text{xy}_{33}+K_{145}\overline{A}_0\partial\text{y}_{2}\partial\text{xx}_{02}+(K_{146}\overline{A}_1+K_{147}\overline{B}_0)\partial\text{y}_{2}\partial\text{xx}_{03}+(K_{148}\overline{A}_1+K_{149}\overline{B}_0)\partial\text{y}_{2}\partial\text{xx}_{12}+(K_{150}\overline{A}_2+K_{151}\overline{B}_1)\partial\text{y}_{2}\partial\text{xx}_{13}+K_{152}\overline{A}_3\partial\text{y}_{2}\partial\text{xx}_{23}+(K_{153}A_2+K_{154}B_1)\partial\text{y}_{2}\partial\text{xy}_{00}+K_{155}A_3\partial\text{y}_{2}\partial\text{xy}_{01}+(K_{156}A_1+K_{157}B_0)\partial\text{y}_{2}\partial\text{xy}_{10}+(K_{158}A_2+K_{159}B_1)\partial\text{y}_{2}\partial\text{xy}_{11}+K_{160}A_3\partial\text{y}_{2}\partial\text{xy}_{12}+K_{161}A_0\partial\text{y}_{2}\partial\text{xy}_{20}+(K_{162}A_1+K_{163}B_0)\partial\text{y}_{2}\partial\text{xy}_{21}+(K_{164}A_2+K_{165}B_1)\partial\text{y}_{2}\partial\text{xy}_{22}+K_{166}A_3\partial\text{y}_{2}\partial\text{xy}_{23}+K_{167}A_0\partial\text{y}_{2}\partial\text{xy}_{31}+(K_{168}A_1+K_{169}B_0)\partial\text{y}_{2}\partial\text{xy}_{32}+(K_{170}A_2+K_{171}B_1)\partial\text{y}_{2}\partial\text{xy}_{33}+K_{172}\overline{A}_0\partial\text{y}_{3}\partial\text{xx}_{03}+K_{173}\overline{A}_0\partial\text{y}_{3}\partial\text{xx}_{12}+(K_{174}\overline{A}_1+K_{175}\overline{B}_0)\partial\text{y}_{3}\partial\text{xx}_{13}+(K_{176}\overline{A}_2+K_{177}\overline{B}_1)\partial\text{y}_{3}\partial\text{xx}_{23}+K_{178}A_3\partial\text{y}_{3}\partial\text{xy}_{00}+(K_{179}A_2+K_{180}B_1)\partial\text{y}_{3}\partial\text{xy}_{10}+K_{181}A_3\partial\text{y}_{3}\partial\text{xy}_{11}+(K_{182}A_1+K_{183}B_0)\partial\text{y}_{3}\partial\text{xy}_{20}+(K_{184}A_2+K_{185}B_1)\partial\text{y}_{3}\partial\text{xy}_{21}+K_{186}A_3\partial\text{y}_{3}\partial\text{xy}_{22}+K_{187}A_0\partial\text{y}_{3}\partial\text{xy}_{30}+(K_{188}A_1+K_{189}B_0)\partial\text{y}_{3}\partial\text{xy}_{31}+(K_{190}A_2+K_{191}B_1)\partial\text{y}_{3}\partial\text{xy}_{32}+K_{192}A_3\partial\text{y}_{3}\partial\text{xy}_{33}$
\end{itemize}
\end{flushleft}
The bivector $\pi=(\beta^1+\alpha^1)+(\beta^2+\alpha^2)$ is the candidate Poisson bivector.  While it certainly extends $\brac_Q$ as a bivector, the proper coefficients $K_{ij}$ must be chosen so that $[\pi,\pi]=0$.   We first begin by attempting to solve the equation 
$$[\pi^1,\pi^2]=[\beta^1+\alpha^1,\beta^2+\alpha^2]=0$$
which is a (very large) system of equations linear in the $K_{ij}$.  Once these coefficents are determined, a solution of the equation $[\pi^2,\pi^2]=0$ follows easily.

Simply setting $[\pi^1,\pi^2]$ equal to zero is too unwieldy to directly determine the coefficients.  Thus, we will first define the contraction for multivectors and analyze $[\pi^1,\pi^2]$ when first contracted with 
\begin{enumerate}
\item elements of $\bigwedge^3(\g\times Z)^*$ for which $[\beta^1,\beta^2]$ is nonzero,
\item elements of $\bigwedge^3\g^*$, and finally
\item all other elements in $\bigwedge^3(\g\times Z)^*$.
\end{enumerate}
The contraction operator $[\pi^1,\pi^2](du,dv,dw)$ is defined as:
\begin{itemize}
\item[$\text{i20}:$] $ \text{contract} = $\\
$(du,dv,dw)\! \to\! \text{Sbr(Sbr(Sbr(}\pi^2,u),v),\text{Sbr}(w,\pi^1))+\text{Sbr(Sbr(Sbr}(\pi^2,w),u),\text{Sbr}(v,\pi^1))+$\\
$\text{Sbr(Sbr(Sbr}(\pi^2,v),w),\text{Sbr}(u,\pi^1))+\text{ Sbr(Sbr(Sbr}(\pi^1,u),v),\text{Sbr}(w,\pi^2))+$\\
$\text{Sbr(Sbr(Sbr}(\pi^1,w),u),\text{Sbr}(v,\pi^2))+\text{Sbr(Sbr(Sbr}(\pi^1,v),w),\text{Sbr}(u,\pi^2));$
\end{itemize}
Given a list LL of elements of $\bigwedge^3\g^*$ as prescribed in the list above, we contract each element with $[\pi^1,\pi^2]$, and extract the $K_{ij}$-coefficients of the resulting elements of $\g$:
\begin{itemize}
\item[$\text{i21}:$]  $\text{apply}(\text{LL},\text{contract});$\\
			$\text{apply(oo},x\to \text{lift}(x,T));$    --  lifts the elements of XT to T\\
			$\text{apply(oo},x\to ((\text{coefficients } x)_1)_0);$  -- makes a list of coefficients \\
			$\text{apply(oo,entries)}; $ \\
\end{itemize}
When applied to the list $LL=\{(da_i,dx_i,dy_i)\}$, this procedure produces the system of equations in Table \ref{table:Maple} below.

\begin{table}[ph!]
\label{table:Maple}
\caption{System of equations obtained by setting $[\pi^1,\pi^2]=0$}
\vspace{.5cm}
 $
\left\{ \begin{array}{l}
   -2 {l}_{0}+2 {K}_{97} {L}_{10}+{K}_{98} {L}_{10}+2 {K}_{8} {L}_{33}+{K}_{9} {L}_{33}-2 {K}_{30}
     {L}_{33}-{K}_{31} {L}_{33}=0,\\-{l}_{0}+{K}_{97} {L}_{10}-{K}_{98} {L}_{10}+{K}_{8} {L}_{33}-{K}_{9}
     {L}_{33}-{K}_{30} {L}_{33}+{K}_{31} {L}_{33}-18=0,\\-{l}_{0}+{K}_{99} {L}_{10}+{K}_{100}
     {L}_{10}+{K}_{11} {L}_{33}+{K}_{12} {L}_{33}-{K}_{57} {L}_{33}-{K}_{58} {L}_{33}=0,\\-2 {l}_{0}+2
     {K}_{99} {L}_{10}-{K}_{100} {L}_{10}+2 {K}_{11} {L}_{33}-{K}_{12} {L}_{33}-2 {K}_{57}
     {L}_{33}+{K}_{58} {L}_{33}-18=0,\\2 {K}_{119} {L}_{10}+{K}_{120} {L}_{10}+2 {K}_{13} {L}_{33}+{K}_{14}
     {L}_{33}-2 {K}_{59} {L}_{33}-{K}_{60} {L}_{33}+6=0,\\{K}_{119} {L}_{10}-{K}_{120} {L}_{10}+{K}_{13}
     {L}_{33}-{K}_{14} {L}_{33}-{K}_{59} {L}_{33}+{K}_{60} {L}_{33}-6=0,\\-{l}_{1}+{K}_{123}
     {L}_{10}+{K}_{124} {L}_{10}+{K}_{39} {L}_{33}+{K}_{40} {L}_{33}-{K}_{63} {L}_{33}-{K}_{64}
     {L}_{33}+4=0,\\-2 {l}_{1}+2 {K}_{123} {L}_{10}-{K}_{124} {L}_{10}+2 {K}_{39} {L}_{33}-{K}_{40}
     {L}_{33}-2 {K}_{63} {L}_{33}+{K}_{64} {L}_{33}+2=0,\\{K}_{121} {L}_{10}+{K}_{122} {L}_{10}+{K}_{17}
     {L}_{33}+{K}_{18} {L}_{33}-{K}_{83} {L}_{33}-{K}_{84} {L}_{33}=0,\\
     \vdots \\
    (44 \mbox{ equations in total}) \\
     \vdots \\
      2 {K}_{50} {L}_{14}+{K}_{51} {L}_{14}-2
     {K}_{115} {L}_{33}-{K}_{116} {L}_{33}+2 {K}_{182} {L}_{33}+{K}_{183} {L}_{33}+18=0,\\{K}_{50}
     {L}_{14}-{K}_{51} {L}_{14}-{K}_{115} {L}_{33}+{K}_{116} {L}_{33}+{K}_{182} {L}_{33}-{K}_{183}
     {L}_{33}=0,\\{K}_{54} {L}_{14}+{K}_{55} {L}_{14}-{K}_{140} {L}_{33}-{K}_{141} {L}_{33}+{K}_{184}
     {L}_{33}+{K}_{185} {L}_{33}+6=0,\\2 {K}_{54} {L}_{14}-{K}_{55} {L}_{14}-2 {K}_{140} {L}_{33}+{K}_{141}
     {L}_{33}+2 {K}_{184} {L}_{33}-{K}_{185} {L}_{33}-6=0,\\2 {l}_{3}+2 {K}_{78} {L}_{14}+{K}_{79}
     {L}_{14}-2 {K}_{143} {L}_{33}-{K}_{144} {L}_{33}+2 {K}_{188} {L}_{33}+{K}_{189}
     {L}_{33}+18=0,\\{l}_{3}+{K}_{78} {L}_{14}-{K}_{79} {L}_{14}-{K}_{143} {L}_{33}+{K}_{144}
     {L}_{33}+{K}_{188} {L}_{33}-{K}_{189} {L}_{33}=0,\\{l}_{3}+{K}_{80} {L}_{14}+{K}_{81}
     {L}_{14}-{K}_{170} {L}_{33}-{K}_{171} {L}_{33}+{K}_{190} {L}_{33}+{K}_{191} {L}_{33}+18=0,\\2 {l}_{3}+2
     {K}_{80} {L}_{14}-{K}_{81} {L}_{14}-2 {K}_{170} {L}_{33}+{K}_{171} {L}_{33}+2 {K}_{190}
     {L}_{33}-{K}_{191} {L}_{33}=0,\\3 {K}_{102} {L}_{10}+3 {K}_{38} {L}_{33}-3 {K}_{62} {L}_{33}-6=0,3
     {K}_{77} {L}_{14}-3 {K}_{142} {L}_{33}+3 {K}_{167} {L}_{33}+6=0,\\3 {K}_{173} {L}_{10}+3 {K}_{43}
     {L}_{33}-3 {K}_{67} {L}_{33}+6=0,3 {K}_{6} {L}_{14}-3 {K}_{130} {L}_{33}+3 {K}_{155} {L}_{33}-6=0
 \end{array}\right.
 $
 \end{table}
 
 These linear equations in $K_{ij}$ are then solved using Maple${}^\mathrm{TM}$ \cite{Maple10}, allowing one to eliminate a number of the coefficients.  This elimination results in a simplified $\alpha^2$ for which we can repeat this process  contracted with a different list of elements of $\bigwedge(\g\times Z)^*$.  Continuing this procedure yields an extension $\pi$ given below:

\begin{flushleft}
\begin{itemize}
\item[$\text{i22}:$] $\pi=60\text{a}_{2}\partial\text{a}_{0}\partial\text{a}_{2}-60\text{a}_{2}\partial\text{a}_{1}\partial\text{a}_{2}-60\text{a}_{3}\partial\text{a}_{0}\partial\text{a}_{3}+60\text{a}_{3}\partial\text{a}_{1}\partial\text{a}_{3}+(60\text{a}_{0}-60\text{a}_{1}+60\text{aa}_{23})\partial\text{a}_{2}\partial\text{a}_{3}-180\text{x}_{0}\partial\text{a}_{1}\partial\text{x}_{0}+60\text{ax}_{20}\partial\text{a}_{2}\partial\text{x}_{0}+(-180\text{x}_{1}-60\text{ax}_{30})\partial\text{a}_{3}\partial\text{x}_{0}-60\text{x}_{1}\partial\text{a}_{0}\partial\text{x}_{1}-120\text{x}_{1}\partial\text{a}_{1}\partial\text{x}_{1}+(-60\text{x}_{0}+120\text{ax}_{21})\partial\text{a}_{2}\partial\text{x}_{1}+(-120\text{x}_{2}-120\text{ax}_{31})\partial\text{a}_{3}\partial\text{x}_{1}+420\text{xx}_{01}\partial\text{x}_{0}\partial\text{x}_{1}-120\text{x}_{2}\partial\text{a}_{0}\partial\text{x}_{2}-60\text{x}_{2}\partial\text{a}_{1}\partial\text{x}_{2}+(-120\text{x}_{1}+180\text{ax}_{22})\partial\text{a}_{2}\partial\text{x}_{2}+(-60\text{x}_{3}-180\text{ax}_{32})\partial\text{a}_{3}\partial\text{x}_{2}+360\text{xx}_{02}\partial\text{x}_{0}\partial\text{x}_{2}+180\text{xx}_{12}\partial\text{x}_{1}\partial\text{x}_{2}-180\text{x}_{3}\partial\text{a}_{0}\partial\text{x}_{3}+(-180\text{x}_{2}+240\text{ax}_{23})\partial\text{a}_{2}\partial\text{x}_{3}-240\text{ax}_{33}\partial\text{a}_{3}\partial\text{x}_{3}+300\text{xx}_{03}\partial\text{x}_{0}\partial\text{x}_{3}+60\text{xx}_{13}\partial\text{x}_{1}\partial\text{x}_{3}-180\text{xx}_{23}\partial\text{x}_{2}\partial\text{x}_{3}+180\text{y}_{0}\partial\text{a}_{1}\partial\text{y}_{0}+(180\text{y}_{1}+120\text{ay}_{20})\partial\text{a}_{2}\partial\text{y}_{0}-120\text{ay}_{30}\partial\text{a}_{3}\partial\text{y}_{0}+(900\text{a}_{0}\text{a}_{1}+540\text{a}_{1}^2-900\text{a}_{2}\text{a}_{3}-60\text{xy}_{00})\partial\text{x}_{0}\partial\text{y}_{0}+(360\text{a}_{1}\text{a}_{3}-180\text{xy}_{10})\partial\text{x}_{1}\partial\text{y}_{0}+(180\text{a}_{3}^2-300\text{xy}_{20})\partial\text{x}_{2}\partial\text{y}_{0}-420\text{xy}_{30}\partial\text{x}_{3}\partial\text{y}_{0}+60\text{y}_{1}\partial\text{a}_{0}\partial\text{y}_{1}+120\text{y}_{1}\partial\text{a}_{1}\partial\text{y}_{1}+(120\text{y}_{2}+60\text{ay}_{21})\partial\text{a}_{2}\partial\text{y}_{1}+(60\text{y}_{0}-60\text{ay}_{31})\partial\text{a}_{3}\partial\text{y}_{1}+360\text{a}_{1}\text{a}_{2}\partial\text{x}_{0}\partial\text{y}_{1}+(420\text{a}_{0}\text{a}_{1}+60\text{a}_{1}^2-180\text{a}_{2}\text{a}_{3}-60\text{xy}_{11})\partial\text{x}_{1}\partial\text{y}_{1}+(120\text{a}_{0}\text{a}_{3}+120\text{a}_{1}\text{a}_{3}-120\text{xy}_{21})\partial\text{x}_{2}\partial\text{y}_{1}+(180\text{a}_{3}^2-180\text{xy}_{31})\partial\text{x}_{3}\partial\text{y}_{1}-60\text{yy}_{01}\partial\text{y}_{0}\partial\text{y}_{1}+120\text{y}_{2}\partial\text{a}_{0}\partial\text{y}_{2}+60\text{y}_{2}\partial\text{a}_{1}\partial\text{y}_{2}+60\text{y}_{3}\partial\text{a}_{2}\partial\text{y}_{2}+120\text{y}_{1}\partial\text{a}_{3}\partial\text{y}_{2}+(180\text{a}_{2}^2+60\text{xy}_{02})\partial\text{x}_{0}\partial\text{y}_{2}+(120\text{a}_{0}\text{a}_{2}+120\text{a}_{1}\text{a}_{2}+60\text{xy}_{12})\partial\text{x}_{1}\partial\text{y}_{2}+(60\text{a}_{0}^2+420\text{a}_{0}\text{a}_{1}-180\text{a}_{2}\text{a}_{3}+60\text{xy}_{22})\partial\text{x}_{2}\partial\text{y}_{2}+(360\text{a}_{0}\text{a}_{3}+60\text{xy}_{32})\partial\text{x}_{3}\partial\text{y}_{2}+60\text{yy}_{02}\partial\text{y}_{0}\partial\text{y}_{2}+60\text{yy}_{12}\partial\text{y}_{1}\partial\text{y}_{2}+180\text{y}_{3}\partial\text{a}_{0}\partial\text{y}_{3}-60\text{ay}_{23}\partial\text{a}_{2}\partial\text{y}_{3}+(180\text{y}_{2}+60\text{ay}_{33})\partial\text{a}_{3}\partial\text{y}_{3}+120\text{xy}_{03}\partial\text{x}_{0}\partial\text{y}_{3}+(180\text{a}_{2}^2+180\text{xy}_{13})\partial\text{x}_{1}\partial\text{y}_{3}+(360\text{a}_{0}\text{a}_{2}+240\text{xy}_{23})\partial\text{x}_{2}\partial\text{y}_{3}+(540\text{a}_{0}^2+900\text{a}_{0}\text{a}_{1}-900\text{a}_{2}\text{a}_{3}+300\text{xy}_{33})\partial\text{x}_{3}\partial\text{y}_{3}+180\text{yy}_{03}\partial\text{y}_{0}\partial\text{y}_{3}+120\text{yy}_{13}\partial\text{y}_{1}\partial\text{y}_{3}+60\text{yy}_{23}\partial\text{y}_{2}\partial\text{y}_{3}+(60\text{aa}_{02}+60\text{aa}_{12})\partial\text{a}_{2}\partial\text{aa}_{01}+(-60\text{aa}_{03}-60\text{aa}_{13})\partial\text{a}_{3}\partial\text{aa}_{01}+180\text{ax}_{00}\partial\text{x}_{0}\partial\text{aa}_{01}+(120\text{ax}_{01}-60\text{ax}_{11})\partial\text{x}_{1}\partial\text{aa}_{01}+(60\text{ax}_{02}-120\text{ax}_{12})\partial\text{x}_{2}\partial\text{aa}_{01}-180\text{ax}_{13}\partial\text{x}_{3}\partial\text{aa}_{01}-180\text{ay}_{00}\partial\text{y}_{0}\partial\text{aa}_{01}+(-120\text{ay}_{01}+60\text{ay}_{11})\partial\text{y}_{1}\partial\text{aa}_{01}+(-60\text{ay}_{02}+120\text{ay}_{12})\partial\text{y}_{2}\partial\text{aa}_{01}+180\text{ay}_{13}\partial\text{y}_{3}\partial\text{aa}_{01}+60\text{aa}_{02}\partial\text{a}_{0}\partial\text{aa}_{02}-60\text{aa}_{02}\partial\text{a}_{1}\partial\text{aa}_{02}+(60\text{aa}_{01}-60\text{aa}_{23})\partial\text{a}_{3}\partial\text{aa}_{02}+(60\text{ax}_{00}-60\text{ax}_{21})\partial\text{x}_{1}\partial\text{aa}_{02}+(120\text{ax}_{01}-120\text{ax}_{22})\partial\text{x}_{2}\partial\text{aa}_{02}+(180\text{ax}_{02}-180\text{ax}_{23})\partial\text{x}_{3}\partial\text{aa}_{02}-180\text{ay}_{01}\partial\text{y}_{0}\partial\text{aa}_{02}+(-120\text{ay}_{02}+60\text{ay}_{21})\partial\text{y}_{1}\partial\text{aa}_{02}+(-60\text{ay}_{03}+120\text{ay}_{22})\partial\text{y}_{2}\partial\text{aa}_{02}+180\text{ay}_{23}\partial\text{y}_{3}\partial\text{aa}_{02}-60\text{aa}_{03}\partial\text{a}_{0}\partial\text{aa}_{03}+60\text{aa}_{03}\partial\text{a}_{1}\partial\text{aa}_{03}+(-60\text{aa}_{01}-60\text{aa}_{23})\partial\text{a}_{2}\partial\text{aa}_{03}+180\text{ax}_{01}\partial\text{x}_{0}\partial\text{aa}_{03}+(120\text{ax}_{02}-60\text{ax}_{31})\partial\text{x}_{1}\partial\text{aa}_{03}+(60\text{ax}_{03}-120\text{ax}_{32})\partial\text{x}_{2}\partial\text{aa}_{03}-180\text{ax}_{33}\partial\text{x}_{3}\partial\text{aa}_{03}+(-60\text{ay}_{00}+60\text{ay}_{31})\partial\text{y}_{1}\partial\text{aa}_{03}+(-120\text{ay}_{01}+120\text{ay}_{32})\partial\text{y}_{2}\partial\text{aa}_{03}+(-180\text{ay}_{02}+180\text{ay}_{33})\partial\text{y}_{3}\partial\text{aa}_{03}+60\text{aa}_{12}\partial\text{a}_{0}\partial\text{aa}_{12}-60\text{aa}_{12}\partial\text{a}_{1}\partial\text{aa}_{12}+(60\text{aa}_{01}+60\text{aa}_{23})\partial\text{a}_{3}\partial\text{aa}_{12}-180\text{ax}_{20}\partial\text{x}_{0}\partial\text{aa}_{12}+(60\text{ax}_{10}-120\text{ax}_{21})\partial\text{x}_{1}\partial\text{aa}_{12}+(120\text{ax}_{11}-60\text{ax}_{22})\partial\text{x}_{2}\partial\text{aa}_{12}+180\text{ax}_{12}\partial\text{x}_{3}\partial\text{aa}_{12}+(-180\text{ay}_{11}+180\text{ay}_{20})\partial\text{y}_{0}\partial\text{aa}_{12}+(-120\text{ay}_{12}+120\text{ay}_{21})\partial\text{y}_{1}\partial\text{aa}_{12}+(-60\text{ay}_{13}+60\text{ay}_{22})\partial\text{y}_{2}\partial\text{aa}_{12}-60\text{aa}_{13}\partial\text{a}_{0}\partial\text{aa}_{13}+60\text{aa}_{13}\partial\text{a}_{1}\partial\text{aa}_{13}+(-60\text{aa}_{01}+60\text{aa}_{23})\partial\text{a}_{2}\partial\text{aa}_{13}+(180\text{ax}_{11}-180\text{ax}_{30})\partial\text{x}_{0}\partial\text{aa}_{13}+(120\text{ax}_{12}-120\text{ax}_{31})\partial\text{x}_{1}\partial\text{aa}_{13}+(60\text{ax}_{13}-60\text{ax}_{32})\partial\text{x}_{2}\partial\text{aa}_{13}+180\text{ay}_{30}\partial\text{y}_{0}\partial\text{aa}_{13}+(-60\text{ay}_{10}+120\text{ay}_{31})\partial\text{y}_{1}\partial\text{aa}_{13}+(-120\text{ay}_{11}+60\text{ay}_{32})\partial\text{y}_{2}\partial\text{aa}_{13}-180\text{ay}_{12}\partial\text{y}_{3}\partial\text{aa}_{13}+(-60\text{aa}_{02}+60\text{aa}_{12})\partial\text{a}_{2}\partial\text{aa}_{23}+(-60\text{aa}_{03}+60\text{aa}_{13})\partial\text{a}_{3}\partial\text{aa}_{23}+180\text{ax}_{21}\partial\text{x}_{0}\partial\text{aa}_{23}+(120\text{ax}_{22}-60\text{ax}_{30})\partial\text{x}_{1}\partial\text{aa}_{23}+(60\text{ax}_{23}-120\text{ax}_{31})\partial\text{x}_{2}\partial\text{aa}_{23}-180\text{ax}_{32}\partial\text{x}_{3}\partial\text{aa}_{23}+180\text{ay}_{31}\partial\text{y}_{0}\partial\text{aa}_{23}+(-60\text{ay}_{20}+120\text{ay}_{32})\partial\text{y}_{1}\partial\text{aa}_{23}+(-120\text{ay}_{21}+60\text{ay}_{33})\partial\text{y}_{2}\partial\text{aa}_{23}-180\text{ay}_{22}\partial\text{y}_{3}\partial\text{aa}_{23}-180\text{ax}_{00}\partial\text{a}_{1}\partial\text{ax}_{00}-60\text{ax}_{20}\partial\text{a}_{2}\partial\text{ax}_{00}+(-180\text{ax}_{01}+60\text{ax}_{30})\partial\text{a}_{3}\partial\text{ax}_{00}-60\text{xx}_{01}\partial\text{x}_{1}\partial\text{ax}_{00}-120\text{xx}_{02}\partial\text{x}_{2}\partial\text{ax}_{00}-180\text{xx}_{03}\partial\text{x}_{3}\partial\text{ax}_{00}+60\text{xy}_{01}\partial\text{y}_{1}\partial\text{ax}_{00}+120\text{xy}_{02}\partial\text{y}_{2}\partial\text{ax}_{00}+180\text{xy}_{03}\partial\text{y}_{3}\partial\text{ax}_{00}-60\text{ax}_{01}\partial\text{a}_{0}\partial\text{ax}_{01}-120\text{ax}_{01}\partial\text{a}_{1}\partial\text{ax}_{01}+(-60\text{ax}_{00}-60\text{ax}_{21})\partial\text{a}_{2}\partial\text{ax}_{01}+(-120\text{ax}_{02}+60\text{ax}_{31})\partial\text{a}_{3}\partial\text{ax}_{01}-120\text{xx}_{12}\partial\text{x}_{2}\partial\text{ax}_{01}-180\text{xx}_{13}\partial\text{x}_{3}\partial\text{ax}_{01}+60\text{xy}_{11}\partial\text{y}_{1}\partial\text{ax}_{01}+120\text{xy}_{12}\partial\text{y}_{2}\partial\text{ax}_{01}+180\text{xy}_{13}\partial\text{y}_{3}\partial\text{ax}_{01}-120\text{ax}_{02}\partial\text{a}_{0}\partial\text{ax}_{02}-60\text{ax}_{02}\partial\text{a}_{1}\partial\text{ax}_{02}+(-120\text{ax}_{01}-60\text{ax}_{22})\partial\text{a}_{2}\partial\text{ax}_{02}+(-60\text{ax}_{03}+60\text{ax}_{32})\partial\text{a}_{3}\partial\text{ax}_{02}+60\text{xx}_{12}\partial\text{x}_{1}\partial\text{ax}_{02}-180\text{xx}_{23}\partial\text{x}_{3}\partial\text{ax}_{02}+60\text{xy}_{21}\partial\text{y}_{1}\partial\text{ax}_{02}+120\text{xy}_{22}\partial\text{y}_{2}\partial\text{ax}_{02}+180\text{xy}_{23}\partial\text{y}_{3}\partial\text{ax}_{02}-180\text{ax}_{03}\partial\text{a}_{0}\partial\text{ax}_{03}+(-180\text{ax}_{02}-60\text{ax}_{23})\partial\text{a}_{2}\partial\text{ax}_{03}+60\text{ax}_{33}\partial\text{a}_{3}\partial\text{ax}_{03}+60\text{xx}_{13}\partial\text{x}_{1}\partial\text{ax}_{03}+120\text{xx}_{23}\partial\text{x}_{2}\partial\text{ax}_{03}+60\text{xy}_{31}\partial\text{y}_{1}\partial\text{ax}_{03}+120\text{xy}_{32}\partial\text{y}_{2}\partial\text{ax}_{03}+180\text{xy}_{33}\partial\text{y}_{3}\partial\text{ax}_{03}-180\text{ax}_{10}\partial\text{a}_{1}\partial\text{ax}_{10}+60\text{ax}_{20}\partial\text{a}_{2}\partial\text{ax}_{10}+(-180\text{ax}_{11}-60\text{ax}_{30})\partial\text{a}_{3}\partial\text{ax}_{10}-120\text{xx}_{01}\partial\text{x}_{1}\partial\text{ax}_{10}-60\text{xx}_{02}\partial\text{x}_{2}\partial\text{ax}_{10}+180\text{xy}_{00}\partial\text{y}_{0}\partial\text{ax}_{10}+120\text{xy}_{01}\partial\text{y}_{1}\partial\text{ax}_{10}+60\text{xy}_{02}\partial\text{y}_{2}\partial\text{ax}_{10}-60\text{ax}_{11}\partial\text{a}_{0}\partial\text{ax}_{11}-120\text{ax}_{11}\partial\text{a}_{1}\partial\text{ax}_{11}+(-60\text{ax}_{10}+60\text{ax}_{21})\partial\text{a}_{2}\partial\text{ax}_{11}+(-120\text{ax}_{12}-60\text{ax}_{31})\partial\text{a}_{3}\partial\text{ax}_{11}+180\text{xx}_{01}\partial\text{x}_{0}\partial\text{ax}_{11}-60\text{xx}_{12}\partial\text{x}_{2}\partial\text{ax}_{11}+180\text{xy}_{10}\partial\text{y}_{0}\partial\text{ax}_{11}+120\text{xy}_{11}\partial\text{y}_{1}\partial\text{ax}_{11}+60\text{xy}_{12}\partial\text{y}_{2}\partial\text{ax}_{11}-120\text{ax}_{12}\partial\text{a}_{0}\partial\text{ax}_{12}-60\text{ax}_{12}\partial\text{a}_{1}\partial\text{ax}_{12}+(-120\text{ax}_{11}+60\text{ax}_{22})\partial\text{a}_{2}\partial\text{ax}_{12}+(-60\text{ax}_{13}-60\text{ax}_{32})\partial\text{a}_{3}\partial\text{ax}_{12}+180\text{xx}_{02}\partial\text{x}_{0}\partial\text{ax}_{12}+120\text{xx}_{12}\partial\text{x}_{1}\partial\text{ax}_{12}+180\text{xy}_{20}\partial\text{y}_{0}\partial\text{ax}_{12}+120\text{xy}_{21}\partial\text{y}_{1}\partial\text{ax}_{12}+60\text{xy}_{22}\partial\text{y}_{2}\partial\text{ax}_{12}-180\text{ax}_{13}\partial\text{a}_{0}\partial\text{ax}_{13}+(-180\text{ax}_{12}+60\text{ax}_{23})\partial\text{a}_{2}\partial\text{ax}_{13}-60\text{ax}_{33}\partial\text{a}_{3}\partial\text{ax}_{13}+180\text{xx}_{03}\partial\text{x}_{0}\partial\text{ax}_{13}+120\text{xx}_{13}\partial\text{x}_{1}\partial\text{ax}_{13}+60\text{xx}_{23}\partial\text{x}_{2}\partial\text{ax}_{13}+180\text{xy}_{30}\partial\text{y}_{0}\partial\text{ax}_{13}+120\text{xy}_{31}\partial\text{y}_{1}\partial\text{ax}_{13}+60\text{xy}_{32}\partial\text{y}_{2}\partial\text{ax}_{13}+60\text{ax}_{20}\partial\text{a}_{0}\partial\text{ax}_{20}-240\text{ax}_{20}\partial\text{a}_{1}\partial\text{ax}_{20}+(-60\text{ax}_{00}+60\text{ax}_{10}-180\text{ax}_{21})\partial\text{a}_{3}\partial\text{ax}_{20}-120\text{xx}_{01}\partial\text{x}_{2}\partial\text{ax}_{20}-180\text{xx}_{02}\partial\text{x}_{3}\partial\text{ax}_{20}+180\text{xy}_{01}\partial\text{y}_{0}\partial\text{ax}_{20}+120\text{xy}_{02}\partial\text{y}_{1}\partial\text{ax}_{20}+60\text{xy}_{03}\partial\text{y}_{2}\partial\text{ax}_{20}-180\text{ax}_{21}\partial\text{a}_{1}\partial\text{ax}_{21}-60\text{ax}_{20}\partial\text{a}_{2}\partial\text{ax}_{21}+(-60\text{ax}_{01}+60\text{ax}_{11}-120\text{ax}_{22})\partial\text{a}_{3}\partial\text{ax}_{21}+60\text{xx}_{01}\partial\text{x}_{1}\partial\text{ax}_{21}-180\text{xx}_{12}\partial\text{x}_{3}\partial\text{ax}_{21}+180\text{xy}_{11}\partial\text{y}_{0}\partial\text{ax}_{21}+120\text{xy}_{12}\partial\text{y}_{1}\partial\text{ax}_{21}+60\text{xy}_{13}\partial\text{y}_{2}\partial\text{ax}_{21}-60\text{ax}_{22}\partial\text{a}_{0}\partial\text{ax}_{22}-120\text{ax}_{22}\partial\text{a}_{1}\partial\text{ax}_{22}-120\text{ax}_{21}\partial\text{a}_{2}\partial\text{ax}_{22}+(-60\text{ax}_{02}+60\text{ax}_{12}-60\text{ax}_{23})\partial\text{a}_{3}\partial\text{ax}_{22}+60\text{xx}_{02}\partial\text{x}_{1}\partial\text{ax}_{22}+120\text{xx}_{12}\partial\text{x}_{2}\partial\text{ax}_{22}+180\text{xy}_{21}\partial\text{y}_{0}\partial\text{ax}_{22}+120\text{xy}_{22}\partial\text{y}_{1}\partial\text{ax}_{22}+60\text{xy}_{23}\partial\text{y}_{2}\partial\text{ax}_{22}-120\text{ax}_{23}\partial\text{a}_{0}\partial\text{ax}_{23}-60\text{ax}_{23}\partial\text{a}_{1}\partial\text{ax}_{23}-180\text{ax}_{22}\partial\text{a}_{2}\partial\text{ax}_{23}+(-60\text{ax}_{03}+60\text{ax}_{13})\partial\text{a}_{3}\partial\text{ax}_{23}+60\text{xx}_{03}\partial\text{x}_{1}\partial\text{ax}_{23}+120\text{xx}_{13}\partial\text{x}_{2}\partial\text{ax}_{23}+180\text{xx}_{23}\partial\text{x}_{3}\partial\text{ax}_{23}+180\text{xy}_{31}\partial\text{y}_{0}\partial\text{ax}_{23}+120\text{xy}_{32}\partial\text{y}_{1}\partial\text{ax}_{23}+60\text{xy}_{33}\partial\text{y}_{2}\partial\text{ax}_{23}-60\text{ax}_{30}\partial\text{a}_{0}\partial\text{ax}_{30}-120\text{ax}_{30}\partial\text{a}_{1}\partial\text{ax}_{30}+(60\text{ax}_{00}-60\text{ax}_{10})\partial\text{a}_{2}\partial\text{ax}_{30}-180\text{ax}_{31}\partial\text{a}_{3}\partial\text{ax}_{30}-180\text{xx}_{01}\partial\text{x}_{0}\partial\text{ax}_{30}-120\text{xx}_{02}\partial\text{x}_{1}\partial\text{ax}_{30}-60\text{xx}_{03}\partial\text{x}_{2}\partial\text{ax}_{30}+60\text{xy}_{00}\partial\text{y}_{1}\partial\text{ax}_{30}+120\text{xy}_{01}\partial\text{y}_{2}\partial\text{ax}_{30}+180\text{xy}_{02}\partial\text{y}_{3}\partial\text{ax}_{30}-120\text{ax}_{31}\partial\text{a}_{0}\partial\text{ax}_{31}-60\text{ax}_{31}\partial\text{a}_{1}\partial\text{ax}_{31}+(60\text{ax}_{01}-60\text{ax}_{11}-60\text{ax}_{30})\partial\text{a}_{2}\partial\text{ax}_{31}-120\text{ax}_{32}\partial\text{a}_{3}\partial\text{ax}_{31}-120\text{xx}_{12}\partial\text{x}_{1}\partial\text{ax}_{31}-60\text{xx}_{13}\partial\text{x}_{2}\partial\text{ax}_{31}+60\text{xy}_{10}\partial\text{y}_{1}\partial\text{ax}_{31}+120\text{xy}_{11}\partial\text{y}_{2}\partial\text{ax}_{31}+180\text{xy}_{12}\partial\text{y}_{3}\partial\text{ax}_{31}-180\text{ax}_{32}\partial\text{a}_{0}\partial\text{ax}_{32}+(60\text{ax}_{02}-60\text{ax}_{12}-120\text{ax}_{31})\partial\text{a}_{2}\partial\text{ax}_{32}-60\text{ax}_{33}\partial\text{a}_{3}\partial\text{ax}_{32}+180\text{xx}_{12}\partial\text{x}_{0}\partial\text{ax}_{32}-60\text{xx}_{23}\partial\text{x}_{2}\partial\text{ax}_{32}+60\text{xy}_{20}\partial\text{y}_{1}\partial\text{ax}_{32}+120\text{xy}_{21}\partial\text{y}_{2}\partial\text{ax}_{32}+180\text{xy}_{22}\partial\text{y}_{3}\partial\text{ax}_{32}-240\text{ax}_{33}\partial\text{a}_{0}\partial\text{ax}_{33}+60\text{ax}_{33}\partial\text{a}_{1}\partial\text{ax}_{33}+(60\text{ax}_{03}-60\text{ax}_{13}-180\text{ax}_{32})\partial\text{a}_{2}\partial\text{ax}_{33}+180\text{xx}_{13}\partial\text{x}_{0}\partial\text{ax}_{33}+120\text{xx}_{23}\partial\text{x}_{1}\partial\text{ax}_{33}+60\text{xy}_{30}\partial\text{y}_{1}\partial\text{ax}_{33}+120\text{xy}_{31}\partial\text{y}_{2}\partial\text{ax}_{33}+180\text{xy}_{32}\partial\text{y}_{3}\partial\text{ax}_{33}+180\text{ay}_{00}\partial\text{a}_{1}\partial\text{ay}_{00}+(180\text{ay}_{01}-60\text{ay}_{20})\partial\text{a}_{2}\partial\text{ay}_{00}+60\text{ay}_{30}\partial\text{a}_{3}\partial\text{ay}_{00}+60\text{xy}_{10}\partial\text{x}_{1}\partial\text{ay}_{00}+120\text{xy}_{20}\partial\text{x}_{2}\partial\text{ay}_{00}+180\text{xy}_{30}\partial\text{x}_{3}\partial\text{ay}_{00}+60\text{yy}_{01}\partial\text{y}_{1}\partial\text{ay}_{00}+120\text{yy}_{02}\partial\text{y}_{2}\partial\text{ay}_{00}+180\text{yy}_{03}\partial\text{y}_{3}\partial\text{ay}_{00}+60\text{ay}_{01}\partial\text{a}_{0}\partial\text{ay}_{01}+120\text{ay}_{01}\partial\text{a}_{1}\partial\text{ay}_{01}+(120\text{ay}_{02}-60\text{ay}_{21})\partial\text{a}_{2}\partial\text{ay}_{01}+(60\text{ay}_{00}+60\text{ay}_{31})\partial\text{a}_{3}\partial\text{ay}_{01}+60\text{xy}_{11}\partial\text{x}_{1}\partial\text{ay}_{01}+120\text{xy}_{21}\partial\text{x}_{2}\partial\text{ay}_{01}+180\text{xy}_{31}\partial\text{x}_{3}\partial\text{ay}_{01}+120\text{yy}_{12}\partial\text{y}_{2}\partial\text{ay}_{01}+180\text{yy}_{13}\partial\text{y}_{3}\partial\text{ay}_{01}+120\text{ay}_{02}\partial\text{a}_{0}\partial\text{ay}_{02}+60\text{ay}_{02}\partial\text{a}_{1}\partial\text{ay}_{02}+(60\text{ay}_{03}-60\text{ay}_{22})\partial\text{a}_{2}\partial\text{ay}_{02}+(120\text{ay}_{01}+60\text{ay}_{32})\partial\text{a}_{3}\partial\text{ay}_{02}+60\text{xy}_{12}\partial\text{x}_{1}\partial\text{ay}_{02}+120\text{xy}_{22}\partial\text{x}_{2}\partial\text{ay}_{02}+180\text{xy}_{32}\partial\text{x}_{3}\partial\text{ay}_{02}-60\text{yy}_{12}\partial\text{y}_{1}\partial\text{ay}_{02}+180\text{yy}_{23}\partial\text{y}_{3}\partial\text{ay}_{02}+180\text{ay}_{03}\partial\text{a}_{0}\partial\text{ay}_{03}-60\text{ay}_{23}\partial\text{a}_{2}\partial\text{ay}_{03}+(180\text{ay}_{02}+60\text{ay}_{33})\partial\text{a}_{3}\partial\text{ay}_{03}+60\text{xy}_{13}\partial\text{x}_{1}\partial\text{ay}_{03}+120\text{xy}_{23}\partial\text{x}_{2}\partial\text{ay}_{03}+180\text{xy}_{33}\partial\text{x}_{3}\partial\text{ay}_{03}-60\text{yy}_{13}\partial\text{y}_{1}\partial\text{ay}_{03}-120\text{yy}_{23}\partial\text{y}_{2}\partial\text{ay}_{03}+180\text{ay}_{10}\partial\text{a}_{1}\partial\text{ay}_{10}+(180\text{ay}_{11}+60\text{ay}_{20})\partial\text{a}_{2}\partial\text{ay}_{10}-60\text{ay}_{30}\partial\text{a}_{3}\partial\text{ay}_{10}+180\text{xy}_{00}\partial\text{x}_{0}\partial\text{ay}_{10}+120\text{xy}_{10}\partial\text{x}_{1}\partial\text{ay}_{10}+60\text{xy}_{20}\partial\text{x}_{2}\partial\text{ay}_{10}+120\text{yy}_{01}\partial\text{y}_{1}\partial\text{ay}_{10}+60\text{yy}_{02}\partial\text{y}_{2}\partial\text{ay}_{10}+60\text{ay}_{11}\partial\text{a}_{0}\partial\text{ay}_{11}+120\text{ay}_{11}\partial\text{a}_{1}\partial\text{ay}_{11}+(120\text{ay}_{12}+60\text{ay}_{21})\partial\text{a}_{2}\partial\text{ay}_{11}+(60\text{ay}_{10}-60\text{ay}_{31})\partial\text{a}_{3}\partial\text{ay}_{11}+180\text{xy}_{01}\partial\text{x}_{0}\partial\text{ay}_{11}+120\text{xy}_{11}\partial\text{x}_{1}\partial\text{ay}_{11}+60\text{xy}_{21}\partial\text{x}_{2}\partial\text{ay}_{11}-180\text{yy}_{01}\partial\text{y}_{0}\partial\text{ay}_{11}+60\text{yy}_{12}\partial\text{y}_{2}\partial\text{ay}_{11}+120\text{ay}_{12}\partial\text{a}_{0}\partial\text{ay}_{12}+60\text{ay}_{12}\partial\text{a}_{1}\partial\text{ay}_{12}+(60\text{ay}_{13}+60\text{ay}_{22})\partial\text{a}_{2}\partial\text{ay}_{12}+(120\text{ay}_{11}-60\text{ay}_{32})\partial\text{a}_{3}\partial\text{ay}_{12}+180\text{xy}_{02}\partial\text{x}_{0}\partial\text{ay}_{12}+120\text{xy}_{12}\partial\text{x}_{1}\partial\text{ay}_{12}+60\text{xy}_{22}\partial\text{x}_{2}\partial\text{ay}_{12}-180\text{yy}_{02}\partial\text{y}_{0}\partial\text{ay}_{12}-120\text{yy}_{12}\partial\text{y}_{1}\partial\text{ay}_{12}+180\text{ay}_{13}\partial\text{a}_{0}\partial\text{ay}_{13}+60\text{ay}_{23}\partial\text{a}_{2}\partial\text{ay}_{13}+(180\text{ay}_{12}-60\text{ay}_{33})\partial\text{a}_{3}\partial\text{ay}_{13}+180\text{xy}_{03}\partial\text{x}_{0}\partial\text{ay}_{13}+120\text{xy}_{13}\partial\text{x}_{1}\partial\text{ay}_{13}+60\text{xy}_{23}\partial\text{x}_{2}\partial\text{ay}_{13}-180\text{yy}_{03}\partial\text{y}_{0}\partial\text{ay}_{13}-120\text{yy}_{13}\partial\text{y}_{1}\partial\text{ay}_{13}-60\text{yy}_{23}\partial\text{y}_{2}\partial\text{ay}_{13}+60\text{ay}_{20}\partial\text{a}_{0}\partial\text{ay}_{20}+120\text{ay}_{20}\partial\text{a}_{1}\partial\text{ay}_{20}+180\text{ay}_{21}\partial\text{a}_{2}\partial\text{ay}_{20}+(-60\text{ay}_{00}+60\text{ay}_{10})\partial\text{a}_{3}\partial\text{ay}_{20}+60\text{xy}_{00}\partial\text{x}_{1}\partial\text{ay}_{20}+120\text{xy}_{10}\partial\text{x}_{2}\partial\text{ay}_{20}+180\text{xy}_{20}\partial\text{x}_{3}\partial\text{ay}_{20}+180\text{yy}_{01}\partial\text{y}_{0}\partial\text{ay}_{20}+120\text{yy}_{02}\partial\text{y}_{1}\partial\text{ay}_{20}+60\text{yy}_{03}\partial\text{y}_{2}\partial\text{ay}_{20}+120\text{ay}_{21}\partial\text{a}_{0}\partial\text{ay}_{21}+60\text{ay}_{21}\partial\text{a}_{1}\partial\text{ay}_{21}+120\text{ay}_{22}\partial\text{a}_{2}\partial\text{ay}_{21}+(-60\text{ay}_{01}+60\text{ay}_{11}+60\text{ay}_{20})\partial\text{a}_{3}\partial\text{ay}_{21}+60\text{xy}_{01}\partial\text{x}_{1}\partial\text{ay}_{21}+120\text{xy}_{11}\partial\text{x}_{2}\partial\text{ay}_{21}+180\text{xy}_{21}\partial\text{x}_{3}\partial\text{ay}_{21}+120\text{yy}_{12}\partial\text{y}_{1}\partial\text{ay}_{21}+60\text{yy}_{13}\partial\text{y}_{2}\partial\text{ay}_{21}+180\text{ay}_{22}\partial\text{a}_{0}\partial\text{ay}_{22}+60\text{ay}_{23}\partial\text{a}_{2}\partial\text{ay}_{22}+(-60\text{ay}_{02}+60\text{ay}_{12}+120\text{ay}_{21})\partial\text{a}_{3}\partial\text{ay}_{22}+60\text{xy}_{02}\partial\text{x}_{1}\partial\text{ay}_{22}+120\text{xy}_{12}\partial\text{x}_{2}\partial\text{ay}_{22}+180\text{xy}_{22}\partial\text{x}_{3}\partial\text{ay}_{22}-180\text{yy}_{12}\partial\text{y}_{0}\partial\text{ay}_{22}+60\text{yy}_{23}\partial\text{y}_{2}\partial\text{ay}_{22}+240\text{ay}_{23}\partial\text{a}_{0}\partial\text{ay}_{23}-60\text{ay}_{23}\partial\text{a}_{1}\partial\text{ay}_{23}+(-60\text{ay}_{03}+60\text{ay}_{13}+180\text{ay}_{22})\partial\text{a}_{3}\partial\text{ay}_{23}+60\text{xy}_{03}\partial\text{x}_{1}\partial\text{ay}_{23}+120\text{xy}_{13}\partial\text{x}_{2}\partial\text{ay}_{23}+180\text{xy}_{23}\partial\text{x}_{3}\partial\text{ay}_{23}-180\text{yy}_{13}\partial\text{y}_{0}\partial\text{ay}_{23}-120\text{yy}_{23}\partial\text{y}_{1}\partial\text{ay}_{23}-60\text{ay}_{30}\partial\text{a}_{0}\partial\text{ay}_{30}+240\text{ay}_{30}\partial\text{a}_{1}\partial\text{ay}_{30}+(60\text{ay}_{00}-60\text{ay}_{10}+180\text{ay}_{31})\partial\text{a}_{2}\partial\text{ay}_{30}+180\text{xy}_{10}\partial\text{x}_{0}\partial\text{ay}_{30}+120\text{xy}_{20}\partial\text{x}_{1}\partial\text{ay}_{30}+60\text{xy}_{30}\partial\text{x}_{2}\partial\text{ay}_{30}+120\text{yy}_{01}\partial\text{y}_{2}\partial\text{ay}_{30}+180\text{yy}_{02}\partial\text{y}_{3}\partial\text{ay}_{30}+180\text{ay}_{31}\partial\text{a}_{1}\partial\text{ay}_{31}+(60\text{ay}_{01}-60\text{ay}_{11}+120\text{ay}_{32})\partial\text{a}_{2}\partial\text{ay}_{31}+60\text{ay}_{30}\partial\text{a}_{3}\partial\text{ay}_{31}+180\text{xy}_{11}\partial\text{x}_{0}\partial\text{ay}_{31}+120\text{xy}_{21}\partial\text{x}_{1}\partial\text{ay}_{31}+60\text{xy}_{31}\partial\text{x}_{2}\partial\text{ay}_{31}-60\text{yy}_{01}\partial\text{y}_{1}\partial\text{ay}_{31}+180\text{yy}_{12}\partial\text{y}_{3}\partial\text{ay}_{31}+60\text{ay}_{32}\partial\text{a}_{0}\partial\text{ay}_{32}+120\text{ay}_{32}\partial\text{a}_{1}\partial\text{ay}_{32}+(60\text{ay}_{02}-60\text{ay}_{12}+60\text{ay}_{33})\partial\text{a}_{2}\partial\text{ay}_{32}+120\text{ay}_{31}\partial\text{a}_{3}\partial\text{ay}_{32}+180\text{xy}_{12}\partial\text{x}_{0}\partial\text{ay}_{32}+120\text{xy}_{22}\partial\text{x}_{1}\partial\text{ay}_{32}+60\text{xy}_{32}\partial\text{x}_{2}\partial\text{ay}_{32}-60\text{yy}_{02}\partial\text{y}_{1}\partial\text{ay}_{32}-120\text{yy}_{12}\partial\text{y}_{2}\partial\text{ay}_{32}+120\text{ay}_{33}\partial\text{a}_{0}\partial\text{ay}_{33}+60\text{ay}_{33}\partial\text{a}_{1}\partial\text{ay}_{33}+(60\text{ay}_{03}-60\text{ay}_{13})\partial\text{a}_{2}\partial\text{ay}_{33}+180\text{ay}_{32}\partial\text{a}_{3}\partial\text{ay}_{33}+180\text{xy}_{13}\partial\text{x}_{0}\partial\text{ay}_{33}+120\text{xy}_{23}\partial\text{x}_{1}\partial\text{ay}_{33}+60\text{xy}_{33}\partial\text{x}_{2}\partial\text{ay}_{33}-60\text{yy}_{03}\partial\text{y}_{1}\partial\text{ay}_{33}-120\text{yy}_{13}\partial\text{y}_{2}\partial\text{ay}_{33}-180\text{yy}_{23}\partial\text{y}_{3}\partial\text{ay}_{33}-60\text{xx}_{01}\partial\text{a}_{0}\partial\text{xx}_{01}-300\text{xx}_{01}\partial\text{a}_{1}\partial\text{xx}_{01}-120\text{xx}_{02}\partial\text{a}_{3}\partial\text{xx}_{01}+(-240\text{a}_{3}\text{x}_{0}+1380\text{a}_{0}\text{x}_{1}+240\text{a}_{1}\text{x}_{1}-1380\text{a}_{2}\text{x}_{2})\partial\text{y}_{0}\partial\text{xx}_{01}+(-315\text{a}_{0}\text{x}_{0}+315\text{a}_{2}\text{x}_{1})\partial\text{y}_{1}\partial\text{xx}_{01}-120\text{xx}_{02}\partial\text{a}_{0}\partial\text{xx}_{02}-240\text{xx}_{02}\partial\text{a}_{1}\partial\text{xx}_{02}-120\text{xx}_{01}\partial\text{a}_{2}\partial\text{xx}_{02}+(-60\text{xx}_{03}-180\text{xx}_{12})\partial\text{a}_{3}\partial\text{xx}_{02}+(-960\text{a}_{3}\text{x}_{1}+660\text{a}_{0}\text{x}_{2}+960\text{a}_{1}\text{x}_{2}-660\text{a}_{2}\text{x}_{3})\partial\text{y}_{0}\partial\text{xx}_{02}+(-120\text{a}_{3}\text{x}_{0}-120\text{a}_{0}\text{x}_{1}+120\text{a}_{1}\text{x}_{1}+120\text{a}_{2}\text{x}_{2})\partial\text{y}_{1}\partial\text{xx}_{02}+(900\text{a}_{0}\text{x}_{0}-900\text{a}_{2}\text{x}_{1})\partial\text{y}_{2}\partial\text{xx}_{02}-180\text{xx}_{03}\partial\text{a}_{0}\partial\text{xx}_{03}-180\text{xx}_{03}\partial\text{a}_{1}\partial\text{xx}_{03}-180\text{xx}_{02}\partial\text{a}_{2}\partial\text{xx}_{03}-180\text{xx}_{13}\partial\text{a}_{3}\partial\text{xx}_{03}+(-360\text{a}_{3}\text{x}_{2}+360\text{a}_{1}\text{x}_{3})\partial\text{y}_{0}\partial\text{xx}_{03}+(-540\text{a}_{3}\text{x}_{1}+540\text{a}_{1}\text{x}_{2})\partial\text{y}_{1}\partial\text{xx}_{03}+(-420\text{a}_{3}\text{x}_{0}+1020\text{a}_{0}\text{x}_{1}+420\text{a}_{1}\text{x}_{1}-1020\text{a}_{2}\text{x}_{2})\partial\text{y}_{2}\partial\text{xx}_{03}+(1260\text{a}_{0}\text{x}_{0}-1260\text{a}_{2}\text{x}_{1})\partial\text{y}_{3}\partial\text{xx}_{03}-180\text{xx}_{12}\partial\text{a}_{0}\partial\text{xx}_{12}-180\text{xx}_{12}\partial\text{a}_{1}\partial\text{xx}_{12}-60\text{xx}_{02}\partial\text{a}_{2}\partial\text{xx}_{12}-60\text{xx}_{13}\partial\text{a}_{3}\partial\text{xx}_{12}+(-360\text{a}_{3}\text{x}_{1}+90\text{a}_{0}\text{x}_{2}+360\text{a}_{1}\text{x}_{2}-90\text{a}_{2}\text{x}_{3})\partial\text{y}_{1}\partial\text{xx}_{12}+(-60\text{a}_{3}\text{x}_{0}+660\text{a}_{0}\text{x}_{1}+60\text{a}_{1}\text{x}_{1}-660\text{a}_{2}\text{x}_{2})\partial\text{y}_{2}\partial\text{xx}_{12}+(1440\text{a}_{0}\text{x}_{0}-1440\text{a}_{2}\text{x}_{1})\partial\text{y}_{3}\partial\text{xx}_{12}-240\text{xx}_{13}\partial\text{a}_{0}\partial\text{xx}_{13}-120\text{xx}_{13}\partial\text{a}_{1}\partial\text{xx}_{13}+(-60\text{xx}_{03}-180\text{xx}_{12})\partial\text{a}_{2}\partial\text{xx}_{13}-120\text{xx}_{23}\partial\text{a}_{3}\partial\text{xx}_{13}+(-1260\text{a}_{3}\text{x}_{2}+1260\text{a}_{1}\text{x}_{3})\partial\text{y}_{1}\partial\text{xx}_{13}+(-420\text{a}_{3}\text{x}_{1}+660\text{a}_{0}\text{x}_{2}+420\text{a}_{1}\text{x}_{2}-660\text{a}_{2}\text{x}_{3})\partial\text{y}_{2}\partial\text{xx}_{13}+(300\text{a}_{3}\text{x}_{0}+840\text{a}_{0}\text{x}_{1}-300\text{a}_{1}\text{x}_{1}-840\text{a}_{2}\text{x}_{2})\partial\text{y}_{3}\partial\text{xx}_{13}-300\text{xx}_{23}\partial\text{a}_{0}\partial\text{xx}_{23}-60\text{xx}_{23}\partial\text{a}_{1}\partial\text{xx}_{23}-120\text{xx}_{13}\partial\text{a}_{2}\partial\text{xx}_{23}+(-1440\text{a}_{3}\text{x}_{2}+1440\text{a}_{1}\text{x}_{3})\partial\text{y}_{2}\partial\text{xx}_{23}+(540\text{a}_{0}\text{x}_{2}-540\text{a}_{2}\text{x}_{3})\partial\text{y}_{3}\partial\text{xx}_{23}+60\text{yy}_{01}\partial\text{a}_{0}\partial\text{yy}_{01}+300\text{yy}_{01}\partial\text{a}_{1}\partial\text{yy}_{01}+120\text{yy}_{02}\partial\text{a}_{2}\partial\text{yy}_{01}+(1740\text{a}_{2}\text{y}_{0}-2220\text{a}_{0}\text{y}_{1}-1740\text{a}_{1}\text{y}_{1}+2220\text{a}_{3}\text{y}_{2})\partial\text{x}_{0}\partial\text{yy}_{01}+(1080\text{a}_{0}\text{y}_{0}-1080\text{a}_{3}\text{y}_{1})\partial\text{x}_{1}\partial\text{yy}_{01}+120\text{yy}_{02}\partial\text{a}_{0}\partial\text{yy}_{02}+240\text{yy}_{02}\partial\text{a}_{1}\partial\text{yy}_{02}+(60\text{yy}_{03}+180\text{yy}_{12})\partial\text{a}_{2}\partial\text{yy}_{02}+120\text{yy}_{01}\partial\text{a}_{3}\partial\text{yy}_{02}+(-3900\text{a}_{2}\text{y}_{1}+2220\text{a}_{0}\text{y}_{2}+3900\text{a}_{1}\text{y}_{2}-2220\text{a}_{3}\text{y}_{3})\partial\text{x}_{0}\partial\text{yy}_{02}+(780\text{a}_{2}\text{y}_{0}+1500\text{a}_{0}\text{y}_{1}-780\text{a}_{1}\text{y}_{1}-1500\text{a}_{3}\text{y}_{2})\partial\text{x}_{1}\partial\text{yy}_{02}+(2700\text{a}_{0}\text{y}_{0}-2700\text{a}_{3}\text{y}_{1})\partial\text{x}_{2}\partial\text{yy}_{02}+180\text{yy}_{03}\partial\text{a}_{0}\partial\text{yy}_{03}+180\text{yy}_{03}\partial\text{a}_{1}\partial\text{yy}_{03}+180\text{yy}_{13}\partial\text{a}_{2}\partial\text{yy}_{03}+180\text{yy}_{02}\partial\text{a}_{3}\partial\text{yy}_{03}+(-2040\text{a}_{2}\text{y}_{2}+2040\text{a}_{1}\text{y}_{3})\partial\text{x}_{0}\partial\text{yy}_{03}+(-900\text{a}_{2}\text{y}_{1}+1440\text{a}_{0}\text{y}_{2}+900\text{a}_{1}\text{y}_{2}-1440\text{a}_{3}\text{y}_{3})\partial\text{x}_{1}\partial\text{yy}_{03}+(-140\text{a}_{2}\text{y}_{0}+1720\text{a}_{0}\text{y}_{1}+140\text{a}_{1}\text{y}_{1}-1720\text{a}_{3}\text{y}_{2})\partial\text{x}_{2}\partial\text{yy}_{03}+(1860\text{a}_{0}\text{y}_{0}-1860\text{a}_{3}\text{y}_{1})\partial\text{x}_{3}\partial\text{yy}_{03}+180\text{yy}_{12}\partial\text{a}_{0}\partial\text{yy}_{12}+180\text{yy}_{12}\partial\text{a}_{1}\partial\text{yy}_{12}+60\text{yy}_{13}\partial\text{a}_{2}\partial\text{yy}_{12}+60\text{yy}_{02}\partial\text{a}_{3}\partial\text{yy}_{12}+(-1800\text{a}_{2}\text{y}_{2}+1800\text{a}_{1}\text{y}_{3})\partial\text{x}_{0}\partial\text{yy}_{12}+(-1440\text{a}_{2}\text{y}_{1}+1080\text{a}_{0}\text{y}_{2}+1440\text{a}_{1}\text{y}_{2}-1080\text{a}_{3}\text{y}_{3})\partial\text{x}_{1}\partial\text{yy}_{12}+(240\text{a}_{2}\text{y}_{0}+780\text{a}_{0}\text{y}_{1}-240\text{a}_{1}\text{y}_{1}-780\text{a}_{3}\text{y}_{2})\partial\text{x}_{2}\partial\text{yy}_{12}+(3960\text{a}_{0}\text{y}_{0}-3960\text{a}_{3}\text{y}_{1})\partial\text{x}_{3}\partial\text{yy}_{12}+240\text{yy}_{13}\partial\text{a}_{0}\partial\text{yy}_{13}+120\text{yy}_{13}\partial\text{a}_{1}\partial\text{yy}_{13}+120\text{yy}_{23}\partial\text{a}_{2}\partial\text{yy}_{13}+(60\text{yy}_{03}+180\text{yy}_{12})\partial\text{a}_{3}\partial\text{yy}_{13}+(-2250\text{a}_{2}\text{y}_{2}+2250\text{a}_{1}\text{y}_{3})\partial\text{x}_{1}\partial\text{yy}_{13}+(-900\text{a}_{2}\text{y}_{1}+360\text{a}_{0}\text{y}_{2}+900\text{a}_{1}\text{y}_{2}-360\text{a}_{3}\text{y}_{3})\partial\text{x}_{2}\partial\text{yy}_{13}+(2130\text{a}_{2}\text{y}_{0}+1410\text{a}_{0}\text{y}_{1}-2130\text{a}_{1}\text{y}_{1}-1410\text{a}_{3}\text{y}_{2})\partial\text{x}_{3}\partial\text{yy}_{13}+300\text{yy}_{23}\partial\text{a}_{0}\partial\text{yy}_{23}+60\text{yy}_{23}\partial\text{a}_{1}\partial\text{yy}_{23}+120\text{yy}_{13}\partial\text{a}_{3}\partial\text{yy}_{23}+(-7020\text{a}_{2}\text{y}_{2}+7020\text{a}_{1}\text{y}_{3})\partial\text{x}_{2}\partial\text{yy}_{23}+(2460\text{a}_{2}\text{y}_{1}-420\text{a}_{0}\text{y}_{2}-2460\text{a}_{1}\text{y}_{2}+420\text{a}_{3}\text{y}_{3})\partial\text{x}_{3}\partial\text{yy}_{23}+180\text{xy}_{01}\partial\text{a}_{2}\partial\text{xy}_{00}-180\text{xy}_{10}\partial\text{a}_{3}\partial\text{xy}_{00}+(1440\text{a}_{0}\text{x}_{0}-1440\text{a}_{2}\text{x}_{1})\partial\text{x}_{0}\partial\text{xy}_{00}+(300\text{a}_{3}\text{x}_{0}+480\text{a}_{0}\text{x}_{1}-300\text{a}_{1}\text{x}_{1}-480\text{a}_{2}\text{x}_{2})\partial\text{x}_{1}\partial\text{xy}_{00}+(-240\text{a}_{3}\text{x}_{1}+300\text{a}_{0}\text{x}_{2}+240\text{a}_{1}\text{x}_{2}-300\text{a}_{2}\text{x}_{3})\partial\text{x}_{2}\partial\text{xy}_{00}+(-900\text{a}_{3}\text{x}_{2}+900\text{a}_{1}\text{x}_{3})\partial\text{x}_{3}\partial\text{xy}_{00}+(1080\text{a}_{0}\text{y}_{0}-1080\text{a}_{3}\text{y}_{1})\partial\text{y}_{0}\partial\text{xy}_{00}+(240\text{a}_{2}\text{y}_{0}+420\text{a}_{0}\text{y}_{1}-240\text{a}_{1}\text{y}_{1}-420\text{a}_{3}\text{y}_{2})\partial\text{y}_{1}\partial\text{xy}_{00}+(-300\text{a}_{2}\text{y}_{1}+960\text{a}_{0}\text{y}_{2}+300\text{a}_{1}\text{y}_{2}-960\text{a}_{3}\text{y}_{3})\partial\text{y}_{2}\partial\text{xy}_{00}+(-1620\text{a}_{2}\text{y}_{2}+1620\text{a}_{1}\text{y}_{3})\partial\text{y}_{3}\partial\text{xy}_{00}+60\text{xy}_{01}\partial\text{a}_{0}\partial\text{xy}_{01}-60\text{xy}_{01}\partial\text{a}_{1}\partial\text{xy}_{01}+120\text{xy}_{02}\partial\text{a}_{2}\partial\text{xy}_{01}+(60\text{xy}_{00}-180\text{xy}_{11})\partial\text{a}_{3}\partial\text{xy}_{01}+(1620\text{a}_{0}\text{x}_{0}-1620\text{a}_{2}\text{x}_{1})\partial\text{x}_{1}\partial\text{xy}_{01}+(60\text{a}_{3}\text{x}_{0}+420\text{a}_{0}\text{x}_{1}-60\text{a}_{1}\text{x}_{1}-420\text{a}_{2}\text{x}_{2})\partial\text{x}_{2}\partial\text{xy}_{01}+(-540\text{a}_{3}\text{x}_{1}+1080\text{a}_{0}\text{x}_{2}+540\text{a}_{1}\text{x}_{2}-1080\text{a}_{2}\text{x}_{3})\partial\text{x}_{3}\partial\text{xy}_{01}+(-480\text{a}_{2}\text{y}_{0}+960\text{a}_{0}\text{y}_{1}+480\text{a}_{1}\text{y}_{1}-960\text{a}_{3}\text{y}_{2})\partial\text{y}_{0}\partial\text{xy}_{01}+(-180\text{a}_{2}\text{y}_{1}+720\text{a}_{0}\text{y}_{2}+180\text{a}_{1}\text{y}_{2}-720\text{a}_{3}\text{y}_{3})\partial\text{y}_{1}\partial\text{xy}_{01}+(-720\text{a}_{2}\text{y}_{2}+720\text{a}_{1}\text{y}_{3})\partial\text{y}_{2}\partial\text{xy}_{01}+120\text{xy}_{02}\partial\text{a}_{0}\partial\text{xy}_{02}-120\text{xy}_{02}\partial\text{a}_{1}\partial\text{xy}_{02}+60\text{xy}_{03}\partial\text{a}_{2}\partial\text{xy}_{02}+(120\text{xy}_{01}-180\text{xy}_{12})\partial\text{a}_{3}\partial\text{xy}_{02}+(360\text{a}_{0}\text{x}_{0}-360\text{a}_{2}\text{x}_{1})\partial\text{x}_{2}\partial\text{xy}_{02}+(120\text{a}_{3}\text{x}_{0}+1200\text{a}_{0}\text{x}_{1}-120\text{a}_{1}\text{x}_{1}-1200\text{a}_{2}\text{x}_{2})\partial\text{x}_{3}\partial\text{xy}_{02}+(-720\text{a}_{2}\text{y}_{1}+360\text{a}_{0}\text{y}_{2}+720\text{a}_{1}\text{y}_{2}-360\text{a}_{3}\text{y}_{3})\partial\text{y}_{0}\partial\text{xy}_{02}+(-1440\text{a}_{2}\text{y}_{2}+1440\text{a}_{1}\text{y}_{3})\partial\text{y}_{1}\partial\text{xy}_{02}+180\text{xy}_{03}\partial\text{a}_{0}\partial\text{xy}_{03}-180\text{xy}_{03}\partial\text{a}_{1}\partial\text{xy}_{03}+(180\text{xy}_{02}-180\text{xy}_{13})\partial\text{a}_{3}\partial\text{xy}_{03}+(900\text{a}_{0}\text{x}_{0}-900\text{a}_{2}\text{x}_{1})\partial\text{x}_{3}\partial\text{xy}_{03}+(-900\text{a}_{2}\text{y}_{2}+900\text{a}_{1}\text{y}_{3})\partial\text{y}_{0}\partial\text{xy}_{03}-60\text{xy}_{10}\partial\text{a}_{0}\partial\text{xy}_{10}+60\text{xy}_{10}\partial\text{a}_{1}\partial\text{xy}_{10}+(-60\text{xy}_{00}+180\text{xy}_{11})\partial\text{a}_{2}\partial\text{xy}_{10}-120\text{xy}_{20}\partial\text{a}_{3}\partial\text{xy}_{10}+(200\text{a}_{3}\text{x}_{0}+2780\text{a}_{0}\text{x}_{1}-200\text{a}_{1}\text{x}_{1}-2780\text{a}_{2}\text{x}_{2})\partial\text{x}_{0}\partial\text{xy}_{10}+(-360\text{a}_{3}\text{x}_{1}+360\text{a}_{0}\text{x}_{2}+360\text{a}_{1}\text{x}_{2}-360\text{a}_{2}\text{x}_{3})\partial\text{x}_{1}\partial\text{xy}_{10}+(-360\text{a}_{3}\text{x}_{2}+360\text{a}_{1}\text{x}_{3})\partial\text{x}_{2}\partial\text{xy}_{10}+(360\text{a}_{0}\text{y}_{0}-360\text{a}_{3}\text{y}_{1})\partial\text{y}_{1}\partial\text{xy}_{10}+(360\text{a}_{2}\text{y}_{0}+540\text{a}_{0}\text{y}_{1}-360\text{a}_{1}\text{y}_{1}-540\text{a}_{3}\text{y}_{2})\partial\text{y}_{2}\partial\text{xy}_{10}+(-120\text{a}_{2}\text{y}_{1}+420\text{a}_{0}\text{y}_{2}+120\text{a}_{1}\text{y}_{2}-420\text{a}_{3}\text{y}_{3})\partial\text{y}_{3}\partial\text{xy}_{10}+(-60\text{xy}_{01}+120\text{xy}_{12})\partial\text{a}_{2}\partial\text{xy}_{11}+(60\text{xy}_{10}-120\text{xy}_{21})\partial\text{a}_{3}\partial\text{xy}_{11}+(-945\text{a}_{0}\text{x}_{0}+945\text{a}_{2}\text{x}_{1})\partial\text{x}_{0}\partial\text{xy}_{11}+(-60\text{a}_{3}\text{x}_{0}+1200\text{a}_{0}\text{x}_{1}+60\text{a}_{1}\text{x}_{1}-1200\text{a}_{2}\text{x}_{2})\partial\text{x}_{1}\partial\text{xy}_{11}+(240\text{a}_{3}\text{x}_{1}+420\text{a}_{0}\text{x}_{2}-240\text{a}_{1}\text{x}_{2}-420\text{a}_{2}\text{x}_{3})\partial\text{x}_{2}\partial\text{xy}_{11}+(900\text{a}_{0}\text{y}_{0}-900\text{a}_{3}\text{y}_{1})\partial\text{y}_{0}\partial\text{xy}_{11}+(60\text{a}_{2}\text{y}_{0}+1140\text{a}_{0}\text{y}_{1}-60\text{a}_{1}\text{y}_{1}-1140\text{a}_{3}\text{y}_{2})\partial\text{y}_{1}\partial\text{xy}_{11}+(-180\text{a}_{2}\text{y}_{1}+180\text{a}_{0}\text{y}_{2}+180\text{a}_{1}\text{y}_{2}-180\text{a}_{3}\text{y}_{3})\partial\text{y}_{2}\partial\text{xy}_{11}+(-1080\text{a}_{2}\text{y}_{2}+1080\text{a}_{1}\text{y}_{3})\partial\text{y}_{3}\partial\text{xy}_{11}+60\text{xy}_{12}\partial\text{a}_{0}\partial\text{xy}_{12}-60\text{xy}_{12}\partial\text{a}_{1}\partial\text{xy}_{12}+(-60\text{xy}_{02}+60\text{xy}_{13})\partial\text{a}_{2}\partial\text{xy}_{12}+(120\text{xy}_{11}-120\text{xy}_{22})\partial\text{a}_{3}\partial\text{xy}_{12}+(720\text{a}_{0}\text{x}_{0}-720\text{a}_{2}\text{x}_{1})\partial\text{x}_{1}\partial\text{xy}_{12}+(120\text{a}_{3}\text{x}_{0}+120\text{a}_{0}\text{x}_{1}-120\text{a}_{1}\text{x}_{1}-120\text{a}_{2}\text{x}_{2})\partial\text{x}_{2}\partial\text{xy}_{12}+(-60\text{a}_{3}\text{x}_{1}+660\text{a}_{0}\text{x}_{2}+60\text{a}_{1}\text{x}_{2}-660\text{a}_{2}\text{x}_{3})\partial\text{x}_{3}\partial\text{xy}_{12}+(60\text{a}_{2}\text{y}_{0}+240\text{a}_{0}\text{y}_{1}-60\text{a}_{1}\text{y}_{1}-240\text{a}_{3}\text{y}_{2})\partial\text{y}_{0}\partial\text{xy}_{12}+(-540\text{a}_{2}\text{y}_{1}+720\text{a}_{0}\text{y}_{2}+540\text{a}_{1}\text{y}_{2}-720\text{a}_{3}\text{y}_{3})\partial\text{y}_{1}\partial\text{xy}_{12}+(-1620\text{a}_{2}\text{y}_{2}+1620\text{a}_{1}\text{y}_{3})\partial\text{y}_{2}\partial\text{xy}_{12}+120\text{xy}_{13}\partial\text{a}_{0}\partial\text{xy}_{13}-120\text{xy}_{13}\partial\text{a}_{1}\partial\text{xy}_{13}-60\text{xy}_{03}\partial\text{a}_{2}\partial\text{xy}_{13}+(180\text{xy}_{12}-120\text{xy}_{23})\partial\text{a}_{3}\partial\text{xy}_{13}+(360\text{a}_{0}\text{x}_{0}-360\text{a}_{2}\text{x}_{1})\partial\text{x}_{2}\partial\text{xy}_{13}+(120\text{a}_{3}\text{x}_{0}+120\text{a}_{0}\text{x}_{1}-120\text{a}_{1}\text{x}_{1}-120\text{a}_{2}\text{x}_{2})\partial\text{x}_{3}\partial\text{xy}_{13}+(-420\text{a}_{2}\text{y}_{1}+1020\text{a}_{0}\text{y}_{2}+420\text{a}_{1}\text{y}_{2}-1020\text{a}_{3}\text{y}_{3})\partial\text{y}_{0}\partial\text{xy}_{13}+(-720\text{a}_{2}\text{y}_{2}+720\text{a}_{1}\text{y}_{3})\partial\text{y}_{1}\partial\text{xy}_{13}-120\text{xy}_{20}\partial\text{a}_{0}\partial\text{xy}_{20}+120\text{xy}_{20}\partial\text{a}_{1}\partial\text{xy}_{20}+(-120\text{xy}_{10}+180\text{xy}_{21})\partial\text{a}_{2}\partial\text{xy}_{20}-60\text{xy}_{30}\partial\text{a}_{3}\partial\text{xy}_{20}+(-624\text{a}_{3}\text{x}_{1}+672\text{a}_{0}\text{x}_{2}+624\text{a}_{1}\text{x}_{2}-672\text{a}_{2}\text{x}_{3})\partial\text{x}_{0}\partial\text{xy}_{20}+(-144\text{a}_{3}\text{x}_{2}+144\text{a}_{1}\text{x}_{3})\partial\text{x}_{1}\partial\text{xy}_{20}+(720\text{a}_{0}\text{y}_{0}-720\text{a}_{3}\text{y}_{1})\partial\text{y}_{2}\partial\text{xy}_{20}+(60\text{a}_{2}\text{y}_{0}+240\text{a}_{0}\text{y}_{1}-60\text{a}_{1}\text{y}_{1}-240\text{a}_{3}\text{y}_{2})\partial\text{y}_{3}\partial\text{xy}_{20}-60\text{xy}_{21}\partial\text{a}_{0}\partial\text{xy}_{21}+60\text{xy}_{21}\partial\text{a}_{1}\partial\text{xy}_{21}+(-120\text{xy}_{11}+120\text{xy}_{22})\partial\text{a}_{2}\partial\text{xy}_{21}+(60\text{xy}_{20}-60\text{xy}_{31})\partial\text{a}_{3}\partial\text{xy}_{21}+(-180\text{a}_{3}\text{x}_{0}-180\text{a}_{0}\text{x}_{1}+180\text{a}_{1}\text{x}_{1}+180\text{a}_{2}\text{x}_{2})\partial\text{x}_{0}\partial\text{xy}_{21}+(-60\text{a}_{3}\text{x}_{1}+255\text{a}_{0}\text{x}_{2}+60\text{a}_{1}\text{x}_{2}-255\text{a}_{2}\text{x}_{3})\partial\text{x}_{1}\partial\text{xy}_{21}+(-1260\text{a}_{3}\text{x}_{2}+1260\text{a}_{1}\text{x}_{3})\partial\text{x}_{2}\partial\text{xy}_{21}+(1440\text{a}_{0}\text{y}_{0}-1440\text{a}_{3}\text{y}_{1})\partial\text{y}_{1}\partial\text{xy}_{21}+(180\text{a}_{0}\text{y}_{1}-180\text{a}_{3}\text{y}_{2})\partial\text{y}_{2}\partial\text{xy}_{21}+(-120\text{a}_{2}\text{y}_{1}+60\text{a}_{0}\text{y}_{2}+120\text{a}_{1}\text{y}_{2}-60\text{a}_{3}\text{y}_{3})\partial\text{y}_{3}\partial\text{xy}_{21}+(-120\text{xy}_{12}+60\text{xy}_{23})\partial\text{a}_{2}\partial\text{xy}_{22}+(120\text{xy}_{21}-60\text{xy}_{32})\partial\text{a}_{3}\partial\text{xy}_{22}+(-6300\text{a}_{0}\text{x}_{0}+6300\text{a}_{2}\text{x}_{1})\partial\text{x}_{0}\partial\text{xy}_{22}+(480\text{a}_{3}\text{x}_{0}-2580\text{a}_{0}\text{x}_{1}-480\text{a}_{1}\text{x}_{1}+2580\text{a}_{2}\text{x}_{2})\partial\text{x}_{1}\partial\text{xy}_{22}+(-540\text{a}_{3}\text{x}_{1}+720\text{a}_{0}\text{x}_{2}+540\text{a}_{1}\text{x}_{2}-720\text{a}_{2}\text{x}_{3})\partial\text{x}_{2}\partial\text{xy}_{22}+(-720\text{a}_{3}\text{x}_{2}+720\text{a}_{1}\text{x}_{3})\partial\text{x}_{3}\partial\text{xy}_{22}+(360\text{a}_{0}\text{y}_{0}-360\text{a}_{3}\text{y}_{1})\partial\text{y}_{0}\partial\text{xy}_{22}+(60\text{a}_{2}\text{y}_{0}+1140\text{a}_{0}\text{y}_{1}-60\text{a}_{1}\text{y}_{1}-1140\text{a}_{3}\text{y}_{2})\partial\text{y}_{1}\partial\text{xy}_{22}+(-180\text{a}_{2}\text{y}_{1}+540\text{a}_{0}\text{y}_{2}+180\text{a}_{1}\text{y}_{2}-540\text{a}_{3}\text{y}_{3})\partial\text{y}_{2}\partial\text{xy}_{22}+60\text{xy}_{23}\partial\text{a}_{0}\partial\text{xy}_{23}-60\text{xy}_{23}\partial\text{a}_{1}\partial\text{xy}_{23}-120\text{xy}_{13}\partial\text{a}_{2}\partial\text{xy}_{23}+(180\text{xy}_{22}-60\text{xy}_{33})\partial\text{a}_{3}\partial\text{xy}_{23}+(-900\text{a}_{0}\text{x}_{0}+900\text{a}_{2}\text{x}_{1})\partial\text{x}_{1}\partial\text{xy}_{23}+(-120\text{a}_{3}\text{x}_{0}+420\text{a}_{0}\text{x}_{1}+120\text{a}_{1}\text{x}_{1}-420\text{a}_{2}\text{x}_{2})\partial\text{x}_{2}\partial\text{xy}_{23}+(-840\text{a}_{3}\text{x}_{1}+420\text{a}_{0}\text{x}_{2}+840\text{a}_{1}\text{x}_{2}-420\text{a}_{2}\text{x}_{3})\partial\text{x}_{3}\partial\text{xy}_{23}+(-180\text{a}_{2}\text{y}_{0}+1260\text{a}_{0}\text{y}_{1}+180\text{a}_{1}\text{y}_{1}-1260\text{a}_{3}\text{y}_{2})\partial\text{y}_{0}\partial\text{xy}_{23}+(-300\text{a}_{2}\text{y}_{1}+240\text{a}_{0}\text{y}_{2}+300\text{a}_{1}\text{y}_{2}-240\text{a}_{3}\text{y}_{3})\partial\text{y}_{1}\partial\text{xy}_{23}+(-1440\text{a}_{2}\text{y}_{2}+1440\text{a}_{1}\text{y}_{3})\partial\text{y}_{2}\partial\text{xy}_{23}-180\text{xy}_{30}\partial\text{a}_{0}\partial\text{xy}_{30}+180\text{xy}_{30}\partial\text{a}_{1}\partial\text{xy}_{30}+(-180\text{xy}_{20}+180\text{xy}_{31})\partial\text{a}_{2}\partial\text{xy}_{30}+(540\text{a}_{0}\text{y}_{0}-540\text{a}_{3}\text{y}_{1})\partial\text{y}_{3}\partial\text{xy}_{30}-120\text{xy}_{31}\partial\text{a}_{0}\partial\text{xy}_{31}+120\text{xy}_{31}\partial\text{a}_{1}\partial\text{xy}_{31}+(-180\text{xy}_{21}+120\text{xy}_{32})\partial\text{a}_{2}\partial\text{xy}_{31}+60\text{xy}_{30}\partial\text{a}_{3}\partial\text{xy}_{31}+(-540\text{a}_{3}\text{x}_{1}+540\text{a}_{1}\text{x}_{2})\partial\text{x}_{0}\partial\text{xy}_{31}+(900\text{a}_{0}\text{y}_{0}-900\text{a}_{3}\text{y}_{1})\partial\text{y}_{2}\partial\text{xy}_{31}+(420\text{a}_{2}\text{y}_{0}+600\text{a}_{0}\text{y}_{1}-420\text{a}_{1}\text{y}_{1}-600\text{a}_{3}\text{y}_{2})\partial\text{y}_{3}\partial\text{xy}_{31}-60\text{xy}_{32}\partial\text{a}_{0}\partial\text{xy}_{32}+60\text{xy}_{32}\partial\text{a}_{1}\partial\text{xy}_{32}+(-180\text{xy}_{22}+60\text{xy}_{33})\partial\text{a}_{2}\partial\text{xy}_{32}+120\text{xy}_{31}\partial\text{a}_{3}\partial\text{xy}_{32}+(2220\text{a}_{3}\text{x}_{0}-4980\text{a}_{0}\text{x}_{1}-2220\text{a}_{1}\text{x}_{1}+4980\text{a}_{2}\text{x}_{2})\partial\text{x}_{0}\partial\text{xy}_{32}+(-360\text{a}_{0}\text{x}_{2}+360\text{a}_{2}\text{x}_{3})\partial\text{x}_{1}\partial\text{xy}_{32}+(-6300\text{a}_{3}\text{x}_{2}+6300\text{a}_{1}\text{x}_{3})\partial\text{x}_{2}\partial\text{xy}_{32}+(1620\text{a}_{0}\text{y}_{0}-1620\text{a}_{3}\text{y}_{1})\partial\text{y}_{1}\partial\text{xy}_{32}+(360\text{a}_{2}\text{y}_{0}+900\text{a}_{0}\text{y}_{1}-360\text{a}_{1}\text{y}_{1}-900\text{a}_{3}\text{y}_{2})\partial\text{y}_{2}\partial\text{xy}_{32}+(540\text{a}_{2}\text{y}_{1}+540\text{a}_{0}\text{y}_{2}-540\text{a}_{1}\text{y}_{2}-540\text{a}_{3}\text{y}_{3})\partial\text{y}_{3}\partial\text{xy}_{32}-180\text{xy}_{23}\partial\text{a}_{2}\partial\text{xy}_{33}+180\text{xy}_{32}\partial\text{a}_{3}\partial\text{xy}_{33}+(-1440\text{a}_{0}\text{x}_{0}+1440\text{a}_{2}\text{x}_{1})\partial\text{x}_{0}\partial\text{xy}_{33}+(192\text{a}_{3}\text{x}_{0}-672\text{a}_{0}\text{x}_{1}-192\text{a}_{1}\text{x}_{1}+672\text{a}_{2}\text{x}_{2})\partial\text{x}_{1}\partial\text{xy}_{33}+(-312\text{a}_{3}\text{x}_{1}+264\text{a}_{0}\text{x}_{2}+312\text{a}_{1}\text{x}_{2}-264\text{a}_{2}\text{x}_{3})\partial\text{x}_{2}\partial\text{xy}_{33}+(-900\text{a}_{3}\text{x}_{2}+900\text{a}_{1}\text{x}_{3})\partial\text{x}_{3}\partial\text{xy}_{33}+(900\text{a}_{0}\text{y}_{0}-900\text{a}_{3}\text{y}_{1})\partial\text{y}_{0}\partial\text{xy}_{33}+(420\text{a}_{2}\text{y}_{0}+780\text{a}_{0}\text{y}_{1}-420\text{a}_{1}\text{y}_{1}-780\text{a}_{3}\text{y}_{2})\partial\text{y}_{1}\partial\text{xy}_{33}+(240\text{a}_{2}\text{y}_{1}+420\text{a}_{0}\text{y}_{2}-240\text{a}_{1}\text{y}_{2}-420\text{a}_{3}\text{y}_{3})\partial\text{y}_{2}\partial\text{xy}_{33}+(-360\text{a}_{2}\text{y}_{2}+360\text{a}_{1}\text{y}_{3})\partial\text{y}_{3}\partial\text{xy}_{33};$

\item[$\text{i23}:$] $\text{Sbr}(\pi,\pi)$
\item[$\text{o23}:$] $0$
\end{itemize}
\end{flushleft}

\bibliographystyle{plain}

\end{document}